\documentclass[10.5pt,reqno]{amsart}
\allowdisplaybreaks[4]
\RequirePackage[OT1]{fontenc}
\RequirePackage{amsthm,amsmath}
\RequirePackage[numbers]{natbib}
\RequirePackage[colorlinks,linkcolor=blue,citecolor=blue,urlcolor=blue]{hyperref}
\usepackage{amssymb}
\usepackage{anysize}
\usepackage{hyperref}
\usepackage{extarrows}
\usepackage{multirow}
\usepackage{natbib}
\usepackage{stmaryrd}
\usepackage{color}
\usepackage{arydshln}

\numberwithin{equation}{section}
\newtheorem{thm}{Theorem}[section]
\newtheorem{lem}[thm]{Lemma}
\newtheorem{pro}[thm]{Proposition}
\newtheorem{cor}[thm]{Corollary}
\newtheorem{defi}[thm]{Definition}
\newtheorem{rem}[thm]{Remark}
\newtheorem{assu}[thm]{Assumption}
\newtheorem{Exa}{Example}[section]
\newcommand{\be}{\begin{equation}}
\newcommand{\ee}{\end{equation}}
\newcommand{\bea}{\begin{eqnarray*}}
\newcommand{\eea}{\end{eqnarray*}}

\makeatletter

\newcommand{\Rmnum}[1]{\expandafter\@slowromancap\romannumeral #1@}
\makeatother
\makeatletter 
\@addtoreset{equation}{section}
\makeatother  

\marginsize{25mm}{25mm}{25mm}{30mm}

\newcommand{\beqa}{\begin{eqnarray}}
\newcommand{\eeqa}{\end{eqnarray}}

\renewcommand{\Re}{\mathsf{Re}}
\renewcommand{\Im}{\mathsf{Im}}

\newcommand{\bu}{{\bf{u}}}

\newcommand{\R}{{\mathbb R }}

\begin{document}

\title[]{Delocalization for a class of random block band matrices}
\author{Zhigang Bao}
\author{L\'{a}szl\'{o} Erd\H{o}s}
\
\thanks{Z.G. Bao was supported by ERC Advanced Grant RANMAT No.338804;
L. Erd\H{o}s was partially supported by ERC Advanced Grant RANMAT No.338804 }

\address{IST Austria, Am Campus 1, 3400 Klosterneuberg, Austria}
\email{maomie2007@gmail.com}

\address{IST Austria, Am Campus 1, 3400 Klosterneuberg, Austria}
\email{laszlo.erdos@ist.ac.at}

\subjclass[2010]{15B52, 82B44}

\date{\today}

\keywords{Random band matrix, supersymmetry, Green's function comparison, local semicircle law, delocalization}

\maketitle

\begin{abstract}
We consider $N\times N$ Hermitian random matrices $H$ consisting of blocks of size $M\geq N^{6/7}$. The matrix elements are i.i.d. within the blocks, close to a Gaussian in the four moment matching sense, but their distribution varies from block to block to form a block-band structure, with an essential band width $M$. We show that the entries of the Green's function $G(z)=(H-z)^{-1}$ satisfy the local semicircle law with spectral parameter $z=E+\mathbf{i}\eta$ down to the real axis for any $\eta \gg N^{-1}$, using a combination of the supersymmetry method inspired by 
\cite{Sh2014} and the Green's function comparison strategy. Previous estimates were valid only for $\eta\gg M^{-1}$. The new estimate also implies that the eigenvectors in the middle of the spectrum are fully delocalized.
\end{abstract}

%

\section{Introduction} \label{s.1}
The Hamiltonian of quantum systems on a  graph $\varGamma$ is a self-adjoint matrix 
  $H= (h_{ab})_{a,b\in \varGamma}$, $H=H^*$.
The matrix elements $h_{ab}$ represent the quantum transition rates
from vertex $a$ to $b$. Disordered quantum systems have random matrix elements. We 
 assume they are centered, $\mathbb{E}h_{ab} =0$, and  independent subject to the
basic symmetry constraint $h_{ab} = \bar h_{ba}$. The variance $\sigma_{ab}^2: = \mathbb{E} |h_{ab}|^2$
represents the strength of the transition from $a$ to $b$ and we use a scaling where the norm  $\| H\|$
is typically order 1.
The simplest case is the mean field model, where $h_{ab}$ are identically distributed; this
is the standard Wigner matrix ensemble \cite{W}.  The other prominent example is the
Anderson model \cite{A}  or
random Schr\"odinger operator, $H= \Delta +V$, 
where the kinetic energy $\Delta$ is the (deterministic) graph Laplacian and
the potential $V= ( V_x)_{x\in \varGamma}$ is an on-site multiplication
operator with random multipliers. If $\varGamma$ is  a discrete $\mathsf{d}$-dimensional torus,
then only  few matrix elements $h_{ab}$  are nonzero and they connect nearest neighbor points  in the torus,
$\text{dist}(a,b)\le 1$. This is in sharp contrast to the mean field character of the Wigner matrices.

Random band matrices naturally interpolate between the mean field Wigner matrices and the short range
Anderson model. They are characterized by a parameter $M$, called the {\it band width},
such that the matrix elements $h_{ab}$ for
$\text{dist}(a,b)\ge M$ are zero or negligible.  If $M$ is comparable with the diameter $N$ of
the system then we are in the mean field regime, while $M\sim 1$ corresponds to the short range model.

The Anderson model exhibits a metal-insulator phase transition:   at high disorder
the system is in the localized (insulator) regime, while at small disorder it is in the delocalized (metallic) regime,
 at least in $\mathsf{d}\ge 3$ dimensions
and away from the spectral edges. The localized  regime is characterized by
 exponentially decaying eigenfunctions and off diagonal decay of the Green's function, while
 in the complementary regime the  eigenfunctions are supported
in the whole physical space. In terms of the  {\it localization length $\ell$},
the characteristic length scale of the decay,  the localized regime corresponds to $\ell \ll N$, 
while in the delocalized regime $\ell\sim N$. Starting from the basic papers \cite{AM,FS}, 
the localized regime is  well understood, but the delocalized regime is still an open  mathematical problem for the $\mathsf{d}$-dimensional torus.

Since the eigenvectors of the mean field Wigner matrices are always delocalized \cite{ESY2,EYY2012}, while the short range models are localized, by varying the parameter $M$
in the random band matrix,  one expects a (de)localization phase transition.  Indeed, for $\mathsf{d}=1$ it is conjectured
(and supported by non rigorous supersymmetric calculations \cite{Fy}) that the system is delocalized for broad bands,
$M\gg N^{1/2}$ and localized for $M\ll N^{1/2}$. The optimal power 1/2 has not yet been achieved from either sides.
Localization has been shown for $M\ll N^{1/8}$ in \cite{Sche}, while delocalization in a certain sense was proven for $M\gg N^{4/5}$ in 
\cite{EKYY3}.  Interestingly, for a special Gaussian model
even the sine kernel behavior of the $2$-point correlation function of the characteristic polynomials could be proven down
to the optimal band width $M\gg N^{1/2}$, see \cite{Sh2014b, Sh2014c}. Note that the sine kernel
is consistent with the delocalization but does not imply it. We remark that our discussion concerns the bulk of the spectrum; the transition at the spectral edge is much better understood. In \cite{Sodin} it was shown that the edge spectrum follows the Tracy Widom distribution, characteristic to mean field model, for $M\gg N^{5/6}$, but it yields a different distribution for narrow bands, $M\ll N^{5/6}$.

Delocalization is closely related to estimates on the diagonal elements of the resolvent $G(z)=(H-z)^{-1}$
at spectral parameters with small imaginary part $\eta =\Im z$. Indeed, if $G_{ii}(E+i\eta)$ is bounded for all $i$
and all $E\in \R$, then each $\ell^2$-normalized eigenvector $\bu$  of $H$ 
is delocalized on scale $\eta^{-1}$ in a sense that $\max_i |u_i|^2 \lesssim \eta$, i.e. $u$ is supported
on at least $\eta^{-1}$ sites. In particular, if $G_{ii}$ can be controlled down to the scale $\eta \sim 1/N$,
then the system is in the complete delocalized regime. Moreover, boundedness of $G_{ii}$ also implies that the local semicircle law holds for the same regime of $\eta$.

For band matrices with band width $M$, or even under the more general condition  $ \sigma_{ab}^2\le M^{-1}$, the boundedness of $G_{ii}$ was shown down to scale $\eta\gg M^{-1}$
in \cite{EYY2012} (see also \cite{EKYY2013b}). If $M\gg N^{1/2}$, it is expected that $G_{ii}$ remains bounded 
even  down to $\eta\gg N^{-1}$ which is the typical eigenvalue spacing, 
the smallest relevant scale in the  model. However, with the standard approach \cite{EYY2012, EKYY2013b} via
 the self-consistent equations for the Green's function does not seem to work for $\eta\le 1/M$;
 the fluctuation is hard to control. The more subtle approach using the self-consistent {\it matrix}
 equation    in \cite{EKYY3}
 could  prove delocalization and the off-diagonal Green's function profile
that are consistent with the conventional quantum diffusion picture, but it was valid only for relatively large
$\eta$, far from $M^{-1}$. 
 Moment methods, even with a delicate renormalization scheme \cite{Sod}
 could not break the barrier $\eta\sim M^{-1}$ either.

In this paper we  attack the problem differently; with supersymmetric (SUSY) techniques.
 Our main result is 
that $G_{ii}(z)$ is bounded, and the local semicircle law holds for any $\eta \gg N^{-1}$, i.e. down to the optimal scale,
if the band width is not too small, $M\gg N^{6/7}$, but
 under two technical assumptions.
First, we consider a generalization of Wegner's $n$-orbital model \cite{SW1980,Wegner1979}, namely,
we assume that the band matrix has a block structure, i.e. it
consists of $M\times M$ blocks and the matrix elements within each block
have the same distribution. This assumption is essential to reduce the number
of integration variables in the supersymmetric representation, since, roughly speaking,
each $M\times M$ block will be represented by a single supermatrix with $16$ supersymmetric variables.
Second, we assume that the distribution of the matrix elements matches a Gaussian up to four moments in the spirit of \cite{TV2011}.
Supersymmetry heavily uses Gaussian integrations,
in fact all mathematically rigorous works on random band matrices with supersymmetric method
assume that the matrix elements are Gaussian, see \cite{DPS2002, DS2010, DSZ2010, Sh2014b, Sh2014, Sh2014c, Spencer2012, SZ2004}. The Green's function comparison method \cite{EYY2012}
allows one to compare Green's functions of two matrix ensembles provided that
the distributions match up to four moments and provided that $G_{ii}$ are bounded.
This was an important motivation to reach the optimal scale $\eta \gg N^{-1}$.

In the next subsections we introduce the model precisely and state our main results.
While SUSY approach is ubiquitous  in physics, see e.g. the basic monograph by Efetov \cite{Efetov},
its application in rigorous proofs is  notoriously difficult. Initiated by T. Spencer (see \cite{Spencer2012}
for a summary) and starting with the  paper \cite{DPS2002} by Disertori, Pinsker and Spencer, only a handful of
mathematical papers have succeeded in exploiting this powerful tool. Our supersymmetric analysis was inspired by \cite{Sh2014}, but our observable, $G_{ab}$,
requires a partly different formalism, in particular we use the singular version
of the superbosonization formula \cite{BEKYZ}. Moreover, our analysis is considerably more involved
since we consider relatively narrow bands. In Section \ref{s.1.2}, we explain our novelties 
compared with \cite{Sh2014}.

\subsection{Matrix model}
Let
\begin{eqnarray*}
H_N=(h_{ab})
\end{eqnarray*}
be an $N\times N$ random Hermitian matrix, in which the entries are independent  (up to symmetry), centered, complex variables. In this paper, we are concerned with $H_N$ possessing a block band structure. To define this structure explicitly, we set the additional parameters $M\equiv M(N)$ and $W\equiv W(N)$ satisfying
\begin{eqnarray*}
W=N/M.
\end{eqnarray*}
For simplicity, we assume that both $M$ and $W$ are integers. Let $S=(\mathfrak{s}_{jk})$ be a $W\times W$ symmetric matrix, which will be chosen as a weighted Laplacian of a connected graph on $W$ vertices. Now, we decompose $H_N$ into $W\times W$ blocks of size $M\times M$, and relabel
\begin{eqnarray*}
h_{jk,\alpha\beta}:=h_{ab},\qquad  j,k=1,\ldots, W,\quad \alpha,\beta=1,\ldots,M,
\end{eqnarray*}
where $(j,k)=(\lceil a/M\rceil,\lceil b/M\rceil)$ is the index of the block containing $h_{ab}$, and 
\[(\alpha,\beta)=\Big{(}a-(j-1)M,b-(k-1)M\Big{)}\]
describes the location of the entry in the block. Moreover, we assume
\begin{eqnarray}
 \mathbb{E} h_{jk,\alpha\beta}h_{j'k',\alpha'\beta'}=\frac{1}{M}\delta_{jk'}\delta_{j'k}\delta_{\alpha\beta'}\delta_{\beta\alpha'}(\delta_{jk}+\mathfrak{s}_{jk}). \label{0127105}
\end{eqnarray}
That means, the variance profile of the random matrix $\sqrt{M}H_N$ is given by 
\begin{eqnarray}
\tilde{S}=(\tilde{\mathfrak{s}}_{jk}):=I+S, \label{021601}
\end{eqnarray} 
in which each entry represents the common variance of the entries in the corresponding block of $\sqrt{M}H_N$. Moreover, if $h_{jk,\alpha\beta}$'s are Gaussian, (\ref{0127105}) also implies that for each  off-diagonal entry $h_{jk,\alpha\beta}$, its real part and  imaginary part are i.i.d. $N(0,\tilde{\mathfrak{s}}_{jk}/2M)$ variables.
\subsection{Assumptions and main results}
 In the sequel, for some matrix $A=(a_{ij})$ and some index sets $\mathsf{I}$ and $\mathsf{J}$, we introduce the notation $A^{(\mathsf{I}|\mathsf{J})}$ to denote the submatrix  obtained by deleting the $i$-th row and $j$-th column of $A$ for all $i\in \mathsf{I}$ and $j\in \mathsf{J}$. We will adopt the abbreviation
\begin{eqnarray}
A^{(i|j)}:=A^{(\{i\}|\{j\})},\quad i\neq j,\qquad A^{(i)}:=A^{(\{i\}|\{i\})}. \label{020970}
\end{eqnarray}
In addition,  we use $||A||_{\max}:=\max_{i,j}|a_{ij}|$ to denote the max norm of $A$. Throughout the paper, we need some assumptions on $S$.
\begin{assu}[On $S$] \label{assu.1} Let $\mathcal{G}=(\mathcal{V},\mathcal{E})$ be a connected simple graph with $\mathcal{V}=\{1,\ldots, W\}$. Assume that $S$ is a $W\times W$ symmetric matrix satisfying the following four conditions. 
\begin{itemize}
\item[(i)] $S$ is a weighted Laplacian on $\mathcal{G}$, i.e. for $i\neq j$, we have $\mathfrak{s}_{ij}>0$ if $\{i,j\}\in \mathcal{E}$ and $\mathfrak{s}_{ij}=0$ if $\{i,j\}\not\in \mathcal{E}$, and for the diagonal entries, we have
\begin{eqnarray*}
\mathfrak{s}_{ii}=-\sum_{j:j\neq i}\mathfrak{s}_{ij}, \qquad \forall\; i=1,\ldots,W. 
\end{eqnarray*}
\item[(ii)] $\tilde{S}$ defined in (\ref{021601}) is strictly diagonally dominant, i.e.,  there exists some constant $c_0>0$ such that
\begin{eqnarray*}
1+2\mathfrak{s}_{ii}>c_0,\qquad \forall\; i=1,\ldots, W.
\end{eqnarray*}
\item[(iii)] For the discrete Green's functions, we assume that there exist some positive constants $C$ and $\gamma$ such that
\begin{eqnarray*}
\max_{i=1,\ldots, W}||(S^{(i)})^{-1}||_{\max}\leq CW^{\gamma}.
\end{eqnarray*}
\item[(iv)] There exists a spanning tree $\mathcal{G}_0=(\mathcal{V},\mathcal{E}_0)\subset\mathcal{G}$, on which the weights are bounded below, i.e. for some constant $c>0$, we have
\begin{eqnarray*}
\mathfrak{s}_{ij}\geq c,\qquad \text{if}\quad \{i,j\}\in \mathcal{E}_0
\end{eqnarray*}
\end{itemize}
\end{assu}
\begin{rem} From Assumption \ref{assu.1}  (ii), we easily see that 
\begin{eqnarray}
\tilde{S}\geq c_0I \label{0129120}
\end{eqnarray}
for the same positive constant $c_0$. In addition, the lower bound $c$ in (iv) can be weakened to $N^{-\varepsilon}$ for some sufficiently small constant $\varepsilon>0$. But for simplicity, we will not try to optimize this bound in this paper.
\end{rem}
\begin{Exa}\label{Exa1} Let $\Delta$ be the standard discrete Laplacian on the $\mathsf{d}$-dimensional torus $[1,\mathfrak{w}]^{\mathsf{d}}\cap \mathbb{Z}^{\mathsf{d}}$, with periodic boundary condition, where $\mathfrak{w}=W/\mathsf{d}$. Here by {\it standard} we mean the weights on the edges of the box are all $1$. Now let $S=a\Delta$ for some positive constant $a<1/4\mathsf{d}$. It is then easy to check Assumption \ref{assu.1} (i), (ii) and (iv) are satisfied. In addition, if $\mathsf{d}=1$, it is well known that we can choose $\gamma=1$ in Assumption \ref{assu.1} (iii). For $\mathsf{d}\geq 3$, one can choose $\gamma=0$. For $\mathsf{d}=2$, one can choose $\gamma=\varepsilon$ for arbitrarily small constant $\varepsilon$. For instance, one can refer to \cite{Ellis2003} for more details.
\end{Exa}
For simplicity, we also introduce the notation
\begin{eqnarray}
\sigma^2_{ij}:=\mathbb{E}|h_{ij}|^2,\qquad \mathcal{T}:=(\sigma^2_{ij})_{N,N}=\frac{1}{M}\tilde{S}\otimes \mathbf{1}_M\mathbf{1}_M',\quad i,j=1,\ldots,N, \label{030740}
\end{eqnarray}
where $\mathbf{1}_M$ is the $M$-dimensional vector whose components are all $1$ and $\tilde{S}$ is the variance matrix in (\ref{021601}).
It is elementary that
\begin{eqnarray}
\text{Spec}(\mathcal{T})=\text{Spec}(\tilde{S})\cup\{0\}\subset[0,1]. \label{0307100}
\end{eqnarray}
Our assumption on $M$ depends on the constant $\gamma$  in Assumption \ref{assu.1} (iii).
\begin{assu}[On $M$] \label{assu.2}We assume that there exists a (small) positive constant $\varepsilon_1$ such that
\begin{eqnarray}
M\geq W^{4+2\gamma+\varepsilon_1}. \label{0124101111111}
\end{eqnarray}
\end{assu}
\begin{rem} A direct consequence of (\ref{0124101111111})  and $N=MW$ is
\begin{eqnarray}
M\geq N^{\frac{4+2\gamma+\varepsilon_1}{5+2\gamma+\varepsilon_1}}. \label{0305100}
\end{eqnarray}
Especially, when $\gamma=1$, one has $M\gg N^{6/7}$. Actually, through a more involved analysis, (\ref{0124101111111}) (or (\ref{0305100})) can be further improved. At least, for $\gamma\leq 1$, we expect that $M\gg N^{4/5}$ is enough. However, we will not pursue this direction here.
\end{rem}
Besides Assumption \ref{assu.1} on the variance profile of $H$, we need to impose some additional assumption on the distribution of its entries. To this end, we temporarily employ the notation $H^g=(h^g_{ab})$ to represent a random block band matrix with Gaussian entries, satisfying (\ref{0127105}), Assumption \ref{assu.1} and Assumption \ref{assu.2}. 

\begin{assu}[On distribution] \label{assu.3}We assume 
that for each $a,b\in\{1,\ldots, N\}$, the moments of the entry $h_{ab}$ match those of $h_{ab}^g$ up to the 4th order, i.e.
\begin{eqnarray}
\mathbb{E}(\Re h_{ab})^k(\Im h_{ab})^\ell=\mathbb{E}(\Re h_{ab}^g)^k(\Im h_{ab}^g)^\ell,\quad \forall\; k,\ell\in \mathbb{N}, \quad \text{s.t.}\quad  k+\ell\leq 4. \label{031901}
\end{eqnarray} 
In addition, we assume the distribution of $h_{ab}$ possesses a subexponential tail, namely, there exist positive constants $c_1$ and $c_2$ such that for any $\tilde{\gamma}>0$,
\begin{eqnarray}
\mathbb{P}\Big{(}|h_{ab}|\geq \tilde{\gamma}^{c_1}(\mathbb{E}|h_{ab}|^2)^{\frac12}\Big{)}\leq c_2 e^{-\tilde{\gamma}} \label{021020}
\end{eqnarray}
holds uniformly for all $a,b=1,\ldots, N$.
\end{assu}
The four moment condition (\ref{031901}) in the context of random matrices first appeared in Tao and Vu's work \cite{TV2011}.

To state our results, we will  need the following notion on the comparison of two random sequences, which was introduced in \cite{EKY2013} and \cite{EKYY2013b}.
\begin{defi}[Stochastic domination] \label{defi.1}For some possibly $N$-dependent parameter set $\mathsf{U}_N$, and two families of random variables $\mathsf{X}=(\mathsf{X}_N(u): N\in\mathbb{N},u\in \mathsf{U}_{N})$ and $\mathsf{Y}=(\mathsf{Y}_N(u): N\in\mathbb{N},u\in \mathsf{U}_N)$, we say that $\mathsf{X}$ is stochastically dominated by $\mathsf{Y}$, if for all $\varepsilon>0$ and $D>0$ we have
\begin{eqnarray}
\sup_{u\in \mathsf{U}_N}\mathbb{P}\Big{(}\mathsf{X}_N(u)\geq N^\varepsilon\mathsf{Y}_N(u)\Big{)}\leq N^{-D} \label{020950}
\end{eqnarray}  
for all sufficiently large $N\geq N_0(\varepsilon, D)$. In this case we write
\begin{eqnarray*}
\mathsf{X}\prec \mathsf{Y}.
\end{eqnarray*}
\end{defi}
For example, by (\ref{0127105}) and Assumption \ref{assu.3}, we have
\begin{eqnarray}
|h_{ab}|\prec \frac{1}{\sqrt{M}},\quad \forall\; a,b=1,\ldots, N. \label{031302}
\end{eqnarray} 

 Note that $\tilde{S}$ is doubly stochastic. It is known that the empirical eigenvalue distribution of $H_N$  converges to the semicircle law, whose density function is given by
\begin{eqnarray*}
\varrho_{sc}(x):=\frac{1}{2\pi}\sqrt{4-x^2}\cdot\mathbf{1}(|x|\leq 2).
\end{eqnarray*}
We denote the Green's function of $H_N$ by
\begin{eqnarray*}
G(z)\equiv G_N(z):=(H_N-z)^{-1},\quad z=E+\mathbf{i}\eta\in \mathbb{C}^+:=\{w\in\mathbb{C}: \Im w>0\}
\end{eqnarray*}
and its $(a,b)$ matrix element is $G_{ab}(z)$.
Throughout the paper, we will always use $E$ and $\eta$ to denote the real and imaginary part of $z$ without further mention.  In addition, for simplicity, we suppress the subscript $N$ from the notation of the matrices here and there.
The Stieltjes transform of $\varrho_{sc}(x)$ is
\begin{eqnarray*}
m_{sc}(z)=\int_{-2}^2\frac{\varrho_{sc}(x)}{x-z}{\rm d}x=\frac{-z+\sqrt{z^2-4}}{2},
\end{eqnarray*}
where we chose the branch of the square root with positive imaginary part for $z\in\mathbb{C}^+$. Note that $m_{sc}(z)$ is a solution to the following self-consistent equation
\begin{eqnarray}
m_{sc}(z)=\frac{1}{-z-m_{sc}(z)}. \label{030702}
\end{eqnarray} 

The semicircle law also holds in a local sense, see Theorem 2.3 in \cite{EKYY2013b}. For simplicity, we cite this result  with a slight modification adjusted to our assumption.  
\begin{pro}[Erd\H{o}s, Knowles, Yau, Yin, \cite{EKYY2013b}]\label{pro021301} Let $H$ be a random block band matrix satisfying Assumptions \ref{assu.1}, \ref{assu.2} and \ref{assu.3}. Then
\begin{eqnarray}
\max_{a,b}|G_{ab}(z)-\delta_{ab}m_{sc}(z)|\prec\frac{1}{\sqrt{M\eta}},\quad \text{if}\quad E\in[-2+\kappa,2-\kappa]\quad \text{and}\quad M^{-1+\varepsilon}\leq \eta\leq 10 \label{0127100}
\end{eqnarray}
for any fixed small positive constants $\kappa$ and $\varepsilon$. 
\end{pro}
\begin{rem} We remark that Theorem 2.3 in \cite{EKYY2013b} was established under a more general assumption $\sum_k \sigma_{jk}^2=1$ and $\sigma_{jk}^2\leq C/M$. Especially, the block structure on the variance profile is not needed. In addition, Theorem 2.3 in \cite{EKYY2013b} also covers the edges of the spectrum, which will not be discussed in this paper. We also refer to \cite{EYY2012} for a previous result, see Theorem 2.1 therein. 
\end{rem}

Our aim in this paper is to  extend the local semicircle law to the regime $\eta\gg N^{-1}$ and replace $M$ with $N$ in (\ref{0127100}).  More specifically, we will work in the following set, defined  for  arbitrarily small constant $\kappa>0$ and any sufficiently small positive constant $ \varepsilon_2:=\varepsilon_2(\varepsilon_1)$,
\begin{eqnarray}
\mathbf{D}(N,\kappa, \varepsilon_2):=\Big{\{}z=E+\mathbf{i}\eta\in\mathbb{C}: |E|\leq \sqrt{2}-\kappa,N^{-1+\varepsilon_2}\leq \eta\leq M^{-1}N^{\varepsilon_2}\Big{\}}. \label{021101}
\end{eqnarray}
Throughout the paper, we will assume that $\varepsilon_2$ is much smaller than $\varepsilon_1$, see (\ref{0124101111111}) for the latter. Specifically, there exists some large enough constant $C$ such that $\varepsilon_2\leq \varepsilon_1/C$.

\begin{thm}[Local semicircle law]\label{thm.012801}Suppose that $H$ is a random block band matrix satisfying Assumptions \ref{assu.1}, \ref{assu.2} and \ref{assu.3}. Let $\kappa$ be an arbitrarily small positive constant  and $\varepsilon_2$ be any sufficiently small positive constant. Then
\begin{eqnarray}
\max_{a,b}|G_{ab}(z)-\delta_{ab}m_{sc}(z)|\prec \frac{1}{\sqrt{N\eta}} \label{020961}
\end{eqnarray}
holds uniformly on $\mathbf{D}(N,\kappa,\varepsilon_2)$.
\end{thm}
\begin{rem} The restriction $|E|\leq\sqrt{2}-\kappa$ in (\ref{021101}) is technical. We believe the result can be extended to the whole bulk regime of the spectrum, i.e.,  $|E|\leq 2-\kappa$, see Section \ref{s.13} for further comment. The upper bound of $\eta$ in (\ref{021101}) is also technical. However, for $\eta> M^{-1}N^{\varepsilon_2}$, one can control the Green's function by (\ref{0127100}) directly.  
\end{rem}
Let $\lambda_1,\ldots,\lambda_N$ be the eigenvalues of $H_N$. We denote by $\mathbf{u}_i:=(u_{i1},\ldots, u_{iN})$ the normalized eigenvector of $H_N$ corresponding to $\lambda_i$. From Theorem \ref{thm.012801}, we can also get the following delocalization property for the eigenvectors.
\begin{thm}[Complete delocalization] \label{thm.031301}Let $H$ be a random block band matrix satisfying Assumptions \ref{assu.1}, \ref{assu.2} and \ref{assu.3}. We have
\begin{eqnarray}
\max_{i:|\lambda_i|\leq \sqrt{2}-\kappa} ||\mathbf{u}_i||_\infty\prec N^{-\frac12}. \label{031335}
\end{eqnarray}
\end{thm}
\begin{rem} We remark that delocalization in a certain weak sense was proven in \cite{EKYY3} for an even more general class of random band matrices if $M\gg N^{4/5}$. However, Theorem \ref{thm.031301} asserts delocalization for {\bf all} eigenvectors in a very strong sense (supremum norm), while Proposition 7.1 of \cite{EKYY3} stated that {\bf most} eigenvectors are delocalized in a sense that their substantial support cannot be too small.
\end{rem}

\subsection{Outline of the proof strategy and novelties}\label{s.1.2} In this section, we briefly outline the strategy for the proof of Theorem \ref{thm.012801}. 

The first step, which is the main task of the whole proof, is to establish the following Theorem \ref{lem.012802}, namely, a prior estimate of the Green's function in the Gaussian case. 
For technical reason, we need the following slight modification of Assumption \ref{assu.2}, to state the result.
\begin{assu}[On $M$] \label{assu.4}Let $\varepsilon_1$ be the small positive constant in Assumption \ref{assu.2}. We assume
\begin{eqnarray}
N(\log N)^{-10}\geq M\geq W^{4+2\gamma+\varepsilon_1}. \label{012410}
\end{eqnarray}
\end{assu}
In the regime $M\geq N(\log N)^{-10}$, we see that (\ref{020961}) anyway follows from (\ref{0127100}) directly. 
\begin{thm} \label{lem.012802}Assume that $H$ is a Gaussian block band matrix,  satisfying Assumptions \ref{assu.1} and \ref{assu.4}. Let $n$ be any fixed positive integer. Let $\kappa$ be an arbitrarily small positive constant  and  $\varepsilon_2$ be any sufficiently small positive constant. There is $N_0=N_0(n)$, such that for all $N\geq N_0$ and all $z\in \mathbf{D}(N,\kappa,\varepsilon_2)$, we have
\begin{eqnarray}
\mathbb{E}|G_{ab}(z)|^{2n} \leq N^{C_0}\Big{(}\delta_{ab}+\frac{1}{(N\eta)^{n}}\Big{)},\quad \forall\; a,b=1,\ldots, N \label{090102}
\end{eqnarray}
for some positive constant $C_0$ independent of $n$ and $z$. 
\end{thm}
\begin{rem}\label{rem.031501} Much more delicate analysis can show that the prefactor $N^{C_0}$ can be improved to some $n$-dependent constant $C_n$. We refer to Section \ref{s.13} for further comment on this issue.
\end{rem}
Using the definition of stochastic domination in Definition \ref{defi.1}, a simple Markov inequality shows that  (\ref{090102}) implies 
\begin{eqnarray}
|G_{ab}(z)|\prec \delta_{ab}+\frac{1}{\sqrt{N\eta}}, \quad \forall\; a,b=1,\ldots, N. \label{021051}
\end{eqnarray}

The proof of Theorem \ref{lem.012802} is the main task of our paper. We will use the supersymmetry method. We partially rely on the  arguments from Shcherbina's work \cite{Sh2014} concerning universality of the local $2$-point function and
we develop new techniques to treat our observable, the high moment of the entries of $G(z)$, under a more general setting. We will comment on the novelties later in this subsection.

The second step is to generalize Theorem \ref{lem.012802} from the Gaussian case to more general distribution satisfying Assumption \ref{assu.3}, via a Green's function comparison strategy initiated in \cite{EYY2012}, see Lemma \ref{lem021256} below.

The last step is to use Lemma \ref{lem021256} and its Corollary \ref{cor.1} to prove our main theorems. Using (\ref{021051}) below to bound the error term in the self-consistent equation for the Green's function, we can prove Theorem \ref{thm.012801} by a continuity argument in $z$, with the aid of the initial estimate for large $\eta$ provided in Proposition \ref{pro021301}. Theorem \ref{thm.031301} will then easily follow from Theorem \ref{thm.012801}. 

The second and the last steps are carried out in Section \ref{s.14}. The main body of this paper, Sections \ref{s.3}--\ref{s.12} is devoted to the proof of Theorem \ref{lem.012802}.

One of the main novelty of this work is to combine the supersymmetry method and the Green's function comparison strategy to go beyond the Gaussian ensemble, which was so far the only random band matrix ensemble amenable to the supersymmetry method, as mentioned at the beginning. The comparison strategy requires an apriori control on the individual matrix elements of the Green's function with high probability (see (\ref{021051})), this is one of our main motivations behind Theorem \ref{lem.012802}.

Although we consider a different observable than \cite{Sh2014}, many technical aspects of the supersymmetric analysis overlaps with \cite{Sh2014}. For the convenience of the reader, we now briefly introduce the strategy of \cite{Sh2014}, and highlight the main novelties of our work.

In \cite{Sh2014}, the author considers the $2$-point correlation function of the trace of the resolvent of the Gaussian block band matrix $H$, with the variance profile $\tilde{S}=1+a\Delta$, under the assumption $M\sim N$ (note that we use $M$ instead of $W$ in \cite{Sh2014} for the size of the blocks). The $2$-point correlation function can be expressed in terms of a superintegral of a superfunction $F(\{\breve{\mathcal{S}}_i\}_{i=1}^W)$  with a collection of $4\times 4$ supermatrices $\breve{\mathcal{S}}_i:=\mathcal{Z}^*_i\mathcal{Z}_i$. Here for each $i$, $\mathcal{Z}_i=(\Psi_{1,i},\Psi_{2,i},\Phi_{1,i},\Phi_{2,i})$ is an $M\times 4$ matrix and $\mathcal{Z}^*_i$ is its conjugate transpose, where $\Psi_{1,i}$ and $\Psi_{2,i}$ are Grassmann $M$-vectors whilst $\Phi_{1,i}$ and $\Phi_{2,i}$ are complex $M$-vectors. Then, by using the superbosonization formula in the nonsingular case ($M\geq 4$) from \cite{LSZ}, one can transform the superintegral of  $F(\{\breve{\mathcal{S}}_i\}_{i=1}^W)$ to a superintegral of $F(\{\mathcal{S}_i\}_{i=1}^W)$, where each $\mathcal{S}_i$ is a supermatrix akin to $\breve{\mathcal{S}}_i$, but only consists of $16$ independent variables (either complex or Grassmann). We will call the integral representation of the observable after using the superbosonization formula as the {\it final integral representation}. Schematically it has the form 
\begin{eqnarray}
\int \mathsf{g}(\mathcal{S}_c)e^{M\mathsf{f}_c(\mathcal{S}_c)+\mathsf{f}_g(\mathcal{S}_g, \mathcal{S}_c)}{\rm d}\mathcal{S}, \label{031801}
\end{eqnarray} 
for some functions $\mathsf{g}(\cdot)$, $\mathsf{f}_c(\cdot)$ and $\mathsf{f}_g(\cdot)$, where we used the abbreviation $\mathcal{S}:=\{\mathcal{S}_i\}_{i=1}^W$ and $\mathcal{S}_c$ and $\mathcal{S}_g$ represents the collection of all complex variables and Grassmann variables in $\mathcal{S}$, respectively. Here,  $\mathsf{g}(\mathcal{S}_c)$ and $\mathsf{f}_c(\mathcal{S}_c)$ are some complex functions and $\mathsf{f}_g(\mathcal{S}_g, \mathcal{S}_c)$ will be mostly regarded as a function of the Grassmann variables with complex variables as its parameters. The number of variables (either complex or Grassmann) in the final integral representation then turns out to be of order $W$, which is much smaller than the original order $N$. In fact, in \cite{Sh2014} it is assumed that $W=O(1)$ although the author also mentions the possibility to deal with the case $W\sim N^{\varepsilon}$ for some small positive $\varepsilon$, see the remark below Theorem 1 therein.

Performing a saddle point analysis for the complex measure $\exp\{M\mathsf{f}_c(\mathcal{S}_c)\}$, one can restrict the integral in a small vicinity of some saddle point, say, $\mathcal{S}_c=\mathcal{S}_{c0}$. It turns out that $\mathsf{f}_c(\mathcal{S}_{c0})=0$ and $\mathsf{f}_c(\mathcal{S}_c)$ decays quadratically away from $\mathcal{S}_{c0}$. Consequently, by plugging in the saddle point $\mathcal{S}_{c0}$, one can estimate $\mathsf{g}(\mathcal{S}_c)$ by $\mathsf{g}(\mathcal{S}_{c0})$ directly. However, for $\exp\{M\mathsf{f}_c(\mathcal{S}_c)\}$ and $\exp\{\mathsf{f}_g(\mathcal{S}_g, \mathcal{S}_c)\}$, one shall expand them around the saddle point. Roughly speaking, in some vicinity of $\mathcal{S}_{c0}$, one will find that the expansions read
\begin{eqnarray}
e^{M\mathsf{f}_c(\mathcal{S}_c)}=\exp\{-\mathbf{u}'\mathbb{A}\mathbf{u}+\mathsf{e}_c(\mathbf{u})\},\qquad e^{\mathsf{f}_g(\mathcal{S}_g, \mathcal{S}_c)}=\exp\{-\boldsymbol{\rho}'\mathbb{H}\boldsymbol{\tau}\}\mathsf{p}(\boldsymbol{\rho},\boldsymbol{\tau},\mathbf{u}), \label{031805}
\end{eqnarray}
where $\mathbf{u}$ is a complex vector of dimension $O(W)$, which is essentially a vectorization of $\sqrt{M}(\mathcal{S}_c-\mathcal{S}_{c0})$; $\mathsf{e}_c(\mathbf{u})=o(1)$ is some error term; $\boldsymbol{\rho}$ and $\boldsymbol{\tau}$ are two Grassmann vectors of dimension $O(W)$; $\mathbb{A}$ is a complex matrix with positive-definite Hermitian part and $\mathbb{H}$ is a complex matrix; $\mathsf{p}(\boldsymbol{\rho},\boldsymbol{\tau},\mathbf{u})$ is the expansion of $\exp\{\mathsf{f}_g(\mathcal{S}_g, \mathcal{S}_c)-\mathsf{f}_g(\mathcal{S}_g, \mathcal{S}_{c0})\}$, which possesses the form
\begin{eqnarray}
\mathsf{p}(\boldsymbol{\rho},\boldsymbol{\tau},\mathbf{u})=\sum_{\ell=0}^{O(W)} M^{-\frac{\ell}{2}}\mathsf{p}_\ell(\boldsymbol{\rho},\boldsymbol{\tau},\mathbf{u}), \label{031920}
\end{eqnarray}
where $\mathsf{p}_\ell(\boldsymbol{\rho},\boldsymbol{\tau},\mathbf{u})$ is a polynomial of the components of $\boldsymbol{\rho}$ and $\boldsymbol{\tau}$ with degree $2\ell$, regarding $\mathbf{u}$ as fixed parameters. Now, keeping the leading order term of $\mathsf{p}(\boldsymbol{\rho},\boldsymbol{\tau},\mathbf{u})$, and discarding the remainder terms, one can get the final estimate of the integral by taking the Gaussian integral over $\mathbf{u}$, $\boldsymbol{\rho}$ and $\boldsymbol{\tau}$. This completes the summary of \cite{Sh2014}. 

Similarly to \cite{Sh2014}, we also use the superbosonization formula to reduce the number of variables and perform the saddle point analysis on the resulting integral. However, owing to the following three main aspects, our analysis is significantly different from \cite{Sh2014}.\\

\noindent$\bullet$(Different observable)  Our objective is to compute high moments of the single entry of the Green's function. By using Wick's formula (see Proposition \ref{pro.012801}), we express $\mathbb{E}|G_{jk}|^{2n}$ in terms of a superintegral of some superfunction of the form 
\[\tilde{F}\Big(\{\Psi_{a,j},\Psi^*_{a,j},\Phi_{a,j}, \Phi^*_{a,j}\}_{\substack{a=1,2;\\j=1,\ldots,W}}\Big):=\big(\bar{\phi}_{1,q,\beta}\phi_{1,p,\alpha}\bar{\phi}_{2,p,\alpha}\phi_{2,q,\beta}\big)^nF(\{\breve{\mathcal{S}}_i\}_{i=1}^W)\] for some $p,q\in\{1,\ldots, W\}$ and $\alpha,\beta\in\{1,\ldots,M\}$, where $\phi_{1,p,\alpha}$ is the $\alpha$-th coordinate of $\Phi_{1,p}$, and the others are defined analogously. Unlike the case in \cite{Sh2014}, $\tilde{F}$ is not a function of  $\{\breve{\mathcal{S}}_i\}_{i=1}^W$ only. Hence, using the superbosonization formula to change $\breve{\mathcal{S}}_i$ to $\mathcal{S}_i$ directly is not feasible in our case. In order to handle the factor $\big(\bar{\phi}_{1,q,\beta}\phi_{1,p,\alpha}\bar{\phi}_{2,p,\alpha}\phi_{2,q,\beta}\big)^n$, the main idea is to split off certain rank-one supermatrices from $\breve{\mathcal{S}}_p$ and $\breve{\mathcal{S}}_q$ such that this factor can be expressed in terms of the entries of these rank-one supermatrices. Then we use the superbosonization formula not only in the nonsingular case from \cite{LSZ} but also in the singular case from \cite{BEKYZ} to change and reduce the variables, resulting the final integral representation of $\mathbb{E}|G_{jk}|^{2n}$. Though this final integral representation, very schematically, is still of the form (\ref{031801}), due to the decomposition of the supermatrices $\breve{\mathcal{S}}_p$ and $\breve{\mathcal{S}}_q$,  it is considerably more complicated than its counterpart in \cite{Sh2014}. Especially, the function $\mathsf{g}(\mathcal{S}_c)$ differs from its counterpart in \cite{Sh2014}, and its estimate at the saddle point follows from a different argument. \\

\noindent$\bullet$(Small band width) In \cite{Sh2014}, the author considers the case that the band width $M$ is comparable with $N$, i.e. the number of blocks $W$ is finite.  Though the derivation of the $2$-point correlation function is highly nontrivial even with such a large band width, our objective, the local semicircle law and delocalization of the eigenvectors, however, can be proved for the case $M\sim N$ in  a similar manner as for the Wigner matrix ($M=N$), see \cite{EKYY2013b,EYY2012}. In our work, we will work with much smaller band width to go beyond the results in \cite{EKYY2013b,EYY2012}, see Assumption \ref{assu.2}. Several main difficulties stemming from a narrow band width can be heuristically explained as follows. 

At first, let us focus on the integral over the small vicinity of the saddle point, in which the exponential functions in the integrand in (\ref{031801}) approximately look like (\ref{031805}). 

We regard the first term in (\ref{031805}) as a complex Gaussian measure, of dimension $O(W)$. When $W\sim 1$, one can discard the error term $\mathsf{e}_c(\mathbf{u})$ directly and perform the Gaussian integral over $\mathbf{u}$, due to the fact $\int {\rm d}\mathbf{u}\exp\{-\mathbf{u}'\Re (\mathbb{A})\mathbf{u}\}|\mathsf{e}_c(\mathbf{u})|=o(1)$. However, such an estimate is not allowed when $W\sim N^{\varepsilon}$ (say), because the normalization of the measure $\exp\{-\mathbf{u}'\Re(\mathbb{A})\mathbf{u}\}$ might be exponentially larger than that of $\exp\{-\mathbf{u}'\mathbb{A}\mathbf{u}\}$. In order to handle this issue, we shall do a second deformation of the contours of the complex variables in the vicinity of the saddle, following the steepest descent paths exactly, whereby we can transform the complex Gaussian measure to a real one, thus the error term of the integral can be controlled. 

Now, we turn to the second term in (\ref{031805}). When $W\sim 1$, there are only finitely many Grassmann variables. Hence, the complex coefficient of each term in the polynomial  $\mathsf{p}(\boldsymbol{\rho}, \boldsymbol{\tau},\mathbf{u})$, which is of order $M^{-\ell/2}$ for some $\ell\in\mathbb{N}$ (see (\ref{031920})), actually controls the magnitude of the integral of this term against the Gaussian measure $\exp\{-\boldsymbol{\rho}'\mathbb{H}\boldsymbol{\tau}\}$. Consequently, in case of $W\sim 1$, it suffices to keep the leading order term (according to $M^{-\ell/2}$), one may discard the others trivially, and compute the Gaussian integral over $\boldsymbol{\rho}$ and $\boldsymbol{\tau}$ explicitly. However, when $W\sim N^{\varepsilon}$ (say), in light of the Wick's formula (\ref{0129102}) and the fact that the coefficients are of order $M^{-\ell/2}$, the order of the integral of each term of $\mathsf{p}(\boldsymbol{\rho}, \boldsymbol{\tau},\mathbf{u})$ against the Gaussian measure reads $M^{-\ell/2}\det\mathbb{H}^{(\mathsf{I}|\mathsf{J})}$ for some index sets $\mathsf{I}$ and $\mathsf{J}$ and some $\ell\in\mathbb{N}$. Due to the fact $W\sim N^{\varepsilon}$, $\det\mathbb{H}^{(\mathsf{I}|\mathsf{J})}$ is typically exponential in $W$. Hence, it is much more complicated to determine and compare the orders of the integrals of all $e^{O(W)}$ terms. In our discussion, we perform a unified estimate for the integrals of all the terms, rather than simply compare them by $M^{-\ell/2}$.

In addition, the analysis for the integral away from the vicinity of the saddle point in our work is also quite different from \cite{Sh2014}.  Actually, the integral over the complement of the vicinity can be trivially ignored in \cite{Sh2014}, since each factor in the integrand of (\ref{031801}) is of order $1$, thus gaining any $o(1)$ factor for the integrand outside the vicinity is enough for the estimate. However, in our case, either $\exp\{M\mathsf{f}_c(\mathcal{S}_c)\}$ or $\int {\rm d} \mathcal{S}_g \exp\{\mathsf{f}_g(\mathcal{S}_g, \mathcal{S}_c)\}$ is essentially exponential in $W$.  This fact forces us to provide an apriori bound for $\int {\rm d} \mathcal{S}_g \exp\{\mathsf{f}_g(\mathcal{S}_g, \mathcal{S}_c)\}$ in the full domain of $\mathcal{S}_c$ rather than in the vicinity of the saddle point only. In addition, an analysis of the tail behavior of the measure $\exp\{M\mathsf{f}_c(\mathcal{S}_c)\}$ needs also to be performed. \\

\noindent$\bullet$(General variance profile $\tilde{S}$)
In \cite{Sh2014}, the authors considered the special case $S=a\Delta$ with $a<1/4\mathsf{d}$. We generalize the discussion to more general weighted Laplacians $S$ satisfying Assumption \ref{assu.1}, which, as a special case, includes the standard Laplacian $\Delta$ for any fixed dimension $\mathsf{d}$.

\subsection{Notation and organization} \label{s.1.3} Throughout the paper, we will need some notation. At first, we conventionally use $U(r)$ to denote the unitary group of degree $r$, as well, $U(1,1)$ represents the $U(1,1)$ group. Furthermore, we denote 
\begin{eqnarray*}
\mathring{U}(r)=U(r)/U(1)^r,\qquad \mathring{U}(1,1)=U(1,1)/U(1)^2.
\end{eqnarray*}
Recalling the real part $E$ of $z$, we will frequently need the following two parameters
\begin{eqnarray*}
a_+=\frac{\mathbf{i}E+\sqrt{4-E^2}}{2},\qquad a_-=\frac{\mathbf{i}E-\sqrt{4-E^2}}{2}.
\end{eqnarray*}  
Correspondingly, we define the following four matrices
\begin{eqnarray}
D_{\pm}=\text{diag}(a_+,a_-), \quad D_{\mp}=\text{diag}(a_-,a_+),\quad D_{+}=\text{diag}(a_+,a_+),\quad D_{-}=\text{diag}(a_-,a_-). \label{0129101}
\end{eqnarray}
We remark here $D_\pm$ does not mean ``$D_+$ or $D_-$''. For simplicity, we introduce the following notation for some domains used throughout the paper.
\begin{eqnarray*}
\mathbb{I}:=[0,1], \quad \mathbb{L}:=[0,2\pi),\quad \Sigma:\quad  \text{unit circle}, \quad \mathbb{R}_+:=[0,\infty),\quad \mathbb{R}_-:=-\mathbb{R}_+,\quad \Gamma:=a_+\mathbb{R}_+.
\end{eqnarray*}

 For some $\ell\times \ell$ Hermitian matrix $A$, we use $\lambda_1(A)\leq\ldots\leq \lambda_\ell(A)$ to represent its ordered eigenvalues.
For some possibly $N$-dependent parameter set $\mathsf{U}_N$, and two families of complex functions  $\{a_N(u): N\in\mathbb{N}, u\in \mathsf{U}_N\}$ and $\{b_N(u): N\in\mathbb{N}, u\in \mathsf{U}_N\}$, if there exists a positive constant $C>1$ such that $C^{-1}|b_N(u)|\leq |a_N(u)|\leq C |b_N(u)|$ holds uniformly in $N$ and $u$, we write
\begin{eqnarray*}
a_N(u)\sim b_N(u).
\end{eqnarray*}
Conventionally, we use $\{\mathbf{e}_i:i=1,\ldots, \ell\}$ to denote the standard basis of $\mathbb{R}^{\ell}$, in which the dimension $\ell$ has been suppressed for simplicity. For some real quantities $a$ and $b$, we use $a\wedge b$ and $a\vee b$ to represent $\min\{a,b\}$ and $\max\{a,b\}$, respectively.

Throughout the paper, $c$, $c'$, $c_1$, $c_2$, $C$, $C'$, $C_1$, $C_2$ represent some generic positive constants that are possibly $n$-dependent and may differ from line to line. In contrast, we use $C_0$ to denote some generic positive constant independent of $n$.

The paper will be organized in the following way. In Section \ref{s.14}, we prove Theorem \ref{thm.012801} and Theorem \ref{thm.031301}, with Theorem \ref{lem.012802}. The proof of Theorem \ref{lem.012802} will be done in Section \ref{s.3}--Section \ref{s.12}.
More specifically, in Section \ref{s.3}, we use the supersymmetric formalism to represent $\mathbb{E}|G_{ij}|^{2n}$ in terms of a superintegral, in which the integrand can be factorized into several functions; Section \ref{s.5}  is devoted to a preliminary analysis on these functions; Section \ref{s.6}--Section \ref{s.10} are responsible for different steps of the saddle point analysis, whose organization will be further clarified at the end of Section \ref{s.6}; Section \ref{s.12} is devoted to the final proof of Theorem \ref{lem.012802}, by summing up the discussions in  \ref{s.3}--Section \ref{s.10}. Finally, in Section \ref{s.13}, we make some further comments on possible improvements. 

\section{Proofs of Theorem \ref{thm.012801} and Theorem \ref{thm.031301}}  \label{s.14}
At first, (\ref{090102}) can be generalized to the generally distributed matrix with the  four moment matching condition via the Green's function comparison strategy.
\begin{lem}\label{lem021256} Assume that $H$ is a random block band matrix, satisfying Assumptions \ref{assu.1}, \ref{assu.3} and \ref{assu.4}. Let $\kappa$ be an arbitrarily small positive constant  and  $\varepsilon_2$ be any sufficiently small positive constant. There is $N_0=N_0(n)$, such that for all $N\geq N_0$ and all $z\in \mathbf{D}(N,\kappa,\varepsilon_2)$, we have 
\begin{eqnarray}
\mathbb{E}|G_{ab}(z)|^{2n} \leq N^{C_0}\Big{(}\delta_{ab}+\frac{1}{(N\eta)^{n}}\Big{)},\quad \forall\; a,b=1,\ldots, N \label{09010222}
\end{eqnarray}
for some positive constant $C_0$ uniform in $n$ and $z$. 
\end{lem}
By the definition of stochastic domination in Definition \ref{defi.1}, we can get the following corollary from Lemma \ref{lem021256} immediately. 
\begin{cor} \label{cor.1}Under the assumptions of Lemma \ref{lem021256},  we have 
\begin{eqnarray}
|G_{ab}(z)|\prec \delta_{ab}+\frac{1}{\sqrt{N\eta}}, \quad \forall\; a,b=1,\ldots, N \label{032001}
\end{eqnarray}
uniformly on $\mathbf{D}(N,\kappa,\varepsilon_2)$.
\end{cor}
In the sequel, at first, we prove Lemma \ref{lem021256} from Theorem \ref{lem.012802} via the Green's function comparison strategy. Then we prove Theorem \ref{thm.012801}, using Lemma \ref{lem021256}. Finally, we will show that Theorem \ref{thm.031301} follows from Theorem \ref{thm.012801} simply.
\subsection{Green's function comparison: Proof of Lemma \ref{lem021256}} 
To show (\ref{09010222}),  we use Lindeberg's replacement strategy to compare the Green's functions of the Gaussian case and the general case. That means, we will replace the entries of $H^g$ by those of $H$ one by one, and compare the Green's functions step by step. Choose and fix a bijective ordering map 
\begin{eqnarray}
\varpi: \{(i,j): 1\leq i\leq j\leq N\}\to \Big{\{}1,\ldots, \varsigma(N)\Big{\}},\quad \varsigma(N):=\frac{N(N+1)}{2}. \label{021046}
\end{eqnarray}
Then we use $H_k$ to represent the $N\times N$ random Hermitian matrix whose $(i,j)$-th entry is $h_{ij}$ if $\varpi(i,j)\leq k$, and is $h^g_{ij}$ otherwise. Especially, we have $H_0=H^g$ and $H_{\varsigma(N)}=H$. Correspondingly, we define the Green's functions by
\begin{eqnarray*}
G_k(z):=\Big{(}H_k-z\Big{)}^{-1},\quad k=1,\ldots, \varsigma(N).
\end{eqnarray*} 
Fix $k$ and denote
\begin{eqnarray}
\varpi^{-1}(k)=(a,b). \label{021231}
\end{eqnarray}
Then,  we write
\begin{eqnarray*}
&&H_{k-1}=H_k^0+\mathsf{V}_{ab},\qquad \mathsf{V}_{ab}:=\Big(1-\frac{\delta_{ab}}{2}\Big)\big(h_{ab}^g \mathbf{e}_a\mathbf{e}_b^*+h_{ba}^g\mathbf{e}_b\mathbf{e}_a^*\big),\\
&& H_k=H_k^0+\mathsf{W}_{ab}, \qquad \mathsf{W}_{ab}:=\Big(1-\frac{\delta_{ab}}{2}\Big)\big(h_{ab} \mathbf{e}_a\mathbf{e}_b^*+h_{ba}\mathbf{e}_b\mathbf{e}_a^*\big),
\end{eqnarray*}
where $H_k^0$ is obtained via replacing $h_{ab}$ and $h_{ba}$ by $0$ in $H_k$ (or replacing $h^g_{ab}$ and $h^g_{ba}$ by $0$ in $H_{k-1}$). In addition, we denote
\begin{eqnarray*}
G_k^0(z)=(H_k^0-z)^{-1}.
\end{eqnarray*}

Set $\varepsilon_3\equiv\varepsilon_3(\gamma,\varepsilon_1)$ to be a sufficiently small positive constant, satisfying (say)
\begin{eqnarray}
\varepsilon_3\leq \frac{1}{100}\cdot \frac{\varepsilon_1}{5+2\gamma+\varepsilon_1}, \label{0305101}
\end{eqnarray}
where $\gamma$ is from Assumption \ref{assu.1} (iii) and $\varepsilon_1$ is from (\ref{0124101111111}). For simplicity, we introduce the following parameters for $\ell=1,\ldots, \varsigma(N)$ and $i,j=1,\ldots, N$,
\begin{eqnarray}
\widehat{\Theta}_0:=N^{C_0},\quad \widehat{\Theta}_{\ell,ij}:=\widehat{\Theta}_0\bigg{(}1+C\Big(\frac{N^{\varepsilon_3}}{\sqrt{M}}\Big)^{5}\bigg{)}^\ell\prod_{\varpi(a,b)\leq \ell}\bigg(1+C\delta_{\{i,j\}\{a,b\}}\Big(\frac{N^{\varepsilon_3}\sqrt{N\eta}}{\sqrt{M}}\Big)^5\bigg), \label{021250}
\end{eqnarray}
where $C$ is a positive constant. Here we used the notation $\delta_{\mathsf{I}\mathsf{J}}=1$ if two index sets $\mathsf{I}$ and $\mathsf{J}$ are the same and $\delta_{\mathsf{I}\mathsf{J}}=0$ otherwise. It is easy to see that for $\eta\leq M^{-1}N^{\varepsilon_2}$, we have
\begin{eqnarray}
\widehat{\Theta}_{\ell,ij}\leq 2 \widehat{\Theta}_0,\qquad \forall\; \ell=1,\ldots, \varsigma(N),\quad i,j=1,\ldots, N, \label{021235}
\end{eqnarray}
by using (\ref{0305100}).

Now,  we compare $G_{k-1}(z)$ and $G_k(z)$. We will prove the following lemma.
\begin{lem} \label{lem.021090}Suppose that the assumptions in Lemma \ref{lem021256} hold. Additionally, we assume that for some sufficiently small positive constant $\varepsilon_3$ satisfying (\ref{0305101}), 
\begin{eqnarray}
|(G_\ell)_{ij}(z)|\prec N^{\varepsilon_3},\quad |(G_\ell^0)_{ij}(z)|\prec N^{\varepsilon_3},\quad \forall\; \ell=1,\ldots, \varsigma(N), \quad \forall\;  i,j=1,\ldots, N\label{0212020}
\end{eqnarray}
uniformly for $z\in \mathbf{D}(N,\kappa,\varepsilon_2)$.
Let  $n\in \mathbb{N}$ be any given integer. Then, if
\begin{eqnarray}
\mathbb{E}|(G_{k-1})_{ij}(z)|^{2n}\leq \widehat{\Theta}_{k-1, ij} \Big{(}\delta_{ij}+\frac{1}{(N\eta)^{n}}\Big{)},\quad \forall\; i,j=1,\ldots, N,\label{021201}
\end{eqnarray}
we also have
\begin{eqnarray}
\mathbb{E}|(G_{k})_{ij}(z)|^{2n}\leq\widehat{\Theta}_{k,ij}\Big{(}\delta_{ij}+\frac{1}{(N\eta)^{n}}\Big{)}, \qquad\qquad \forall\; i,j=1,\ldots, N
\end{eqnarray}
for any $k=1,\ldots, \varsigma(N)$.
\end{lem} 
\begin{proof}[Proof of Lemma \ref{lem.021090}] Fix $k$ and omit the argument $z$ from now on. At first, under the conditions (\ref{0212020}) and (\ref{021201}), we show that 
\begin{eqnarray}
\mathbb{E}|(G_k^0)_{ij}|^{2n}\leq  3\widehat{\Theta}_0\Big(\delta_{ij}+\frac{1}{(N\eta)^{n}}\Big), \quad \forall\; i,j=1,\ldots, N.\label{021041}
\end{eqnarray}
To see this, we use the expansion with (\ref{021231})
\begin{eqnarray*}
(G_k^0)_{ij}=(G_{k-1})_{ij}+(G_{k-1} \mathsf{W}_{ab}G_k^0)_{ij},
\end{eqnarray*}
which implies that for a sufficiently large constant $D>0$
\begin{eqnarray*}
\mathbb{E}|(G_k^0)_{ij}|^{2n}\leq \mathbb{E}\Big||(G_{k-1})_{ij}|+\frac{N^{2\varepsilon_3}}{\sqrt{M}}\Big|^{2n}+\eta^{-2n}N^{-D}\leq 3\widehat{\Theta}_0\Big(\delta_{ij}+\frac{1}{(N\eta)^{n}}\Big),
\end{eqnarray*}
where the first step follows from (\ref{031302}), (\ref{0212020}), Definition \ref{defi.1} and the trivial bound $\eta^{-1}$ for the Green's functions, and 
the second step follows from (\ref{021201}), (\ref{021235}) and  the fact $N^{2\varepsilon_3}/\sqrt{M}\ll 1/ \sqrt{N\eta}$ for $z\in \mathbf{D}(N,\kappa,\varepsilon_2)$.

Now, recall  (\ref{021231}) again and expand $G_{k-1}(z)$ and $G_k(z)$ around $G_k^0(z)$, namely
\begin{eqnarray}
&&G_{k-1}=G_{k}^0+\sum_{\ell=1}^m(-1)^\ell(G_k^0 \mathsf{V}_{ab})^{\ell} G_k^0+(-1)^{m+1}(G_k^0 \mathsf{V}_{ab})^{m+1}G_{k-1},\nonumber\\
&&G_k=G_k^0+\sum_{\ell=1}^m(-1)^\ell(G_k^0 \mathsf{W}_{ab})^{\ell} G_k^0+(-1)^{m+1}(G_k^0 \mathsf{W}_{ab})^{m+1}G_k. \label{021055}
\end{eqnarray}
We always choose $m$ to be sufficiently large, depending on $\varepsilon_3$ but independent of $N$.
Then, we can write
\begin{eqnarray}
&&(G_{k-1})_{ij}=(G_k^0)_{ij}+\sum_{\ell=1}^m \mathsf{R}_{\ell,ij}+\tilde{\mathsf{R}}_{m+1,ij},\nonumber\\
 &&(G_k)_{ij}=(G_k^0)_{ij}+\sum_{\ell=1}^m\mathsf{S}_{\ell, ij}+\tilde{\mathsf{S}}_{m+1,ij}, \label{021071}
\end{eqnarray}
where
\begin{eqnarray}
&&\mathsf{R}_{\ell,ij}:=(-1)^{\ell}\Big((G_k^0 \mathsf{V}_{ab})^{\ell} G_k^0\Big)_{ij},\quad \mathsf{S}_{\ell,ij}:=(-1)^{\ell}\Big((G_k^0 \mathsf{W}_{ab})^{\ell} G_k^0\Big)_{ij},\quad \ell=1,\ldots,m, \nonumber\\
&&\tilde{\mathsf{R}}_{m+1,ij}:=(-1)^{m+1}\Big((G_k^0 \mathsf{V}_{ab})^{m+1}G_{k-1}\Big)_{ij},\quad 
\tilde{\mathsf{S}}_{m+1,ij}:=(-1)^{m+1}\Big((G_k^0 \mathsf{W}_{ab})^{m+1}G_k\Big)_{ij}. \label{021072}
\end{eqnarray} 
At first, by taking $m$ sufficiently large, from (\ref{0212020}) and (\ref{031302}), we have the trivial bound
\begin{eqnarray}
 |\tilde{\mathsf{R}}_{m+1,ij}|, |\tilde{\mathsf{S}}_{m+1,ij}|\prec M^{-{\frac{m+1}{2}}}N^{(m+2)\varepsilon_3}\ll \frac{1}{M^3\sqrt{N\eta}}.\label{021321}
\end{eqnarray}
For $\mathsf{R}_{\ell,ij}$ and $\mathsf{S}_{\ell,ij}$, we split the discussion into off-diagonal case and diagonal case. 
In the case of $i\neq j$, we keep the first and the last factors of the terms in the expansions of $((G_k^0 \mathsf{V}_{ab})^{\ell} G_k^0)_{ij}$ and $((G_k^0 \mathsf{W}_{ab})^{\ell} G_k^0)_{ij}$, namely, $(G_k^0)_{ij'}$ and $(G_k^0)_{i'j}$ for some $i',j'=a$ or $b$, and bound the  factors in between by using (\ref{031302}) and (\ref{0212020}), resulting the bound
\begin{eqnarray}
|\mathsf{R}_{\ell,ij}|, |\mathsf{S}_{\ell,ij}|\prec M^{-\frac{\ell}{2}}N^{(\ell-1)\varepsilon_3}\sum_{i',j'=a,b}|(G_k^0)_{ij'}(G_k^0)_{i'j}|,\quad \ell=1,\ldots,m.\label{021208}
\end{eqnarray}
For $i=j$, we only keep the first factor of the terms in the expansions of $((G_k^0 \mathsf{V}_{ab})^{\ell} G_k^0)_{ii}$ and $((G_k^0 \mathsf{W}_{ab})^{\ell} G_k^0)_{ii}$, and bound the others by using (\ref{031302}) and (\ref{0212020}), resulting the bound 
\begin{eqnarray}
|\mathsf{R}_{\ell,ii}|, |\mathsf{S}_{\ell,ii}|\prec M^{-\frac{\ell}{2}}N^{\ell\varepsilon_3}\Big{(}|(G_k^0)_{ia}|+|(G_k^0)_{ib}|\Big{)},\quad \ell=1,\ldots,m. \label{021215}
\end{eqnarray}
Observe that, in case $i\neq j$,  if $\{i,j\}\neq \{a,b\}$,  at least one of $(G_k^0)_{ij'}$ and  $(G_k^0)_{i'j}$ is an off-diagonal entry of $G_k^0$ for $i',j'=a$ or $b$. 

Now we compare the $2n$-th moment of $|(G_{k-1})_{ij}|$ and $|(G_k)_{ij}|$.  At first, we write
\begin{eqnarray}
\mathbb{E}|(G_d)_{ij}|^{2n}=\mathbb{E}((G_d)_{ij})^n(\overline{(G_d)_{ij}})^n,\quad d=k-1,k \label{021075}
\end{eqnarray}
By substituting the expansion (\ref{021071}) into (\ref{021075}), we can write
\begin{eqnarray}
\mathbb{E}|(G_d)_{ij}|^{2n}=\mathbf{A}(i,j)+\mathbf{R}_{d}(i,j), \quad d=k-1,k, \label{021330}
\end{eqnarray}
where $\mathbf{A}(i,j)$ is the sum of the terms which depend only  on $H_k^0$ and the first four moments of $h_{ab}$, and $\mathbf{R}_d(i,j)$ is the sum of all the other terms. We claim that $\mathbf{R}_d(i,j)$ satisfies the bound
\begin{eqnarray}
|\mathbf{R}_d(i,j)|\leq C\widehat{\Theta}_0\Big(\frac{N^{\varepsilon_3}}{\sqrt{M}}\Big)^5\Big{(}\delta_{ij}+\frac{\delta_{\{i,j\}\{a,b\}}}{(N\eta)^{n-\frac52}}+\frac{1}{(N\eta)^{n}}\Big{)},\qquad  d=k-1,k, \label{021206}
\end{eqnarray}
for some positive constant $C$.
Now, we verify (\ref{021206}). According to  (\ref{021041}) and the fact that the sequence $\mathsf{R}_{1,ij},\ldots, \mathsf{R}_{m,ij}, \tilde{\mathsf{R}}_{m+1,ij}$, as well as  $\mathsf{S}_{1,ij},\ldots, \mathsf{S}_{m,ij}, \tilde{\mathsf{S}}_{m+1,ij}$, decreases by a factor $N^{\varepsilon_3}/\sqrt{M}$ in magnitude, it is not difficult to check the leading order terms of $\mathbf{R}_{k-1}(i,j)$ are of the form
\begin{eqnarray}
\mathbb{E}\left((G_k^0)_{ij}\right)^{p}\left(\overline{(G_k^0)_{ij}}\right)^{2n-p-\sum_{\ell=1}^5(q_\ell+q'_\ell)} \prod_{\ell=1}^{5}\mathsf{R}_{\ell,ij}^{q_{\ell}} \bar{\mathsf{R}}_{\ell,ij}^{q'_\ell}, \label{021210}
\end{eqnarray}
 and those of $\mathbf{R}_{k}(i,j)$ are of the form
\begin{eqnarray}
\mathbb{E}\left((G_k^0)_{ij}\right)^{p}\left(\overline{(G_k^0)_{ij}}\right)^{2n-p-\sum_{\ell=1}^5(q_\ell+q'_\ell)} \prod_{\ell=1}^{5}\mathsf{S}_{\ell,ij}^{q_{\ell}} \bar{\mathsf{S}}_{\ell,ij}^{q'_\ell}, \label{0305110}
\end{eqnarray}
with some $p,q_\ell,q'_\ell\in\mathbb{N}$ such that
\begin{eqnarray}
\sum_{\ell=1}^5\ell(q_\ell+q'_\ell)=5,\qquad 0\leq p\leq 2n-\sum_{\ell=1}^5(q_\ell+q'_\ell). \label{0306211}
\end{eqnarray}
Every other term has at least $6$ factors of $h_{ab}$ or $h_{ab}^g$ or their conjugates, thus their sizes are typically controlled by $M^{-3}(N\eta)^{-n}$, i.e. they are subleading.
Hence, it suffices to bound (\ref{021210}) and (\ref{0305110}). In the sequel, we only estimate (\ref{021210}) in details, (\ref{0305110}) can be handled in the same manner.

Now, the five factors of $h_{ab}$ or $h_{ba}$ within the $\mathsf{R}_{\ell,ij}$'s in (\ref{021210}) are independent of the rest and estimated by $M^{-5/2}$. For the remaining factors from $G^0_k$, we use (\ref{021041}) to bound $2n$ of them and use (\ref{0212020}) to bound the rest.  In the case that $i\neq j$ and $\{i,j\}\neq \{a,b\}$, by the discussion above, we must have an off-diagonal entry of $G_k^0$ in the product $(G_k^0)_{ij'} (G_k^0)_{i'j}$ for any choice of $i',j'=a$ or $b$. Then, in the bound for $\mathsf{R}_{\ell,ij}$ in (\ref{021208}), for each $(G_k^0)_{ij'} (G_k^0)_{i'j}$, we keep the off-diagonal entry and bound the other by $N^{\varepsilon_3}$ from assumption  (\ref{0212020}). Hence, by using (\ref{021208}) and (\ref{0306211}), we see that for some $i_r,j_r\in\{i,j,a,b\}$ with $i_r\neq j_r$, $r=1,\ldots, \sum(q_\ell+q'_\ell)$, the following bound holds
\begin{eqnarray}
(\ref{021210})\leq \Big(\frac{N^{\varepsilon_3}}{\sqrt{M}}\Big)^{5}\mathbb{E}\Big(|(G_k^0)_{ij}|^{2n-\sum_{\ell=1}^5(q_\ell+q'_\ell)}\prod_{r=1}^{\sum_{\ell=1}^5(q_\ell+q'_\ell)}|(G_k^0)_{i_r j_r}|\Big)\leq 3\Big(\frac{N^{\varepsilon_3}}{\sqrt{M}}\Big)^{5}\frac{\widehat{\Theta}_0}{(N\eta)^{n}}, \label{030601}
\end{eqnarray}
where the last step follows from (\ref{021041}) and H\"{o}lder's inequality. In case of $i\neq j$ but $\{i,j\}=\{a,b\}$, we keep an entry in the product $(G_k^0)_{ij'} (G_k^0)_{i'j}$ and bound the other by $N^{\varepsilon_3}$. We remark here in this case the entry being kept can be either diagonal or off-diagonal. Consequently, for some $i_r,j_r\in\{i,j,a,b\},r=1,\ldots,\sum(q_\ell+q'_\ell)$, we have the bound
\begin{eqnarray}
(\ref{021210})\leq \Big(\frac{N^{\varepsilon_3}}{\sqrt{M}}\Big)^5\mathbb{E}\Big(|(G_k^0)_{ij}|^{2n-\sum_{\ell=1}^5(q_\ell+q'_\ell)}\prod_{r=1}^{\sum_{\ell=1}^5(q_\ell+q'_\ell)}|(G_k^0)_{i_r j_r}|\Big)\leq 3\Big(\frac{N^{\varepsilon_3}}{\sqrt{M}}\Big)^{5}\frac{\widehat{\Theta}_0}{(N\eta)^{n-\frac52}} \label{030602}
\end{eqnarray}
by using (\ref{021041}) and H\"{o}lder's inequality again.  Hence, we have shown (\ref{021206}) in the case of $i\neq j$. 
For $i=j$, it is analogous to show 
\begin{eqnarray}
(\ref{021210})\leq 3\Big(\frac{N^{\varepsilon_3}}{\sqrt{M}}\Big)^{5}\widehat{\Theta}_0 \label{030606}
\end{eqnarray}
by using (\ref{021041}), (\ref{021215}) and H\"{o}lder's inequality.
Hence, we verified (\ref{021206}). Consequently, by Assumption \ref{assu.3},  (\ref{021330}) and (\ref{021206}) we have
\begin{eqnarray*}
\Big{|}\mathbb{E}|(G_{k-1})_{ij}|^{2n}-\mathbb{E}|(G_k)_{ij}|^{2n}\Big{|}\leq C\widehat{\Theta}_0\Big(\frac{N^{\varepsilon_3}}{\sqrt{M}}\Big)^5\Big{(}\delta_{ij}+\frac{\delta_{\{i,j\}\{a,b\}}}{(N\eta)^{n-\frac52}}+\frac{1}{(N\eta)^{n}}\Big{)},
\end{eqnarray*}
which together with the assumption (\ref{021201}) for $\mathbb{E}|(G_{k-1})_{ij}|^{2n}$ and the definition of $\widehat{\Theta}_{\ell,ij}$'s in (\ref{021250}), we can get 
\begin{eqnarray*}
\mathbb{E}|(G_k)_{ij}|^{2n}\leq \widehat{\Theta}_{k,ij}\Big{(}\delta_{ij}+\frac{1}{(N\eta)^{n}}\Big{)}.
\end{eqnarray*}
Hence, we completed the proof of Lemma \ref{lem.021090}.
\end{proof}
To show (\ref{09010222}), we also need the following lemma.
\begin{lem} \label{lem.021118}Suppose that the assumptions in Lemma \ref{lem021256} hold. Fix the indices $a,b\in \{1,\ldots N\}$. Let $H^0$ be a matrix obtained from $H$ with its $(a,b)$-th entry replaced by $0$. Then, if for some $\eta_0\geq 1/N$ there exists
\begin{eqnarray}
|G_{ii}(z)|\prec 1,\quad |(G^0)_{ii}(z)|\prec 1 \quad \text{for} \quad \eta\geq \eta_0, \quad \forall\; i=1,\ldots,N,\label{021290}
\end{eqnarray}
then we also have
\begin{eqnarray*}
|G_{ij}(z)|\prec \frac{\eta_0}{\eta},\quad |(G^0)_{ij}(z)|\prec \frac{\eta_0}{\eta}, \quad \text{for}\quad  \frac{1}{N}<\eta\leq \eta_0, \quad  \forall\; i,j=1,\ldots,N. 
\end{eqnarray*}
\end{lem}
\begin{proof}[Proof of Lemma \ref{lem.021118}] The proof is almost the same as the discussion on pages 2311--2312 in \cite{EKYY2013a}. For the convenience of the reader, we sketch it below. At first, according to the discussion below (4.28) in \cite{EKYY2013a}, for any $i,j=1,\ldots, N$, we have
\begin{eqnarray*}
|G_{ij}(E+\mathbf{i}\eta)|\leq C\max_{\ell}\sum_{k\geq 0} \Im G_{\ell\ell}(E+\mathbf{i}2^k\eta).
\end{eqnarray*}
Now, we set
\begin{eqnarray*}
k_1:=\max\{k: 2^k\eta<\eta_0\},\qquad k_2:=\max\{k: 2^k\eta<1\}.
\end{eqnarray*}
According to our assumption, both $k_1$ and $k_2$ are of the order $\log N$. Now, we have
\begin{align*}
\sum_{k\geq 0} \Im G_{\ell\ell}(E+\mathbf{i}2^k\eta)&=\sum_{k=0}^{k_1} \Im G_{\ell\ell}(E+\mathbf{i}2^k\eta)+\sum_{k=k_1}^{k_2} \Im G_{\ell\ell}(E+\mathbf{i}2^k\eta)+\sum_{k=k_2+1}^{\infty} \Im G_{\ell\ell}(E+\mathbf{i}2^k\eta)\nonumber\\
&\prec \frac{\eta_0}{\eta}\sum_{k=0}^{k_1} \frac{1}{2^k}\Im G_{\ell\ell}(E+\mathbf{i}\eta_0)+(k_2-k_1)+1\prec \frac{\eta_0}{\eta}
\end{align*}
where in the second step, we used the fact that the function $y\mapsto y\Im G_{\ell\ell} (E+\mathbf{i}y)$ is monotonically increasing, the condition (\ref{021290}) and the fact $\eta\leq \eta_0$. Hence, we conclude the proof of Lemma \ref{lem.021118}.
\end{proof}
Now, with Theorem \ref{lem.012802}, Lemma \ref{lem.021090} and Lemma \ref{lem.021118}, we can prove Lemma \ref{lem021256}.
\begin{proof}[Proof for Lemma \ref{lem021256}] The proof relies on the following bootstrap argument, namely, we show that  once 
\begin{eqnarray}
|(G_\ell)_{ij}|\prec 1, \quad \forall\; \ell=1,\ldots,\varsigma(N),\quad \forall\; i,j=1,\ldots, N \label{031310}
\end{eqnarray}
holds for $\eta\geq \eta_0$ with $\eta_0\in[N^{-1+\varepsilon_2+\varepsilon_3},M^{-1}N^{\varepsilon_2}]$, it also holds for $\eta\geq \eta_0N^{-\varepsilon_3}$ for any $\varepsilon_3$ satisfying (\ref{0305101}). Assuming (\ref{031310}) holds for $\eta\geq\eta_0$, we see that  
\begin{align*}
&\max_{i,j}|(G^0_{\ell})_{ij}|=\max_{i,j}|(G_\ell)_{ij}+((G_\ell)\mathsf{W}_{ab}G^0_\ell)_{ij}|\nonumber\\
&\prec \max_{i,j}|(G_\ell)_{ij}|+\frac{1}{\sqrt{M}}\max_{i,j}|(G_\ell)_{ij}|\cdot \max_{i,j}|(G^0_\ell)_{ij}|\prec 1+\frac{1}{\sqrt{M}} \max_{i,j}|(G^0_\ell)_{ij}|.
\end{align*} 
Consequently, for $\eta\geq\eta_0$, we also have
\begin{eqnarray}
|(G^0_\ell)_{ij}|\prec 1, \quad \forall\; \ell=1,\ldots,\varsigma(N),\quad \forall\; i,j=1,\ldots, N. \label{031322}
\end{eqnarray}
Therefore, (\ref{021290}) holds. Then, by Lemma \ref{lem.021118}, we see that (\ref{0212020}) holds for  $\eta\geq \eta_0N^{-\varepsilon_3}$.
Furthermore, by Lemma \ref{lem.021090} and Theorem \ref{lem.012802} for $G_0$, i.e. the Gaussian case, one can get that for any given $n$,
\begin{eqnarray}
&\mathbb{E}|(G_{\ell})_{ij}|^{2n}\leq \widehat{\Theta}_{\ell,ij}\Big{(}\delta_{ij}+\frac{1}{(N\eta)^{n}}\Big{)}\leq 2\widehat{\Theta}_0\Big{(}\delta_{ij}+\frac{1}{(N\eta)^{n}}\Big{)},\quad \text{for}\quad M^{-1}N^{\varepsilon_2}\geq \eta\geq \eta_0N^{-\varepsilon_3},&\nonumber\\
& \forall\; \ell=1,\ldots, \varsigma(N),\quad \forall\; i,j=1,\ldots,N.& \label{021291}
\end{eqnarray}
Note that since (\ref{021291}) holds for any given $n$, we get (\ref{031310}) for $M^{-1}N^{\varepsilon_2}\geq\eta\geq \eta_0N^{-\varepsilon_3}$. 

Now we start from $\eta_0=M^{-1}N^{\varepsilon_2}$. By Proposition \ref{pro021301} we see that (\ref{031310}) holds for all $\eta\geq \eta_0$. Then we can use the bootstrap argument above finitely many times to show (\ref{031310}) holds for all $\eta\geq N^{-1+\varepsilon_2}$. Consequently, we have (\ref{0212020}) for all $\eta\geq N^{-1+\varepsilon_2}$. Then,  Lemma \ref{lem021256} follows from Lemma \ref{lem.021090} and Theorem \ref{lem.012802} immediately.
\end{proof}

\subsection{Proof of Theorem \ref{thm.012801}}
Without loss of generality we can assume that $M\leq N(\log N)^{-10}$, otherwise, Proposition \ref{pro021301} implies (\ref{020961}) immediately. Now, recalling the notation defined in (\ref{020970}), we denote the Green's function of $H^{(i)}$ as
\begin{eqnarray*}
G^{(i)}(z):=(H^{(i)}-z)^{-1},
\end{eqnarray*}
with a little abuse of notation. We only need to consider the diagonal entries $G_{ii}$ below, since the bound for the off-diagonal entires of $G(z)$ is implied by (\ref{09010222}) directly. 
Set 
\begin{eqnarray}
\Delta_i\equiv\Delta_i(z):=\frac{1}{G_{ii}}+z+\sum_{a}\sigma_{ai}^2G_{aa}. \label{030701}
\end{eqnarray}
We introduce the notation
\begin{eqnarray*}
\Lambda_d\equiv\Lambda_d(z):=\max_{i}|G_{ii}(z)-m_{sc}(z)|.
\end{eqnarray*}
We have the following lemma.
\begin{lem} \label{lem.090205}Suppose that $H$ satisfies Assumptions \ref{assu.1}, \ref{assu.3} and \ref{assu.4}. We have
\begin{eqnarray} 
|\varDelta_i(z)|\prec \frac{1}{\sqrt{N\eta}}, \quad \forall\; i=1,\ldots,N, \label{021311}
\end{eqnarray}
uniformly for $z\in \mathbf{D}(N,\kappa,\varepsilon_2)$.
\end{lem}
The proof of Lemma \ref{lem.090205} will be postponed. Using Lemma \ref{lem.090205}, we see that, with high probability, (\ref{030701}) is a small perturbation of the self-consistent equation of $m_{sc}$, i.e. (\ref{030702}), considering $\sum_{a}\sigma_{ai}^2=1$. To control $\Lambda_d$, we use a {\emph{continuity argument}} from \cite{EKYY2013b}.

We remind here that in the sequel, the parameter set of the stochastic dominance is always $\mathbf{D}(N,\kappa,\varepsilon_2)$, without further mention. We need to show that 
\begin{eqnarray}
\Lambda_d\prec \frac{1}{\sqrt{N\eta}}, \label{030705}
\end{eqnarray} 
and first we claim that it suffices to show that  
\begin{eqnarray}
\mathbf{1}(\Lambda_d\leq N^{-\frac{\varepsilon_2}{4}})\Lambda_d\prec \frac{1}{\sqrt{N\eta}}.  \label{030706}
\end{eqnarray} 
Indeed, if (\ref{030706}) were proven, we see that with high probability either $\Lambda_d>N^{-\frac{\varepsilon_2}{4}}$ or $\Lambda_d\prec 1/\sqrt{N\eta}\leq N^{-\frac{\varepsilon_2}{2}}$ for $z\in \mathbf{D}(N,\kappa,\varepsilon_2)$. That means, there is a gap in the possible range of $\Lambda_d$. Now, choosing $\varepsilon$ in (\ref{0127100}) to be sufficiently small, we are able to get for $\eta=M^{-1}N^{\varepsilon_2}$,
\begin{eqnarray}
\Lambda_d\prec N^{-\frac{\varepsilon_2}{2}},\quad \forall\; E\in[-2+\kappa,2-\kappa],\quad \forall\; i=1,\ldots, N. \label{021009}
\end{eqnarray}
By the fact that $\Lambda_d$ is continuous in $z$, we see that with high probability, $\Lambda_d$ can only stay in one side of the range, namely, (\ref{030705}) holds. The rigorous details of this argument involve considering a fine discrete grid of the $z$-parameter and using that $G(z)$ is Lipschitz continuous (albeit with a large Lipschitz constant $1/\eta$). The details are found in Section 5.3 of \cite{EKYY2013b}.  

Hence, what remains is to verify (\ref{030706}). The proof of (\ref{030706}) is almost the same as that for Lemma 3.5 in \cite{EYY2012}. For the convenience of the reader, we sketch it below without reproducing the details. We set
\begin{eqnarray*}
\bar{m}\equiv\bar{m}(z):=\frac{1}{N}\sum_{i=1}^N G_{ii}(z),\quad \mathsf{u}_i\equiv\mathsf{u}_i(z):=G_{ii}-\bar{m},\quad i=1,\ldots, N.
\end{eqnarray*}
We also denote $\vec{\mathsf{u}}:=(\mathsf{u}_1,\ldots, \mathsf{u}_N)$.
By the assumption $\Lambda_d\leq N^{-\frac{\varepsilon_2}{4}}$, we have
\begin{eqnarray}
\mathsf{u}_i= O(N^{-\frac{\varepsilon_2}{4}}). \label{030710}
\end{eqnarray}
Now we rewrite (\ref{030701}) as 
\begin{eqnarray}
0=G_{ii}+\frac{1}{z+\sum_{a}\sigma_{ai}^2G_{aa}-\varDelta_i}=:G_{ii}+\frac{1}{z+\bar{m}(z)}+\varOmega_i. \label{030722}
\end{eqnarray}
By using (\ref{021311}), Lemma 5.1 in \cite{EYY2012}, and the assumption $\Lambda_d\leq N^{-\frac{\varepsilon_2}{4}}$, we can 
show that
\begin{eqnarray}
\varOmega_i=-\frac{\sum_{a}\sigma_{ai}^2\mathsf{u}_a}{(z+\bar{m}(z))^2}+O\big{(}||\vec{\mathsf{u}}||_{\infty}^2\big{)}+O\big{(}\max_i|\varDelta_i|\big{)}. \label{030723}
\end{eqnarray}
One can refer to the derivation of (5.14) in \cite{EYY2012} for more details. Averaging over $i$ for (\ref{030722}) and (\ref{030723}) leads to 
\begin{eqnarray}
\bar{m}(z)+\frac{1}{z+\bar{m}(z)}=-\varOmega. \label{030712}
\end{eqnarray} 
and
\begin{eqnarray}
\varOmega:=\frac{1}{N}\sum_{i=1}^N\varOmega_i=O\big{(}||\vec{\mathsf{u}}||_{\infty}^2\big{)}+O\big{(}\max_i|\varDelta_i|\big{)} \label{030711}
\end{eqnarray}
Plugging (\ref{030710}) and (\ref{021311}) into (\ref{030711}) yields
\begin{eqnarray}
|\varOmega|\prec N^{-\frac{\varepsilon_2}{2}}. \label{030720}
\end{eqnarray}
Using (\ref{030720}), the fact $|\bar{m}(z)-m_{sc}(z)|\leq \Lambda_d\leq N^{-\frac{\varepsilon_2}{4}}$, and Lemma 5.2 in \cite{EYY2012}, to (\ref{030712}), we have
\begin{eqnarray}
|\bar{m}(z)-m_{sc}(z)|\leq  |\varOmega|=O\big{(}||\vec{\mathsf{u}}||_{\infty}^2\big{)}+O\big{(}\max_i|\varDelta_i|\big{)},  \label{030730}
\end{eqnarray}
where in the first step we have used the fact that $z\in \mathbf{D}(N,\kappa,\varepsilon_2)$ thus away from the edges of the semicircle law.
Now, we combine (\ref{030722}), (\ref{030723}) and (\ref{030712}), resulting
\begin{eqnarray}
\mathsf{u}_i=\frac{\sum_{a}\sigma_{ai}^2\mathsf{u}_a}{(z+\bar{m}(z))^2}+\varOmega+O\big{(}||\vec{\mathsf{u}}||_{\infty}^2\big{)}+O\big{(}\max_i|\varDelta_i|\big{)}=\mathsf{w}_i+\frac{\sum_{a}\sigma_{ai}^2\mathsf{u}_a}{(z+m_{sc}(z))^2},\qquad i=1,\ldots, N. \label{030727}
\end{eqnarray}
We just take the above identity as the definition of $\mathsf{w}_i$. Analogously, we set $\vec{\mathsf{w}}:=(\mathsf{w}_1,\ldots, \mathsf{w}_N)'$. Then (\ref{030711}) and (\ref{030727}) imply
\begin{align}
||\vec{\mathsf{w}}||_\infty&=O\big{(}||\vec{\mathsf{u}}||_{\infty}^2\big{)}+O\big{(}\max_i|\varDelta_i|\big{)}+O\big(||\vec{\mathsf{u}}||_{\infty}\cdot|(z+\bar{m}(z))^{-2}-(z+m_{sc}(z))^{-2}|\big)\nonumber\\
&\leq O\big{(}||\vec{\mathsf{u}}||_{\infty}^2\big{)}+O\big{(}\max_i|\varDelta_i|\big{)}+O\big(||\vec{\mathsf{u}}||_{\infty}\cdot|\bar{m}(z)-m_{sc}(z)|\big)\nonumber\\
&\leq  O\big{(}||\vec{\mathsf{u}}||_{\infty}^2\big{)}+O\big{(}\max_i|\varDelta_i|\big{)} \label{030767}
\end{align}
where the second step follows from the fact $|z+m_{sc}(z)|\geq 1$ in $\mathbf{D}(N,\kappa,\varepsilon_2)$ (see (5.1) in \cite{EYY2012} for instance), (\ref{030720}) and (\ref{030730}), and in the last step we used (\ref{030730}) again. 

Now, using the fact $m_{sc}^2(z)=(m_{sc}(z)+z)^{-2}$ (see (\ref{030702})), we rewrite (\ref{030727}) in terms of the matrix $\mathcal{T}$ introduced in (\ref{030740}) as
\begin{eqnarray*}
\vec{\mathsf{u}}=\big(1-m_{sc}^2(z) \mathcal{T}\big)^{-1}\vec{\mathsf{w}}.
\end{eqnarray*}
Consequently, we have
\begin{eqnarray}
||\vec{\mathsf{u}}||_{\infty}\leq \Big|\Big| \big(1-m_{sc}^2(z) \mathcal{T}\big)^{-1} \Big|\Big|_{\ell^\infty\to\ell^\infty} ||\vec{\mathsf{w}}||_\infty:=\varGamma(z)||\vec{\mathsf{w}}||_\infty. \label{030768}
\end{eqnarray}
Then for $z\in \mathbf{D}(N,\kappa,\varepsilon_2)$, using (\ref{0307100}) and Proposition A.2 (ii) in \cite{EKYY2013b} (with $\delta_-=1$ and $\theta>c$), we can get 
\begin{eqnarray}
\varGamma(z)=O(\log N). \label{030751}
\end{eqnarray}
Plugging (\ref{030751}) and (\ref{030767}) into (\ref{030768}) yields
\begin{eqnarray*}
||\vec{\mathsf{u}}||_{\infty}\prec O(||\vec{\mathsf{u}}||_{\infty}^2+O\big{(}\max_i|\varDelta_i|\big{)})\prec ||\vec{\mathsf{u}}||_{\infty}^2+\frac{1}{\sqrt{N\eta}}, 
\end{eqnarray*}
where the second step follows from (\ref{021311}).  Then (\ref{030710}) further implies that
\begin{eqnarray*}
||\vec{\mathsf{u}}||_{\infty}\prec \frac{1}{\sqrt{N\eta}}, 
\end{eqnarray*}
which together with (\ref{030730}) and (\ref{021311}) also implies
\begin{eqnarray*}
|\bar{m}(z)-m_{sc}(z)|\prec \frac{1}{\sqrt{N\eta}}.
\end{eqnarray*}
Hence
\begin{eqnarray*}
\Lambda_d\leq ||\vec{\mathsf{u}}||_{\infty}+|\bar{m}(z)-m_{sc}(z)|\prec \frac{1}{\sqrt{N\eta}}.
\end{eqnarray*}
Therefore, we completed the proof of Theorem \ref{thm.012801}.

\begin{proof}[Proof of Lemma \ref{lem.090205}]  For simplicity, we omit the variable $z$ from the notation below. At first, we recall the elementary identity by Schur's complement,  namely,
\begin{eqnarray}
G_{ii}=\frac{1}{h_{ii}-z-(\mathbf{h}_i^{\langle i\rangle})^* G^{(i)}\mathbf{h}_i^{\langle i\rangle}}. \label{0902100}
\end{eqnarray}
where we used the notation $\mathbf{h}_i^{\langle i\rangle}$ to denote the $i$-th column of $H$, with the $i$-th component deleted.
Now, we use the identity for $a,b\neq i$ (see Lemma 4.5 in \cite{EKYY2013b} for instance),
\begin{eqnarray}
G_{ab}^{(i)}=G_{ab}-G_{ai}G_{ib}(G_{ii})^{-1}=G_{ab}-G_{ai}G_{ib}\Big{(}h_{ii}-z-(\mathbf{h}_i^{\langle i\rangle})^* G^{(i)}\mathbf{h}_i^{\langle i\rangle}\Big{)}. \label{090291}
\end{eqnarray}
By using (\ref{021020}) and the large deviation estimate for the quadratic form (see Theorem C.1 of \cite{EKYY2013b} for instance), we have
\begin{eqnarray}
\Big{|}(\mathbf{h}_i^{\langle i\rangle})^* G^{(i)}\mathbf{h}_i^{\langle i\rangle}-\sum_{a\neq i} \sigma_{ai}^2\cdot G_{aa}^{(i)}\Big{|}\prec \sqrt{\frac{1}{M}\max_a|G_{aa}^{(i)}|^2+\max_{a\neq b} |G_{ab}^{(i)}|^2}, \label{021002}
\end{eqnarray}
which implies that
\begin{eqnarray}
\Big{|}(\mathbf{h}_i^{\langle i\rangle})^* G^{(i)}\mathbf{h}_i^{\langle i\rangle}\Big{|}\prec \max_{a\neq i}|G_{aa}^{(i)}|+\sqrt{\frac{1}{M}\max_{a\neq i}|G_{aa}^{(i)}|^2+\max_{a\neq b} |G_{ab}^{(i)}|^2}\leq 3\max_{a,b\neq i}|G_{ab}^{(i)}|, \label{090290}
\end{eqnarray}
where we have used the fact that $\sum_{a}\sigma_{ai}^2=1$  in the first inequality above. Plugging (\ref{021051}) and (\ref{090290}) into (\ref{090291}) and using Corollary \ref{cor.1} we obtain
\begin{eqnarray*}
\max_{a,b\neq i}|G_{ab}^{(i)}|\prec 1+\frac{1}{N\eta}\Big{(}1+3\max_{a,b\neq i}|G_{ab}^{(i)}|\Big{)},
\end{eqnarray*}
which  implies 
\begin{eqnarray}
\max_{a,b\neq i}|G_{ab}^{(i)}|\prec 1,\qquad \Big{|}(\mathbf{h}_i^{\langle i\rangle})^* G^{(i)}\mathbf{h}_i^{\langle i\rangle}\Big{|}\prec 1. \label{021001}
\end{eqnarray}
In addition, (\ref{021051}),  (\ref{090291})  and (\ref{021001}) lead to the fact that
\begin{eqnarray}
|G_{ab}(z)-G_{ab}^{(i)}(z)|\prec \frac{1}{N\eta},\qquad |G_{ab}^{(i)}(z)| \prec \delta_{ab}+\frac{1}{\sqrt{N\eta}},\quad \forall\; a, b\neq i. \label{021310}
\end{eqnarray}  
Now, using (\ref{030701}), (\ref{0902100}), (\ref{021002}) and (\ref{021310}), we can see that
\begin{eqnarray}
|\varDelta_i|=\big|-h_{ii}+(\mathbf{h}_i^{\langle i\rangle})^* G^{(i)}\mathbf{h}_i^{\langle i\rangle}-\sum_{a} \sigma_{ai}^2 G_{aa}\big|\prec\frac{1}{\sqrt{N\eta}}. \label{021401}
\end{eqnarray}
Therefore, we completed the proof of Lemma \ref{lem.090205}.
\end{proof}
\subsection{Proof of Theorem \ref{thm.031301}} With Theorem \ref{thm.012801}, we can prove Theorem \ref{thm.031301} routinely. 
At first, due to Definition \ref{defi.1} and the fact that $G_{ab}(z)$ and $m_{sc}(z)$ are Lipschitz functions of $z$ with Lipschitz constant $\eta^{-1}$, it is easy to strengthen  (\ref{020961}) to 
\begin{eqnarray*}
\max_{a,b}\sup_{z\in\mathbf{D}(N,\kappa,\varepsilon_2)}|G_{ab}(z)-\delta_{ab}m_{sc}(z)|\prec \frac{1}{\sqrt{N\eta}},
\end{eqnarray*}
which implies that
\begin{eqnarray}
\max_{a}\sup_{z\in\mathbf{D}(N,\kappa,\varepsilon_2)}|G_{aa}(z)|\prec C \label{031330}
\end{eqnarray}
for some positive constant $C$ due to the fact that $m_{sc}(z)\sim 1$. Recalling the normalized eigenvector $\mathbf{u}_i=(u_{i1},\ldots, u_{iN})$ corresponding to $\lambda_i$, and using the spectral decomposition, we have
\begin{eqnarray}
\max_a\Im G_{aa}(z)=\max_a\sum_{i=1}^N\frac{|u_{ia}|^2\eta}{|\lambda_i-E|^2+\eta^2}=\sum_{i=1}^N\frac{||\mathbf{u}_i||_\infty^2\eta}{|\lambda_i-E|^2+\eta^2}. \label{031331}
\end{eqnarray}
For any $|\lambda_i|\leq \sqrt{2}-\kappa$, we set $E=\lambda_i$ on the r.h.s. of (\ref{031331}) and use (\ref{031330}) to bound the l.h.s. of it. Then we obtain
\begin{eqnarray*}
\frac{||\mathbf{u}_{i}||^2_\infty}{\eta}\prec 1.
\end{eqnarray*}
Choosing $\eta=N^{-1+\varepsilon_2}$ above and using the fact that $\varepsilon_2$ can be arbitrarily small, we can get (\ref{031335}). Hence, we completed the proof of Theorem \ref{thm.031301}.
\section{Supersymmetric formalism and integral representation for the Green's function} \label{s.3}
In this section, we will represent $\mathbb{E}|G_{ij}(z)|^{2n}$ for the Gaussian case by a superintegral.  The final representation is stated in (\ref{122707}). We make the convention here, for any real argument in an integral below, its region of the integral is always $\mathbb{R}$, unless specified otherwise.
\subsection{Gaussian integrals and superbosonization formulas}
Let $\boldsymbol{\phi}=(\phi_1,\ldots, \phi_k)'$ be a vector of complex components, $\boldsymbol{\psi}=(\psi_1,\ldots,\psi_k)'$ be a vector of Grassmann components. In addition, let $\boldsymbol{\phi}^*$ and $\boldsymbol{\psi}^*$ be the conjugate transposes of $\boldsymbol{\phi}$ and $\boldsymbol{\psi}$, respectively. We recall the following well-known formulas for Gaussian integrals.
\begin{pro}[Gaussian integrals or Wick's formulas] \label{pro.012801} ~~
\begin{itemize}
\item[(i)] Let $\mathrm{A}$ be a $k\times k$ complex matrix with positive-definite Hermitian part, i.e. $\Re A>0$. Then for any $\ell\in \mathbb{N}$, and $i_1,\ldots, i_\ell, j_1,\ldots, j_\ell\in \{1,\ldots, k\}$, we have
\begin{eqnarray}
\int \prod_{a=1}^k\frac{{\rm d}\Re \phi_a {\rm d}\Im \phi_a}{\pi}\;\exp\{-\boldsymbol{\phi}^*\mathrm{A}\boldsymbol{\phi}\}\prod_{b=1}^\ell \bar{\phi}_{i_b}\phi_{j_b}=\frac{1}{\det \mathrm{A}} \; \sum_{\sigma\in \mathbb{P}(\ell)} \prod_{b=1}^\ell (\mathrm{A}^{-1})_{j_b,i_{\sigma(b)}}, \label{0129103}
\end{eqnarray}
where $\mathbb{P}(\ell)$ is the permutation group of degree $\ell$.
\item[(ii)] Let $\mathrm{B}$ be any $k\times k$ matrix. Then for any $\ell\in \{ 0,\ldots, k\}$, any $\ell$ distinct integers  $i_1,\ldots, i_\ell$ and  another $\ell$ distinct integers $j_1,\ldots, j_\ell\in \{1,\ldots, k\}$, we have
\begin{eqnarray}
\int \prod_{a=1}^k{\rm d}\bar{\psi}_a {\rm d}\psi_a\; \exp\{-\boldsymbol{\psi}^*\mathrm{B}\boldsymbol{\psi}\}\prod_{b=1}^\ell \bar{\psi}_{i_b}\psi_{j_b} =(-1)^{\ell+\sum_{\alpha=1}^\ell(i_\alpha+j_\alpha)}\det\mathrm{B}^{(\mathsf{I}|\mathsf{J})}, \label{0129102}
\end{eqnarray}
where $\mathsf{I}=\{i_1,\ldots, i_\ell\}$, and $\mathsf{J}=\{j_1,\ldots, j_\ell\}$.
\end{itemize}
\end{pro}
Now, we introduce the superbosonization formula for superintegrals.  Let $\boldsymbol{\chi}=(\chi_{ij})$ be an  $\ell\times r$ matrix with Grassmann entries, $\mathbf{f}=(f_{ij})$ be an $\ell\times r$ matrix with complex entries. In addition, we denote their conjugate transposes by $\boldsymbol{\chi}^*$ and $\mathbf{f}^*$ respectively. Let $F$ be a function of the entries of the matrix 
\begin{eqnarray*}
\mathcal{S}(\mathbf{f},\mathbf{f}^*;\boldsymbol{\chi},\boldsymbol{\chi}^*):=\bigg(\begin{array}{ccc}
\boldsymbol{\chi}^*\boldsymbol{\chi} & \boldsymbol{\chi}^*\mathbf{f}\\
\mathbf{f}^*\boldsymbol{\chi} &\mathbf{f}^*\mathbf{f}
\end{array}\bigg).
\end{eqnarray*}
Let $\mathcal{A}(\boldsymbol{\chi},\boldsymbol{\chi}^*)$ be the Grassmann algebra generated by $\chi_{ij}$'s and $\bar{\chi}_{ij}$'s. Then we can regard $F$ as a function defined on a complex vector space, taking values in $\mathcal{A}(\boldsymbol{\chi},\boldsymbol{\chi}^*)$. Hence, we can and do view $F(\mathcal{S}(\mathbf{f},\mathbf{f}^*;\boldsymbol{\chi},\boldsymbol{\chi}^*))$ as a polynomial in $\chi_{ij}$'s and $\bar{\chi}_{ij}$'s, in which the coefficients are functions of $f_{ij}$'s and $\bar{f}_{ij}$'s. Under this viewpoint, we state the assumption on $F$ as follows.
\begin{assu}\label{assu.030501} Suppose that $F(\mathcal{S}(\mathbf{f},\mathbf{f}^*;\boldsymbol{\chi},\boldsymbol{\chi}^*))$ is a holomorphic function of $f_{ij}$'s and $\bar{f}_{ij}$'s if they are  regarded as independent variables, and $F$ is a Schwarz function of $\Re f_{ij}$'s and $\Im f_{ij}$'s, by those we mean that all of the coefficients of $F(\mathcal{S}(\mathbf{f},\mathbf{f}^*;\boldsymbol{\chi},\boldsymbol{\chi}^*))$, as functions of $f_{ij}$'s and $\bar{f}_{ij}$'s, possess the above properties.
\end{assu}
\begin{pro}[Superbosonization formula for the nonsingular case,\cite{LSZ}] \label{pro.1}Suppose that $F$ satisfies Assumption \ref{assu.030501}. For $\ell\geq r$, we have
\begin{eqnarray}
&&\int F\bigg(
\begin{array}{ccc}
\boldsymbol{\chi}^*\boldsymbol{\chi} & \boldsymbol{\chi}^*\mathbf{f}\\
\mathbf{f}^*\boldsymbol{\chi} &\mathbf{f}^*\mathbf{f}
\end{array}
\bigg){\rm d}\mathbf{f}{\rm d}\boldsymbol{\chi}
=(\mathbf{i}\pi)^{-r(r-1)} \int {\rm d}\hat{\mu}(\mathbf{x}){\rm d}\hat{\nu}(\mathbf{y}){\rm d}\boldsymbol{\omega}{\rm d} \boldsymbol{\xi}  \; F\bigg(
\begin{array}{ccc}
\mathbf{x} & \boldsymbol{\omega}\\
\boldsymbol{\xi} &\mathbf{y}
\end{array}
\bigg) \frac{\det^\ell \mathbf{y}}{\det^\ell(\mathbf{x}-\boldsymbol{\omega}\mathbf{y}^{-1}\boldsymbol{\xi})}, \nonumber\\\label{010501}
\end{eqnarray}
where $\mathbf{x}=(x_{ij})$ is a unitary matrix; $\mathbf{y}=(y_{ij})$ is a positive-definite Hermitian matrix; $\boldsymbol{\omega}$ and $\boldsymbol{\xi}$ are two Grassmann matrices, and all of them are $r\times r$. Here 
\begin{eqnarray*}
&&{\rm d}\mathbf{f}=\prod_{i,j}\frac{{\rm d}\Re f_{ij} {\rm d} \Im f_{ij}}{\pi},\qquad {\rm d}\boldsymbol{\chi}=\prod_{i,j} {\rm d}\bar{\chi}_{ij} {\rm d}\chi_{ij},\\
&&{\rm d}\hat{\nu}(\mathbf{y})=\mathbf{1}(\mathbf{y}>0) \prod_{i=1}^r {\rm d} y_{ii} \prod_{j>k} {\rm d}\Re y_{jk} {\rm d}\Im y_{jk}, \qquad {\rm d}\boldsymbol{\omega}{\rm d} \boldsymbol{\xi}= \prod_{i,j=1}^r {\rm d}\omega_{ij}{\rm d} \xi_{ij},
\end{eqnarray*}
and ${\rm d}\hat{\mu}(\cdot)$ is defined by
\begin{eqnarray*}
&&{\rm d}\hat{\mu}(\mathbf{x})=\frac{\pi^{r(r-1)/2}}{\prod_{i=1}^r i!}\cdot\prod_{i=1}^r\frac{{\rm d} x_i }{2\pi \mathbf{i}} \cdot (\Delta(x_1,\ldots, x_r))^2 \cdot {\rm d}\mu(V),
\end{eqnarray*}
under the parametrization induced by the eigendecomposition, namely,
\begin{eqnarray*}
\mathbf{x}=V^*\hat{\mathbf{x}}V,\quad \hat{\mathbf{x}}={\rm diag}(x_1,\ldots,x_r),\quad V\in \mathring{U}(r).
\end{eqnarray*}
Here ${\rm d}\mu(V)$ is the Haar measure on $\mathring{U}(r)$, and $\Delta(\cdot)$ is the Vandermonde determinant. In addition, the integral w.r.p.t. $\mathbf{x}$ ranges over $U(2)$, that w.r.p.t. $\mathbf{y}$ ranges over all positive-definite matrices. 
\end{pro}

For the singular case, i.e. $r>\ell$,  we only state the formula for the case of $r=2$ and $\ell=1$, which is enough for our purpose. We can refer to formula (11) in \cite{BEKYZ} for the result under more general setting on $r$ and $\ell$.
\begin{pro}[Superbosonization formula for the singular case,\cite{BEKYZ}] \label{pro.021602}Suppose that $F$ satisfies Assumption \ref{assu.030501}. If $r=2$ and $\ell=1$,  we have
\begin{eqnarray}
 &&\int F\bigg(
\begin{array}{ccc}
\boldsymbol{\chi}^*\boldsymbol{\chi} & \boldsymbol{\chi}^*\mathbf{f}\\
\mathbf{f}^*\boldsymbol{\chi} &\mathbf{f}^*\mathbf{f}
\end{array}
\bigg){\rm d}\mathbf{f}{\rm d}\boldsymbol{\chi}=\frac{-1}{\pi^2}\int {\rm d}\mathbf{w} {\rm d}\hat{\mu}(\mathbf{x})\cdot\mathbf{1}(y\geq 0){\rm d}y\cdot {\rm d}\boldsymbol{\omega}{\rm d} \boldsymbol{\xi} \; F\bigg(
\begin{array}{ccc}
\mathbf{x} & \boldsymbol{\omega}\mathbf{w}^*\\
\mathbf{w}\boldsymbol{\xi} &y\mathbf{w}\mathbf{w}^*
\end{array}
\bigg)\frac{y(y-\boldsymbol{\xi}\mathbf{x}^{-1}\boldsymbol{\omega})^{2}}{\det^2\mathbf{x}}, \nonumber\\\label{010502}
\end{eqnarray}
where $y$ is a positive variable; $\mathbf{x}$ is a $2$-dimensional unitary matrix; $\boldsymbol{\omega}=(\omega_1,\omega_2)'$ and $\boldsymbol{\xi}=(\xi_1,\xi_2)$ are two vectors with Grassmann components. In addition, $\mathbf{w}$ is a unit vector, which can be parameterized by
\begin{eqnarray*}
\mathbf{w}=\left(1/\sqrt{1+|w|^2},w/\sqrt{1+|w|^2}\right)',\quad w\in\mathbb{C}.
\end{eqnarray*}
Moreover, the differentials are defined as
\begin{eqnarray*}
{\rm d}\mathbf{w}=\frac{1}{(1+|w|^2)^2} {\rm d}\Re w {\rm d}\Im w,\qquad {\rm d}\boldsymbol{\omega}{\rm d}\boldsymbol{\xi}=\prod_{i=1,2} {\rm d}\omega_i{\rm d}\xi_i.
\end{eqnarray*}
In addition, the integral w.r.p.t. $\mathbf{x}$ ranges over $U(2)$.
\end{pro} 
In our discussion, for $\mathbf{w}$, we will adopt the parametrization 
\begin{eqnarray*}
v=1/{\sqrt{1+|w|^2}},\quad e^{\mathbf{i}\theta}=w/{|w|},\quad v\in \mathbb{I}, \quad \theta\in \mathbb{L}
\end{eqnarray*}
for convenience.  Accordingly, we can get
\begin{eqnarray*}
{\rm d}\mathbf{w}=v{\rm d}v {\rm d}\theta.
\end{eqnarray*}
\subsection{Initial representation}
For $a=1,2$ and $j=1,\ldots,W$, we set
\begin{eqnarray*}
&&\Phi_{a, j}=\big(\phi_{a,j,1},\ldots,\phi_{a,j,M}\big)',\quad \Psi_{a, j}=\big(\psi_{a,j,1},\ldots,\psi_{a,j,M}\big)'\\
&&\Phi_a=\big(\Phi_{a,1}',\ldots,\Phi_{a,W}'\big)',\quad \Psi_a=\big(\Psi_{a,1}',\ldots,\Psi_{a,W}'\big)'.
\end{eqnarray*}
For each $j$ and each $a$, $\Phi_{a,j}$  is a vector with complex components, and $\Psi_{a,j}$ is a vector  with Grassmann components. In addition, we use $\Phi^*_{a,j}$ and $\Psi^*_{a,j}$ to represent the conjugate transposes of $\Phi_{a,j}$ and $\Psi_{a,j}$ respectively. Analogously, we adopt the notation $\Phi^*_{a}$ and $\Psi^*_{a}$ to represent the  conjugate transposes of $\Phi_{a}$ and $\Psi_{a}$, respectively.
We have the following integral representation for the moments of the Green's function.
\begin{lem}
For any $p,q=1,\ldots, W$ and $\alpha, \beta=1,\ldots, M$, we have 
\begin{align}
|G_{pq,\alpha\beta}(z)|^{2n}&=\frac{1}{(n!)^2}\int {\rm d} \Phi {\rm d}\Psi \; \big(\bar{\phi}_{1,q,\beta}\phi_{1,p,\alpha}\bar{\phi}_{2,p,\alpha}\phi_{2,q,\beta}\big)^n\nonumber\\
&\hspace{2ex}\times \exp\Big\{\mathbf{i}\Psi_1^*(z-H)\Psi_1+\mathbf{i}\Phi_1^*(z-H)\Phi_1-\mathbf{i}\Psi_2^*(\bar{z}-H)\Psi_2-\mathbf{i}\Phi^*_2(\bar{z}-H)\Phi_2\Big\}, \label{122601}
\end{align}
where
\begin{eqnarray*}
{\rm d}\Phi=\prod_{a=1,2}\prod_{j=1}^W\prod_{\alpha'=1}^M\frac{{\rm d}\Re \phi_{a,j,\alpha'}{\rm d} \Im \phi_{a,j,\alpha'}}{\pi},\quad {\rm d}\Psi=\prod_{a=1,2}  \prod_{j=1}^W\prod_{\alpha'=1}^M{\rm d} \bar{\psi}_{a,j,\alpha'}{\rm d}\psi_{a,j,\alpha'}.
\end{eqnarray*}
\end{lem}
\begin{proof} By using Proposition \ref{pro.012801} (i)  with $\ell=n$ and Proposition \ref{pro.012801} (ii) with $\ell=0$, we can get (\ref{122601}) immediately.
\end{proof}
\subsection{Averaging over the Gaussian random matrix} Recall the variance profile $\tilde{S}$ in (\ref{021601}).
Now, we take expectation of the Green's function, i.e average over the random matrix. By elementary Gaussian integral, we get
\begin{align}
\mathbb{E}|G_{pq,\alpha\beta}(z)|^{2n}&= \frac{1}{(n!)^2}\int {\rm d} \Phi {\rm d}\Psi \; \big(\bar{\phi}_{1,q,\beta}\phi_{1,p,\alpha}\bar{\phi}_{2,p,\alpha}\phi_{2,q,\beta}\big)^n\; \exp\Big{\{} \mathbf{i}\sum_{j=1}^W (Tr \breve{X}_jJZ+Tr \breve{Y}_j J Z)\Big{\}}\nonumber\\
&\hspace{2ex}\times\exp\Big{\{}\frac{1}{2M}\sum_{j,k}\tilde{\mathfrak{s}}_{jk}Tr\breve{X}_jJ\breve{X}_k J-\frac{1}{2M}\sum_{j,k}\tilde{\mathfrak{s}}_{jk}Tr\breve{Y}_jJ\breve{Y}_k J\Big{\}} \nonumber\\
&\hspace{2ex}\times\exp\Big{\{} -\frac{1}{M}\sum_{j,k}\tilde{\mathfrak{s}}_{jk} Tr \breve{\Omega}_jJ\breve{\Xi}_kJ\Big{\}}, \label{120201}
\end{align} 
where
\begin{eqnarray*}
J=\text{diag}(1,-1),\quad Z=\text{diag}(z,\bar{z}),
\end{eqnarray*}
and for each $j=1,\ldots, W$, the matrices $\breve{X}_j$, $\breve{Y}_j$, $\breve{\Omega}_j$ and $\breve{\Xi}_j$ are $2\times 2$ blocks of a supermatrix, namely, 
\begin{eqnarray*}
\breve{\mathcal{S}}_j=\left( \begin{array}{c|c} \breve{X}_j &  \breve{\Omega}_j \\ \hline
\breve{\Xi}_j & \breve{Y}_j \end{array} \right) := \left(\begin{array}{cc|cc} \Psi_{1,j}^*\Psi_{1,j} &\Psi_{1,j}^*\Psi_{2,j} &  \Psi_{1,j}^*\Phi_{1,j} &\Psi_{1,j}^*\Phi_{2,j} \\
\Psi_{2,j}^*\Psi_{1,j} &\Psi_{2,j}^*\Psi_{2,j} &  \Psi_{2,j}^*\Phi_{1,j} &\Psi_{2,j}^*\Phi_{2,j}\\ \hline
\Phi_{1,j}^*\Psi_{1,j} &\Phi_{1,j}^*\Psi_{2,j} &  \Phi_{1,j}^*\Phi_{1,j} &\Phi_{1,j}^*\Phi_{2,j} \\
\Phi_{2,j}^*\Psi_{1,j} &\Phi_{2,j}^*\Psi_{2,j} &  \Phi_{2,j}^*\Phi_{1,j} &\Phi_{2,j}^*\Phi_{2,j}\end{array}\right) .
\end{eqnarray*}
\begin{rem}
The derivation of (\ref{120201}) from (\ref{122601}) is quite standard. We refer to the proof of (2.14) in \cite{Sh2014} for more details and will not reproduce it here.
\end{rem}
\subsection{Decomposition of the supermatrices}  
From now on, we split the discussion into the following three cases
\begin{itemize}
\item(Case 1): Entries in the off-diagonal blocks, i.e. $p\neq q$,
\item (Case 2): Off-diagonal entries in the diagonal blocks, i.e. $p=q, \quad \alpha\neq \beta$,
\item (Case 3): Diagonal entries, i.e. $p=q, \quad\alpha=\beta$.
\end{itemize}
For each case, we will perform a decomposition for the supermatrix  $\breve{\mathcal{S}}_j$ ($j=p$ or $q$).  For a vector $\mathbf{v}$ and some index set $\mathsf{I}$, we use $\mathbf{v}^{\langle \mathsf{I}\rangle}$  to denote the subvector  obtained by deleting the $i$-th component of $\mathbf{v}$ for all $i\in \mathsf{I}$. Then, we adopt the notation 
\begin{eqnarray*}
&&\breve{\mathcal{S}}_j^{\langle \mathsf{I}\rangle}=\left( \begin{array}{c|c} \breve{X}_j^{\langle \mathsf{I}\rangle} &  \breve{\Omega}_j^{\langle \mathsf{I}\rangle} \\ \hline
\breve{\Xi}_j^{\langle \mathsf{I}\rangle} & \breve{Y}_j^{\langle \mathsf{I}\rangle} \end{array} \right),\qquad  \breve{\mathcal{S}}_j^{[i]}=\left( \begin{array}{c|c} \breve{X}_j^{[i]} &  \breve{\Omega}_j^{[i]} \\ \hline
\breve{\Xi}_j^{[i]} & \breve{Y}_j^{[i]} \end{array} \right).
\end{eqnarray*}
Here, for $\mathsf{A}=\breve{X}_j$, $\breve{Y}_j$, $\breve{\Omega}_j$ or $\breve{\Xi}_j$, the notation $\mathsf{A}^{\langle \mathsf{I}\rangle}$ is defined via replacing $\Phi_{a,j}$, $\Psi_{a,j}$, $\Phi_{a,j}^*$ and $\Psi_{a,j}^*$ by $\Phi_{a,j}^{\langle \mathsf{I}\rangle}$, $\Psi_{a,j}^{\langle \mathsf{I}\rangle}$, $(\Phi_{a,j}^*)^{\langle \mathsf{I}\rangle}$ and $(\Psi_{a,j}^*)^{\langle \mathsf{I}\rangle}$, respectively, for $a=1,2$, in the definition of  $\mathsf{A}$. In addition,  the notation $\mathsf{A}^{[i]}$ is  defined via replacing $\Phi_{a,j}$, $\Psi_{a,j}$, $\Phi_{a,j}^*$ and $\Psi_{a,j}^*$ by $\phi_{a,j,i}$, $\psi_{a,j,i}$, $\bar{\phi}_{a,j,i}$ and $\bar{\psi}_{a,j,i}$ respectively, for $a=1,2$, in the definition of  $\mathsf{A}$. Moreover, for $\mathsf{A}=\breve{\mathcal{S}}_j$, $\breve{X}_j$, $\breve{Y}_j$, $\breve{\Omega}_j$ or $\breve{\Xi}_j$, we will simply abbreviate $\mathsf{A}^{\langle\{a,b\}\rangle}$ and $\mathsf{A}^{\langle\{a\}\rangle}$ by $\mathsf{A}^{\langle a,b\rangle}$ and $\mathsf{A}^{\langle a\rangle}$, respectively. Note that $\breve{\mathcal{S}}_j^{[i]}$ is of rank-one.
 
For Case 1,  due to symmetry, we can assume $\alpha=\beta=1$. Then we extract two rank-one supermatrices from $\breve{\mathcal{S}}_p$ and $\breve{\mathcal{S}}_q$ such that the quantities  $\bar{\phi}_{2,p,1}\phi_{1,p,1}$ and $\bar{\phi}_{1,q,1}\phi_{2,q,1}$  can be expressed in terms of the entries of these supermatrices.   More specifically, we decompose the supermatrices
\begin{eqnarray}
\breve{\mathcal{S}}_p=\breve{\mathcal{S}}_p^{\langle 1\rangle}+\breve{\mathcal{S}}_p^{[1]},\quad \breve{\mathcal{S}}_q=\breve{\mathcal{S}}_q^{\langle 1\rangle}+\breve{\mathcal{S}}_q^{[1]}. \label{012830}
\end{eqnarray}
Consequently, we can write
\begin{eqnarray}
\bar{\phi}_{1,q,1}\phi_{1,p,1}\bar{\phi}_{2,p,1}\phi_{2,q,1}=(\breve{Y}_q^{[1]})_{12}(\breve{Y}_p^{[1]})_{21}. \label{020903}
\end{eqnarray}

For Case 2, due to symmetry, we can assume that $\alpha=1$, $\beta=2$. Then we extract two rank-one supermatrices from $\breve{\mathcal{S}}_p$, namely,
\begin{eqnarray}
\breve{\mathcal{S}}_p=\breve{\mathcal{S}}_p^{\langle 1,2\rangle}+\breve{\mathcal{S}}_p^{[1]}+\breve{\mathcal{S}}_p^{[2]}. \label{020904}
\end{eqnarray}
Consequently, we can write
\begin{eqnarray}
\bar{\phi}_{1,p,2}\phi_{1,p,1}\bar{\phi}_{2,p,1}\phi_{2,p,2}=(\breve{Y}_p^{[2]})_{12}(\breve{Y}_p^{[1]})_{21}. \label{020910}
\end{eqnarray}

Finally, for Case 3, due to symmetry, we can assume that $\alpha=1$. Then we extract only one rank-one supermatrix from $\breve{\mathcal{S}}_p$, namely,
\begin{eqnarray}
\breve{\mathcal{S}}_p=\breve{\mathcal{S}}_p^{\langle1\rangle}+\breve{\mathcal{S}}_p^{[1]}. \label{020905}
\end{eqnarray}
Consequently, we can write
\begin{eqnarray*}
\bar{\phi}_{1,p,1}\phi_{1,p,1}\bar{\phi}_{2,p,1}\phi_{2,p,1}=(\breve{Y}_p^{[1]})_{12}(\breve{Y}_p^{[1]})_{21}=(\breve{Y}_p^{[1]})_{11}(\breve{Y}_p^{[1]})_{22}.
\end{eqnarray*}

Since the discussion for all three cases are similar, we will only present the details for Case 1. More specifically, in the remaining part of this section and Section \ref{s.5} to Section \ref{s.10}, we will only treat Case 1. In Section \ref{s.12}, we will sum up the discussions in the previous sections and explain how to adapt them to Case 2 and Case 3, resulting a final proof of Theorem \ref{lem.012802}.

\subsection{Variable reduction by superbosonization formulae}  \label{s.4.4}
We will work with Case 1. Recall the decomposition (\ref{012830}). We use the superbosonization formulae to reduce the number of variables. We shall treat $\breve{\mathcal{S}}_k$ ($k\neq p,q$) and $\breve{\mathcal{S}}_j^{\langle1\rangle}$ ($j=p,q$) on an equal footing and use the formula (\ref{010501}) with $r=2,\ell=M$ for the former and $r=2,\ell=M-1$ for the latter, while we separate the terms  $\breve{\mathcal{S}}_j^{[1]}$ ($j=p,q$) and use the formula (\ref{010502}). For simplicity, we introduce the notation
\begin{eqnarray}
\tilde{\mathcal{S}}_j=\left\{\begin{array}{cc}
\breve{\mathcal{S}}_j,\qquad \text{if}\quad j\neq p, q,\\\\
 \breve{\mathcal{S}}_j^{\langle1\rangle},\quad\text{if}\quad j=p, q. 
\end{array}\right.   \label{020907}
\end{eqnarray} 
Accordingly, we will use $\tilde{X}_j$, $\tilde{\Omega}_j$, $\tilde{\Xi}_j$ and $\tilde{Y}_j$ to denote four blocks of $\tilde{\mathcal{S}}_j$. With this notation, we can rewrite (\ref{120201}) with $\alpha=\beta=1$ as
\begin{eqnarray}
&&\mathbb{E}|G_{pq,11}(z)|^{2n}= \frac{1}{(n!)^2}\int {\rm d} \Phi {\rm d}\Psi \; \Big(\bar{\phi}_{1,q,1}\phi_{1,p,1}\bar{\phi}_{2,p,1}\phi_{2,q,1}\Big)^n\;\exp\Big{\{} \mathbf{i}\sum_{j=1}^W \Big(Tr \tilde{X}_jJZ+ Tr \tilde{Y}_j J Z\Big)\Big{\}}\nonumber\\
&&\times\exp\Big{\{}\frac{1}{2M}\sum_{j,k}\tilde{\mathfrak{s}}_{jk}\Big(Tr\tilde{X}_jJ\tilde{X}_k J-Tr\tilde{Y}_jJ\tilde{Y}_k J \Big)\Big{\}}\;\exp\Big{\{} -\frac{1}{M}\sum_{j,k}\tilde{\mathfrak{s}}_{jk} Tr \tilde{\Omega}_jJ\tilde{\Xi}_kJ\Big{\}} \nonumber\\
&&\times\prod_{k=p,q}\exp\Big{\{} \mathbf{i} Tr \breve{X}_k^{[1]}JZ+\mathbf{i} Tr \breve{Y}_k^{[1]}JZ\Big{\}}\;\prod_{k=p,q}\exp\Big{\{}\frac{1}{M}\sum_{j=1}^W\tilde{\mathfrak{s}}_{jk}\Big{(} Tr\tilde{X}_jJ\breve{X}_k^{[1]}J-Tr\tilde{Y}_jJ\breve{Y}_k^{[1]}J\Big{)}\Big{\}}\nonumber\\
&&\times\prod_{k,\ell=p,q}\exp\Big{\{}\frac{\tilde{\mathfrak{s}}_{k\ell}}{2M}\Big{(} Tr\breve{X}_k^{[1]}J\breve{X}_\ell^{[1]}J-Tr\breve{Y}_k^{[1]}J\breve{Y}_\ell^{[1]}J\Big{)}\Big{\}}\;\prod_{k,\ell=p,q}\exp\Big{\{} -\frac{\tilde{\mathfrak{s}}_{k\ell}}{M}Tr\breve{\Omega}_k^{[1]} J\breve{\Xi}_\ell^{[1]} J\Big{\}}\nonumber\\
&&\times \prod_{k=p,q}\exp\Big{\{} -\frac{1}{M}\sum_{j}\tilde{\mathfrak{s}}_{jk}\Big{(} Tr\tilde{\Omega}_j J\breve{\Xi}_k^{[1]} J+Tr\breve{\Omega}_k^{[1]} J\tilde{\Xi}_j J\Big{)}\Big{\}} \label{020908}
\end{eqnarray}

Now, we use the superbosonization formulae, i.e., (\ref{010501}) and (\ref{010502}),  to change to the reduced variables as
\begin{eqnarray}
&&\tilde{X}_j\to X_j,\quad \tilde{Y}_j\to Y_j,\quad \tilde{\Omega}_j\to \Omega_j,\quad \tilde{\Xi}_j\to \Xi_j,\quad j=1,\ldots, W,\nonumber\\
&& \breve{X}_k^{[1]}\to X_k^{[1]},\quad \breve{\Omega}_k^{[1]}\to \boldsymbol{\omega}_k^{[1]} (\mathbf{w}^{[1]}_k)^*, \quad \breve{\Xi}_k^{[1]}\to \mathbf{w}_k^{[1]}\boldsymbol{\xi}_k^{[1]},\quad \breve{Y}_k^{[1]}\to Y_k^{[1]}:=y_k^{[1]}\mathbf{w}^{[1]}_k(\mathbf{w}^{[1]}_k)^*, \quad k=p,q.
\nonumber\\\label{0204110}
\end{eqnarray}
Here, for $j=1,\ldots, W$,  $X_j$ is a $2\times 2$ unitary  matrix; $Y_j$ is a $2\times 2$ positive-definite matrix; $\Omega_j=(\omega_{j,\alpha\beta})$ and $\Xi_j=(\xi_{j,\alpha\beta})$ are $2\times 2$ Grassmann matrices. For $k=p$ or $q$, $X_k^{[1]}$ is a $2\times 2$ unitary matrix;  $y_k^{[1]}$ is a positive variable; $\boldsymbol{\omega}_{k}^{[1]}=(\omega_{k,1}^{[1]}, \omega_{k,2}^{[1]})'$ is a column vector with Grassmann components;   $\boldsymbol{\xi}_{k}^{[1]}=(\xi_{k,1}^{[1]},\xi_{k,2}^{[1]})$  is a row vector with  Grassmann components. In addition, for $k=p,q$,
\begin{eqnarray}
\mathbf{w}^{[1]}_k=\Big({v}^{[1]}_k, {u}^{[1]}_ke^{\mathbf{i}\sigma_k^{[1]}}\Big)',\quad {u}_k^{[1]}=\sqrt{1-({v}_k^{[1]})^2},\quad {v}_k^{[1]}\in \mathbb{I},\quad \sigma_k^{[1]}\in\mathbb{L}. \label{0204100}
\end{eqnarray}
Then by using superbosonization formulae, we arrive at the representation
\begin{eqnarray}
&&\mathbb{E}|G_{pq,11}(z)|^{2n}\nonumber\\
&&=\frac{(-1)^{W}}{(n!)^2\pi^{2W+4}}\int  {\rm d}X^{[1]} {\rm d}\mathbf{y}^{[1]} {\rm d} \mathbf{w}^{[1]}  {\rm d}\boldsymbol{\omega}^{[1]} {\rm d}\boldsymbol{\xi}^{[1]} {\rm d}X {\rm d}Y {\rm d}\Omega {\rm d} \Xi \;\Big{(}y_p^{[1]}y_q^{[1]}\big(\mathbf{w}^{[1]}_q(\mathbf{w}^{[1]}_q)^*\big)_{12}\big(\mathbf{w}^{[1]}_p(\mathbf{w}^{[1]}_p)^*\big)_{21}\Big{)}^n\nonumber\\
&&\times\exp\Big{\{} \mathbf{i}\sum_{j=1}^W \Big(Tr X_jJZ+Tr Y_j J Z\Big)\Big{\}}\;\exp\Big{\{}\frac{1}{2M}\sum_{j,k}\tilde{\mathfrak{s}}_{jk}\Big{(}TrX_jJX_k J-TrY_jJY_k J\Big{)} \Big{\}}\nonumber\\
&&\times\exp\Big{\{} -\frac{1}{M}\sum_{j,k}\tilde{\mathfrak{s}}_{jk} Tr\Omega_jJ\Xi_kJ\Big{\}} \;\prod_{j}\frac{\det^M Y_j}{\det^M\big(X_j- \Omega_j Y_j^{-1}\Xi_j\big)}\; \prod_{k=p,q}\frac{\det\big(X_k-\Omega_kY_k^{-1}\Xi_k\big)}{\det Y_k}\nonumber\\
&&\times \prod_{k=p,q}\exp\Big{\{} \mathbf{i} Tr X_k^{[1]}JZ+\mathbf{i} Tr Y_k^{[1]}JZ \Big{\}}\;\prod_{k=p,q}\exp\Big{\{}\frac{1}{M}\sum_{j=1}^W\tilde{\mathfrak{s}}_{jk} \Big{(}TrX_jJX_k^{[1]}J-TrY_jJY_k^{[1]}J\Big{)}\Big{\}}\nonumber\\
&&\times\prod_{k,\ell=p,q}\exp\Big{\{}\frac{\tilde{\mathfrak{s}}_{k\ell} }{2M}\Big{(}TrX_k^{[1]}JX_\ell^{[1]}J-TrY_k^{[1]}JY_\ell^{[1]}J\Big{)}\Big{\}}\;\prod_{k,\ell=p,q}\exp\Big{\{} -\frac{\tilde{\mathfrak{s}}_{k\ell}}{M}Tr\boldsymbol{\omega}_k^{[1]}(\mathbf{w}_k^{[1]})^* J\mathbf{w}_\ell^{[1]}\boldsymbol{\xi}_\ell^{[1]} J\Big{\}}\nonumber\\
&& \times\prod_{k=p,q} \exp\Big{\{} -\frac{1}{M}\sum_{j=1}^W\tilde{\mathfrak{s}}_{jk} Tr \Omega_j J\mathbf{w}_k^{[1]}\boldsymbol{\xi}_k^{[1]} J\Big{\}}\;\prod_{k=p,q}\exp\Big{\{} -\frac{1}{M}\sum_{j=1}^W\tilde{\mathfrak{s}}_{jk} Tr\boldsymbol{\omega}_k^{[1]}(\mathbf{w}_k^{[1]})^* J\Xi_j J\Big{\}}\nonumber\\
&&\times \prod_{k=p,q}\frac{y_k^{[1]}\Big(y_k^{[1]}-\boldsymbol{\xi}_k^{[1]}(X_k^{[1]})^{-1}\boldsymbol{\omega}_k^{[1]}\Big)^{2}}{\det^2 (X_k^{[1]})},
\label{090230}
\end{eqnarray}
where we used the notation $\mathbf{y}^{[1]}:=(y_p^{[1]},y_q^{[1]})$, $\mathbf{w}^{[1]}:=(\mathbf{w}_p^{[1]}, \mathbf{w}_q^{[1]})$. The differentials in (\ref{090230}) are defined by
\begin{eqnarray*}
&& {\rm d}X^{[1]}:=\prod_{j=p,q} {\rm d} \hat{\mu}(X_p^{[1]}) {\rm d}\hat{\mu}(X_q^{[1]}),\qquad {\rm d}\mathbf{y}^{[1]}:=\prod_{j=p,q}\mathbf{1}(y_j^{[1]}>0) {\rm d}y_j^{[1]},\\\\\
&& {\rm d}\mathbf{w}^{[1]}:=\prod_{j=p,q}{\rm d}\mathbf{w}^{[1]}_j=\prod_{j=p,q}{v}_j^{[1]}{\rm d} {v}_j^{[1]} {\rm d}\sigma_j^{[1]},\qquad {\rm d}\boldsymbol{\omega}^{[1]} {\rm d}\boldsymbol{\xi}^{[1]}:=\prod_{\alpha=1,2} \prod_{j=p,q} \omega_{j,\alpha}^{[1]}\xi_{j,\alpha}^{[1]}.\\\\
&&{\rm d}X:=\prod_{j=1}^W {\rm d} \hat{\mu}(X_j),\qquad {\rm d}Y:=\prod_{j=1}^W {\rm d} \hat{\nu}(Y_j),\qquad {\rm d}\Omega {\rm d}\Xi:= \prod_{\alpha,\beta=1,2}\prod_{j=1}^W {\rm d}\omega_{j,\alpha\beta}{\rm d}\xi_{j,\alpha\beta}.
\end{eqnarray*}
The regions of the integral of $X_j$'s and $X_k^{[1]}$'s  are all $U(2)$, and those of $Y_j$'s are the set of all positive-definite matrices. The integral of $v_k^{[1]}$ ranges over $\mathbb{I}$ and that of $\sigma_k^{[1]}$ ranges over $\mathbb{L}$, for $k=p,q$.

Now we change the variables as $X_jJ\to X_j,Y_jJ \to B_j,\Omega_jJ\to \Omega_j, \Xi_jJ\to\Xi_j$
and perform the scaling $X_j\to -MX_j$, $B_j\to MB_j$, $\Omega_j\to \sqrt{M} \Omega_j$ and $\Xi_j\to\sqrt{M}\Xi_j$.
Consequently, we can write
\begin{eqnarray}
&&\mathbb{E}|G_{pq,11}(z)|^{2n}
= \frac{(-1)^{W}M^{4W}}{(n!)^2\pi^{2W+4}} \int  {\rm d}X^{[1]} {\rm d}\mathbf{y}^{[1]} {\rm d} \mathbf{w}^{[1]}  {\rm d}\boldsymbol{\omega}^{[1]} {\rm d}\boldsymbol{\xi}^{[1]} {\rm d}X {\rm d}B {\rm d} \Omega {\rm d} \Xi \; \exp\Big\{-M\big(K(X)+L(B)\big)\Big\}\nonumber\\
&&\hspace{15ex}\times \mathcal{P}( \Omega, \Xi, X, B)\cdot \mathcal{Q}\big( \Omega, \Xi, \boldsymbol{\omega}^{[1]},\boldsymbol{\xi}^{[1]}, X^{[1]}, \mathbf{y}^{[1]},\mathbf{w}^{[1]}\big) \cdot \mathcal{F}\big(X,B,X^{[1]}, \mathbf{y}^{[1]},\mathbf{w}^{[1]}\big), \label{012871}
\end{eqnarray}
where the functions in the integrand are defined as
\begin{eqnarray}
&&K(X):=-\frac{1}{2}\sum_{j,k}\tilde{\mathfrak{s}}_{jk}Tr X_jX_k+\mathbf{i}E\sum_j TrX_j +\sum_j \log\det X_j,\nonumber\\ 
&&L(B):=\frac12\sum_{j,k}\tilde{\mathfrak{s}}_{jk}Tr B_jB_k-\mathbf{i}E\sum_j Tr B_j -\sum_j \log\det B_j,\nonumber\\ 
&&\mathcal{P}( \Omega, \Xi, X, B):=\exp\Big{\{} -\sum_{j,k}\tilde{\mathfrak{s}}_{jk} Tr  \Omega_j\Xi_k\Big{\}}
\;\prod_j\frac{1}{\det^M\big(1+M^{-1}X_j^{-1} \Omega_jB_j^{-1}\Xi_j\big)}\nonumber\\
&&\hspace{25ex}\times\prod_{k=p,q}\frac{\det\big(X_k+M^{-1} \Omega_k B_k^{-1}\Xi_k\big)}{\det B_k},\nonumber\\ 
&&\mathcal{Q}( \Omega, \Xi, \boldsymbol{\omega}^{[1]},\boldsymbol{\xi}^{[1]}, X^{[1]}, \mathbf{y}^{[1]},\mathbf{w}^{[1]}):=\prod_{k=p,q} \Big(1-(y_k^{[1]})^{-1}\boldsymbol{\xi}_k^{[1]}(X_k^{[1]})^{-1}\boldsymbol{\omega}_k^{[1]}\Big)^{2}\nonumber\\
&&\hspace{10ex}\times\prod_{k=p,q}\exp\Big{\{} -\frac{1}{\sqrt{M}}\sum_{j}\tilde{\mathfrak{s}}_{jk}\Big{(} Tr \Omega_j \mathbf{w}_k^{[1]}\boldsymbol{\xi}_k^{[1]}J +Tr\boldsymbol{\omega}_k^{[1]}(\mathbf{w}_k^{[1]})^* J\Xi_j\Big{)}\Big{\}}\nonumber\\
&&\hspace{10ex}\times\prod_{k,\ell=p,q}\exp\Big{\{} -\frac{1}{M}\tilde{\mathfrak{s}}_{k\ell}Tr\boldsymbol{\omega}_k^{[1]}(\mathbf{w}_k^{[1]})^* J\mathbf{w}_\ell^{[1]}\boldsymbol{\xi}_\ell^{[1]}J \Big{\}} , \nonumber\\  \nonumber\\
&&\hspace{2ex}\mathcal{F}(X,B,X^{[1]}, \mathbf{y}^{[1]},\mathbf{w}^{[1]}):=f(X,X^{[1]})\;g(B, \mathbf{y}^{[1]},\mathbf{w}^{[1]}) , 
\label{012130}
\end{eqnarray}
with
\begin{align}
&\hspace{10ex}f(X,X^{[1]}):=\exp\Big{\{} M\eta\sum_{j=1}^W Tr X_jJ\Big{\}}\; \prod_{k,\ell=p,q}\exp\Big{\{}\frac{\tilde{\mathfrak{s}}_{k\ell}}{2M} TrX_k^{[1]}JX_\ell^{[1]}J\Big{\}}\nonumber\\
&\hspace{15ex}\times\prod_{k=p,q}\frac{1}{\det^2(X_k^{[1]})} \exp\Big{\{} \mathbf{i} Tr X_k^{[1]}JZ-\sum_j\tilde{\mathfrak{s}}_{jk} TrX_jX_k^{[1]}J\Big{\}},\label{012150}\\
&g(B, \mathbf{y}^{[1]},\mathbf{w}^{[1]}):=\exp\Big{\{} -M\eta\sum_{j=1}^W Tr B_jJ\Big{\}}\; \prod_{k,\ell=p,q}\exp\Big{\{}-\frac{\tilde{\mathfrak{s}}_{k\ell}}{2M} TrY_k^{[1]}JY_\ell^{[1]} J\Big{\}}\nonumber\\
&\hspace{2ex}\times  \Big{(}\big(\mathbf{w}^{[1]}_q(\mathbf{w}^{[1]}_q)^*\big)_{12}\big(\mathbf{w}^{[1]}_p(\mathbf{w}^{[1]}_p)^*\big)_{21}\Big{)}^n\;\prod_{k=p,q}(y_k^{[1]})^{n+3}\exp\Big{\{}\mathbf{i} Tr Y_k^{[1]}JZ-\sum_j\tilde{\mathfrak{s}}_{jk} TrB_jY_k^{[1]}J\Big{\}}.\label{0125130}
\end{align}
In (\ref{012871}), the regions of $X_j$'s and $X_k^{[1]}$'s  are all $U(2)$, and those of $B_j$'s are the set of the matrices $A$ satisfying $AJ>0$.  Roughly speaking, here we collected the terms containing the Grassmann variables from $\boldsymbol{\omega}^{[1]}_k$'s and $\boldsymbol{\xi}^{[1]}_k$'s, resulting the factor $\mathcal{Q}( \Omega, \Xi, \boldsymbol{\omega}^{[1]},\boldsymbol{\xi}^{[1]}, X^{[1]}, \mathbf{y}^{[1]},\mathbf{w}^{[1]})$. Then, we put the terms  containing the Grassmann variables in $\Omega_j$'s and $\Xi_j$'s together (except those in $\mathcal{Q}( \Omega, \Xi, \boldsymbol{\omega}^{[1]},\boldsymbol{\xi}^{[1]}, X^{[1]}, \mathbf{y}^{[1]},\mathbf{w}^{[1]})$), resulting the factor $\mathcal{P}( \Omega, \Xi, X, B)$.  Finally, we sorted out the terms containing $\eta$, $y_{k}^{[1]}$ and the variables from $X_k^{[1]}$'s and $\mathbf{w}_k^{[1]}$'s (except those in $\mathcal{Q}( \Omega, \Xi, \boldsymbol{\omega}^{[1]},\boldsymbol{\xi}^{[1]}, X^{[1]}, \mathbf{y}^{[1]},\mathbf{w}^{[1]})$), resulting the factor $\mathcal{F}(X,B,X^{[1]}, \mathbf{y}^{[1]},\mathbf{w}^{[1]})$. This separation indicates the order in which we will perform the integrations.

\subsection{Parametrization for $X$, $B$}  \label{s.4}
Similarly to the discussion in \cite{Sh2014}, we start with some preliminary parameterization. At first, we do the eigendecomposition
\begin{eqnarray}
X_j=P_j^*\hat{X}_j P_j,\quad B_j=Q_j^{-1}\hat{B}_jQ_j,\qquad P_j\in \mathring{U}(2),\quad Q_j\in \mathring{U}(1,1), \label{021707}
\end{eqnarray}
where 
\begin{eqnarray}
\hat{X}_j=\text{diag}(x_{j,1},x_{j,2}),\quad \hat{B}_j=\text{diag}(b_{j,1},-b_{j,2}),\quad x_{j,1},x_{j,2}\in \Sigma,\quad b_{j,1}, b_{j,2}\in \mathbb{R}_+. \label{021613}
\end{eqnarray}
Further, we introduce
\begin{eqnarray}
V_j=P_jP_1^*\in \mathring{U}(2), \quad T_j=Q_jQ_1^{-1}\in \mathring{U}(1,1),\quad j=1,\ldots, W. \label{0129100}
\end{eqnarray}
Especially, we have $V_1=T_1=I$. Now, we parameterize $P_1$, $Q_1$, $V_j$ and $T_j$ for all $j=2,\ldots, W$ as follows
\begin{eqnarray}
&&P_1=\left(\begin{array}{ccc}u &ve^{\mathbf{i}\theta}\\ -ve^{-\mathbf{i}\theta} &u\end{array}\right),\quad V_j=\left(\begin{array}{ccc}u_j &v_je^{\mathbf{i}\theta_j}\\ -v_je^{-\mathbf{i}\theta_j} &u_j\end{array}\right),\nonumber\\ \nonumber\\
&&u=\sqrt{1-v^2},\quad u_j=\sqrt{1-v_j^2},\quad  v, v_j\in\mathbb{I}, \quad \theta,\theta_j\in\mathbb{L},\nonumber\\\nonumber\\
&&Q_1=\left(\begin{array}{ccc}s &te^{\mathbf{i}\sigma}\\ te^{-\mathbf{i}\sigma} &s\end{array}\right),\quad T_j=\left(\begin{array}{ccc}s_j &t_je^{\mathbf{i}\sigma_j}\\ t_je^{-\mathbf{i}\sigma_j} &s_j\end{array}\right),\nonumber\\\nonumber\\
&&s=\sqrt{1+t^2}, \quad s_j=\sqrt{1+t_j^2}, \quad t, t_j\in\mathbb{R}_+,\quad \sigma,\sigma_j\in\mathbb{L}. \label{122708}
\end{eqnarray}
Under the parametrization above, we can express the corresponding differentials as follows.
\begin{eqnarray}
&&{\rm d}X{\rm d}B ={\rm d}\mu(P_1){\rm d}\nu(Q_1)\cdot \prod_{j=2}^W {\rm d}\mu(V_j) {\rm d}\nu(T_j) \cdot \prod_{j=1}^W {\rm d}b_{j,1} {\rm d} b_{j,2} \cdot\frac{{\rm d} x_{j,1}}{2\pi\mathbf{i}} \frac{{\rm d}x_{j,2}}{2\pi \mathbf{i}}\nonumber\\
&&\hspace{10ex} \times 2^W(\pi/2)^{2W} \prod_{j=1}^W(x_{j,1}-x_{j,2})^2 (b_{j,1}+b_{j,2})^2, \label{120901}
\end{eqnarray}
where
\begin{eqnarray*}
{\rm d}\mu(P_1)=2v {\rm d} v \cdot\frac{{\rm d}\theta}{2\pi},\quad {\rm d} \mu(V_j)= 2v_j {\rm d}v_j \cdot \frac{{\rm d} \theta_j}{2\pi},\quad {\rm d}\nu(Q_1)=2t {\rm d}t \cdot \frac{{\rm d}\sigma}{2\pi},\quad {\rm d} \nu(T_j)= 2t_j {\rm d} t_j\cdot \frac{{\rm d}\sigma_j}{2\pi}.
\end{eqnarray*}
In addition, for simplicity, we do the change of variables  
\begin{eqnarray}
 \Omega_j\to P_1^* \Omega_j Q_1,\quad \Xi_j\to Q_1^{-1} \Xi_j P_1.\label{121001}
\end{eqnarray}
Note that the Berezinian of such a change is $1$. After this change, $\mathcal{P}( \Omega, \Xi, X,B,\mathbf{y}^{[1]},\mathbf{w}^{[1]})$ turns out to be independent of $P_1$ and $Q_1$.

To adapt to the new parametrization, we change the notation
\begin{eqnarray}
&&K(X)\to K(\hat{X},V),\quad L(B)\to L(\hat{B},T),\quad \mathcal{P}( \Omega, \Xi, X, B)\to \mathcal{P}( \Omega, \Xi, \hat{X}, \hat{B}, V, T),\nonumber\\
&& \mathcal{Q}( \Omega, \Xi, \boldsymbol{\omega}^{[1]},\boldsymbol{\xi}^{[1]}, X^{[1]}, \mathbf{y}^{[1]},\mathbf{w}^{[1]})\to \mathcal{Q}( \Omega, \Xi, \boldsymbol{\omega}^{[1]},\boldsymbol{\xi}^{[1]}, P_1, Q_1, X^{[1]}, \mathbf{y}^{[1]},\mathbf{w}^{[1]}),\nonumber\\
&&\mathcal{F}(X,B, X^{[1]}, \mathbf{y}^{[1]},\mathbf{w}^{[1]})\to \mathcal{F}(\hat{X},\hat{B}, V, T, P_1, Q_1, X^{[1]}, \mathbf{y}^{[1]},\mathbf{w}^{[1]}),\nonumber\\
&&f(X, X^{[1]})\to f(P_1, V, \hat{X}, X^{[1]}),\quad g(B, \mathbf{y}^{[1]},\mathbf{w}^{[1]})\to g(Q_1, T,\hat{B},\mathbf{y}^{[1]},\mathbf{w}^{[1]}).\label{020206}
\end{eqnarray}
We recall here that $K(X)$ does not depend on $P_1$, as well, 
$L(B)$ does not depend on $Q_1$. Moreover, according to the change (\ref{121001}), we have
\begin{align}
\mathcal{P}( \Omega, \Xi, \hat{X}, \hat{B}, V, T)&=\exp\Big{\{} -\sum_{j,k}\tilde{\mathfrak{s}}_{jk} Tr  \Omega_j\Xi_k\Big{\}}\cdot\prod_j\frac{1}{\det^M\big(1+M^{-1}V_j^*\hat{X}_j^{-1}V_j \Omega_jT_j^{-1}\hat{B}_j^{-1}T_j\Xi_j\big)}\nonumber\\
&\times\prod_{k=p,q}\frac{\det\big(V_k^*\hat{X}_kV_k+M^{-1} \Omega_k T_k^{-1}\hat{B}_k^{-1}T_k\Xi_k\big)}{\det \hat{B}_k} \label{012870}
\end{align}
and 
\begin{align}
&\mathcal{Q}( \Omega, \Xi, \boldsymbol{\omega}^{[1]},\boldsymbol{\xi}^{[1]}, P_1, Q_1, X^{[1]}, \mathbf{y}^{[1]},\mathbf{w}^{[1]})\nonumber\\\nonumber\\
&=\prod_{k=p,q}\exp\Big{\{} -\frac{1}{\sqrt{M}}\sum_{j}\tilde{\mathfrak{s}}_{jk} \Big{(}TrP_1^* \Omega_jQ_1 \mathbf{w}_k^{[1]}\boldsymbol{\xi}_k^{[1]}J +Tr\boldsymbol{\omega}_k^{[1]}(\mathbf{w}_k^{[1]})^* JQ_1^{-1}\Xi_j P_1\Big{)}\Big{\}}\nonumber\\
&\times\prod_{k,\ell=p,q}\exp\Big{\{} -\frac{1}{M}\tilde{\mathfrak{s}}_{k\ell}Tr\boldsymbol{\omega}_k^{[1]}(\mathbf{w}_k^{[1]})^* J(\mathbf{w}_\ell^{[1]})\boldsymbol{\xi}_\ell^{[1]}J \Big{\}}\cdot \prod_{k=p,q}\left(1-(y_k^{[1]})^{-1}\boldsymbol{\xi}_k^{[1]}(X_k^{[1]})^{-1}\boldsymbol{\omega}_k^{[1]}\right)^{2}. \label{122806}
\end{align}

Consequently, using (\ref{120901}), from (\ref{012871}) we can write
\begin{eqnarray}
&&\mathbb{E}|G_{pq,11}(z)|^{2n}
= \frac{M^{4W}}{(n!)^2 8^W\pi^{2W+4}} \int \prod_{j=2}^W {\rm d}\mu(V_j) {\rm d}\nu(T_j) \; \int_{\mathbb{R}_+^{2W}} \prod_{j=1}^W {\rm d}b_{j,1} {\rm d} b_{j,2} \; \oint_ {\Sigma^{2W}} \prod_{j=1}^W{\rm d}x_{j,1}{\rm d}x_{j,2}  \nonumber\\
&&\hspace{5ex}\times \exp\left\{-M\big(K(\hat{X},V)+L(\hat{B},T)\big)\right\}\cdot\prod_{j=1}^W(x_{j,1}-x_{j,2})^2(b_{j,1}+b_{j,2})^2\cdot \mathsf{A}(\hat{X}, \hat{B}, V, T). \label{122707}
\end{eqnarray}
where we introduced the notation
\begin{align}
\mathsf{A}(\hat{X}, \hat{B}, V, T)&:=\int  {\rm d}X^{[1]} {\rm d}\mathbf{y}^{[1]} {\rm d} \mathbf{w}^{[1]}  {\rm d}\boldsymbol{\omega}^{[1]} {\rm d}\boldsymbol{\xi}^{[1]}  {\rm d} \Omega {\rm d} \Xi {\rm d}\mu(P_1){\rm d}\nu(Q_1)\;
\mathcal{P}( \Omega, \Xi, \hat{X}, \hat{B}, V, T)\nonumber\\
&\hspace{2ex}\times \mathcal{Q}( \Omega, \Xi, \boldsymbol{\omega}^{[1]},\boldsymbol{\xi}^{[1]}, P_1, Q_1, X^{[1]}, \mathbf{y}^{[1]},\mathbf{w}^{[1]}) \cdot \mathcal{F}(\hat{X},\hat{B}, V, T, P_1, Q_1, X^{[1]}, \mathbf{y}^{[1]},\mathbf{w}^{[1]}). \label{013115}
\end{align}
In (\ref{122707}), the regions of $V_j$'s are all $\mathring{U}(2)$, and those of $T_j$'s are all $\mathring{U}(1,1)$.
Observe that all Grassmann variables are inside the integrand of the integral $\mathsf{A}(\hat{X}, \hat{B}, V, T)$. Hence, (\ref{122707}) separates the saddle point calculation from the observable $\mathsf{A}(\hat{X}, \hat{B}, V, T)$. 

To facilitate the discussions in the remaining part, we introduce some additional terms and notation here. Henceforth, we will employ the notation
\begin{eqnarray*}
(X^{[1]})^{-1}=\Big\{(X^{[1]}_p)^{-1}, (X^{[1]}_q)^{-1}\Big\},\qquad (\mathbf{y}^{[1]})^{-1}=\Big\{(y_p^{[1]})^{-1}, (y_q^{[1]})^{-1}\Big\}
\end{eqnarray*}
for the collection of inverse matrices and reciprocals, respectively.
For a matrix or a vector $A$ under discussion, we will use the term {\emph{$A$-variables}} to refer to all the variables parametrizing it. For example, $\hat{X}_j$-variables means $x_{j,1}$ and  $x_{j,2}$, and $\hat{X}$-variables refer to the collection of all $\hat{X}_j$-variables. Analogously, we can define the terms $T$-variables, $\mathbf{y}^{[1]}$-variables , $\Omega$-variables and so on. We use another term {\emph{$A$-entries}} to refer to the non-zero entries of $A$. Note that $\hat{X}_j$-variables are just $\hat{X}_j$-entries. However, for $T_j$, they are different, namely,
\begin{eqnarray*}
T_j\text{-variables}:\quad  t_j, \sigma_j,\quad \text{vs.}\quad T_j\text{-entries}: \quad s_j, t_je^{\mathbf{i}\sigma_j}, t_je^{-\mathbf{i}\sigma_j}.
\end{eqnarray*} 
Analogously, we will also use the term {\emph{T}-entries} to refer to the collection of all $T_j$-entries. Then $V$-entries, $\mathbf{w}^{[1]}$-entries, etc. are defined in the same manner. It is easy to check that $Q_1^{-1}$-entries are the same as $Q_1$-entries, up to a sign, as well, $T_j^{-1}$-entries are the same as $T_j$-entries, for all $j=2,\ldots, W$.

Moreover, to simplify the notation, we make the convention here that we will frequently use a {\it dot} to represent all the arguments of a function. That means, for instance, we will write $\mathcal{P}( \Omega, \Xi, \hat{X}, \hat{B}, V, T)$ as $\mathcal{P}(\cdot)$ if there is no confusion.  Analogously, we will also use the abbreviation $\mathcal{Q}(\cdot)$, $\mathcal{F}(\cdot)$, $\mathsf{A}(\cdot)$, and so on.  

Let $\mathbf{a}:=\{a_1,\ldots, a_\ell\}$ be a set of variables, we will adopt the notation
\begin{eqnarray*}
\mathfrak{Q}(\mathbf{a}; \kappa_1, \kappa_2, \kappa_3)
\end{eqnarray*}
to denote the class of all multivariate polynomials $\mathfrak{p}(\mathbf{a})$ in the arguments $a_1,\ldots, a_\ell$ such that the following three conditions are satisfied: (i) The total number of the monomials in $\mathfrak{p}(\mathbf{a})$ is bounded by $\kappa_1$; (ii) the coefficients of all monomials in  $\mathfrak{p}(\mathbf{a})$ are bounded by $\kappa_2$ in magnitude; (iii) the  power of each $a_i$ in each monomial is bounded by $\kappa_3$, for all $i=1,\ldots, \ell$. For example, 
\begin{eqnarray}
5b_{j,1}^{-1}+3b_{j,1}t_j^2+1\in \mathfrak{Q}\big(\{b_{j,1}^{-1}, b_{j,1}, t_j\}; 3, 5, 2\big). \label{021610}
\end{eqnarray}
In addition, we define the subset of $\mathfrak{Q}(\mathbf{a}; \kappa_1, \kappa_2, \kappa_3)$, namely,
\begin{eqnarray*}
\mathfrak{Q}_{\text{deg}}\big(\mathbf{a}; \kappa_1, \kappa_2, \kappa_3\big)
\end{eqnarray*}
consisting of those polynomials in $\mathfrak{Q}(\mathbf{a}; \kappa_1, \kappa_2, \kappa_3)$ such that the degree is bounded by $\kappa_3$, i.e. the total degree of each monomial is bounded by $\kappa_3$. For example
\begin{eqnarray*}
5b_{j,1}^{-1}+3b_{j,1}t_j^2+1\in \mathfrak{Q}_{\text{deg}}\big(\{b_{j,1}^{-1}, b_{j,1}, t_j\}; 3, 5, 3\big).
\end{eqnarray*}

\section{Preliminary discussion on the integrand} \label{s.5}
In this section, we perform a preliminary analysis on the factors of the integrand in (\ref{012871}). For convenience, we introduce the matrix
\begin{eqnarray}
\mathfrak{I}=\bigg(\begin{array}{ccc}
0 &1\\
1&0
\end{array}\bigg). \label{021961}
\end{eqnarray}
\subsection{$\exp\{-M(K(\hat{X},V)+L(\hat{B},T))\}$} \label{s.5.1} Recall the parametrization of $\hat{B}_j$, $\hat{X}_j$, $T_j$ and $V_j$ in 
(\ref{021613}) and (\ref{122708}), as well as the matrices defined in (\ref{0129101}). According to the discussion in \cite{Sh2014}, there are three types of saddle points of this function, namely, 
\begin{itemize} 
\item Type I :\hspace{3ex} For each $j$, \hspace{5ex}$\displaystyle (\hat{B}_j, T_j,\hat{X}_j)=(D_{\pm}, I, D_{\pm})\quad \text{or} \quad (D_{\pm}, I, D_{\mp})$, 

$\hspace{10ex}\theta_j\in \mathbb{L}$, $~~~~v_j=0$  if  $\hat{X}_j=\hat{X}_1$, and  \hspace{1ex}$v_j=1$  if  $\hat{X}_j\neq \hat{X}_1$.
\item Type II :\hspace{3ex}For each $j$, \hspace{5ex}$\displaystyle (\hat{B}_j, T_j,\hat{X}_j)=(D_{\pm}, I, D_{+})$ and $V_j\in \mathring{U}(2)$.
\item Type III : \hspace{1ex} For each $j$, \hspace{5ex}$\displaystyle (\hat{B}_j, T_j,\hat{X}_j)=(D_{\pm}, I, D_{-})$ and $V_j\in \mathring{U}(2)$.
\end{itemize}
(Actually, since $\theta_j$ and $v_j$ vary on continuous sets, it would be more appropriate to use the term {\it saddle manifolds}.)
Note that at each type of saddle points, we have $(\hat{B}_j, T_j)=(D_{\pm}, I)$ for all $j$. We will see that the main contribution to the integral (\ref{012871}) comes from some small vicinities of the Type I saddle points. Furthermore, the contributions from all the Type I saddle points are the same, which can be explained as follows.
At first, by the definition in  (\ref{0129100}), we have $V_1=I$. If we regard $\theta_j$'s in the parametrization of $V_j$'s as fixed parameters, it is easy to see that there are totally $2^W$ choices of Type I saddle points.  Moreover, if $v_j=1$, we can do the transform
\begin{eqnarray*}
\hat{X}_j\to \mathfrak{I}\hat{X}_j\mathfrak{I}=\hat{X}_1,\qquad V_j\to \mathfrak{I} V_j=I \quad \text{in}\quad  \mathring{U}(1,1).
\end{eqnarray*}
Consequently, it suffices to consider two saddle points 
\begin{eqnarray}
(\hat{B}_j, T_j,\hat{X}_j, V_j)=(D_{\pm}, I, D_{\pm}, I),\quad\text{or}\quad (D_{\pm}, I, D_{\mp}, I), \label{021620}
\end{eqnarray} 
corresponding to $\hat{X}_1=D_{\pm}$ or $D_{\mp}$, respectively.  Furthermore, the contributions to the integral (\ref{012871}) from the vicinities of these two saddle points are also the same. To see this, we recall the fact that the original integrand in (\ref{012871}) is a function of the entries of $X_j=P_j^{-1}\hat{X}_jP_j$. Now we do the transform $P_j\to \mathfrak{I}P_j$  and $\hat{X}_j\to \mathfrak{I}\hat{X}_j\mathfrak{I}$ for all $j=1,\ldots, W$ to  change one saddle in (\ref{021620}) to the other. Now,  since the Haar measure on $\mathring{U}(2)$ is invariant under the shift $P_1\to \mathfrak{I}P_1$ , the integral over $P_1$-variables is unchanged.  That means, for Type I saddle points, it suffices to consider 
\begin{itemize}
\item Type I' :\hspace{3ex} For each $j$, \hspace{5ex}$\displaystyle (\hat{B}_j, T_j,\hat{X}_j, V_j)=(D_{\pm}, I, D_{\pm}, I)$. 
\end{itemize}
In summary, the total contribution to the integral (\ref{012871}) from all Type I saddle points is $2^W$ times that from the Type I' saddle point.

Following the discussion in \cite{Sh2014}, we will show in Section \ref{s.6}  that  both  $K(\hat{X},V)-K(D_{\pm}, I)$ and $L(\hat{B},T)-L(D_{\pm}, I)$ have positive real parts, bounded by some positive quadratic forms from below, which allows us to perform the saddle point analysis. In addition, it will be seen that in a vicinity of Type I' saddle point, $\exp\{-M(K(\hat{X},V)+L(\hat{B},T))\}$ is approximately Gaussian.

\subsection{$\mathcal{Q}( \Omega, \Xi, \boldsymbol{\omega}^{[1]},\boldsymbol{\xi}^{[1]}, P_1, Q_1, X^{[1]}, \mathbf{y}^{[1]},\mathbf{w}^{[1]})$} \label{s.5.20}The function $\mathcal{Q}(\cdot)$  contains both the $ \Omega, \Xi$-variables from $\mathcal{P}(\cdot)$, and the $P_1$, $Q_1$, $X^{[1]},\mathbf{y}^{[1]}, \mathbf{w}^{[1]}$-variables from $\mathcal{F}(\cdot)$. In addition, note that in the integrand in (\ref{012871}),  $\mathcal{Q}(\cdot)$ is the only factor containing the $\boldsymbol{\omega}^{[1]}$ and $\boldsymbol{\xi}^{[1]}$- variables. Hence, we can compute the integral
\begin{eqnarray}
\mathsf{Q}\big( \Omega, \Xi, P_1, Q_1, X^{[1]}, \mathbf{y}^{[1]},\mathbf{w}^{[1]}\big):=\int d\boldsymbol{\omega}^{[1]}d\boldsymbol{\xi}^{[1]}\; \mathcal{Q}\big( \Omega, \Xi, \boldsymbol{\omega}^{[1]},\boldsymbol{\xi}^{[1]}, P_1, Q_1, X^{[1]}, \mathbf{y}^{[1]},\mathbf{w}^{[1]}\big) \label{030311}
\end{eqnarray}
at first.
The explicit formula for $\mathsf{Q}( \cdot)$ is complicated and irrelevant for us. From (\ref{122806}) and the definition of the Grassmann integral, it is not difficult to see that $\mathsf{Q}( \cdot)$ is a polynomial of the $(X^{[1]})^{-1}$, $(\mathbf{y}^{[1]})^{-1}$, $\mathbf{w}^{[1]}$, $P_1$, $Q_1$, $\Omega$ and $\Xi$-entries. 
In principle, for each monomial in the polynomial $\mathsf{Q}(\cdot)$, we can combine the Grassmann variables with $\mathcal{P}(\cdot)$, then  perform the integral over $\Omega$ and $\Xi$-variables, whilst we combine the complex variables with $\mathcal{F}(\cdot)$, and perform the integral over $X^{[1]}$, $\mathbf{y}^{[1]}$, $\mathbf{w}^{[1]}$, $P_1$ and $Q_1$-variables. A formal discussion on $\mathsf{Q}(\cdot)$ will be given in Section \ref{s.7.1}. However, the terms from $\mathsf{Q}(\cdot)$ turn out to be irrelevant in our proof. Therefore,  in the arguments with $\mathsf{Q}(\cdot)$  involved, a typical strategy that we will adopt is as follows: we usually neglect $\mathsf{Q}(\cdot)$ at first, and perform the discussion on $\mathcal{P}(\cdot)$ and $\mathcal{F}(\cdot)$ separately, at the end, we make necessary comments on how to slightly modify the discussions to take $\mathsf{Q}(\cdot)$ into account.

\subsection{$\mathcal{P}( \Omega, \Xi, \hat{X}, \hat{B}, V, T)$} We will mainly regard $\mathcal{P}(\cdot)$  as a function of the $\Omega$ and $\Xi$-variables. As mentioned above, we also have some $\Omega$ and $\Xi$-variables from the irrelevant term $\mathsf{Q}(\cdot)$. But we  temporarily ignore them  and regard as if the integral over $ \Omega$ and $\Xi$-variables reads
\begin{eqnarray}
\mathsf{P}(\hat{X}, \hat{B}, V, T):=\int d \Omega d\Xi \; \mathcal{P}( \Omega, \Xi, \hat{X}, \hat{B}, V, T). \label{013116}
\end{eqnarray}
We shall estimate $\mathsf{P}(\cdot)$ in three different regions: (1) the complement of the vicinities of the saddle points; (2) the vicinity of Type I saddle point; (3) the vicinities of Type II and III saddle points, which will be done in Sections \ref{s.7.2}, \ref{s.10.1} and \ref{s.11.1}, respectively. (Definition \ref{021701} gives the precise definition of the vicinities.) In each case we will decompose the function $\mathcal{P}(\cdot)$ as a product of a Gaussian measure  and a multivariate polynomial of Grassmann variables. Consequently, we can employ (\ref{0129102}) to perform the integral of this polynomial against the Gaussian measure, whereby  $\mathsf{P}(\cdot)$ can be estimated. Especially, it turns out that in the vicinity of Type I saddle points, $\mathsf{P}(\cdot)$  is approximately the normalizing constant of the Gaussian measure obtained from $\exp\{-M(K(\hat{X},V)+L(\hat{B},T))\}$.

\subsection{$\mathcal{F}(\hat{X},\hat{B}, V, T, P_1, Q_1, X^{[1]}, \mathbf{y}^{[1]},\mathbf{w}^{[1]})$}
Observe that  $\mathcal{F}$ is the only term containing the energy scale $\eta$. As in the previous discussion of $\mathcal{P}(\cdot)$, here we also ignore the $P_1$, $Q_1$, $X^{[1]},\mathbf{y}^{[1]}, \mathbf{w}^{[1]}$-variables from the irrelevant term $\mathsf{Q}(\cdot)$ temporarily, and investigate the integral 
\begin{align}
\mathsf{F}(\hat{X},\hat{B}, V, T)&=\int d X^{[1]} d\mathbf{y}^{[1]} {\rm d}\mathbf{w}^{[1]} d\mu(P_1) d\nu(Q_1)\; \mathcal{F}(\hat{X},\hat{B}, V, T, P_1, Q_1, X^{[1]}, \mathbf{y}^{[1]},\mathbf{w}^{[1]})\nonumber\\
&=\int d X^{[1]}d\mu(P_1)\; f(\hat{X}, V, P_1)\cdot \int d\mathbf{y}^{[1]} {\rm d}\mathbf{w}^{[1]} d\nu(Q_1)\; g(\hat{B}, T, Q_1, \mathbf{y}^{[1]},\mathbf{w}^{[1]}). \label{013117}
\end{align}
We shall also estimate $\mathsf{F}(\cdot)$ in three different regions: (1) the complement of the vicinities of the saddle points; (2) the vicinity of Type I saddle point; (3) the vicinities of Type II and III saddle points, which will be done in Sections \ref{s.7.3}, \ref{s.10.2} and \ref{s.11.2}, respectively. 

Especially, when we restrict the $\hat{X}$, $\hat{B}$, $V$ and $T$-variables to the vicinity of the Type I saddle points, the above integral  can be performed approximately, resulting our main term, a factor of order $1/(N\eta)^{n+2}$. This step will be done in  Section \ref{s.11}. It is instructive to give a heuristic sketch of this calculation.  At first, we plug the Type I saddle points into (\ref{013117}). We will show  that the integral of $f(\cdot)$ approximately reads
\begin{eqnarray*}
e^{-(a_+-a_-)N\eta}\int d X^{[1]}d\mu(P_1)\; f(D_{\pm}, I, P_1)\sim \frac{1}{N\eta},
\end{eqnarray*}
which is the easy part.
Then, recalling  the definition of $g(\cdot)$ in (\ref{0125130}) and the parameterization (\ref{0204100}), we will show that the integral of $g(\cdot)$ approximately reads
\begin{eqnarray}
&&e^{(a_+-a_-)N\eta}\int d\mathbf{y}^{[1]} {\rm d}\mathbf{w}^{[1]} d\nu(Q_1)\; g(D_{\pm}, I, Q_1, \mathbf{y}^{[1]},\mathbf{w}^{[1]})\nonumber\\
&&\sim \int_0^\infty 2t d t\int_{\mathbb{L}^2} d\sigma_p^{[1]} d\sigma_q^{[1]}  \; e^{\mathbf{i}n\sigma_p^{[1]}}e^{-\mathbf{i}n\sigma_q^{[1]}}\cdot e^{-cN\eta t^2+c_1e^{-\mathbf{i}\sigma_p^{[1]}}t+c_2e^{\mathbf{i}\sigma_q^{[1]}}t}\nonumber\\
&&\sim  \int_0^\infty 2t d t\cdot t^{2n}\cdot e^{-cN\eta t^2}\sim \frac{1}{(N\eta)^{n+1}}, \label{020920}
\end{eqnarray}
where in the second step above we used the fact that
\begin{eqnarray*}
\int_{\mathbb{L}} d\sigma \cdot e^{\mathbf{i}n\sigma}\;e^{ce^{-\mathbf{i}\sigma}t}\sim t^n.
\end{eqnarray*}
We notice that the factor $e^{\mathbf{i}n\sigma_p^{[1]}}e^{-\mathbf{i}n\sigma_q^{[1]}}$ in (\ref{020920}) actually comes from the term 
\begin{eqnarray*}
\Big{(}\big(\mathbf{w}^{[1]}_q(\mathbf{w}^{[1]}_q)^*\big)_{12}\big(\mathbf{w}^{[1]}_p(\mathbf{w}^{[1]}_p)^*\big)_{21}\Big{)}^n
\end{eqnarray*} 
in (\ref{0125130}). This factor brings a strong oscillation to the integrand in the integral (\ref{020920}). In Case 2, an analogous factor will appear, resulting the same estimate as  (\ref{020920}). However, in Case 3, such an oscillating factor is absent, then the estimate for the counterpart of the integral in (\ref{020920}) is of order $1/N\eta$ instead of $1/(N\eta)^{n+1}$. The detailed analysis will be presented in Sections \ref{s.10} and \ref{s.12}.
  
\section{Saddle points and vicinities} \label{s.6} 
In this section, we study the saddle points of $K(\hat{X},V)$ and $ L(\hat{B},T)$ and deform the contours of the $\hat{B}$-variables to pass through the saddle points. Then we introduce and classify some small vicinities of these saddle points. The derivation of the saddle points of $K(\hat{X},V)$ and $ L(\hat{B},T)$ in Section \ref{s.6.1} and \ref{s.6.2} below is essentially the same as the counterpart in \cite{Sh2014}, the only difference is that we are working under a more general setting on $S$. Hence, in Section \ref{s.6.1} and \ref{s.6.2}, we  just sketch the discussion,  list the results, and make necessary modifications to adapt to our setting. In the sequel, we employ the notation
\begin{eqnarray}
&&\mathbf{b}_a:=(b_{1,a},\ldots, b_{W,a}),\quad \mathbf{x}_a:=(x_{1,a},\ldots, x_{W,a}), \qquad a=1,2,\nonumber\\
&&\mathbf{t}:=(t_2,\ldots, t_W),\quad \mathbf{v}:=(v_2,\ldots, v_W),\quad \boldsymbol{\sigma}:=(\sigma_2,\ldots,\sigma_W),\quad \boldsymbol{\theta}:=(\theta_2,\ldots,\theta_W). \label{021705}
\end{eqnarray} 

As mentioned above, later we also need to deform the contours, and discuss the integral over some vicinities of the saddle points, thus it is convenient to introduce a notation for the integral over  specific domains. To this end, for $a=1,2$, we use $\mathbf{I}^b_{a}$ and $\mathbf{I}^x_a$ to denote generic domains of $\mathbf{b}_a$ and $\mathbf{x}_a$ respectively. Analogously, we use $\mathbf{I}^t$  and $\mathbf{I}^v$ to represent generic domains of  $\mathbf{t}$ and $\mathbf{v}$, respectively. These domains will be specified later.
Now, for some collection of domains, we introduce the notation
\begin{eqnarray}
&&\mathcal{I}(\mathbf{I}^{b}_1, \mathbf{I}^b_2, \mathbf{I}^x_1, \mathbf{I}^x_2, \mathbf{I}^t,\mathbf{I}^v):=\frac{M^{4W}}{(n!)^28^{W}\pi^{2W+4}}\int_{\mathbb{L}^{2W-2}} \prod_{j=2}^W\frac{{\rm d}\theta_j}{2\pi}\prod_{j=2}^W \frac{{\rm d}\sigma_j}{2\pi} \nonumber\\
&&  \times\int_{\mathbf{I}^b_1} \prod_{j=1}^W {\rm d} b_{j,1} \int_{\mathbf{I}^b_2} \prod_{j=1}^W {\rm d} b_{j,2} \int_{\mathbf{I}^x_1} \prod_{j=1}^W {\rm d} x_{j,1} \int_{\mathbf{I}^x_2} \prod_{j=1}^W {\rm d} x_{j,2}  \int_{\mathbf{I}^t} \prod_{j=2}^W 2t_j {\rm d} t_j \int_{\mathbf{I}^v} \prod_{j=2}^W 2v_j {\rm d} v_j\nonumber\\
&& \times \exp\left\{-M\big(K(\hat{X},V)+L(\hat{B},T)\big)\right\}\cdot\prod_{j=1}^W (x_{j,1}-x_{j,2})^2(b_{j,1}+b_{j,2})^2\cdot \mathsf{A}(\hat{X}, \hat{B}, V, T). \label{010114}
\end{eqnarray}
For example, we can write (\ref{122707}) as
\begin{eqnarray}
\mathbb{E}|G_{pq,11}(z)|^{2n}=\mathcal{I}(\mathbb{R}_+^W, \mathbb{R}_+^W, \Sigma^W, \Sigma^W, \mathbb{R}^{W-1}_+,\mathbb{I}^{W-1}), \label{020901}
\end{eqnarray}
which is the integral over the full domain.

\subsection{Saddle points of $L(\hat{B},T)$}\label{s.6.1} We introduce the function
\begin{eqnarray}
\Bbbk(a):=\frac{a^2}{2}-\mathbf{i}Ea-\log a, \quad a\in\mathbb{C}. \label{0129999}
\end{eqnarray}
Recalling the definition of $L(\cdot)$ in (\ref{012130}), the decomposition of $B_j$'s in (\ref{021707}) and the definition of $T_j$'s  in (\ref{0129100}), we can write
\begin{align}
L(\hat{B},T)&=-\frac14\sum_{j,k} \mathfrak{s}_{jk} Tr \big(T_j^{-1}\hat{B}_j T_j-T_k^{-1}\hat{B}_kT_k\big)^2+\sum_{j} \Big(\frac12 Tr \hat{B}_j^2-\mathbf{i}ETr \hat{B}_j-\log\det\hat{B}_j\Big)\nonumber\\
&= -\frac14\sum_{j,k} \mathfrak{s}_{jk} Tr (\hat{B}_j -\hat{B}_k)^2+\sum_{j} \Big(\frac12 Tr \hat{B}_j^2-\mathbf{i}ETr \hat{B}_j-\log\det\hat{B}_j\Big)\nonumber\\
&\hspace{2ex}+\frac12\sum_{j,k} \mathfrak{s}_{jk} |(T_kT_j^{-1})_{12}|^2(b_{j,1}+b_{j,2})(b_{k,1}+b_{k,2})\nonumber\\
&=:\ell(\mathbf{b}_1)+\ell(-\mathbf{b}_2)+\ell_S(\hat{B},T), \label{121110}
\end{align}
where we used the notation introduced in (\ref{021705}), and the functions $\ell(\cdot)$ and $\ell_S(\cdot)$ are defined as
\begin{eqnarray}
&&\ell(\mathbf{a}):=-\frac14\sum_{j,k}\mathfrak{s}_{jk}(a_j-a_k)^2+\sum_j\Bbbk(a_j),\quad  \mathbf{a}=(a_1,\ldots, a_W)\in \mathbb{C}^W,\nonumber\\
&& \ell_S(\hat{B},T):=\frac12\sum_{j,k} \mathfrak{s}_{jk} |(T_kT_j^{-1})_{12}|^2(b_{j,1}+b_{j,2})(b_{k,1}+b_{k,2}).\label{020401}
\end{eqnarray}

Following the discussion in \cite{Sh2014} with slight modification (see Section 3 therein), we see that for $|E|\leq\sqrt{2}-\kappa$, the saddle point of $ L(\hat{B},T)$ is 
\begin{eqnarray}
(\hat{B}_j,T_j)=(D_{\pm}, I), \quad \forall\; j=1,\ldots,W, \label{010901}
\end{eqnarray}
where $D_{\pm}$ is defined in (\ref{0129101}).
For simplicity, we will write (\ref{010901}) as $(\hat{B},T)=(D_{\pm},I)$ in the sequel. Observe that
\begin{eqnarray}
L(D_{\pm}, I)=\ell(a_+)+\ell(a_-),\qquad \ell_S(D_{\pm},I)=0.  \label{0129110}
\end{eqnarray}
We introduce the notation
\begin{eqnarray}
\mathring{\ell}_{++}(\mathbf{a}):=\ell(\mathbf{a})-\ell(a_+),\quad \mathring{\ell}_{+-}(\mathbf{a}):=\ell(\mathbf{a})-\ell(a_-)\quad \mathring{\ell}_{--}(\mathbf{a}):=\ell(-\mathbf{a})-\ell(a_-), \label{012645}
\end{eqnarray}
where $\ell(a_+)$  represents the value of $\ell(\mathbf{a})$ at the point $\mathbf{a}=(a_+,\ldots, a_+)$, and $\ell(a_-)$ is defined analogously. Correspondingly, we adopt the notation 
\begin{eqnarray}
\mathring{L}(\hat{B},T):=L(\hat{B},T)-L(D_{\pm}, I)=\mathring{\ell}_{++}(\mathbf{b}_1)+\mathring{\ell}_{--}(\mathbf{b}_2)+\ell_S(\hat{B},T), \label{0129111}
\end{eqnarray}
which is implied by (\ref{121110}), (\ref{0129110}) and (\ref{012645}). Now, for each $j=1,\ldots,W$, we deform the contours of $b_{j,1}$ and $b_{j,2}$  to 
\begin{eqnarray}
b_{j,1}\in \Gamma:=\{ra_+| r\in\mathbb{R}_+\},\quad b_{j,2}\in \bar{\Gamma}=\{-ra_-|r\in \mathbb{R}_+\} \label{122705}
\end{eqnarray}
to pass  through the saddle points of $\hat{B}$-variables, based on the following lemma which will be proved in Section \ref{s.8}.
\begin{lem} \label{lem.010601}With the notation introduced in (\ref{010114}), we have
\begin{eqnarray*}
\mathcal{I}\big(\Gamma^W, \bar{\Gamma}^W, \Sigma^W, \Sigma^W, \mathbb{R}_+^{W-1},\mathbb{I}^{W-1}\big)=\mathcal{I}\big(\mathbb{R}_+^W, \mathbb{R}_+^W, \Sigma^W, \Sigma^W, \mathbb{R}_+^{W-1},\mathbb{I}^{W-1}\big)=\mathbb{E}|G_{pq,11}(z)|^{2n}.
\end{eqnarray*}
\end{lem}
We introduce the notation
\begin{eqnarray}
r_{j,1}=|b_{j,1}|,\quad r_{j,2}=|b_{j,2}|,\quad  j=1,\ldots,W. \label{020138}
\end{eqnarray}
 Along the new contours, we have the following lemma.
\begin{lem}  \label{lem.010101}Suppose that $|E|\leq\sqrt{2}-\kappa$. Let $\mathbf{b}_{1}\in \Gamma^W$, $\mathbf{b}_{2}\in \bar{\Gamma}^W$. We have
\begin{eqnarray}
\hspace{5ex}\Re \mathring{L}(\hat{B},T)
\geq  c\sum_{a=1,2}\sum_{j=1}^W \Big((r_{j,a}-1)^2+(r_{j,a}-\log r_{j,a}-1)\Big)+\Re \ell_S(\hat{B},T)\geq  c\sum_{a=1,2}\sum_{j=1}^W (r_{j,a}-1)^2 \nonumber\\\label{0113500}
\end{eqnarray}
for some positive constant $c$.
\end{lem} 
\begin{proof} Since $|E|\leq\sqrt{2}-\kappa$, we have $\Re(b_{j,1}+b_{j,2})(b_{k,1}+b_{k,2})\geq 0$ for $\mathbf{b}_{1}\in \Gamma^W$ and $\mathbf{b}_{2}\in \bar{\Gamma}^W$, thus $\Re \ell_S(\hat{B},T)\geq 0$, in light of the definition in (\ref{020401}). Consequently, according to (\ref{0129111}), it suffices to prove
\begin{eqnarray}
\Re\mathring{\ell}_{++}(\mathbf{b}_1)+\Re\mathring{\ell}_{--}(\mathbf{b}_2)\geq c\sum_{a=1,2}\sum_{j=1}^W \Big((r_{j,a}-1)^2+(r_{j,a}-\log r_{j,a}-1)\Big) \label{0129125}
\end{eqnarray}
for some positive constant $c$. To see this, we observe the following identities obtained via elementary calculation,
\begin{align*}
\Re\mathring{\ell}_{++}(\mathbf{b}_1)&=\frac{E^2-2}{4}\Big{(}\frac{1}{2}\sum_{j,k} \mathfrak{s}_{jk}(r_{j,1}-r_{k,1})^2-\sum_j(r_{j,1}-1)^2\Big{)}+\sum_{j=1}^W\big(r_{j,1}-\log r_{j,1}-1\big)\nonumber\\
\Re\mathring{\ell}_{--}(\mathbf{b}_2)&=\frac{E^2-2}{4}\Big{(}\frac{1}{2}\sum_{j,k} \mathfrak{s}_{jk}(r_{j,2}-r_{k,2})^2-\sum_j(r_{j,2}-1)^2\Big{)}+\sum_{j=1}^W\big(r_{j,2}-\log r_{j,2}-1\big),
\end{align*}
which together with $|E|\leq\sqrt{2}-\kappa$ and (\ref{0129120}) implies (\ref{0129125}) immediately. Hence, we completed the proof of Lemma \ref{lem.010101}.
\end{proof}
\subsection{Saddle points of $K(\hat{X},V)$} \label{s.6.2}
Analogously, recalling the definition in (\ref{020401}), we can write 
\begin{align}
K(\hat{X},V)&=\frac14\sum_{j,k} \mathfrak{s}_{jk} Tr\big( V_j^*\hat{X}_jV_j-V_k^*\hat{X}_kV_k\big)^2-\sum_j\Big(\frac12Tr\hat{X}_j^2-\mathbf{i}ETr\hat{X}_j-\log\det \hat{X}_j\Big)\nonumber\\
&= \frac14\sum_{j,k} \mathfrak{s}_{jk} Tr\big( \hat{X}_j-\hat{X}_k\big)^2-\sum_j\Big(\frac12Tr\hat{X}_j^2-\mathbf{i}ETr\hat{X}_j-\log\det \hat{X}_j\Big)\nonumber\\
&\hspace{2ex}+ \frac12\sum_{j,k} \mathfrak{s}_{jk}|(V_kV_j^*)_{12}|^2(x_{j,1}-x_{j,2})(x_{k,1}-x_{k,2})\nonumber\\
&= -\ell(\mathbf{x}_1)-\ell(\mathbf{x}_2)+ \ell_S(\hat{X},V), \label{020341} 
\end{align}
where $\ell(\cdot)$ is defined in the first line of (\ref{020401}) and $\ell_S(\hat{X},V)$ is defined as
\begin{eqnarray}
\ell_S(\hat{X},V)=\frac12\sum_{j,k} \mathfrak{s}_{jk}|(V_kV_j^*)_{12}|^2(x_{j,1}-x_{j,2})(x_{k,1}-x_{k,2}).  \label{021720}
\end{eqnarray}

Analogously to the notation $L(D_{\pm},I)$, we will use  $K(D_{\pm},I)$ to represent the value of $K(\hat{X},V)$ at $(\hat{X}_j,V_j)=(D_{\pm}, I)$ for all $j=1,\ldots, W$.  In addition, $K(D_{+},I)$ and $K(D_{-},I)$ are defined in the same manner. Observing that
\begin{eqnarray}
\ell_S(D_{\pm},I)=\ell_S(D_{+},I)=\ell_S(D_{-},I)=0, \label{020343}
\end{eqnarray} 
we have
\begin{eqnarray}
K(D_{\pm}, I)=-\ell(a_+)-\ell(a_-),\quad K(D_{+},I)=-2\ell(a_+),\quad K(D_{-},I)=-2\ell(a_-). \label{020344}
\end{eqnarray}
Moreover, we employ the notation
\begin{eqnarray}
\mathring{K}(\hat{X},V)=K(\hat{X},V)-K(D_{\pm}, I)=-\mathring{\ell}_{++}(\mathbf{x}_1)-\mathring{\ell}_{+-}(\mathbf{x}_2)+\ell_S(\hat{X},V). \label{020345}
\end{eqnarray}
We will need the following elementary observations that are easy to check from (\ref{020344}) and (\ref{020401})
\begin{eqnarray}
K(D_{\pm}, I)+L(D_{\pm}, I)=0, \qquad \Re K(D_{+},I)=\Re K(D_{-},I)=\Re K(D_{\pm},I). \label{011537}
\end{eqnarray}

 In addition, we introduce the $W\times W$ matrix 
\begin{eqnarray}
S^v=(\mathfrak{s}^v_{jk}),\quad \mathfrak{s}^v_{jk}:=\mathfrak{s}_{jk}|(V_kV_j^*)_{12}|^2, \label{011520}
\end{eqnarray}
and the $2W\times 2W$ matrices
\begin{eqnarray}
\mathbb{S}=S\oplus S,\qquad \mathbb{S}^v:=\mathbb{S}+\bigg(\begin{array}{cccc} -S^v &S^v\\ S^v & -S^v\end{array}\bigg), \label{021721}
\end{eqnarray}
where $\mathbb{S}^v$ depends on the $V$-variables according to (\ref{011520}). Here we regard $V$-variables as fixed parameters.
Due to the fact $|(V_kV_j^*)_{12}|\in\mathbb{I}$, it is easy to see that $\mathbb{S}^v$ is a weighted Laplacian of a graph with $2W$ vertices. In particular, $\mathbb{S}^v\leq 0$. By the definition (\ref{011520}), one can see that $S^v_{ii}=0$ for all $i=1,\ldots, W$. Consequently, we can obtain
\begin{eqnarray*}
\sum_{k\neq j} \mathbb{S}^v_{jk}=\sum_{k\neq j} \mathbb{S}^v_{kj}=-\mathbb{S}^v_{jj}=\bigg\{\begin{array}{lll}-\mathfrak{s}_{jj},\qquad \qquad\quad\text{if}\quad j=1,\ldots W\\
-\mathfrak{s}_{j-W,j-W},\qquad \text{if} \quad j=W+1,\ldots, 2W
\end{array}
\end{eqnarray*} 
Similarly to (\ref{0129120}), we get
\begin{eqnarray}
I+\mathbb{S}^v\geq c_0I, \label{0129997}
\end{eqnarray}
where $c_0$ is the constant in Assumption \ref{assu.1} (ii). 
Moreover, it is not difficult to see from the definitions in (\ref{021720}), (\ref{011520}) and (\ref{021721}) that
\begin{eqnarray}
 \frac14\sum_{j,k} \mathfrak{s}_{jk} Tr( \hat{X}_j-\hat{X}_k)^2+ \ell_S(\hat{X},V)=-\frac12 \mathbf{x}'\mathbb{S}^v\mathbf{x}, \label{010201}
\end{eqnarray}
where we used the notation $\mathbf{x}:=(\mathbf{x}_1',\mathbf{x}_2')'$. 

Now let
\begin{eqnarray}
\vartheta_j=\arg x_{j,1},\quad \vartheta_{W+j}=\arg x_{j,2},\quad \forall\; j=1,\ldots, W. \label{0129998}
\end{eqnarray}
Then, recalling the parametrization of $V_j$'s in (\ref{122708}), we have the following lemma. 
\begin{lem} \label{lem.0101.1}Assume that $x_{j,1}, x_{j,2}\in \Sigma$ for all $j=1,\ldots, W$. We have
\begin{eqnarray}
\Re \mathring{K}(\hat{X},V)\geq \frac14\sum_{j,k=1}^{2W}(\mathbb{S}^v)_{jk}(\cos \vartheta_j-\cos \vartheta_k)^2+c\sum_{j=1}^{2W}\Big(\sin \vartheta_j-\frac{E}{2}\Big)^2 \label{0129994}
\end{eqnarray}
for some positive constant $c$. In addition, $\Re \mathring{K}(\hat{X},V)$ attains its minimum $0$ at the following three types of saddle points
\begin{itemize}
\item Type I : \hspace{3ex}For each $j$,\hspace{3ex}$\hat{X}_j=D_{\pm}\quad \text{or}  \quad D_{\mp}$, 

$\hspace{10ex}\theta_j\in \mathbb{L}$ $~~~~v_j=0$  if  $\hat{X}_j=\hat{X}_1$, and  $v_j=1$  if  $\hat{X}_j\neq \hat{X}_1$,
\item Type II :  \hspace{3ex}For each $j$, \hspace{3ex}$\hat{X}_j=D_{+},V_j\in \mathring{U}(2)$,
\item Type III : \hspace{3ex}For each $j$, \hspace{3ex}$\hat{X}_j=D_{-},V_j\in\mathring{U}(2)$,
\end{itemize}
which are the restrictions of three types of saddle points in Section \ref{s.5.1}, on $\hat{X}$ and $V$-variables.
\end{lem}
\begin{rem} The Type I saddle points of $(\hat{X},V)$ are exactly those points satisfying 
\begin{eqnarray*}
V_j^*\hat{X}_jV_j=D_{\pm},\quad \forall\; j=1,\ldots, W,\quad \text{or}\quad V_j^*\hat{X}_jV_j=D_{\mp},\quad \forall\; j=1,\ldots, W.
\end{eqnarray*}
In Lemma \ref{lem.0101.1}, we wrote them in terms of $\hat{X}_j$, $v_j$ and $\theta_j$ in order to evoke the parameterization in (\ref{021613}) and (\ref{122708}).
\end{rem}
\begin{proof} By (\ref{020341}), (\ref{010201}), the definitions of the functions $\ell(\cdot)$ in (\ref{010901}) and $\Bbbk(\cdot)$ in (\ref{0129999}), we can write
\begin{eqnarray*}
\mathring{K}(\hat{X},V)=-\frac12 \mathbf{x}^*\mathbb{S}^v\mathbf{x}-\sum_{j=1}^W\Big{(}\Bbbk(x_{j,1})+\Bbbk(x_{j,2})-\Bbbk(a_+)-\Bbbk(a_-)\Big{)}.
\end{eqnarray*}
 By using (\ref{0129998}) and the fact $|x_{j,a}|=1$ for all $j=1,\ldots, W$ and $a=1,2$, we can obtain via elementary calculation
\begin{eqnarray}
&&\Re \mathring{K}(\hat{X},V)=\frac14\sum_{j,k=1}^{2W}(\mathbb{S}^v)_{jk}(\cos \vartheta_j-\cos \vartheta_k)^2\nonumber\\
&&\hspace{5ex}-\frac14\sum_{j,k=1}^{2W}(\mathbb{S}^v)_{jk}(\sin \vartheta_j-\sin \vartheta_k)^2+\sum_{j=1}^{2W}\Big(\sin \vartheta_j-\frac{E}{2}\Big)^2. \label{0129995}
\end{eqnarray}
In light of the fact $\mathbb{S}^v\leq 0$ and (\ref{0129997}), we have
\begin{eqnarray}
I+\frac12\mathbb{S}^v\geq I+\mathbb{S}^v\geq c_0I.\label{0129996}
\end{eqnarray}
Applying (\ref{0129996}) to the last two terms on the r.h.s. of (\ref{0129995}) yields (\ref{0129994}). 

Now, we show that $\Re \mathring{K}(\hat{X},V)$ attains its minimum $0$ at three types of points listed in Lemma \ref{lem.0101.1}.  It is elementary to check that these points are minimum points along the contour by plugging them into the definition of $\Re\mathring{K}(\hat{X},V)$. In the sequel, we show that they are the only solutions to the equation 
\begin{eqnarray}
\Re \mathring{K}(\hat{X},V)=0. \label{021731}
\end{eqnarray} 
At first, by the second term on the r.h.s. of (\ref{0129994}), we see that for any solution to (\ref{021731}), 
\begin{eqnarray*}
\sin \vartheta_j=\frac{E}{2},\quad \forall\; j=1,\ldots, 2W,
\end{eqnarray*}
which implies that $
x_{j,a}=a_+$ or $a_-$ for all $j=1,\ldots, W$ and $a=1,2$, by recalling the definition (\ref{0129998}) and the definitions of $a_+$ and $a_-$ in Section \ref{s.1.3}. Consequently, for each $j$, $\hat{X}_j$ can only be one of $D_{\pm}$, $D_{\mp}$, $D_{+}$ and $D_{-}$. 

Suppose that $\hat{X}_1=D_{+}$, we claim that $\hat{X}_j=D_{+}$ for all $j$. Otherwise, owing to the fact that the graph $\mathcal{G}$ is connected, there exists  $\{i, j\}\in\mathcal{E}$ such that $\mathfrak{s}_{ij}>0$ and $\hat{X}_i=D_{+}$ but $\hat{X}_j=D_{\pm}$, $D_{\mp}$ or $D_{-}$. Without loss of generality, we assume $\hat{X}_j=D_{\pm}$. In this case, we use the fact
\begin{align}
\Re \mathring{K}(\hat{X},V)&\geq \frac14(\mathbb{S}^v)_{ij}(\cos \vartheta_i-\cos \vartheta_j)^2+\frac14(\mathbb{S}^v)_{i+W,j}(\cos \vartheta_{i+W}-\cos \vartheta_j)^2\nonumber\\
&\hspace{-2ex}+\frac14(\mathbb{S}^v)_{i,j+W}(\cos \vartheta_i-\cos \vartheta_{j+W})^2+\frac14(\mathbb{S}^v)_{i+W,j+W}(\cos \vartheta_{i+W}-\cos \vartheta_{j+W})^2, \label{013101}
\end{align}
which follows from (\ref{0129994}) directly. Now, by the assumption $\hat{X}_i=D_{+}$ while $\hat{X}_j=D_{\pm}$, we have
\begin{eqnarray*}
\cos \vartheta_i=\cos \vartheta_{i+W}=\cos \vartheta_j=\Re(a_+)=\frac{\sqrt{4-E^2}}{2},\quad \cos \vartheta_{j+W}=\Re(a_-)=-\frac{\sqrt{4-E^2}}{2},
\end{eqnarray*}
which together with (\ref{013101}) implies that
\begin{eqnarray}
\Re \mathring{K}(\hat{X},V)\geq \frac{4-E^2}{4}\Big((\mathbb{S}^v)_{i,j+W}+(\mathbb{S}^v)_{i+W,j+W}\Big)= \frac{4-E^2}{4}\mathfrak{s}_{ij}>0, \label{021740}
\end{eqnarray}
contradicting to (\ref{021731}). Analogously, we can show that $\hat{X}_j$ can not be $D_{\mp}$ or $D_{-}$. Consequently, for a solution to (\ref{021731}), if $\hat{X}_1=D_{+}$, we have shown that $\hat{X}_j=D_{+}$ for all $j$. Similarly, we can show that if $\hat{X}_1=D_{-}$, then $\hat{X}_j=D_{-}$ for all $j$. These two kinds of solutions are collected as the Type II and Type III saddle points, respectively. 

What remains is to show that if $\hat{X}_1=D_{\pm}$ or $D_{\mp}$, the solution to (\ref{021731}) must be one of the Type I saddle points. We only show the case of $\hat{X}_1=D_{\pm}$. Assume that $\{1,i\}\in\mathcal{E}$ in the graph $\mathcal{G}$, i.e. $\mathfrak{s}_{1i}>0$.  At first, similarly to the discussion from (\ref{013101}) to (\ref{021740}), we can show that $\hat{X}_i$ can only be $D_{\pm}$ or $D_{\mp}$.  If $\hat{X}_i=D_{\pm}$, then by using (\ref{013101}) with $j=1$, we have
\begin{eqnarray*}
\Re \mathring{K}(\hat{X},V)\geq \frac{4-E^2}{4}\Big((\mathbb{S}^v)_{i,1+W}+(\mathbb{S}^v)_{i+W,1}\Big)=\frac{4-E^2}{2}\mathfrak{s}^v_{1i}\geq 0,
\end{eqnarray*}
and the equality holds if and only if $V_i=I$, according to the assumption $V_1=I$ and the definition in  (\ref{011520}). The discussion on the case of $\hat{X}_i=D_{\mp}$ is analogous. Consequently, we have 
\begin{eqnarray}
V^*_i\hat{X}_iV_i=V^*_1\hat{X}_1V_1=D_{\pm}. \label{013103}
\end{eqnarray} 
Since the graph $\mathcal{G}$ is connected, we can show that (\ref{013103}) holds for all $i=1,\ldots, W$. Analogously, if $\hat{X}_1=D_{\mp}$, we can show that $V^*_j\hat{X}_jV_j=D_{\mp}$ for all $j=1,\ldots, W$. Therefore, we completed the proof of Lemma \ref{lem.0101.1}.
\end{proof}

\subsection{Vicinities of the saddle points} \label{s.6.3}
Having studied the saddle points of $L(\hat{B},T)$ and $K(\hat{X},V)$, we then introduce some small vicinities of them. To this end, we introduce the quantity 
\begin{eqnarray}
\Theta\equiv\Theta(N,\varepsilon_0):=WN^{\varepsilon_0} \label{021102}
\end{eqnarray}
for small positive constant $\varepsilon_0$ which will be chosen later.
Let $\mathbf{a}=(a_{1},\ldots, a_{W})\in\mathbb{C}^W$ be any complex vector. In the sequel, we adopt the notation
\begin{eqnarray}
&&\mathbf{a}+d:=\left(a_{1}+d,\ldots,a_{W}+d\right),\quad d\mathbf{a}:=\left(d a_{1},\ldots,d a_{W}\right),\quad \forall\;  d\in\mathbb{C},\nonumber\\
&&\hspace{20ex} \arg(\mathbf{a}):=\big(\arg(a_1),\ldots, \arg(a_W)\big). \label{021906}
\end{eqnarray} 
Now,  we define the following domains . 
\begin{eqnarray}
&& \Upsilon^b_+\equiv \Upsilon^b_+(N,\varepsilon_0):=\Big\{\mathbf{a}\in \Gamma^W: ||\mathbf{a}-a_+||_2^2\leq \frac{\Theta}{M}\Big\},\nonumber\\
&&  \Upsilon^b_-\equiv \Upsilon^b_-(N,\varepsilon_0):=\Big\{\mathbf{a}\in \bar{\Gamma}^W: ||\mathbf{a}+a_-||_2^2\leq \frac{\Theta}{M}\Big\},\nonumber\\
 &&  \Upsilon^x_+\equiv \Upsilon^x_+(N,\varepsilon_0):=\Big\{\mathbf{a}\in \Sigma^W: ||\arg (a_+^{-1}\mathbf{a})||_2^2\leq \frac{\Theta}{M}\Big\},\nonumber\\
 && \Upsilon^x_-\equiv \Upsilon^x_-(N,\varepsilon_0):=\Big\{\mathbf{a}\in \Sigma^W: ||\arg(a_-^{-1}\mathbf{a})||_2^2\leq \frac{\Theta}{M}\Big\},\nonumber\\
 && \Upsilon_S\equiv \Upsilon_S(N,\varepsilon_0):=\Big\{\mathbf{a}\in \mathbb{R}_+^{W-1}: -\mathbf{a}'S^{(1)}\mathbf{a}\leq \frac{\Theta}{M}\Big\}, 
  \label{021907}
\end{eqnarray}
where the superscripts $b$ and $x$ indicate that these will be domains of the corresponding variables. In order to define the vicinities of the  Type I saddle points properly, we introduce the permutation $\epsilon_j$ of $\{1,2\}$, for each triple $(x_{j,1}, x_{j,2}, v_j)$. Specifically, recalling the fact of $u_j=\sqrt{1-v_j^2}$ from (\ref{122708}), we define
\begin{eqnarray*}
v_{j,\epsilon_j}\equiv v_{j,\epsilon_j}(\epsilon_1):=v_j\mathbf{1}(\epsilon_j=\epsilon_1)+u_j\mathbf{1}(\epsilon_j\neq \epsilon_1).
\end{eqnarray*}
Denoting by $\boldsymbol{\epsilon}=(\epsilon_1,\ldots, \epsilon_W)$ and $\boldsymbol{\epsilon}(a)=(\epsilon_1(a),\ldots, \epsilon_W(a))$ for $a=1,2$, we set
\begin{eqnarray}
\mathbf{x}_{\boldsymbol{\epsilon}(a)}=(x_{1,\epsilon_1(a)},\ldots, x_{W,\epsilon_W(a)}),\qquad a=1,2,\qquad \mathbf{v}_{\boldsymbol{\epsilon}}=(v_{2,\epsilon_2},\ldots, v_{W,\epsilon_W}). \label{021920}
\end{eqnarray}

With this notation, we now define the Type I, II, and III vicinities, parameterized by $(\mathbf{b}_1,\mathbf{b}_2,\mathbf{x}_1,\mathbf{x}_2,\mathbf{t},\mathbf{v})$ of the corresponding saddle point types. We also define the special case of the Type I vicinity, namely, Type I' vicinity, corresponding to the Type I' saddle point defined in Section \ref{s.5.1}.
\begin{defi}\label{021701} ~~~
\begin{itemize}
\item  Type I vicinity :  \hspace{2ex}$\big(\mathbf{b}_1,\mathbf{b}_2,\mathbf{x}_{\boldsymbol{\epsilon}(1)},\mathbf{x}_{\boldsymbol{\epsilon}(2)},\mathbf{t},\mathbf{v}_{\boldsymbol{\epsilon}}\big)\in\Upsilon^b_+\times \Upsilon^b_-\times \Upsilon^x_+\times \Upsilon^x_-\times \Upsilon_S\times \Upsilon_S $ for some $\boldsymbol{\epsilon}$.
\item Type I' vicinity : \hspace{2ex}$\big(\mathbf{b}_1,\mathbf{b}_2,\mathbf{x}_1,\mathbf{x}_2,\mathbf{t},\mathbf{v}\big)\in\Upsilon^b_+\times \Upsilon^b_-\times \Upsilon^x_+\times \Upsilon^x_-\times \Upsilon_S\times \Upsilon_S $.
\item Type II vicinity :  \hspace{1ex}$\big(\mathbf{b}_1,\mathbf{b}_2,\mathbf{x}_1,\mathbf{x}_2,\mathbf{t},\mathbf{v}\big)\in\Upsilon^b_+\times \Upsilon^b_-\times \Upsilon^x_+\times \Upsilon^x_+\times \Upsilon_S\times \mathbb{I}^{W-1}$.
\item Type III vicinity :  $\big(\mathbf{b}_1,\mathbf{b}_2,\mathbf{x}_1,\mathbf{x}_2,\mathbf{t},\mathbf{v}\big)\in\Upsilon^b_+\times \Upsilon^b_-\times \Upsilon^x_-\times \Upsilon^x_-\times \Upsilon_S\times \mathbb{I}^{W-1}$.
\end{itemize}
\end{defi}

In the following discussion, the parameter $\varepsilon_0$ in $\Theta$ is allowed to be different from line to line. However, given $\varepsilon_1$ in (\ref{0124101111111}), we shall always choose $\varepsilon_2$ in (\ref{021101}) and $\varepsilon_0$ in (\ref{021102}) according to the rule
\begin{eqnarray}
C\varepsilon_2\leq \varepsilon_0\leq\frac{\varepsilon_1}{C} \label{0218105}
\end{eqnarray}
 for some sufficiently large $C>0$. Consequently, by Assumption \ref{assu.4} we have
 \begin{eqnarray}
N(\log N)^{-10}\geq M=M^{1-4\varepsilon_0}M^{4\varepsilon_0}\geq  W^{(4+2\gamma+\varepsilon_1)(1-4\varepsilon_0)} M^{4\varepsilon_0}\gg W^{(4+2\gamma+4\varepsilon_0)} M^{4\varepsilon_0}=W^{2\gamma}\Theta^4. \label{021105}
 \end{eqnarray}
To prove Theorem \ref{lem.012802}, we split the task into three steps. The first step is to exclude the integral outside the vicinities. Specifically, we will show the following lemma.
\begin{lem} \label{lem.121601}Under Assumptions \ref{assu.1} and \ref{assu.4}, we have,
\begin{align}
\mathcal{I}\big(\Gamma^{W}, \bar{\Gamma}^{W}, \Sigma^W,\Sigma^W, \mathbb{R}^{W-1}_+,\mathbb{I}^{W-1}\big)&=2^W\mathcal{I}\big(\Upsilon^b_+, \Upsilon^b_-, \Upsilon^x_+, \Upsilon^x_-, \Upsilon_S,\Upsilon_S\big)\nonumber\\
&\hspace{2ex}+ \mathcal{I}\big(\Upsilon^b_+, \Upsilon^b_-, \Upsilon^x_+, \Upsilon^x_+, \Upsilon_S,\mathbb{I}^{W-1}\big)\nonumber\\
&\hspace{2ex}+\mathcal{I}\big(\Upsilon^b_+, \Upsilon^b_-, \Upsilon^x_-, \Upsilon^x_-, \Upsilon_S,\mathbb{I}^{W-1}\big)+O( e^{-\Theta}). \label{013106}
\end{align}
\end{lem}
\begin{rem} The first three terms on the r.h.s. of (\ref{013106}) correspond to the integrals over vicinities of the Type I, II, and III saddle points, respectively. Note that for the first term, we have used the argument in Section \ref{s.5.1}, namely, the total contribution of the integral over  the Type I vicinity is $2^W$ times that over the Type I' vicinity.
\end{rem}

The second step, is to estimate the integral over  the Type I vicinity. We have the following  lemma.
\begin{lem} \label{lem.010602}Under Assumptions \ref{assu.1} and \ref{assu.4}, there exists some positive constant $C_0$ uniform in $n$ and some positive number $N_0=N_0(n)$ such that for all $N\geq N_0$, 
\begin{eqnarray}
2^W\big|\mathcal{I}(\Upsilon^b_+, \Upsilon^b_-, \Upsilon^x_+,\Upsilon^x_-, \Upsilon_S, \Upsilon_S)\big|\leq \frac{N^{C_0}}{(N\eta)^{n}}.  \label{020436}
\end{eqnarray}
\end{lem}

The last step is to show that the integral over  the Type II and III vicinities are also negligible.  
\begin{lem}  \label{lem.122801}Under Assumptions \ref{assu.1} and \ref{assu.4}, there exists some positive constant $c$ such that,
\begin{eqnarray*}
\mathcal{I}\big(\Upsilon^b_+, \Upsilon^b_-, \Upsilon^x_+, \Upsilon^x_+, \Upsilon_S,\mathbb{I}^{W-1}\big)=O(e^{-c W}),\qquad 
\mathcal{I}\big(\Upsilon^b_+, \Upsilon^b_-, \Upsilon^x_-, \Upsilon^x_-, \Upsilon_S,\mathbb{I}^{W-1}\big)=O(e^{-c W}).
\end{eqnarray*}
\end{lem}

Therefore, the remaining task  is to prove Lemmas \ref{lem.010601}, \ref{lem.121601},  \ref{lem.010602} and \ref{lem.122801}. For the convenience of the reader, we outline the organization of the subsequent part as follows. 

 At first, the proofs of Lemmas \ref{lem.010601} and \ref{lem.121601} require a discussion on the bound of the integrand, especially on the term $\mathsf{A}(\cdot)$, which contains the integral over all the Grassmann variables. To this end, we perform a crude analysis for the function $\mathsf{A}(\cdot)$ in Section \ref{s.7} in advance, with which we are able to prove Lemmas \ref{lem.010601} and \ref{lem.121601} in Section \ref{s.8}. 
 
 Then, we can restrict ourselves to the integral over the vicinities, i.e.,  prove Lemmas \ref{lem.010602} and \ref{lem.122801}. It will be shown that in the vicinity $(\mathbf{b}_1,\mathbf{b}_2, \mathbf{t})\in \Upsilon_+^b\times \Upsilon_-^b\times \Upsilon_S$, the factor $\exp\{-M\mathring{L}(\hat{B},T)\}$ is approximately the product of a complex Gaussian measure of the $\hat{B}$-variables and a real Gaussian measure of the $\mathbf{t}$-variables. Here,  by ``complex Gaussian measure'' we mean a function of the form $ \exp\{-\mathbf{u}'\mathbb{A}\mathbf{u}\}$, where $\mathbf{u}$ is a real vector, while $\mathbb{A}$ is a complex matrix with positive-definite Hermitian part. In order to estimate the integral against this Gaussian measure (in an approximate sense), we shall  get rid of the $o(1)$ term in the integral of the form
 \begin{eqnarray}
 \int {\rm d} \mathbf{u} \; \exp\big\{-\mathbf{u}'\mathbb{A}\mathbf{u}+o(1)\big\}\; \mathfrak{f}(\mathbf{u}), \label{011550}
 \end{eqnarray}
for some function $\mathfrak{f}$, which however cannot be done directly, owing to the fact that $\mathbb{A}$ is complex. In our case, this problem can be solved by further deforming the contours of the $\hat{B}$-variables, following the steepest descent paths exactly in the vicinity. By doing this, we can get a real Gaussian measure, thus the remainder terms can be easily controlled when integrate against this measure. The situation for $\exp\{-M\mathring{K}(\hat{X},V)\}$ is a little bit more complicated due to different types of the saddle points. However, in the Type I vicinity, we can do the same thing. Hence, in Section \ref{s.9}, we will analyze the Gaussian measure (in an approximate sense) $\exp\{-M(\mathring{K}(\hat{X},V)+\mathring{L}(\hat{B},T))\}$, especially, we will further deform the contours of $\hat{X}$ and $\hat{B}$-variables in the vicinities, whereby we can prove Lemmas \ref{lem.010602}  in Section \ref{s.11}. In the Type II and III vicinities, we will bound $\exp\{-M\mathring{K}(\hat{X},V)\}$ by its absolute value directly. It turns out to be enough for our proof of Lemma \ref{lem.122801}, which is given in Section \ref{s.10}.
\section{Crude bound on $\mathsf{A}(\hat{X}, \hat{B}, V, T)$} \label{s.7}
In this section, we  provide a bound on the function $\mathsf{A}(\cdot)$ in terms of the $\hat{B}, T$-variables, which holds on all the domains under discussion in the sequel. Here, by {\emph{crude bound}} we mean a bound of order $\exp\{O(WN^{\varepsilon_2})\})$, which will be specified in 
Lemma \ref{lem.020203} below. By the definition in (\ref{013115}), we see that $\mathsf{A}(\cdot)$ is an integral of the product of $\mathcal{Q}(\cdot)$, $\mathcal{P}(\cdot)$ and $\mathcal{F}(\cdot)$. We will mainly treat $\mathcal{Q}(\cdot)$ as a function of $\boldsymbol{\omega}^{[1]},\boldsymbol{\xi}^{[1]}$-variables, treat $\mathcal{P}(\cdot)$ as a function of $ \Omega,\Xi$-variables,  and treat $\mathcal{F}(\cdot)$ as a function of $X^{[1]}, \mathbf{y}^{[1]},\mathbf{w}^{[1]},P_1, Q_1$-variables. However, in the function $\mathcal{Q}(\cdot)$, we actually have every argument mentioned above.  Hence, we perform the integral over $\boldsymbol{\omega}^{[1]}$-variables and $\boldsymbol{\xi}^{[1]}$-variables for $\mathcal{Q}(\cdot)$ at first. The resulting function $\mathsf{Q}(\cdot)$ turns out to be a polynomial of the remaining arguments. 
As mentioned in Section \ref{s.5.20}, a typical procedure we will adopt is to ignore $\mathsf{Q}(\cdot)$ at first,  then estimate the integrals of $\mathcal{P}(\cdot)$ and $\mathcal{F}(\cdot)$, which are denoted by $\mathsf{P}(\cdot)$ and $\mathsf{F}(\cdot)$,  respectively (see (\ref{013116}) and (\ref{013117})), finally, we make necessary comment on how to modify the bounding scheme to take $\mathsf{Q}(\cdot)$ into account, whereby we can get the desired bound for $\mathsf{A}(\cdot)$. 

For the sake of simplicity, from now on, we will use the notation
\begin{eqnarray}
&&\omega_{j,1}:=\omega_{j,11},\quad \omega_{j,2}:=\omega_{j,12},\quad \omega_{j,3}:=\omega_{j,21},\quad \omega_{j,4}:=\omega_{j,22},\nonumber\\
&&\xi_{j,1}:=\xi_{j,11},\quad \xi_{j,2}:=\xi_{j,21},\quad \xi_{j,3}:=\xi_{j,12},\quad \xi_{j,4}:=\xi_{j,22}. \label{010301}
\end{eqnarray}
Moreover,  we introduce the domains
\begin{eqnarray}
&&\widehat{\Sigma}:=\Big\{re^{\mathbf{i}\vartheta}: |r-1|\leq \frac{1}{10}, \vartheta\in\mathbb{L}\Big\},\nonumber\\\nonumber\\
&&\mathbb{K}\equiv\mathbb{K}(E):=\left\{\begin{array}{ccc}\Big\{\omega\in\mathbb{C}:0\leq \arg \omega\leq \frac{\arg a_+}{2}+\frac{\pi}{8}\Big\},\quad \text{if}\quad E\geq 0,\\\\
\Big\{\omega\in\mathbb{C}:\frac{\arg a_+}{2}-\frac{\pi}{8}\leq \arg \omega\leq 0\Big\},\quad \text{if}\quad E<0.
\end{array}
\right. \label{030901}
\end{eqnarray}
By the assumption that $|E|\leq \sqrt{2}-\kappa$ in (\ref{021101}), it is easy to see that $|\arg\omega|\leq \pi/4-c$ for all $\omega\in \mathbb{K}\cup\bar{\mathbb{K}}$, where $c$ is some positive constant depending on $\kappa$. Our aim is to show the following lemma.
\begin{lem} \label{lem.020203} Suppose that $\mathbf{b}_1,\mathbf{b}_2,\mathbf{x}_1,\mathbf{x}_2\in \mathbb{K}^W\times \bar{\mathbb{K}}^W\times \widehat{\Sigma}^W\times \widehat{\Sigma}^W$. Under the assumption of Theorem \ref{lem.012802}, we have
\begin{eqnarray*}
|\mathsf{A}(\hat{X},\hat{B}, V, T)|\leq  e^{O(WN^{\varepsilon_2})}\prod_{j=1}^W \Big(r_{j,1}^{-1}+r_{j,2}^{-1}+t_j+1\Big)^C:=  e^{O(WN^{\varepsilon_2})} \mathfrak{p}(\mathbf{r}^{-1},\mathbf{t}). 
\end{eqnarray*}
\end{lem}
\begin{rem}Obviously, using the terminology introduced at the end of Section \ref{s.4},  we have
\begin{eqnarray}
 \mathfrak{p}(\mathbf{r}^{-1},\mathbf{t})\in \mathfrak{Q}\Big(\{r_{j,1}^{-1},r_{j,2}^{-1},t_j\}_{j=1}^{W}; \kappa_1,\kappa_2,\kappa_3\Big)
,\quad 
\kappa_1=e^{O(W)},\quad \kappa_2, \kappa_3=O(1). \label{011150}
\end{eqnarray}
\end{rem}

\subsection{Integral of $\mathcal{Q}$} \label{s.7.1}
In this section, we investigate the function
\begin{eqnarray}
\qquad\mathsf{Q}( \Omega, \Xi, P_1, Q_1, X^{[1]}, \mathbf{y}^{[1]},\mathbf{w}^{[1]}):=\int {\rm d}\boldsymbol{\omega}^{[1]}{\rm d}\boldsymbol{\xi}^{[1]}\; \mathcal{Q}( \Omega, \Xi, \boldsymbol{\omega}^{[1]},\boldsymbol{\xi}^{[1]}, P_1, Q_1, X^{[1]}, \mathbf{y}^{[1]},\mathbf{w}^{[1]}). \label{020130}\end{eqnarray}
Recall $\mathfrak{Q}_{\text{deg}}(\mathbf{a}; \kappa_1, \kappa_2, \kappa_3)$ defined at the end of Section \ref{s.4}, the parameterization in (\ref{0204100}) and (\ref{122708}) and the notation introduced in (\ref{010301}). We shall show the following lemma.
\begin{lem} \label{lem.0202081}If we regard $\sigma$, ${v}_p^{[1]}$, ${v}_q^{[1]}$,  $P_1$ and $(X^{[1]})^{-1}$-entries as fixed parameters, we have
\begin{eqnarray}
\mathsf{Q}(\cdot)\in \mathfrak{Q}_{\text{deg}}(\mathfrak{S}; \kappa_1,\kappa_2,\kappa_3), \qquad \kappa_1=W^{O(1)},\quad \kappa_2, \kappa_3=O(1), \label{020101}
\end{eqnarray} 
where $\mathfrak{S}$ is the set of variables defined by
\begin{eqnarray*}
\mathfrak{S}:=\Big{\{}t,s, (y^{[1]}_k)^{-1}, e^{\mathbf{i}\sigma_k^{[1]}}, e^{-\mathbf{i}\sigma_k^{[1]}}, \frac{\omega_{i,a}\xi_{j,b}}{M}\Big{\}}_{\substack{i,j=1,\ldots,W;\\k=p,q; a,b=1,\ldots, 4}}.
\end{eqnarray*}
\end{lem}
\begin{proof}
Note that $\mathcal{Q}(\cdot)$ can be regarded as a function of the Grassmann variables in $\boldsymbol{\omega}^{[1]}$ and $\boldsymbol{\xi}^{[1]}$. Hence, by the definition in (\ref{122806}), it is a polynomial of these variables with bounded degree.  To have a closer look on this polynomial, we start with
\begin{eqnarray}
\exp\Big{\{} -\frac{1}{\sqrt{M}}\sum_{j}\tilde{\mathfrak{s}}_{jk} TrP_1^* \Omega_jQ_1 \mathbf{w}_k^{[1]}\boldsymbol{\xi}_k^{[1]}J\Big{\}},\qquad k=p,q. \label{020102}
\end{eqnarray}
Observe that, in the exponent of (\ref{020102}), one $\boldsymbol{\xi}_k^{[1]}$-variable must be accompanied by one $\Omega$-variable. In addition, we combine the factor $1/\sqrt{M}$ with $\Omega_j$'s. Then, by Taylor expansion with respect to $\boldsymbol{\xi}_k^{[1]}$-variables, it is easy to find
\begin{eqnarray}
\exp\Big{\{} -\frac{1}{\sqrt{M}}\sum_{j}\tilde{\mathfrak{s}}_{jk} TrP_1^* \Omega_jQ_1 \mathbf{w}_k^{[1]}\boldsymbol{\xi}_k^{[1]}J\Big{\}}\in \mathfrak{Q}_{\text{deg}}(\mathfrak{S}_{1,k}; \kappa_1,\kappa_2,\kappa_3),\quad \kappa_1=W^{O(1)}, \quad \kappa_2, \kappa_3=O(1),\nonumber\\
\label{021741}
\end{eqnarray}
where
\begin{eqnarray*}
\mathfrak{S}_{1,k}:=\Big{\{}t,s,e^{\mathbf{i}\sigma_k^{[1]}},\frac{\omega_{j,a}\xi_{k,b}^{[1]}}{\sqrt{M}}\Big{\}}_{\substack{j=1,\ldots, W;\\a=1,\ldots, 4; b=1,2}},\qquad k=p,q.
\end{eqnarray*}
Analogously, we can show that for $k=p,q$
\begin{align}
&\exp\Big{\{} -\frac{1}{\sqrt{M}}\sum_{j}\tilde{\mathfrak{s}}_{kj} Tr\boldsymbol{\omega}_k^{[1]}(\mathbf{w}_k^{[1]})^* JQ_1^{-1}\Xi_j P_1 \Big{\}}\in \mathfrak{Q}_{\text{deg}}(\mathfrak{S}_{2,k}; \kappa_1,\kappa_2,\kappa_3),\qquad  \kappa_1=W^{O(1)},\quad \kappa_2, \kappa_3=O(1), \nonumber\\
&~\label{021742}
\end{align}
where
\begin{eqnarray*}
\mathfrak{S}_{2,k}:=\Big{\{}t,s,e^{-\mathbf{i}\sigma_k^{[1]}},\frac{\xi_{j,a}\omega_{k,b}^{[1]}}{\sqrt{M}}\Big{\}}_{\substack{j=1,\ldots, W;\\a=1,\ldots, 4; b=1,2}},\quad k=p,q.
\end{eqnarray*}
Here we have used the fact that $Q_1^{-1}$-entries are the same as $Q_1$-entries, up to a sign.
In a similar manner, one can show that for $k,\ell=p,q$,
\begin{eqnarray}
\exp\Big{\{}-\frac{1}{M}\tilde{\mathfrak{s}}_{k\ell}Tr\boldsymbol{\omega}_k^{[1]}(\mathbf{w}_k^{[1]})^* J(\mathbf{w}_\ell^{[1]})\boldsymbol{\xi}_\ell^{[1]} J \Big{\}}\in \mathfrak{Q}_{\text{deg}}(\mathfrak{S}_{3,k,\ell};\kappa_1,\kappa_2,\kappa_3),\quad  \kappa_1,\kappa_2,\kappa_3=O(1), \label{021743} 
\end{eqnarray}
and
\begin{eqnarray}
\Big(1-(y_k^{[1]})^{-1}\boldsymbol{\xi}_k^{[1]}(X_k^{[1]})^{-1}\boldsymbol{\omega}_k^{[1]}\Big)^{2}\in \mathfrak{Q}_{\text{deg}}(\mathfrak{S}_{4,k};\kappa_1,\kappa_2,\kappa_3),\quad \kappa_1,\kappa_2,\kappa_3=O(1), \label{021744}
\end{eqnarray}
where
\begin{eqnarray*}
&&\mathfrak{S}_{3,k,\ell}:=\Big{\{}{e^{\mathbf{i}\sigma_k^{[1]}}}, {e^{-\mathbf{i}\sigma_\ell^{[1]}}},\omega_{k,a}^{[1]}\xi_{\ell,b}^{[1]}\Big{\}}_{a,b=1,2},\quad \mathfrak{S}_{4,k}:=\Big{\{}(y_k^{[1]})^{-1},\omega_{k,a}^{[1]}\xi_{\ell,b}^{[1]}\Big{\}}_{a,b=1,2}.
\end{eqnarray*}

Hence, by (\ref{021741})-(\ref{021744}) and (\ref{122806}), we see that
\begin{eqnarray}
\mathcal{Q}(\cdot)\in \mathfrak{Q}_{\text{deg}}(\mathfrak{S}_5;\kappa_1,\kappa_2,\kappa_3),\quad \kappa_1=W^{O(1)},\quad \kappa_2,\kappa_3=O(1), \label{020131}
\end{eqnarray}
where
\begin{eqnarray*}
\mathfrak{S}_5:=\Big{\{}t,s, (y^{[1]}_k)^{-1}, e^{\mathbf{i}\sigma_k^{[1]}}, e^{-\mathbf{i}\sigma_k^{[1]}},\frac{\omega_{j,r}\xi_{k,b}^{[1]}}{\sqrt{M}},\frac{\xi_{j,r}\omega_{k,b}^{[1]}}{\sqrt{M}},\omega_{k,a}^{[1]}\xi_{\ell,b}^{[1]}\Big{\}}_{\substack{j=1,\ldots, W; k=p,q;\\ r=1,\ldots, 4; a, b=1,2}}.
\end{eqnarray*}
By the definition in (\ref{020130}), $\mathsf{Q}(\cdot)$ is the integral of $\mathcal{Q}(\cdot)$ over the $\boldsymbol{\omega}^{[1]}$ and $\boldsymbol{\xi}^{[1]}$-variables. Now, we regard all the other variables in $\mathfrak{S}_5$, except $\boldsymbol{\omega}^{[1]}$ and $\boldsymbol{\xi}^{[1]}$-variables, as parameters. By the definition of Grassmann integral, we know that only the coefficient of the highest order term $\prod_{k=p,q}\prod_{a=1,2}\omega^{[1]}_{k,a}\xi^{[1]}_{k,a}$ in $\mathcal{Q}(\cdot)$ survives after integrating $\boldsymbol{\omega}^{[1]}$ and $\boldsymbol{\xi}^{[1]}$-variables out.
Then, it is easy to see  (\ref{020101}) from (\ref{020131}), completing the proof.
\end{proof}
\subsection{Integral of  $\mathcal{P}$} \label{s.7.2}In this subsection, we temporarily ignore the $\Omega$ and $\Xi$-variables from $\mathsf{Q}(\cdot)$, and estimate  $\mathsf{P}(\hat{X}, \hat{B}, V, T)$ defined in (\ref{013116}).
Recalling  $r_{j,1}$ and $r_{j,2}$ defined in (\ref{020138}), we can formulate our estimate as follows.
\begin{lem} Suppose that the assumptions in Lemma \ref{lem.020203} hold. We have
\begin{eqnarray}
|\mathsf{P}(\hat{X}, \hat{B}, V, T)|\leq e^{O(W)}\prod_{j=1}^W \big(r_{j,1}^{-1}+r_{j,2}^{-1}+t_j+1\big)^{O(1)}. \label{010711}
\end{eqnarray}
\end{lem}
\begin{proof}
We start with one factor from $\mathcal{P}(\cdot)$ (see (\ref{012870})), namely
\begin{align}
\boldsymbol{\varpi}_j &:=\frac{1}{\det^{M}\big(1+M^{-1}V_j^*\hat{X}_j^{-1}V_j \Omega_jT_j^{-1}\hat{B}_j^{-1}T_j\Xi_j\big)}\nonumber\\
&=\exp\Big{\{} -M\log\det\big(1+M^{-1}V_j^*\hat{X}_j^{-1}V_j \Omega_jT_j^{-1}\hat{B}_j^{-1}T_j\Xi_j\big)\Big{\}}\nonumber\\
&= 1+\sum_{\ell=1}^4 \frac{1}{M^{\ell-1}} \mathfrak{p}_\ell(\hat{X}_j, \hat{B}_j, V_j,T_j, \Omega_j, \Xi_j).\label{011611}
\end{align}
Here $\mathfrak{p}_\ell(\cdot)$ is a polynomial in $\hat{X}_j^{-1}$, $\hat{B}_j^{-1}$, $V_j$, $T_j$, $\Omega_j$ and $\Xi_j$-entries with bounded degree and bounded coefficients. Here we used the fact that $V_j^*$ and $T_j^{-1}$-entries are the same as $V_j$ and $T_j$-entries, respectively, up to a sign. Moreover, if we regard $\mathfrak{p}_\ell(\cdot)$ as a polynomial of $\Omega_j$ and $\Xi_j$-entries, it is homogeneous, with degree $2\ell$, and the total degree for $\Omega_j$-variables is $\ell$, thus that for $\Xi_j$-entries is also $\ell$. More specifically, we can write
\begin{eqnarray*}
\mathfrak{p}_\ell(\hat{X}_j, \hat{B}_j, V_j,T_j, \Omega_j, \Xi_j)=\sum_{\substack{\alpha_1,\ldots,\alpha_\ell,\\ ~~\beta_1,\ldots, \beta_\ell=1}}^4 \mathfrak{p}_{\ell,\boldsymbol{\alpha},\boldsymbol{\beta}}(\hat{X}_j, \hat{B}_j, V_j,T_j) \prod_{i=1}^{\ell} \omega_{j,\alpha_{i}}\xi_{j, \beta_{i}},
\end{eqnarray*}
where we used the notation in (\ref{010301}) and denoted $\boldsymbol{\alpha}=(\alpha_{1},\ldots, \alpha_{\ell})$ and $\boldsymbol{\beta}=(\beta_{1},\ldots, \beta_{\ell})$. It is easy to verify that $\boldsymbol{\varpi}_j$ is of the form (\ref{011611}) by taking Taylor expansion with respect to the Grassmann variables. The expansion in (\ref{011611}) terminates at $\ell=4$, owing to the fact that there are totally $8$ Grassmann variables from $\Omega_j$ and $\Xi_j$.
In addition, it is also easy to check that $\mathfrak{p}_{\ell,\boldsymbol{\alpha},\boldsymbol{\beta}}(\cdot)$ is a polynomial of $\hat{X}_j^{-1}$, $\hat{B}_j^{-1}$, $V_j$, $T_j$-entries with bounded degree and bounded coefficients, which implies that there exist two positive constants $C_1$ and $C_2$, such that
\begin{eqnarray}
|\mathfrak{p}_{\ell,\boldsymbol{\alpha},\boldsymbol{\beta}}(\cdot)|\leq C_1\big(r_{j,1}^{-1}+r_{j,2}^{-1}+t_j+1\big)^{C_2} \label{010912}
\end{eqnarray}
uniformly in $\ell$, $\boldsymbol{\alpha}$ and $\boldsymbol{\beta}$. Here we used the fact that $\hat{X}_{j}^{-1}$ and $V_j$-entries are all bounded and $T_j$-entries are bounded by $1+t_j$. 

Now, we go back to the definition of $\mathcal{P}(\cdot)$ in (\ref{012870}) and study the last factor. Similarly to the discussion above, it is easy to see that for $k=p$ or $q$, 
\begin{eqnarray}
&&\hat{\boldsymbol{\varpi}}_k:=\frac{\det\big(V_k^*\hat{X}_kV_k+M^{-1} \Omega_k T_k^{-1}\hat{B}_k^{-1}T_k\Xi_k\big)}{\det \hat{B}_k}\nonumber\\
&&=\hat{\mathfrak{p}}_0(\hat{X}_k, \hat{B}_k)+\sum_{\ell=1}^4\sum_{\boldsymbol{\alpha},\boldsymbol{\beta}}\hat{\mathfrak{p}}_{\ell,\boldsymbol{\alpha},{\boldsymbol{\beta}}}(\hat{X}_k, \hat{B}_k, V_k,T_k)\prod_{i=1}^\ell \omega_{k,\alpha_{k,i}}\xi_{k,\beta_{k,i}}, \label{020201}
\end{eqnarray}
where $\hat{\mathfrak{p}}_0(\cdot)=\det\hat{X}_k/\det\hat{B}_k$ and  $\hat{\mathfrak{p}}_{\ell,\boldsymbol{\alpha},{\boldsymbol{\beta}}}(\cdot)$'s are some polynomials of $\hat{X}_k$, $\hat{B}_k^{-1}$, $V_k$, $T_k$-entries with bounded degree and bounded coefficients. Similarly, we have
\begin{eqnarray}
|\hat{\mathfrak{p}}_0(\cdot)|, |\hat{\mathfrak{p}}_{\ell,\boldsymbol{\alpha},\boldsymbol{\beta}}(\cdot)|\leq C_1(r_{k,1}^{-1}+r_{k,2}^{-1}+t_k+1)^{C_2}  \label{021751}
\end{eqnarray}
for some positive constants $C_1$ and $C_2$.

According to the definitions in (\ref{011611}) and (\ref{020201}),  we can rewrite (\ref{012870}) as
\begin{eqnarray}
&&\mathcal{P}( \Omega, \Xi, \hat{X}, \hat{B}, V, T)=\exp\Big{\{} -\sum_{j,k}\tilde{\mathfrak{s}}_{jk} Tr  \Omega_j\Xi_k\Big{\}}\prod_{j=1}^W\boldsymbol{\varpi}_j\prod_{k=p,q} \hat{\boldsymbol{\varpi}}_k. \label{020447}
\end{eqnarray}
In light of the discussion above,  $\prod_{j=1}^W\boldsymbol{\varpi}_j\prod_{k=p,q}\hat{\boldsymbol{\varpi}}_k$ is a polynomial of $\hat{X}^{-1}$, $\hat{B}^{-1}$, $V$, $T$, $\Omega$  and $\Xi$-entries, in which each monomial is of the form
\begin{eqnarray}
\mathfrak{q}_{\vec{\ell},\vec{\boldsymbol{\alpha}},\vec{\boldsymbol{\beta}}}(\hat{X}^{-1},\hat{B}^{-1}, V, T)\prod_{j=1}^W\prod_{i=1}^{\ell_j} \omega_{j,\alpha_{j,i}}\xi_{j,\beta_{j,i}}, \label{010911}
\end{eqnarray}
where we used the notation
\begin{eqnarray*}
&&\vec{\ell}=(\ell_1,\ldots, \ell_W),\quad \vec{\boldsymbol{\alpha}}=(\boldsymbol{\alpha}_1,\ldots, \boldsymbol{\alpha}_W),\quad \vec{\boldsymbol{\beta}}=(\boldsymbol{\beta}_1,\ldots, \boldsymbol{\beta}_W),\nonumber\\
 &&\boldsymbol{\alpha}_j=(\alpha_{j,1},\ldots, \alpha_{j,\ell_j}),\quad \boldsymbol{\beta}_j=(\beta_{j,1},\ldots, \beta_{j,\ell_j}),
\end{eqnarray*}
and 
$\mathfrak{q}_{\vec{\ell},\vec{\boldsymbol{\alpha}},\vec{\boldsymbol{\beta}}}(\cdot)$ is a polynomial of $\hat{X}, \hat{X}^{-1}$, $\hat{B}^{-1}$, $V$ and $T$-entries. 
Moreover, all the entries of $\vec{\ell}$, $\vec{\boldsymbol{\alpha}}$ and $\vec{\boldsymbol{\beta}}$ are bounded by $4$.
By (\ref{010912}) and (\ref{021751}), we have
\begin{eqnarray}
|\mathfrak{q}_{\vec{\ell},\vec{\boldsymbol{\alpha}},\vec{\boldsymbol{\beta}}}(\hat{X}^{-1},\hat{B}^{-1}, V, T)|\leq e^{O(W)}\prod_{j=1}^W\big(r_{j,1}^{-1}+r_{j,2}^{-1}+t_j+1\big)^{C}. \label{010920}
\end{eqnarray}
In addition, it is easy to see that the number of the summands of the form (\ref{010911}) in $\prod_{j=1}^W\boldsymbol{\varpi}_j\prod_{k=p,q}\hat{\boldsymbol{\varpi}}_k$ is bounded by  $e^{O(W)}$. 

Define the vectors
\begin{eqnarray}
\vec{\Omega}:=(\boldsymbol{\omega}_1,\boldsymbol{\omega}_2,\boldsymbol{\omega}_3, \boldsymbol{\omega}_4),\qquad \vec{\Xi}:=(\boldsymbol{\xi}_1,\boldsymbol{\xi}_2,\boldsymbol{\xi}_3, \boldsymbol{\xi}_4), \label{012110}
\end{eqnarray}
where
\begin{eqnarray*}
\boldsymbol{\omega}_\alpha=(\omega_{1,\alpha},\ldots, \omega_{W,\alpha}),\quad \boldsymbol{\xi}_\alpha=(\xi_{1,\alpha},\ldots, \xi_{W,\alpha}),\quad \alpha=1,2,3,4.
\end{eqnarray*}
Here we used the notation (\ref{010301}). In addition, we introduce the matrix
\begin{eqnarray*}
\tilde{\mathbb{H}}=\tilde{S}\oplus\tilde{S}\oplus\tilde{S}\oplus\tilde{S}.
\end{eqnarray*}
It is easy to check
\begin{eqnarray*}
\sum_{j,k}\tilde{\mathfrak{s}}_{jk} Tr \Omega_j\Xi_k=\vec{\Omega}\tilde{\mathbb{H}}\vec{\Xi}'.
\end{eqnarray*}
By using the Gaussian integral formula for the Grassmann variables (\ref{0129102}), we see that for each $\vec{\ell}$, $\vec{\boldsymbol{\alpha}}$ and $\vec{\boldsymbol{\beta}}$, we have
\begin{eqnarray}
\Big{|}\int {\rm d} \Omega {\rm d}\Xi \cdot \exp\Big{\{} -\sum_{j,k}\tilde{\mathfrak{s}}_{jk} Tr \Omega_j\Xi_k\Big{\}}\cdot \prod_{j=1}^W\prod_{i=1}^{\ell_j} \omega_{j,\alpha_{j,i}}\xi_{j,\beta_{j,i}}\Big{|}\leq |\det\tilde{\mathbb{H}}^{(\mathsf{I}|\mathsf{J})}|, \label{010921}
\end{eqnarray}
for some index sets $\mathsf{I}$ and $\mathsf{J}$ with $|\mathsf{I}|=|\mathsf{J}|$. By Assumption \ref{assu.1} (i) and (ii), we see that the $2$-norm of each row of $\tilde{S}$ is $O(1)$. Consequently, by using Hadamard's inequality, we have  
\begin{eqnarray}
|\det\tilde{\mathbb{H}}^{(\mathsf{I}|\mathsf{J})}|=e^{O(W)}. \label{020202}
\end{eqnarray}

Therefore, (\ref{020447})-(\ref{020202}) and the bound $e^{O(W)}$ for the total number of summands of the form (\ref{010911}) in $\prod_{j=1}^W\boldsymbol{\varpi}_j\prod_{k=p,q}\hat{\boldsymbol{\varpi}}_k$  imply that
\begin{eqnarray*}
|\mathsf{P}(\hat{X},\hat{B},V, T)|\leq e^{O(W)}\prod_{j=1}^W \big(r_{j,1}^{-1}+r_{j,2}^{-1}+t_j+1\big)^{O(1)}.
\end{eqnarray*}
Thus we completed the proof.
\end{proof}
\subsection{Integral of  $\mathcal{F}$}\label{s.7.3} In this subsection, we also temporarily ignore the  $X^{[1]},\mathbf{y}^{[1]},\mathbf{w}^{[1]},P_1,Q_1$-variables from $\mathsf{Q}(\cdot)$, and estimate $\mathsf{F}(\hat{X},\hat{B}, V, T)$ defined in (\ref{013117}).
We have the following lemma.
\begin{lem} \label{lem.020201}Suppose that the assumptions in Lemma \ref{lem.020203} hold. We have
\begin{eqnarray*}
|\mathsf{F}(\hat{X},\hat{B}, V, T)| \leq  e^{O(WN^{\varepsilon_2})}\prod_{k=p,q}\big(r_{k,1}^{-1}+r_{k,2}^{-1}+t_k+1\big)^{O(1)}.\label{010712}
\end{eqnarray*}
\end{lem}
\begin{proof}
Recalling the decomposition of $\mathcal{F}(\cdot)$ in (\ref{012130}) together with the parameterization in (\ref{020206}), we will study the integrals 
\begin{eqnarray}
&&\mathbb{G}(\hat{B},T):=\int {\rm d}\nu(Q_1) {\rm d}\mathbf{y}^{[1]} {\rm d}\mathbf{w}^{[1]}\;  g(Q_1,T,\hat{B}, \mathbf{y}^{[1]}, \mathbf{w}^{[1]}),\label{120801}\\
&&\mathbb{F}(\hat{X},V):=\int {\rm d} \mu(P_1) {\rm d}X^{[1]} \; f(P_1, V, \hat{X}, X^{[1]}) \label{120802}
\end{eqnarray}
separately. Recalling the convention at the end of Section \ref{s.3}, we use $f(\cdot)$ and $g(\cdot)$ to represent the integrands above. One can refer to (\ref{012150}) and (\ref{0125130}) for the definition. 

From the assumption $\eta\leq M^{-1}N^{\varepsilon_2}$, it is easy to see 
\begin{eqnarray}
|\mathbb{F}(\hat{X},V)|\leq e^{O(WN^{\varepsilon_2})}, \label{010707}
\end{eqnarray}
 since $P_1, V, \hat{X}, X^{[1]}$-variables are all bounded and  $|\det X_p^{[1]}|, |\det X_q^{[1]}|\sim 1$ when $\mathbf{x}_1,\mathbf{x}_2\in \widehat{\Sigma}$ defined in (\ref{030901}). 
 
 For $\mathbb{G}(\hat{B},T)$, we use the facts
\begin{eqnarray}
&&\Re(TrB_jY_k^{[1]}J)\geq0,  \quad \Re(\mathbf{i} Tr Y_k^{[1]}JZ)=-\eta Tr Y_k^{[1]}\leq 0,\quad TrY_k^{[1]}JY_\ell^{[1]} J\geq0,\quad k,\ell=p,q,  \nonumber\\
&& \hspace{20ex} \big|\big(\mathbf{w}^{[1]}_q(\mathbf{w}^{[1]}_q)^*\big)_{12}\big|\leq1,\quad \big|\big(\mathbf{w}^{[1]}_p(\mathbf{w}^{[1]}_p)^*\big)_{21}\big|\leq 1, \label{0204105}
\end{eqnarray}
to estimate trivially several terms, whereby we can get the bound
\begin{align}
|g(\cdot)|&\leq \exp\Big{\{} -M\eta\sum_{j=1}^W Tr \Re(B_j)J\Big{\}}\prod_{k=p,q} (y_k^{[1]})^{n+3}\exp\Big{\{}-\tilde{\mathfrak{s}}_{kk} Tr \Re(B_k)Y_k^{[1]}J\Big{\}}.  \label{123001}
\end{align}
Here  $\Re(B_j)=Q_1^{-1}T_j^{-1} \Re (\hat{B}_j) T_jQ_1$. Hence, integrating $y_p^{[1]}$ and $y_q^{[1]}$ out yields
\begin{eqnarray}
\int_{\mathbb{R}_+^2} {\rm d}y_p^{[1]} {\rm d} y_q^{[1]} |g(Q_1,T,\hat{B}, \mathbf{y}^{[1]}, \mathbf{w}^{[1]})|\leq  C\frac{\exp\Big\{ -M\eta\sum_{j=1}^W Tr \Re(B_j)J\Big\}}{ \prod_{k=p,q}\Big((\mathbf{w}_k^{[1]})^*J \Re(B_k) \mathbf{w}_k^{[1]}\Big)^{C_1}}, \label{021763}
\end{eqnarray}
for some positive constants $C$ and $C_1$ depending on $n$, where we used the elementary facts that $\tilde{\mathfrak{s}}_{kk}\geq c$ for some positive constant $c$ and
\begin{eqnarray}
 Tr \Re(B_j)Y_k^{[1]}J=y_k^{[1]}(\mathbf{w}_k^{[1]})^*J \Re(B_j) \mathbf{w}_k^{[1]},\quad k=p,q,\quad j=1,\ldots, W. \label{011901}
\end{eqnarray}
Now, note that
\begin{eqnarray}
(\mathbf{w}_k^{[1]})^*J \Re B_j \mathbf{w}_k^{[1]}\geq \lambda_1(J\Re B_j),\quad k=p,q,\quad j=1,\ldots, W. \label{011902}
\end{eqnarray}
In addition, it is also easy to see $\lambda_1(T_j)=s_j-t_j$ and $\lambda_1(Q_1)=s-t$, according to the definitions in (\ref{122708}). Now, by the fact $JA^{-1}=AJ$ for any $A\in \mathring{U}(1,1)$, we have
\begin{eqnarray}
J\Re B_j= Q_1 T_j \text{diag}\big( \Re b_{j,1}, \Re b_{j,2}\big)T_jQ_1. \label{021761}
\end{eqnarray}
Consequently, we can get
\begin{eqnarray}
\lambda_1(J\Re B_j)\geq (s_j-t_j)^2(s-t)^2\text{min}\{\Re b_{j,1}, \Re b_{j,2}\}=\frac{\text{min}\{\Re b_{j,1}, \Re b_{j,2}\}}{(s_j+t_j)^2(s+t)^2}, \label{011903}
\end{eqnarray}
by recalling the facts $s^2-t^2=1$ and $s_j^2-t_j^2=1$.
Therefore, combining (\ref{021763}), (\ref{011902}) and (\ref{011903}),  we have
\begin{eqnarray}
\int_{\mathbb{R}_+^2} {\rm d}y_p^{[1]} {\rm d} y_q^{[1]} |g(\cdot)|\leq  C(s+t)^{4C_1}\exp\Big{\{} -M\eta\sum_{j=1}^W Tr (\Re B_j)J\Big{\}} \prod_{k=p,q}\frac{(s_k+t_k)^{2C_1}}{ \big(\text{min}\{\Re b_{k,1}, \Re b_{k,2}\}\big)^{C_1}}.\label{121610}
\end{eqnarray}

Now, what remains is to estimate the exponential function in (\ref{121610}). By elementary calculation from (\ref{021761}) we obtain 
\begin{eqnarray*}
Tr (\Re B_j)J\geq \big(\Re b_{j,1}+\Re b_{j,2}\big)\big((s_j^2+t_j^2)(s^2+t^2)-4sts_j t_j\big).
\end{eqnarray*}
Observe that
\begin{eqnarray*}
(s_j^2+t_j^2)(s^2+t^2)-4sts_j t_j=\frac{(s_j^2+t_j^2)^2(s^2+t^2)^2-16(sts_jt_j)^2}{(s_j^2+t_j^2)(s^2+t^2)+4sts_j t_j}\geq \frac{4t^4+4t^2+4t_j^4+4t_j^2+1}{2(1+2t^2_j)(1+2t^2)}\geq \frac{1+2t^2}{2(1+2t_j^2)}.
\end{eqnarray*}
It implies that
\begin{align}
\exp\Big{\{} -M\eta\sum_{j=1}^W Tr (\Re B_j)J\Big{\}}&\leq \exp\Big{\{} -2M\eta\sum_{j=1}^W \frac{ \Re b_{j,1}+\Re b_{j,2}}{1+2t_j^2} (1+2t^2)\Big{\}}\nonumber\\
&\leq  \exp\Big{\{} -cM\eta \sum_{j=1}^W \frac{ r_{j,1}+r_{j,2}}{1+2t_j^2}(1+2t^2)\Big{\}}, \label{121611}
\end{align}
for some positive constant $c$, where in the last step we used the fact
\begin{eqnarray}
\Re b_{j,\alpha}\geq cr_{j,\alpha},\quad \forall\; j=1,\ldots, W, \alpha=1,2, \label{010704}
\end{eqnarray}
for some positive constant $c$, in light of the assumption $|E|\leq\sqrt{2}-\kappa$ and the definition of $\mathbb{K}$ in (\ref{030901}).
Plugging (\ref{121611})  into (\ref{121610}), estimating $(s+t)^2\leq 2(1+2t^2)$, and integrating $t$ out, we can crudely bound
\begin{eqnarray}
\int_{\mathbb{R}_+^2} {\rm d}y_p^{[1]} {\rm d} y_q^{[1]} \int_{\mathbb{R}^+} 2t{\rm d}t\cdot  |g(\cdot)|\leq C\Big(\frac{1}{M\eta}\Big)^{C_2}\Big{(}\sum_{j=1}^W \frac{ r_{j,1}+r_{j,2}}{1+2t_j^2}\Big{)}^{-C_2}\prod_{k=p,q}\frac{(1+2t_k^2)^{C_1}}{ \big((\text{min}\{\Re b_{k,1}, \Re b_{k,2}\}\big)^{C_1}}. \label{010702}
\end{eqnarray}
Now, we use the trivial bounds
\begin{eqnarray}
\Big{(}\sum_{j=1}^W \frac{ r_{j,1}+r_{j,2}}{1+2t_j^2}\Big{)}^{-C_2}\leq \Big{(}\frac{1+2t_p^2}{ r_{p,1}+r_{p,2}}\Big{)}^{C_2}\leq \Big{(}(1+2t_p^2)(r_{p,1}^{-1}+r_{p,2}^{-1})\Big{)}^{C_2}, \label{011001}
\end{eqnarray}
and 
\begin{eqnarray}
\frac{1+2t_k^2}{ \text{min}\{\Re b_{k,1}, \Re b_{k,2}\}}\leq C(1+2t_k^2)(r_{k,1}^{-1}+r_{k,2}^{-1}). \label{011002}
\end{eqnarray}
Inserting (\ref{011001}) and (\ref{011002}) into (\ref{010702}) and integrating out the remaining variables yields
\begin{eqnarray}
|\mathbb{G}(\hat{B},T)| \leq  C\Big(\frac{1}{M\eta}\Big)^{C_1} \prod_{k=p,q}\Big{(}r_{k,1}^{-1}+r_{k,2}^{-1}+t_k+1\Big{)}^{C_3}.
\label{010708}
\end{eqnarray}
Combining (\ref{010707}) and (\ref{010708})  we can get the bound
\begin{eqnarray*}
|\mathsf{F}(\hat{X},\hat{B}, V, T)| \leq  e^{O(WN^{\varepsilon_2})}\prod_{k=p,q}\Big{(}r_{k,1}^{-1}+r_{k,2}^{-1}+t_k+1\Big{)}^{O(1)}. 
\end{eqnarray*}
Hence, we completed the proof of Lemma \ref{lem.020201}.
\end{proof}
\subsection{Summing up:  Proof of Lemma \ref{lem.020203}} \label{s.7.4} In the discussions in Sections \ref{s.7.2} and \ref{s.7.3}, we ignored the irrelevant factor $\mathsf{Q}(\cdot)$. However, it is easy to modify the discussion slightly to take this factor into account, whereby we can prove Lemma \ref{lem.020203}. 
\begin{proof}[Proof of Lemma \ref{lem.020203}]At first, by the definition in (\ref{020130}), we can rewrite (\ref{013115}) as
\begin{eqnarray*}
\mathsf{A}(\cdot)=\int  {\rm d}X^{[1]} {\rm d}\mathbf{y}^{[1]} {\rm d} \mathbf{w}^{[1]}{\rm d} \Omega {\rm d} \Xi {\rm d}\mu(P_1) {\rm d}\nu(Q_1)\;
\mathcal{P}( \cdot)\mathsf{Q}(\cdot) \mathcal{F}(\cdot).
\end{eqnarray*}
Now, by the conclusion $\kappa_1=W^{O(1)}$ in Lemma \ref{lem.0202081}, it suffices to consider one term in $\mathsf{Q}(\cdot)$, which is a monomial of the form
\begin{eqnarray*}
\mathfrak{p}(t,s, (y^{[1]}_p)^{-1},(y^{[1]}_q)^{-1})\mathfrak{q}(\Omega, \Xi),
\end{eqnarray*} 
regarding $\sigma$, ${v}_p^{[1]}$, ${v}_q^{[1]}$,  $P_1$-variables, $X^{[1]}$-variables and $\mathbf{w}^{[1]}$-variables as bounded parameters. Here $\mathfrak{p}(\cdot)$ is a monomial of $t,s, (y^{[1]}_p)^{-1},(y^{[1]}_q)^{-1}$ and $\mathfrak{q}(\cdot)$ is a monomial of $\Omega,\Xi$-variables, both with bounded coefficients and bounded degrees, according to the fact $\kappa_2,\kappa_3=O(1)$ in Lemma \ref{lem.0202081}.  Now we define
\begin{eqnarray}
&&\mathsf{P}_{\mathfrak{q}}(\hat{X}, \hat{B}, V, T):=\int {\rm d} \Omega {\rm d}\Xi \; \mathcal{P}( \Omega, \Xi, \hat{X}, \hat{B}, V, T)\cdot \mathfrak{q}(\Omega, \Xi),\nonumber\\
&&\mathsf{F}_{\mathfrak{p}}(\hat{X}, \hat{B}, V, T):=\int {\rm d} X^{[1]} {\rm d}\mathbf{y}^{[1]} {\rm d}\mathbf{w}^{[1]} {\rm d}\mu(P_1) {\rm d}\nu(Q_1)\nonumber\\&&\hspace{5ex}\times\mathcal{F}(\hat{X},\hat{B}, V, T, P_1, Q_1, X^{[1]}, \mathbf{y}^{[1]},\mathbf{w}^{[1]})\cdot \mathfrak{p}(t,s, (y^{[1]}_p)^{-1},(y^{[1]}_q)^{-1}). \label{020511}
\end{eqnarray}
By repeating the discussions in Sections \ref{s.7.2} and \ref{s.7.3} with slight modification, we can easily see that 
\begin{eqnarray*}
&&|\mathsf{P}_{\mathfrak{q}}(\hat{X}, \hat{B}, V, T)|\leq e^{O(W)}\prod_{j=1}^W (r_{j,1}^{-1}+r_{j,2}^{-1}+t_j+1)^{O(1)},\nonumber\\
&&|\mathsf{F}_{\mathfrak{p}}(\hat{X}, \hat{B}, V, T)|\leq e^{O(WN^{\varepsilon_2})}\prod_{k=p,q}\Big{(}r_{k,1}^{-1}+r_{k,2}^{-1}+t_k+1\Big{)}^{O(1)}
\end{eqnarray*}
hold as well. Therefore, we completed the proof of Lemma \ref{lem.020203}.
\end{proof}

\section{Proofs of Lemmas \ref{lem.010601} and \ref{lem.121601}} \label{s.8} 
In this section, with the aid of Lemma \ref{lem.020203}, we prove Lemmas \ref{lem.010601} and \ref{lem.121601}. 
According to Lemmas \ref{lem.010101} and \ref{lem.0101.1}, one can see that away from the saddles, $\Re \mathring{L}(\hat{B},T)$ and $\Re \mathring{K}(\hat{X},V)$ increase quadratically in $\hat{B}$-variables and $\hat{X}$-variables, respectively. Hence, it is easy to control the integral (\ref{010114}) over these variables outside the vicinities. However, like $\hat{B}$-variables, the domain of $\mathbf{t}$-variables is also not compact. This forces us to analyze the exponential function 
\begin{eqnarray}
\mathbb{M}(\mathbf{t}):=\exp\Big{\{}-M\Re \big(\ell_S(\hat{B},T)\big)\Big{\}}\label{011011}
\end{eqnarray}
 carefully for any fixed $\hat{B}$-variables. 
 
 Recall the definition of the sector $\mathbb{K}$ in (\ref{030901}). For  $\mathbf{b}_1\in \mathbb{K}^W$ and $\mathbf{b}_2\in \bar{\mathbb{K}}^W$, we have 
\begin{eqnarray}
\min_{j,k}\Re(b_{j,1}+b_{j,2})(b_{k,1}+b_{k,2})\geq c\min_{j,k}\sum_{a,b=1,2} r_{j,a}r_{k,b}\geq c\min_{j,a} r_{j,a}^2:=\mathfrak{A}(\hat{B}),\label{020260}
\end{eqnarray} 
for some positive constant $c$ depending on $\kappa$ from (\ref{021101}).  
From now on, we regard $\mathbb{M}(\mathbf{t})$ as a measure of the $\mathbf{t}$-variables and study it in the following two regions  separately:
 \begin{eqnarray*}
(i):\mathbf{t}\in \mathbb{I}^{W-1},\qquad (ii):\mathbf{t}\in \mathbb{R}_+^{W-1}\setminus\mathbb{I}^{W-1}.
 \end{eqnarray*} 
 Roughly speaking, when $\mathbf{t}\in \mathbb{I}^{W-1}$, we will see that $\mathbb{M}(\mathbf{t})$ can be bounded pointwisely by a  Gaussian measure. More specifically, we have the following lemma.
 \begin{lem}\label{lem.011002} With the notation above, we have
 \begin{eqnarray*}
 \mathbb{M}(\mathbf{t})\leq \exp\bigg{\{}- \frac{M}{12}\mathfrak{A}(\hat{B})\sum_{j,k} \mathfrak{s}_{jk}(t_k-t_j)^2\bigg{\}}, \quad  \forall\;  \mathbf{t}\in \mathbb{I}^{W-1}.
 \end{eqnarray*}
 \end{lem}

 However, the behavior of $\mathbb{M}(\mathbf{t})$ for $\mathbf{t}\in \mathbb{R}_+^{W-1}\setminus\mathbb{I}^{W-1}$  is much more sophisticated. We will not try to provide a pointwise control of $\mathbb{M}(\mathbf{t})$ in this region. Instead, we will bound the integral of $\mathfrak{q}(\mathbf{t})$ against $\mathbb{M}(\mathbf{t})$ over this region, for any given monomial $\mathfrak{q}(\cdot)$ of interest. More specifically, recalling the definition of $\Theta$ in (\ref{021102}) and the spanning tree $\mathcal{G}_0=(\mathcal{V},\mathcal{E}_0)$ in Assumption \ref{assu.1}, and additionally setting 
\begin{eqnarray}
\mathfrak{L}:=\frac{M}{4}\mathfrak{A}(\hat{B})\min_{i,j\in\mathcal{E}_0}\mathfrak{s}_{ij}, \label{031001}
\end{eqnarray}
 we have the following lemma. 
\begin{lem} \label{lem.011001}Let $\mathfrak{q}(\mathbf{t})=\prod_{j=2}^Wt_j^{n_j}$ be a monomial of $\mathbf{t}$-variables, with powers $n_j=O(1)$ for all $j=2,\ldots, W$. We have
\begin{eqnarray}
\int_{\mathbb{R}_+^{W-1}\setminus \mathbb{I}^{W-1}} \prod_{j=2}^W{\rm d} t_j \; \mathbb{M}(\mathbf{t}) \mathfrak{q}(\mathbf{t})\leq \Big(1+\mathfrak{L}^{-\frac12}\Big)^{O(W^2)}\exp\left\{-\Theta^2\mathfrak{A}(\hat{B})+O(W^2\log N)\right\} \label{031025}
\end{eqnarray}
\end{lem}

\begin{rem}
Roughly speaking, by Lemma \ref{lem.011001} we see that the integral of $\mathfrak{q}(\mathbf{t})$-variables against the measure $\mathbb{M}(\mathbf{t})$ over the region $\mathbb{R}_+^{W-1}\setminus \mathbb{I}^{W-1}$ is exponentially small, owing to the fact $\Theta^2\gg W^2\log N$.
\end{rem}

 We will postpone the proofs of Lemmas  \ref{lem.011002} and \ref{lem.011001} to the end of this section. In the sequel, at first, we prove Lemmas \ref{lem.010601} and \ref{lem.121601} with the aid of Lemmas \ref{lem.020203},  \ref{lem.011002} and \ref{lem.011001}. Before commencing the formal proofs, we mention two basic facts which are formulated as the following lemma.
 \begin{lem} Under Assumption \ref{assu.1}, we have the following two facts.
 \begin{itemize}
 \item For the smallest eigenvalue of $S^{(1)}$, there exists some positive constant $c$ such that 
\begin{eqnarray}
\lambda_1(-S^{(1)})\geq \frac{c}{W^2}. \label{011305}
\end{eqnarray} 
\item  Let $\boldsymbol{\varrho}=(\rho_2,\ldots, \rho_{W})'$ be a real vector and $\rho_1=0$. If there is at least one $\alpha\in\{2,\ldots,W\}$ such that $\varrho_\alpha\geq \Theta/\sqrt{M}$, then we have 
 \begin{eqnarray}
\sum_{j,k} \mathfrak{s}_{jk} (\varrho_j-\varrho_k)^2\geq \frac{\Theta}{M}. \label{121801}
\end{eqnarray}
\end{itemize}
 \end{lem}
 \begin{proof}
 Let $\boldsymbol{\varrho}=(\rho_2,\ldots, \rho_{W})'$ be a real vector and $\rho_1=0$. Now, we assume $|\rho_\alpha|=\max_{\beta=2,\ldots, W}|\rho_\beta|$.  Then
\begin{eqnarray*}
\frac{-\boldsymbol{\rho}'S^{(1)}\boldsymbol{\rho}}{||\boldsymbol{\rho}||_2^2}=\frac{\frac12\sum_{j,k} \mathfrak{s}_{jk} (\rho_j-\rho_k)^2}{\sum_j \rho_j^2}\geq \frac{c(\rho_\alpha-\rho_1)^2}{W^2\rho_\alpha^2}=\frac{c}{W^{2}},
\end{eqnarray*}
where the second step follows from Assumption \ref{assu.1} (iv) and Cauchy-Schwarz inequality.
Analogously, we have
\begin{eqnarray*}
\sum_{j,k} \mathfrak{s}_{jk} (\varrho_j-\varrho_k)^2\geq \frac{c}{W} \varrho_\alpha^2 \geq \frac{\Theta}{M}
\end{eqnarray*}
according to the definition of $\Theta$ in (\ref{021102}).  Hence, we completed the proof.
 \end{proof}
Recalling the notation defined in  (\ref{010114}) and the facts $|x_{j,a}|=1$ and $|b_{j,a}|=r_{j,a}$ for all $j=1,\ldots, W$ and $a=1,2$, for any sequence of domains, we have
\begin{eqnarray}
&&\left|\mathcal{I}(\mathbf{I}^{b}_1, \mathbf{I}^b_2, \mathbf{I}^x_1, \mathbf{I}^x_2, \mathbf{I}^t,\mathbf{I}^v)\right|\nonumber\\
&&\leq e^{O(W\log N)}\int_{\mathbf{I}^b_1} \prod_{j=1}^W {\rm d} b_{j,1} \int_{\mathbf{I}^b_2} \prod_{j=1}^W {\rm d} b_{j,2} \int_{\mathbf{I}^x_1} \prod_{j=1}^W {\rm d} x_{j,1} \int_{\mathbf{I}^x_2} \prod_{j=1}^W {\rm d} x_{j,2}  \int_{\mathbf{I}^t} \prod_{j=2}^W 2t_j {\rm d} t_j \int_{\mathbf{I}^v} \prod_{j=2}^W 2v_j {\rm d} v_j\nonumber\\
&&\hspace{5ex} \times  \exp\left\{-M\big(\Re K(\hat{X},V)+\Re L(\hat{B},T)\big)\right\}\cdot|\mathsf{A}(\hat{X}, \hat{B}, V, T)|\cdot \prod_{j=1}^W (r_{j,1}+r_{j,2})^2. \label{011159}
\end{eqnarray}
In addition, according to Lemma \ref{lem.020203}, we have
\begin{eqnarray}
|\mathsf{A}(\hat{X}, \hat{B}, V, T)|\cdot \prod_{j=2}^W 2t_j\cdot \prod_{j=2}^W 2v_j \cdot \prod_{j=1}^W (r_{j,1}+r_{j,2})^2\leq e^{O(WN^{\varepsilon_2})}\tilde{\mathfrak{p}}(\mathbf{r},\mathbf{r}^{-1},\mathbf{t}), \label{011160}
\end{eqnarray}
for some polynomial $\tilde{\mathfrak{p}}(\mathbf{r},\mathbf{r}^{-1},\mathbf{t})$ with positive coefficients, and
\begin{eqnarray}
&&\tilde{\mathfrak{p}}(\mathbf{r},\mathbf{r}^{-1},\mathbf{t})\in \mathfrak{Q}\Big(\{r_{j,1}, r_{j,2}, r_{j,1}^{-1}, r_{j,2}^{-1}, t_j\}_{j=1}^W; \kappa_1,\kappa_2,\kappa_3\Big),\quad   \kappa_1=e^{O(W)}, \quad \kappa_2, \kappa_3=O(1). \label{011161}
\end{eqnarray}

 \subsection{Proof of Lemma \ref{lem.010601}}
At first, since throughout the whole proof, the domains of $\mathbf{x}_1$, $\mathbf{x}_2$, and $\mathbf{v}$-variables, namely,  $\Sigma^{W}$, $\Sigma^W$ and $\mathbb{I}^{W-1}$, will not be involved,  we just use $*$'s to represent them, in order to simplify the notation.

Now, we introduce the following contours with the parameter $\mathfrak{D}\in \mathbb{R}_+$,
\begin{eqnarray*}
\Gamma_\mathbb{\mathfrak{D}}:=\big\{ra_+|r\in[0,\mathfrak{D}]\big\}\subset \Gamma,\quad \mathbb{R}_\mathfrak{D}=[0,(\Re a_+) \mathfrak{D}]\subset \mathbb{R}_+, \quad  \mathcal{L}_\mathfrak{D}:=\big\{(\Re a_+)\mathfrak{D}+\mathbf{i}(\Im a_+)r|r\in[0,\mathfrak{D}]\big\}.
\end{eqnarray*}
In addition, we recall the sector $\mathbb{K}$ defined in (\ref{030901}).
Then, trivially, we have
\begin{eqnarray*}
\mathbb{R}_+,\Gamma, \mathcal{L}_\mathfrak{D}\subset \mathbb{K},\quad \mathbb{R}_+, \bar{\Gamma}, \bar{\mathcal{L}}_\mathfrak{D}\in \bar{\mathbb{K}},\quad \forall\; \mathfrak{D}\in \mathbb{R}_+.
\end{eqnarray*}

We claim that the integrand in (\ref{010114}) is an analytic function of the $\hat{B}$-variables. To see this, we can go back to the integral representation (\ref{012871}) and the definitions of $L(B)$ and $\mathcal{P}(\Omega,\Xi,X,B)$ in  (\ref{012130}). Note that since $\exp\{M\log\det B_j\}=(\det B_j)^M$, actually the logarithmic terms in $L(B)$ do not produce any singularity in the integrand in 
(\ref{012871}). In addition, according to the fact that the $\chi^\ell=0$ for any Grassmann variable $\chi$ and $\ell\geq 2$, the factors $\det^M(1+M^{-1}X_j^{-1}\Omega_jB_j^{-1}\Xi_j)$ is actually a polynomial of $\Omega_j$, $\Xi_j$, $X_j^{-1}$ and $B_j^{-1}$-entries with degree $16$. The other factors containing $\hat{B}$-variables $\mathcal{P}(\cdot)$ can be checked analogously. Hence, it is easy to see that $\exp\{-ML(\cdot)\}\mathcal{P}(\cdot)$ is analytic in $\hat{B}$-variables.
Consequently, we have
\begin{eqnarray*}
\mathcal{I}\Big((\Gamma_{\mathfrak{D}}\cup \mathcal{L}_\mathfrak{D})^W, (\bar{\Gamma}_{\mathfrak{D}}\cup \bar{\mathcal{L}}_\mathfrak{D})^W, *, *, \mathbb{R}_+^{W-1},*\Big)=\mathcal{I}\Big((\mathbb{R}_\mathfrak{D})^W, (\mathbb{R}_\mathfrak{D})^W,*, *, \mathbb{R}_+^{W-1},*\Big).
\end{eqnarray*}
Hence, to prove Lemma \ref{lem.010601}, it suffices to prove the following lemma.
\begin{lem} \label{lem.020208}Suppose that $|E|\leq\sqrt{2}-\kappa$. As $\mathfrak{D}\to \infty$, the following  convergence hold,
\begin{eqnarray}
&&(i):\quad \mathcal{I}\Big((\Gamma_{\mathfrak{D}}\cup \mathcal{L}_\mathfrak{D})^W, (\bar{\Gamma}_{\mathfrak{D}}\cup \bar{\mathcal{L}}_\mathfrak{D})^W, *, *, \mathbb{R}_+^{W-1},*\Big)-\mathcal{I}\Big((\Gamma_{\mathfrak{D}})^W, ({\bar{\Gamma}_{\mathfrak{D}}})^W, *, *, \mathbb{R}_+^{W-1},*\Big)\to 0,
\nonumber\\
&&(ii):\quad\mathcal{I}\Big((\Gamma)^W, (\bar{\Gamma})^W, *, *, \mathbb{R}_+^{W-1},*\Big)-\mathcal{I}\Big((\Gamma_{\mathfrak{D}})^W, ({\bar{\Gamma}_{\mathfrak{D}}})^W, *, *, \mathbb{R}_+^{W-1},*\Big)\to 0,\nonumber\\
&&(iii):\quad  \mathcal{I}\Big(\mathbb{R}_+^W, \mathbb{R}_+^W, *, *, \mathbb{R}_+^{W-1},*\Big)-\mathcal{I}\Big(\mathbb{R}_\mathfrak{D}^W, \mathbb{R}_\mathfrak{D}^W, *, *, \mathbb{R}_+^{W-1},*\Big)\to 0.\nonumber
\end{eqnarray}
\end{lem}
\begin{proof}  For simplicity, we use the notation
\begin{eqnarray*}
&&\mathbf{I}_\mathfrak{D}^{b,1}:=\big(\Gamma_{\mathfrak{D}}\cup \mathcal{L}_\mathfrak{D}\big)^W\times \big(\bar{\Gamma}_{\mathfrak{D}}\cup \bar{\mathcal{L}}_\mathfrak{D}\big)^W \setminus (\Gamma_{\mathfrak{D}})^W\times (\bar{\Gamma}_{\mathfrak{D}})^W,\nonumber\\
&&\mathbf{I}_\mathfrak{D}^{b,2}:=\Gamma^W\times \bar{\Gamma}^W\setminus (\Gamma_{\mathfrak{D}})^W\times ({\bar{\Gamma}_{\mathfrak{D}}})^W,\nonumber\\
&&\mathbf{I}_\mathfrak{D}^{b,3}:=\mathbb{R}_+^W  \times \mathbb{R}_+^W\setminus \mathbb{R}_\mathfrak{D}^W\times \mathbb{R}_\mathfrak{D}^W.
\end{eqnarray*}
Now, recall the definition of the function $\ell(\mathbf{a})$ in (\ref{020401}) and the representation of $L(\hat{B},T)$ in (\ref{121110}). Hence, in light of the definition of $\mathbb{M}(\mathbf{t})$ in (\ref{011011}), we have
\begin{eqnarray}
\exp\big\{-M\Re L(\hat{B},T)\big\}= \exp\big\{-M\big(\Re \ell(\mathbf{b}_1)+\Re \ell(-\mathbf{b}_2)\big)\big\} \mathbb{M}(\mathbf{t}). \label{021801}
\end{eqnarray}
By the assumption $|E|\leq\sqrt{2}-\kappa$, we see that  $\Re b_{j,a}b_{k,a}>0$ for all $b_{j,a},b_{k,a}\in \mathbb{K}\cup \bar{\mathbb{K}}$. Consequently, when $b_{j,1}\in \mathbb{K}$ and $b_{j,2}\in\bar{\mathbb{K}}$ for all $j=1,\ldots, W$, we have for some positive constant $c$ dependent of $\kappa$ in (\ref{021101}),
\begin{align}
\Re \ell(\mathbf{b}_1)+\Re \ell(-\mathbf{b}_2)&\geq \sum_{a=1,2}\sum_{j}\Big(\frac12(1+\mathfrak{s}_{jj}) \Re b_{j,a}^2+(-1)^{a+1}E\Im b_{j,a}- \log r_{j,a}\Big)\nonumber\\
&\geq c\sum_{a=1,2}\sum_{j} r_{j,a}^2-\sum_{a=1,2}\sum_{j} \log r_{j,a},
\label{021810}
\end{align}
where we used Assumption \ref{assu.1} (ii) and the fact that $(-1)^{a+1}E\Im b_{j,a}\geq 0$. 

Now,  when $(\mathbf{b}_1,\mathbf{b}_2)\in \mathbf{I}_\mathfrak{D}^{b,i}$ for $i=1,2, 3$, we have $\sum_{a=1,2}\sum_{j} r_{j,a}^2 \geq c \mathfrak{D}^2$ for some positive constant $c$, which implies the trivial fact
\begin{eqnarray}
\sum_{a=1,2}\sum_{j} r_{j,a}^2 \geq \frac12\sum_{a=1,2}\sum_{j} r_{j,a}^2+\frac{c}{2} \mathfrak{D}^2. \label{021811}
\end{eqnarray}
Consequently, we can get from (\ref{021801}), (\ref{021810}) and (\ref{021811}) that for some positive constant $c$,
\begin{eqnarray*}
\exp\big\{-M\Re L(\hat{B},T)\big\}\leq e^{-cM\mathfrak{D}^2}\prod_{a=1,2}\prod_{j=1}^W e^{-cMr_{j,a}^2 }\;r_{j,a}^M\cdot \mathbb{M}(\mathbf{t})
\end{eqnarray*}
holds  in $\mathbf{I}_\mathfrak{D}^{b,i}$ for $i=1,2,3$.
In addition, by the boundedness of $V$ and $\hat{X}$-variables, we can  get the trivial bound $MK(\hat{X},V)=O(N)$.
Hence, from (\ref{011159}) and (\ref{011160}) we see that the quantities in Lemma \ref{lem.020208} (i), (ii) and (iii) can be bounded by the following integral with $i=1,2,3$, respectively, 
\begin{align}
&e^{O(W\log N)}\int_{\mathbf{I}^{b,i}_\mathfrak{D}} \prod_{j=1}^W {\rm d} b_{j,1} {\rm d} b_{j,2} \int_{\Sigma^W} \prod_{j=1}^W {\rm d}x_{j,1} \int_{\Sigma^W} \prod_{j=1}^W {\rm d} x_{j,2}  \int_{\mathbb{R}_+^{W-1}} \prod_{j=2}^W 2t_j {\rm d} t_j \int_{\mathbb{I}^{W-1}} \prod_{j=2}^W 2v_j {\rm d} v_j\nonumber\\
& \times  \exp\left\{-M\big(\Re K(\hat{X},V)+\Re L(\hat{B},T)\big)\right\}\cdot|\mathsf{A}(\hat{X}, \hat{B}, V, T)|\cdot \prod_{j=1}^W (r_{j,1}+r_{j,2})^2\nonumber\\
&\leq  e^{O(N)} e^{-cM\mathfrak{D}^2} \int_{\mathbf{I}^{b,i}_\mathfrak{D}}\prod_{j=1}^W {\rm d} b_{j,1} {\rm d} b_{j,2} \int_{\mathbb{R}_+^{W-1}}\prod_{j=2}^W  {\rm d} t_j \prod_{a=1,2}\prod_{j=1}^We^{-cMr_{j,a}^2 }\; r_{j,a}^M\cdot \mathbb{M}(\mathbf{t})\cdot \tilde{\mathfrak{p}}(\mathbf{r},\mathbf{r}^{-1},\mathbf{t}). \label{020265}
\end{align} 

According to the facts $\kappa_1=e^{O(W)}$ and $\kappa_2=O(1)$ in (\ref{011161}), it is suffices to consider one monomial in $\tilde{\mathfrak{p}}(\mathbf{r},\mathbf{r}^{-1},\mathbf{t})$ with bounded coefficient. That means, it suffices to estimate the integral
\begin{eqnarray}
 \int_{\mathbf{I}^{b,i}_\mathfrak{D}}\prod_{j=1}^W {\rm d} b_{j,1} {\rm d} b_{j,2} \int_{\mathbb{R}_+^{W-1}} {\rm d} t_j \prod_{a=1,2}\prod_{j=1}^W e^{-cMr_{j,a}^2 }\; r_{j,a}^M\cdot \mathbb{M}(\mathbf{t})\cdot \tilde{\mathfrak{q}}(\mathbf{r},\mathbf{r}^{-1},\mathbf{t}),\quad i=1,2,3, \label{011370}
\end{eqnarray}
for some monomial
\begin{eqnarray}
\quad\tilde{\mathfrak{q}}(\mathbf{r},\mathbf{r}^{-1},\mathbf{t})=\prod_{a=1,2}\prod_{j=1}^W r_{j,a}^{\ell_{j,a}}\prod_{j=2}^W t_j^{n_j},\quad |\ell_{j,a}|, n_j=O(1),\quad \text{for}\quad j=1,\ldots, W, \quad a=1,2, \label{0113200}
\end{eqnarray}
where the bound on $\ell_j$'s and $n_j$'s follows from the fact that $\kappa_3=O(1)$ in (\ref{011161}). 

Bounding $t_j$'s by $1$ trivially in the region $\mathbf{t}\in \mathbb{I}^{W-1}$ and using Lemma \ref{lem.011001} in the region $\mathbf{t}\in \mathbb{R}_+^{W-1}\setminus\mathbb{I}^{W-1}$, we can get for $i=1,2,3$,
\begin{eqnarray*}
(\ref{011370})\leq  e^{O(W^2\log N)}\int_{\mathbf{I}^{b,i}_\mathfrak{D}}\prod_{j=1}^W {\rm d} b_{j,1} {\rm d} b_{j,2} \cdot \prod_{a=1,2}\prod_{j=1}^W e^{-cMr_{j,a}^2 }\; r_{j,a}^{M+\ell_{j,a}}\big(1+\mathfrak{L}^{-\frac12}\big)^{O(W^2)}.
\end{eqnarray*}
By the definition of $\mathfrak{A}(\hat{B})$ in (\ref{020260}) and the assumption $M\gg W^4$ we see that
\begin{eqnarray*}
\prod_{a=1,2}\prod_{j=1}^Wr_{j,a}^{M+\ell_{j,a}}\big(1+\mathfrak{L}^{-\frac12}\big)^{O(W^2)}\leq \prod_{a=1,2}\prod_{j=1}^Wr_{j,a}^{M(1+o(1))}.
\end{eqnarray*}
Consequently, by using elementary Gaussian integral, we can get the trivial bound
\begin{eqnarray*}
(\ref{011370})\leq e^{O(N\log N)},\quad i=1,2,3,
\end{eqnarray*}
and then we have for $i=1,2,3$,
\begin{eqnarray*}
(\ref{020265})\leq e^{-cM\mathfrak{D}^2+O(N\log N)}\to 0,\quad \text{as}\quad \mathfrak{D}\to \infty.
\end{eqnarray*}
Thus we completed the proof.
\end{proof}
 \subsection{Proof of Lemma \ref{lem.121601}} Plugging  the first identity of (\ref{011537}) and (\ref{011160}) into (\ref{011159}), we can  write
 \begin{align}
&\left|\mathcal{I}(\mathbf{I}^{b}_1, \mathbf{I}^b_2, \mathbf{I}^x_1, \mathbf{I}^x_2, \mathbf{I}^t,\mathbf{I}^v)\right|\leq e^{O(WN^{\varepsilon_2})}\int_{\mathbf{I}^b_1} \prod_{j=1}^W {\rm d} b_{j,1} \int_{\mathbf{I}^b_2} \prod_{j=1}^W {\rm d} b_{j,2} \int_{\mathbf{I}^x_1} \prod_{j=1}^W {\rm d} x_{j,1} \int_{\mathbf{I}^x_2} \prod_{j=1}^W {\rm d} x_{j,2}  \nonumber\\
&\hspace{15ex} \times \int_{\mathbf{I}^t} \prod_{j=2}^W {\rm d} t_j \int_{\mathbf{I}^v} \prod_{j=2}^W {\rm d} v_j \exp\left\{-M\left(\Re \mathring{K}(\hat{X},V)+\Re \mathring{L}(\hat{B},T)\right)\right\}\tilde{\mathfrak{p}}(\mathbf{r},\mathbf{r}^{-1},\mathbf{t}),\label{021890}
 \end{align} 
where $\tilde{\mathfrak{p}}(\mathbf{r},\mathbf{r}^{-1},\mathbf{t})$ is specified in (\ref{011161}).

Lemma \ref{lem.121601} immediately follows from the following two lemmas.
\begin{lem}  \label{lem.0113100}
Under Assumptions \ref{assu.1} and \ref{assu.4}, we have  
\begin{eqnarray}
\mathcal{I}\big(\Gamma^{W}, \bar{\Gamma}^{W}, \Sigma^W,\Sigma^W, \mathbb{R}_+^{W-1}\setminus\mathbb{I}^{W-1},\mathbb{I}^{W-1}\big)\leq e^{-\Theta^2}. \label{021894}
 \end{eqnarray}
\end{lem} 
\begin{lem} \label{lem.020211} Under Assumptions \ref{assu.1} and \ref{assu.4}, we have 
\begin{align}
\mathcal{I}\big(\Gamma^{W}, \bar{\Gamma}^{W}, \Sigma^W,\Sigma^W, \mathbb{I}^{W-1},\mathbb{I}^{W-1}\big)&=2^W\mathcal{I}\big(\Upsilon^b_+, \Upsilon^b_-, \Upsilon^x_+, \Upsilon^x_-, \Upsilon_S,\Upsilon_S\big)\nonumber\\
&+ \mathcal{I}\big(\Upsilon^b_+, \Upsilon^b_-, \Upsilon^x_+, \Upsilon^x_+, \Upsilon_S,\mathbb{I}^{W-1}\big)\nonumber\\
&+\mathcal{I}\big(\Upsilon^b_+, \Upsilon^b_-, \Upsilon^x_-, \Upsilon^x_-, \Upsilon_S,\mathbb{I}^{W-1}\big)+O( e^{-\Theta}). \label{030930}
\end{align}
\end{lem}
 In the sequel, we prove Lemmas \ref{lem.0113100} and \ref{lem.020211}.
 \begin{proof}[Proof of Lemma \ref{lem.0113100}] Recall (\ref{021890}) with the choice of the integration domains
 \begin{eqnarray*}
\big(\mathbf{I}^{b}_1, \mathbf{I}^b_2, \mathbf{I}^x_1, \mathbf{I}^x_2, \mathbf{I}^t,\mathbf{I}^v\big)=\big(\Gamma^{W}, \bar{\Gamma}^{W}, \Sigma^W,\Sigma^W, \mathbb{R}_+^{W-1}\setminus\mathbb{I}^{W-1},\mathbb{I}^{W-1}\big).
 \end{eqnarray*}
To simplify the integral on the r.h.s. of (\ref{021890}), we use the fact $\Re \mathring{K}(\hat{X},V)\geq 0$ implied by (\ref{0129994}), together with the facts that the $\mathbf{x}$ and $\mathbf{v}$-variables are bounded by $1$. Consequently, we can eliminate the integral over $\mathbf{x}$ and $\mathbf{v}$-variables from the integral on the r.h.s. of (\ref{021890}). Moreover, according to (\ref{011161}), it suffices to prove
 \begin{eqnarray}
\int_{\Gamma^W} \prod_{j=1}^W {\rm d}  b_{j,1} \int_{\bar{\Gamma}^W} \prod_{j=1}^W {\rm d}  b_{j,2}  \int_{\mathbb{R}_+^{W-1}\setminus \mathbb{I}^{W-1}} \prod_{j=2}^W {\rm d}  t_j \exp\left\{-M\Re \mathring{L}(\hat{B},T)\right\}\cdot\tilde{\mathfrak{q}}(\mathbf{r},\mathbf{r}^{-1},\mathbf{t})\leq e^{-\Theta^2} \label{021893}
 \end{eqnarray}
 instead, where $\tilde{\mathfrak{q}}(\cdot)$ is the monomial defined in (\ref{0113200}). 
 
 Now, by the first inequality of (\ref{0113500}), we have
 \begin{eqnarray}
\exp\big\{-M\Re \mathring{L}(\hat{B},T)\big\}\leq  \exp\Big{\{}-M\sum_{a=1,2}\sum_{j=1}^W\Big(c(r_{j,a}-1)^2+(r_{j,a}-\log r_{j,a}-1)\Big)\Big{\}}\cdot \mathbb{M}(\mathbf{t}). \label{0113701}
 \end{eqnarray}
At first, we integrate $\mathbf{t}$-variables out by using Lemma \ref{lem.011001}, namely,
 \begin{eqnarray}
 &&\int_{\mathbb{R}_+^{W-1}\setminus \mathbb{I}^{W-1}} \prod_{j=2}^W {\rm d}  t_j\;  \mathbb{M}(\mathbf{t})\cdot\tilde{\mathfrak{q}}(\mathbf{r},\mathbf{r}^{-1},\mathbf{t})\nonumber\\
 &&\leq \prod_{a=1,2}\prod_{j=1}^W r_{j,a}^{\ell_{j,a}}\cdot \Big(1+\mathfrak{L}^{-\frac12}\Big)^{O(W^2)}\exp\Big\{-\Theta^2\mathfrak{A}(\hat{B})+O(W^2\log N)\Big\}. \label{0113700}
 \end{eqnarray}
 
Recall the definitions of $\mathfrak{L}$ and $\mathfrak{A}(\hat{B})$ in (\ref{031001}) and (\ref{020260}), and the fact $W^2=o(M^{1/2})$. The proof of the following fact is an elementary exercise 
\begin{eqnarray}
 \sup_{r\in\mathbb{R}_+}\exp\big\{-M(r-\log r-1)\big\}r^{\ell}=O(1), \quad \text{if}\quad \ell=O(M^{\frac12}). \label{011190}
\end{eqnarray}
Hence, we get the bound
\begin{align}
&\prod_{a=1,2}\prod_{j=1}^W\exp\Big\{-M\big(r_{j,a}-\log r_{j,a}-1\big)\Big\}\cdot r_{j,a}^{\ell_{j,a}}\cdot \Big(1+\mathfrak{L}^{-\frac12}\Big)^{O(W^2)}\nonumber\\
&\leq \max_{b=1,2}\max_{k=1,\ldots,W}\prod_{a=1,2} \prod_{j=1}^W\exp\Big\{-M\big(r_{j,a}-\log r_{j,a}-1\big)\Big\}\cdot r_{j,a}^{\ell_{j,a}} \left(1+r_{k,b}^{-1}\right)^{O(W^2)} =e^{O(W^2)}.
\label{0113702}
\end{align}
Consequently, (\ref{0113701})-(\ref{0113702}) imply that
\begin{eqnarray}
&&\int_{\mathbb{R}_+^{W-1}\setminus \mathbb{I}^{W-1}} \prod_{j=2}^W {\rm d}  t_j \;\exp\left\{-M\Re \mathring{L}(\hat{B},T)\right\}\cdot \tilde{\mathfrak{q}}(\mathbf{r},\mathbf{r}^{-1},\mathbf{t})\nonumber\\
&&\leq \exp\left\{-\Theta^2\mathfrak{A}(\hat{B})+O(W^2\log N)\right\}\cdot\prod_{a=1,2}\prod_{j=1}^W \exp\left\{-cM(r_{j,a}-1)^2\right\}\nonumber\\
&&\leq e^{-\Theta^2}\cdot\prod_{a=1,2}\prod_{j=1}^W \exp\left\{-\frac{c}{2}M(r_{j,a}-1)^2\right\}, \label{0113900}
\end{eqnarray}
for some positive constant $c$,
where in the last step we use the obvious fact
\begin{eqnarray*}
\Theta^2\mathfrak{A}(\hat{B})+\frac{c}{2}\sum_{a=1,2}\sum_{j=1}^WM(r_{j,a}-1)^2\geq \Theta^2\gg W^2\log N
\end{eqnarray*}
by (\ref{021105}) and the definition of $\mathfrak{A}(\hat{B})$ in (\ref{020260}). Plugging the bound (\ref{0113900}) into the l.h.s of (\ref{021893}) and taking the integral over $\hat{B}$-variables we can see that (\ref{021893}) holds, which further implies (\ref{021894}). Therefore, we completed the proof of Lemma \ref{lem.0113100}.
 \end{proof}
To prove Lemma \ref{lem.020211}, we split the exponential function into two parts. We use one part to control the integral, and the other will be estimated by its magnitude. More specifically, we shall prove the following two lemmas. 
 \begin{lem} \label{lem.0111110}Under Assumptions \ref{assu.1} and \ref{assu.4}, we have
 \begin{eqnarray}
&&\int_{\Gamma^W} \prod_{j=1}^W {\rm d}  b_{j,1} \int_{\bar{\Gamma}^W} \prod_{j=1}^W {\rm d}  b_{j,2} \int_{\Sigma^W} \prod_{j=1}^W {\rm d}  x_{j,1} \int_{\Sigma^W} \prod_{j=1}^W {\rm d}  x_{j,2}  \int_{\mathbb{I}^{W-1}} \prod_{j=2}^W {\rm d}  t_j \int_{\mathbb{I}^{W-1}} \prod_{j=2}^W {\rm d}  v_j\nonumber\\
&& \times  \exp\Big\{-\frac{1}{2}M\big(\Re \mathring{K}(\hat{X},V)+\Re \mathring{L}(\hat{B},T)\big)\Big\}\cdot\tilde{\mathfrak{p}}(\mathbf{r},\mathbf{r}^{-1},\mathbf{t})\leq e^{O(W)}. \label{0218100}
 \end{eqnarray}
  \end{lem}

  \begin{lem} \label{lem.0113101}If $(\mathbf{b}_1,\mathbf{b}_2,\mathbf{x}_1,\mathbf{x}_2,\mathbf{t}, \mathbf{v})\in \Gamma^W\times\bar{\Gamma}^W\times\Sigma^W\times \Sigma^W\times \mathbb{I}^{W-1}\times\mathbb{I}^{W-1}$, but not in any of the Types I, II, III vicinities in Definition \ref{021701}, we have
  \begin{eqnarray}
  \exp\Big\{-\frac{1}{2}M\big(\Re \mathring{K}(\hat{X},V)+\Re \mathring{L}(\hat{B},T)\big)\Big\}\leq e^{-\Theta}. \label{012351}
  \end{eqnarray}
  \end{lem}
With Lemmas \ref{lem.0111110} and \ref{lem.0113101}, we can prove  Lemma \ref{lem.020211}. 
\begin{proof}[Proof of Lemma \ref{lem.020211}] For the sake of simplicity, in this proof, we temporarily use $\mathcal{I}_{\text{full}}$ to represent the l.h.s. of (\ref{030930}), i.e. the integral over the full domain, and use $\mathcal{I}_{I}$, $\mathcal{I}_{II}$ and $\mathcal{I}_{III}$ to represent the first three terms on the r.h.s. of (\ref{030930}). Now, combining (\ref{021890}), (\ref{012351}) and (\ref{0218100}), we see that,
\begin{eqnarray*}
\left|\mathcal{I}_{\text{full}}-\mathcal{I}_{I}-\mathcal{I}_{II}-\mathcal{I}_{III}\right|\leq e^{O(WN^{\varepsilon_2})}\cdot e^{-\Theta}\cdot e^{O(W)}\leq e^{-\Theta},
\end{eqnarray*}
in light of the definition of $\Theta$ in  (\ref{021102}) and the assumption (\ref{0218105}). Hence, we completed the proof of Lemma \ref{lem.020211}.
\end{proof}

 \begin{proof}[Proof of Lemma \ref{lem.0111110}] At first, again, the polynomial $\tilde{\mathfrak{p}}(\cdot)$ in the integrand can be replaced by the monomial $\tilde{\mathfrak{q}}(\cdot)$ defined in (\ref{0113200}) in the discussion below, owing to the fact that $\kappa_1=\exp\{O(W)\}$ in (\ref{011161}).  Then, the proof is similar to that of Lemma \ref{lem.0113100}, but much simpler, since $\mathbf{t}$-variables are bounded by $1$ now. Consequently, we can eliminate $\hat{X}$, $\mathbf{t}$ and $\mathbf{v}$-variables from the integral directly and use the trivial bounds
 \begin{eqnarray}
 \tilde{\mathfrak{q}}(\mathbf{r},\mathbf{r}^{-1},\mathbf{t})\leq \prod_{a=1,2}\prod_{j=1}^W r_{j,a}^{\ell_{j,a}},\qquad \Re \mathring{L}(\hat{B},T)
&\geq&  c\sum_{a=1,2}\sum_{j=1}^W (r_{j,a}-1)^2,   \label{021904}
 \end{eqnarray}  
 where the latter is from (\ref{0113500}). Hence, it suffices to show
 \begin{eqnarray}
\int_{\Gamma^W} \prod_{j=1}^W {\rm d}  b_{j,1} \int_{\bar{\Gamma}^W} \prod_{j=1}^W {\rm d}  b_{j,2}  \prod_{a=1,2}\prod_{j=1}^W \exp\left\{-cM(r_{j,a}-1)^2\right\} r_{j,a}^{\ell_{j,a}}\leq e^{O(W)}. \label{021901}
 \end{eqnarray}
Note that (\ref{021901}) follows from elementary Gaussian integral  immediately. Therefore, we completed the proof of Lemma \ref{lem.0111110}.
 \end{proof}
 
 \begin{proof}[Proof of Lemma \ref{lem.0113101}] At first, according to  (\ref{0113500}) and (\ref{0129994}), we see both $M\Re\mathring{L}(\hat{B},T)$ and $M\Re \mathring{K}(\hat{X},V)$ are nonnegative on the full domain. Hence, it suffices to show one of them is larger than $\Theta$  outside the Type I, II, III vicinities. 
 
 Note that for each type of vicinity, we have
 \begin{eqnarray}
 (\mathbf{b}_1,\mathbf{b}_2,\mathbf{t})\in  \Upsilon^b_+\times \Upsilon^b_-\times \Upsilon_S. \label{021902}
 \end{eqnarray}
 Now, if (\ref{021902}) is violated, we have $(\mathbf{b}_1,\mathbf{b}_2)\in \Gamma^W\times \bar{\Gamma}^W\setminus  \Upsilon^b_+\times \Upsilon^b_-$ or $\mathbf{t}\in \mathbb{I}^{W-1}\setminus \Upsilon_S$. If the former holds, by using (\ref{0113500}) and the definition of $\Upsilon^b_+$ and $\Upsilon^b_-$ in (\ref{021907}), we have
 \begin{eqnarray*}
 M\Re \mathring{L}(\hat{B},T)\geq  cM\sum_{a=1,2}\sum_{j=1}^W (r_{j,a}-1)^2=cM||\mathbf{b}_1-a_+||_2^2+cM||\mathbf{b}_2+a_-||_2^2\geq \Theta,
 \end{eqnarray*}
 which shows (\ref{012351}) if $(\mathbf{b}_1,\mathbf{b}_2)\in \Gamma^W\times \bar{\Gamma}^W\setminus  \Upsilon^b_+\times \Upsilon^b_-$.
 
 Hence, it suffices to consider the case $(\mathbf{b}_1,\mathbf{b}_2)\in \Upsilon^b_+\times \Upsilon^b_-$, $\mathbf{t}\in \mathbb{I}^{W-1}\setminus \Upsilon_S$. Using Lemma \ref{lem.011002}, we can see that 
 \begin{eqnarray}
 \mathbb{M}(\mathbf{t})\leq \exp\left\{- \frac{M}{6}\mathfrak{A}(\hat{B})(-\mathbf{t}S^{(1)}\mathbf{t})\right\}\leq \exp\left\{- \mathfrak{A}(\hat{B})\Theta\right\}\leq e^{- \Theta}, \label{021910}
 \end{eqnarray}
 where in the second step we used the definition of $\Upsilon_S$ in (\ref{021907}) and in the last step we used the fact $\mathfrak{A}(\hat{B})\geq c$ if $(\mathbf{b}_1,\mathbf{b}_2)\in \Upsilon^b_+\times \Upsilon^b_-$. Then (\ref{0113701}) and (\ref{021910}) also imply (\ref{012351}). Now, we turn to show $M\Re\mathring{K}(\hat{X},V)\geq \Theta$ outside the vicinities. Recalling the definition of $\vartheta_j$'s in (\ref{0129998}), we split the discussion into two cases
 \begin{eqnarray*}
(i): \Big(\sin \vartheta_j-\frac{E}{2}\Big)^2\leq \frac{\Theta}{M}, \quad \forall\; j=1,\ldots, 2W,\qquad (ii): \Big(\sin \vartheta_j-\frac{E}{2}\Big)^2>\frac{\Theta}{M}, \quad \text{for some}\quad  j\in\{1,\ldots, 2W\}.
 \end{eqnarray*}  
 
 Using (\ref{0129994}), we can get $M\Re\mathring{K}(\hat{X},V)\geq \Theta$ in case (ii) immediately, so we can assume case (i) below. Then (i) implies that
 \begin{eqnarray}
 |\arg(a_+^{-1}x_{j,a})|^2\wedge|\arg(a_-^{-1}x_{j,a})|^2\leq \frac{\Theta}{M}, \quad \forall\; j=1,\ldots, W; a=1,2. \label{021930}
 \end{eqnarray}
 Now, we claim that it suffices to focus on the following three subcases of (\ref{021930}),
 \begin{enumerate}
 \item[(i')] There exists a sequence of permutations of $\{1,2\}$, namely,  $\boldsymbol{\epsilon}=(\epsilon_1,\ldots, \epsilon_W)$, such that 
 \begin{eqnarray*}
 |\arg(a_+^{-1}x_{j,\epsilon_j(1)})|^2\leq \frac{\Theta}{M},\quad |\arg(a_-^{-1}x_{j,\epsilon_j(2)})|^2\leq \frac{\Theta}{M}, \quad \forall\; j=1,\ldots, W.
 \end{eqnarray*}
 \item[(ii')] There exists
  \begin{eqnarray*}
 |\arg(a_+^{-1}x_{j,a})|^2\leq \frac{\Theta}{M}, \quad \forall\; j=1,\ldots, W; a=1,2.
 \end{eqnarray*}
 \item[(iii')] There exists
   \begin{eqnarray*}
 |\arg(a_-^{-1}x_{j,a})|^2\leq \frac{\Theta}{M}, \quad \forall\; j=1,\ldots, W; a=1,2.
 \end{eqnarray*}
 \end{enumerate}
 
 To see this, note that for those $\hat{X}$-variables which satisfy (\ref{021930}) but do not belong to any of the case (i'), (ii') or (iii') listed above, there must be a pair $\{i,j\}\in\mathcal{E}$ such that 
  \begin{eqnarray}
 |\arg(a_+^{-1}x_{i,1})|^2, |\arg(a_+^{-1}x_{i,2})|^2\leq \frac{\Theta}{M},\quad  |\arg(a_-^{-1}x_{j,1})|^2, |\arg(a_-^{-1}x_{j,2})|^2\leq \frac{\Theta}{M}, \label{021951}
  \end{eqnarray}
  or there exists a permutation $\epsilon_i$ such that 
 \begin{eqnarray}
 |\arg(a_+^{-1}x_{i,\epsilon_i(1)})|^2, |\arg(a_-^{-1}x_{i,\epsilon_i(2)})|^2\leq \frac{\Theta}{M},\quad  |\arg(a_+^{-1}x_{j,1})|^2, |\arg(a_+^{-1}x_{j,2})|^2\leq \frac{\Theta}{M}, \label{021952}
 \end{eqnarray}
 or 
  \begin{eqnarray}
 |\arg(a_+^{-1}x_{i,\epsilon_i(1)})|^2, |\arg(a_-^{-1}x_{i,\epsilon_i(2)})|^2\leq \frac{\Theta}{M},\quad  |\arg(a_-^{-1}x_{j,1})|^2, |\arg(a_-^{-1}x_{j,2})|^2\leq \frac{\Theta}{M}. \label{021953}
 \end{eqnarray}
 For each of (\ref{021951}), (\ref{021952}) and (\ref{021953}), we can perform a discussion similar to (\ref{013101})-(\ref{021740}), to show that $M\Re\mathring{K}(\hat{X},V)\geq cM\gg \Theta$. 
 
 Hence, it suffices to focus on cases (i'), (ii') and (iii') in the sequel. Now, we denote the domains of $\hat{X}$-variables satisfying (i'), (ii') and (iii') by $\Upsilon^x_{I}$, $\Upsilon^x_{II}$ and $\Upsilon^x_{III}$, respectively. In addition, in the remaining part of this proof, we temporarily use the term Type A vicinity to represent its restriction on $\hat{X}$-variables, for $A=I,II,III$. Obviously, we have
 \begin{eqnarray*}
 \text{Type A vicinity}\subset \Upsilon^x_{A},\qquad A=I,II,III.
 \end{eqnarray*}
If $(\mathbf{x}_1,\mathbf{x}_2)\in \Upsilon^x_I$ but outside the Type I vicinity, 
\begin{eqnarray*}
||\arg (a_+^{-1}\mathbf{x}_{\boldsymbol{\epsilon}(1)})||_2^2+||\arg (a_-^{-1}\mathbf{x}_{\boldsymbol{\epsilon}(2)})||_2^2\geq \frac{\Theta}{M}, \quad \forall\; \boldsymbol{\epsilon},
\end{eqnarray*}
which easily implies that
\begin{eqnarray}
M\Re \mathring{K}(\hat{X},V)\geq cM\sum_{j=1}^{2W}\Big(\sin \vartheta_j-\frac{E}{2}\Big)^2\geq \Theta. \label{021968}
\end{eqnarray}
Now, we turn to the case  that there exists one sequence of permutations $\boldsymbol{\epsilon}$ such that $(\mathbf{x}_{\boldsymbol{\epsilon}(1)},\mathbf{x}_{\boldsymbol{\epsilon}(2)})\in \Upsilon_+^x\times \Upsilon_-^x$, but $\mathbf{v}_{\boldsymbol{\epsilon}}\not\in \Upsilon_S$. In this case, we just go back to the first line of (\ref{020341}) and do the transform
\begin{eqnarray*}
V_j\to \mathfrak{I} V_j, \quad \hat{X}_j\to \mathfrak{I} \hat{X}_j\mathfrak{I}
\end{eqnarray*}
for those $j$ with $\epsilon(j)\neq \epsilon(1)$, where $\mathfrak{I}$ is defined in (\ref{021961}). Then, it suffices to consider
\begin{eqnarray*}
(\mathbf{x}_{1},\mathbf{x}_{2})\in \Upsilon_+^x\times \Upsilon_-^x,\quad \text{but}\quad \mathbf{v}\not\in \Upsilon_S,\qquad \text{or}\qquad 
(\mathbf{x}_{1},\mathbf{x}_{2})\in \Upsilon_-^x\times \Upsilon_+^x,\quad \text{but}\quad \mathbf{v}\not\in \Upsilon_S.
\end{eqnarray*}
In either case, we can show that $M\Re\mathring{K}(\hat{X},V)\geq \Theta$, analogously to case of $(\mathbf{b}_{1},\mathbf{b}_{2})\in \Upsilon_+^x\times \Upsilon_-^x$ but $\mathbf{t}\not\in \Upsilon_S$, in (\ref{021910}).

Now, what remains is to show that for those $(\mathbf{x}_1,\mathbf{x}_2)\in \Upsilon^x_{A}$ but outside the Type A vicinity (A=II, III), we have $M\Re\mathring{K}(\hat{X},V)\geq \Theta$. We only discuss the case $A=II$, the other is analogous. Note that outside the Type II vicinity of $\hat{X}$ variables we have
\begin{eqnarray}
||\arg (a_+^{-1}\mathbf{x}_1)||_2^2+||\arg (a_+^{-1}\mathbf{x}_2)||_2^2\geq \frac{\Theta}{M}. \label{021980}
\end{eqnarray} 
Observe that now we are already in $\Upsilon^x_{II}$, which means that all $x_{j,a}$'s are close to $a_+$ and far away from $a_-$. That means, we have $\sin(\arg(x_{ja}))-E/2\sim \arg(a_+^{-1}x_{ja})$. Consequently, (\ref{021980}) also implies (\ref{021968}). Therefore, we completed the proof of Lemma \ref{lem.0113101}.
 \end{proof}
\subsection{Proof of Lemma \ref{lem.011002}}
Using the definition of $\ell_S(\hat{B},T)$ in (\ref{020401}) and $\mathfrak{A}(\hat{B})$ in (\ref{020260}) and the fact $|(T_kT_j^{-1})_{12}|=|s_jt_ke^{\mathbf{i}\sigma_k}-s_kt_je^{\mathbf{i}\sigma_j}|$, we have
\begin{eqnarray}
\Re \ell_S(\hat{B},T)\geq  \frac12\mathfrak{A}(\hat{B})\sum_{j,k} \mathfrak{s}_{jk}|s_jt_ke^{\mathbf{i}\sigma_k}-s_kt_je^{\mathbf{i}\sigma_j}|^2.\label{021970}
\end{eqnarray}
Simple estimate using $s_j^2=1+t_j^2$ shows that 
\begin{eqnarray}
|s_jt_ke^{\mathbf{i}\sigma_k}-s_kt_je^{\mathbf{i}\sigma_j}|^2\geq\frac{1}{4}(t_k-t_j)^2\Big{(}\frac{1}{1+2t_j^2}+\frac{1}{1+2t_k^2}\Big{)}\geq \frac{1}{6}(t_k-t_j)^2. \label{010740}
\end{eqnarray}
Notice that the assumption $\mathbf{t}\in \mathbb{I}^{W-1}$ was used only in the last inequality.
By (\ref{021970}), (\ref{010740}) and the definition (\ref{011011}), Lemma \ref{lem.011002} follows immediately.
\subsection{Proof of Lemma \ref{lem.011001}}
 Let $\mathbb{I}^c=\mathbb{R}_+\setminus \mathbb{I}$. Now we consider the domain sequence $\vec{\mathbb{J}}=(\mathbb{J}_2,\ldots, \mathbb{J}_{W})\in \{\mathbb{I}, \mathbb{I}^c\}^{W-1}$. 
We decompose the integral in Lemma \ref{lem.011001} as follows
\begin{eqnarray}
\int_{\mathbb{R}^{W-1}\setminus \mathbb{I}^{W-1}} \prod_{j=2}^W {\rm d} t_j\;  \mathbb{M}(\mathbf{t})\mathfrak{q}(\mathbf{t})=\sum_{\substack{\vec{\mathbb{J}}\in\{\mathbb{I},\mathbb{I}^c\}^{W-1}\\\vec{\mathbb{J}}\neq \mathbb{I}^{W-1}}}\int_{\prod_{j=2}^W \mathbb{J}_j} \prod_{j=2}^W {\rm d} t_j\;  \mathbb{M}(\mathbf{t})\mathfrak{q}(\mathbf{t}). \label{121802}
\end{eqnarray}
Note the total number of the choices of such $\vec{\mathbb{J}}$ in the sum above is $2^{W-1}-1$. It suffices to consider one of these sequences $\vec{\mathbb{J}}\in \{\mathbb{I}, \mathbb{I}^c\}^{W-1}$ in which there is at least one $i$ such that $\mathbb{J}_i=\mathbb{I}^c$. 

Recall the spanning tree $\mathcal{G}_0=(\mathcal{V},\mathcal{E}_0)$ in Assumption \ref{assu.1}. The simplest case is that there exists a linear spanning tree (a path) $\mathcal{G}_0$ with
\begin{eqnarray}
\mathcal{E}_0=\{(i,i+1)\}_{i=1}^{W-1}\subset \mathcal{E}. \label{021990}
\end{eqnarray}
We first present the proof in this simplest case.

Now, we only keep the edges in the path $\mathcal{E}_0$, i.e. the terms with $k=j-1$ in (\ref{021970}), we also trivially discard the term $1/(1+2t_{j}^2)$ from the sum $1/(1+2t_{j-1}^2)+1/(1+2t_{j}^2)$ in the estimate (\ref{010740}) (the first inequality), and finally we bound all $M\mathfrak{A}(\hat{B})\frak{s}_{j-1,j}/4$ by $\mathfrak{L}$ defined in (\ref{031001}) from below. That means, we use the bound
\begin{eqnarray}
\mathbb{M}(\mathbf{t})\leq \prod_{j=2}^{W}\exp\Big{\{}-\mathfrak{L}\frac{(t_{j}-t_{j-1})^2}{1+2t_{j-1}^2}\Big{\}}:=\prod_{j=2}^{W}\breve{\mathbb{M}}_j(\mathbf{t}). \label{020301}
\end{eqnarray}
Consequently, we have
\begin{eqnarray}
\int_{\prod_{j=2}^W \mathbb{J}_j} \prod_{j=2}^W {\rm d} t_j\;  {\mathbb{M}}(\mathbf{t})\mathfrak{q}(\mathbf{t})
\leq  \int_{ \prod_{j=2}^{W}\mathbb{J}_j} \prod_{j=2}^{W}{\rm d} t_j\; \prod_{j=2}^{W} t_j^{n_j} \breve{\mathbb{M}}_j(\mathbf{t}).
 \label{011231}
\end{eqnarray}
Note that, as a function of $\mathbf{t}$, $\breve{\mathbb{M}}_j(\mathbf{t})$ only depends on $t_{j-1}$ and $t_{j}$.

Having fixed $\vec{\mathbb{J}}$, assume that $k$ is the largest index such that $\mathbb{J}_{k}=\mathbb{I}^c$, i.e. $t_{k+1},\ldots, t_W\in\mathbb{I}$.
Now, we claim that 
\begin{eqnarray}
\sum_{j=2}^{W} \frac{(t_{j}-t_{j-1})^2}{1+2t_{j-1}^2}\geq \sum_{j=2}^{k} \frac{(t_{j}-t_{j-1})^2}{1+2t_{j-1}^2}\geq \frac{1}{300k^2},\quad \text{if}\quad t_{k}\in\mathbb{I}^c. \label{031003}
\end{eqnarray}
To see (\ref{031003}), we use the following elementary facts
\begin{eqnarray}
\frac{ (t_{j}-t_{j-1})^2}{1+2t_{j-1}^2}\geq \frac13 \frac{(t_{j}-t_{j-1})^2}{t_{j-1}^2}=\frac13(t_{j}/t_{j-1}-1)^2,\quad \text{if}\quad t_{j-1}\in \mathbb{I}^c \label{011201}
\end{eqnarray}
and
\begin{eqnarray}
\frac{(t_{j}-t_{j-1})^2}{1+2t_{j-1}^2}\geq \frac13 (t_{j}-t_{j-1})^2,\quad \text{if}\quad t_{j-1}\in \mathbb{I} \label{011202}
\end{eqnarray}
for all $j=2,\ldots, W$. We show (\ref{031003}) by contradiction. If (\ref{031003}) is violated, we have
\begin{eqnarray*}
\frac{ (t_{j}-t_{j-1})^2}{1+2t_{j-1}^2}\leq \frac{1}{300k^2},\quad \forall\; j=2,\ldots, k, 
\end{eqnarray*}
which together with (\ref{011201}) and (\ref{011202}) implies that
\begin{eqnarray}
t_j\leq t_{j-1}\Big(1+\frac{1}{10k}\Big)+\frac{1}{10k}. \label{031010}
\end{eqnarray}
Using (\ref{031010}) recursively yields
\begin{eqnarray}
t_{k}\leq \Big(1+\frac{1}{10k}\Big)^{k-1}(t_1+1)-1= \Big(1+\frac{1}{10k}\Big)^{k-1}-1\leq \frac12, \label{031011}
\end{eqnarray}
where in the second step we used the fact $t_1=0$. Note that (\ref{031011}) contradicts $t_{k}\in \mathbb{I}^c$. Hence, we verified 
(\ref{031003}). 

Now, we split $\prod_{j=2}^W\breve{\mathbb{M}}_j(\mathbf{t})$ into two parts. We use one to control the integral, and the other will be estimated by (\ref{031003}). Specifically, substituting (\ref{031003}) into (\ref{011231}) we have
\begin{eqnarray}
\int_{\prod_{j=2}^W \mathbb{J}_j} \prod_{j=2}^W {\rm d} t_j\; {\mathbb{M}}(\mathbf{t})\mathfrak{q}(\mathbf{t})\leq e^{-\frac{\mathfrak{L}}{600k^2}} \int_{\mathbb{R}_+^{W-1}}\prod_{j=2}^{W}{\rm d} t_j\; \prod_{j=2}^{W} t_j^{n_j} \big(\breve{\mathbb{M}}_j(\mathbf{t})\big)^{\frac12}.
\label{031015}
\end{eqnarray}
Therefore, what remains is to estimate the integral in (\ref{031015}), which can be done by elementary Gaussian integral step by step.  More specifically, using (\ref{011201}) and (\ref{011202}) and the change of variable $t_j/t_{j-1}-1\to t_j$ in case of $t_{j-1}\in \mathbb{I}^c$ and $t_j-t_{j-1}\to t_j$ in case of $t_{j-1}\in \mathbb{I}$, it is elementary to see that for any $\ell=O(W)$, 
\begin{eqnarray}
\int_{\mathbb{R}_+} dt_j\; t_j^\ell \big(\breve{\mathbb{M}}_j(\mathbf{t})\big)^{\frac12}\leq \ell!! \Big(1+c\mathfrak{L}^{-\frac12}\Big)^{O(\ell)}\; \big(t_{j-1}^{\ell+1}+1\big)\leq  e^{O(W\log N)}\Big(1+\mathfrak{L}^{-\frac12}\Big)^{O(\ell)}\; \big(t_{j-1}^{\ell+1}+1\big). \label{031020}
\end{eqnarray}
Starting from $j=W$, using (\ref{031020}) to integrate (\ref{031015}) successively, the exponent of $t_j$ increases linearly ($n_j=O(1)$), thus we can get 
\begin{eqnarray*}
\int_{\prod_{j=2}^W \mathbb{J}_j} \prod_{j=2}^W {\rm d} t_j\; {\mathbb{M}}(\mathbf{t})\mathfrak{q}(\mathbf{t})\leq e^{-\frac{\mathfrak{L}}{600W^2}}\cdot e^{O(W^2\log N)}\cdot \Big(1+\mathfrak{L}^{-\frac12}\Big)^{O(W^2)} . 
\end{eqnarray*}
Then (\ref{031025}) follows from the definition of $\mathfrak{L}$ in (\ref{031001}) and (\ref{021105}). Hence, we completed the proof for (\ref{031025}) when the spanning tree is given by  (\ref{021990}).

Now, we consider more general spanning tree $\mathcal{G}_0$ and regard $1$ as its root. We start from the generalization of (\ref{020301}), namely,
\begin{eqnarray}
\mathbb{M}(\mathbf{t})\leq \prod_{\{i,j\}\in\mathcal{E}_0}\exp\Big{\{}-\mathfrak{L}\frac{(t_{j}-t_{i})^2}{1+2t_{i}^2}\Big{\}}:=\prod_{\{i,j\}\in\mathcal{E}_0}\breve{\mathbb{M}}_{i,j}(\mathbf{t}). 
\end{eqnarray}
Here we make the convention that $\text{dist}(1,i)=\text{dist}(1,j)-1$ for all $\{i,j\}\in\mathcal{E}_0$, where $\text{dist}(a,b)$ represents the distance between $a$ and $b$. Now, if there is $k'$ such that $\mathbb{J}_{k'}\in \mathbb{I}^c$, we can prove the following analogue of (\ref{031003}), namely,
\begin{eqnarray*}
\sum_{\{i,j\}\in\mathcal{E}_0}\frac{(t_{j}-t_{i})^2}{1+2t_{i}^2}\geq  \frac{1}{300k^2}
\end{eqnarray*}
by performing the argument in (\ref{011201})-(\ref{031011}) on the path connecting $k'$ and the root $1$. Consequently, we can get the analogue of (\ref{031015}) via replacing $\breve{\mathbb{M}}_j(\mathbf{t})$'s by $\breve{\mathbb{M}}_{i,j}(\mathbf{t})$'s.  Finally, integrating $t_j$'s out successively, from the leaves to the root $1$, yields the same conclusion, i.e. (\ref{031025}), for general $\mathcal{G}_0$. Therefore, we completed the proof of Lemma \ref{lem.011001}.
\section{Gaussian measure in the vicinities} \label{s.9}
From now on, we can restrict ourselves to the Type I, II and III vicinities. As a preparation of the proofs of Lemmas  \ref{lem.010602} and \ref{lem.122801}, we will show in this section that the exponential function 
\begin{eqnarray}
\exp\left\{-M\big(\mathring{K}(\hat{X},V)+\mathring{L}(\hat{B},T)\big)\right\} \label{012640}
\end{eqnarray} 
is approximately a Gaussian measure (unnormalized). 
\subsection{Parametrization and initial approximation in the vicinities}  \label{s.9.1}We change the $\mathbf{x}$, $\mathbf{b}$, $\mathbf{t}$, $\mathbf{v}$-variables to a new set of variables, namely, $\mathring{\mathbf{x}}$, $\mathring{\mathbf{b}}$, $\mathring{\mathbf{t}}$ and $\mathring{\mathbf{v}}$. The precise definition of $\mathring{\mathring{x}}$ differs in the different vicinities. To distinguish the parameterization, we set $\varkappa=\pm$, $+$, or $-$, corresponding to Type I, II or III vicinity, respectively. Recalling $D_\varkappa$ from (\ref{0129101}). For each $j$ and each $\varkappa$, we then set
\begin{eqnarray}
&&\hat{X}_j=D_\varkappa\text{diag}\left(\exp\big\{\mathbf{i}\mathring{x}_{j,1}/\sqrt{M}\big\}, \exp\big\{\mathbf{i}\mathring{x}_{j,2}/\sqrt{M}\big\}\right),\quad \mathring{x}_{j,a}/\sqrt{M}\in [-\pi,\pi],\nonumber\\
 &&\hat{B}_j=D_{\pm}+D_{\pm} \text{diag}\left(\mathring{b}_{j,1}/\sqrt{M}, \mathring{b}_{j,2}/\sqrt{M}\right),\qquad t_j=\mathring{t}_j/\sqrt{M}. \label{011615}
\end{eqnarray}
If $\varkappa=\pm$, we also need to parameterize $v_j$ by
\begin{eqnarray}
v_j=\mathring{v}_j/\sqrt{M}. \label{020321}
\end{eqnarray} 

We set the vectors
\begin{eqnarray*}
&&\mathring{\mathbf{b}}_a:=(\mathring{b}_{1,a},\ldots, \mathring{b}_{W,a}),\quad \mathring{\mathbf{x}}_a:=(\mathring{x}_{1,a},\ldots, \mathring{x}_{W,a}), \qquad a=1,2,\nonumber\\
&&\mathring{\mathbf{t}}:=(\mathring{t}_2,\ldots, \mathring{t}_W),\quad \mathring{\mathbf{v}}:=(\mathring{v}_2,\ldots, \mathring{v}_W).
\end{eqnarray*} 
Accordingly, recalling the quantity $\Theta$ from (\ref{021102}), we introduce the domains
\begin{eqnarray*}
\mathring{\Upsilon}\equiv \mathring{\Upsilon}(N, \varepsilon_0):=\{\mathbf{a}\in \mathbb{R}^W: ||\mathbf{a}||_2^2\leq \Theta\}, \quad
\mathring{\Upsilon}_S\equiv \mathring{\Upsilon}_S(N,\varepsilon_0):=\{\mathbf{a}\in \mathbb{R}_+^{W-1}: -\mathbf{a}'S^{(1)}\mathbf{a}\leq \Theta\}.
\end{eqnarray*}
We remind here, as mentioned above, in the sequel, the small constant $\varepsilon_0$ in $\mathring{\Upsilon}$ and $\mathring{\Upsilon}_S$ may be different from line to line, subject to (\ref{0218105}).  Now, by the definition of the Type I', II and III vicinities in Definition \ref{021701} and the parametrization in (\ref{011615}) and (\ref{020321}), we can redefine the vicinities as follows. 
\begin{defi} \label{defi.030301}We can redefine three types of vicinities as follows.
\begin{itemize}
\item  Type I' vicinity :  \hspace{2ex}$\big(\mathring{\mathbf{b}}_1,\mathring{\mathbf{b}}_2,\mathring{\mathbf{x}}_{1},\mathring{\mathbf{x}}_{2},\mathring{\mathbf{t}},\mathring{\mathbf{v}}\big)\in \mathring{\Upsilon}\times \mathring{\Upsilon}\times \mathring{\Upsilon}\times \mathring{\Upsilon} \times \mathring{\Upsilon}_S\times \mathring{\Upsilon}_S $, with $\varkappa=\pm$.
\item Type II vicinity :  \hspace{1ex}$\big(\mathring{\mathbf{b}}_1,\mathring{\mathbf{b}}_2,\mathring{\mathbf{x}}_{1},\mathring{\mathbf{x}}_{2},\mathring{\mathbf{t}},\mathbf{v}\big)\in \mathring{\Upsilon}\times \mathring{\Upsilon}\times \mathring{\Upsilon}\times \mathring{\Upsilon} \times \mathring{\Upsilon}_S\times \mathbb{I}^{W-1}$, with $\varkappa=+$.
\item Type III vicinity :  $\big(\mathring{\mathbf{b}}_1,\mathring{\mathbf{b}}_2,\mathring{\mathbf{x}}_{1},\mathring{\mathbf{x}}_{2},\mathring{\mathbf{t}},\mathbf{v}\big)\in \mathring{\Upsilon}\times \mathring{\Upsilon}\times \mathring{\Upsilon}\times \mathring{\Upsilon} \times \mathring{\Upsilon}_S\times \mathbb{I}^{W-1}$, with $\varkappa=-$.
\end{itemize}
\end{defi}
We recall from (\ref{121801})  the fact
\begin{eqnarray}
\mathring{\mathbf{t}}\in\mathring{\Upsilon}_S\Longrightarrow ||\mathring{\mathbf{t}}||_\infty=O(\Theta). \label{012658}
\end{eqnarray}
Now, we use the representation (\ref{010114}). Then, for the Type I vicinity, we change $\mathbf{x},\mathbf{b},\mathbf{t},\mathbf{v}$-variables to $\mathring{\mathbf{x}},\mathring{\mathbf{b}},\mathring{\mathbf{t}},\mathring{\mathbf{v}}$-variables according to (\ref{011615}) with $\varkappa=\pm$, thus
\begin{align}
&2^W\mathcal{I}\big(\Upsilon^b_+, \Upsilon^b_-, \Upsilon^x_+, \Upsilon^x_-, \Upsilon_S,\Upsilon_S\big)=\frac{M^{2}}{(n!)^24^W\pi^{2W+4}}\int_{\mathbb{L}^{2W-2}} \prod_{j=2}^W\frac{{\rm d} \theta_j}{2\pi}\prod_{j=2}^W \frac{{\rm d} \sigma_j}{2\pi} \int_{\mathring{\Upsilon}} \prod_{j=1}^W {\rm d}  \mathring{b}_{j,1} \int_{\mathring{\Upsilon}} \prod_{j=1}^W {\rm d}  \mathring{b}_{j,2} \nonumber\\
&  \hspace{12ex}\times\int_{\mathring{\Upsilon}} \prod_{j=1}^W {\rm d}  \mathring{x}_{j,1} \int_{\mathring{\Upsilon}} \prod_{j=1}^W {\rm d} \mathring{x}_{j,2}  \int_{\mathring{\Upsilon}_S} \prod_{j=2}^W 2\mathring{t}_j {\rm d}  \mathring{t}_j \int_{\mathring{\Upsilon}_S} \prod_{j=2}^W 2\mathring{v}_j {\rm d}  \mathring{v}_j\; \prod_{j=1}^{W} \exp\Big\{\mathbf{i}\frac{\mathring{x}_{j,1}+\mathring{x}_{j,2}}{\sqrt{M}}\Big\} \nonumber\\
& \hspace{12ex}\times\exp\Big\{-M\big(\mathring{K}(\hat{X},V)+\mathring{L}(\hat{B},T)\big)\Big\}\cdot\prod_{j=1}^W (x_{j,1}-x_{j,2})^2(b_{j,1}+b_{j,2})^2\cdot \mathsf{A}(\hat{X}, \hat{B}, V, T).
 \label{012318}
\end{align}

For the Type II or III vicinities, i.e. $\varkappa=+$ or $-$,  we change $\mathbf{x},\mathbf{b},\mathbf{t}$-variables to $\mathring{\mathbf{x}},\mathring{\mathbf{b}},\mathring{\mathbf{t}}$-variables. Consequently, we have
\begin{align}
&\mathcal{I}\big(\Upsilon^b_+, \Upsilon^b_-, \Upsilon^x_\varkappa, \Upsilon^x_\varkappa, \Upsilon_S,\mathbb{I}^{W-1}\big)=\frac{(-a_\varkappa^2)^{W}}{(n!)^2}\cdot \frac{M^{W+1}}{8^W\pi^{2W+4}}\cdot\int_{\mathbb{L}^{2W-2}} \prod_{j=2}^W\frac{{\rm d} \theta_j}{2\pi}\prod_{j=2}^W \frac{{\rm d} \sigma_j}{2\pi}\int_{\mathring{\Upsilon}} \prod_{j=1}^W {\rm d}  \mathring{b}_{j,1} \int_{\mathring{\Upsilon}} \prod_{j=1}^W {\rm d}  \mathring{b}_{j,2} \nonumber\\
&  \hspace{12ex}\times \int_{\mathring{\Upsilon}} \prod_{j=1}^W {\rm d}  \mathring{x}_{j,1} \int_{\mathring{\Upsilon}} \prod_{j=1}^W {\rm d}  \mathring{x}_{j,2}  \int_{\mathring{\Upsilon}_S} \prod_{j=2}^W 2\mathring{t}_j {\rm d}  \mathring{t}_j \int_{\mathbb{I}^{W-1}} \prod_{j=2}^W 2v_j {\rm d} v_j\; \prod_{j=1}^{W} \exp\Big\{\mathbf{i}\frac{\mathring{x}_{j,1}+\mathring{x}_{j,2}}{\sqrt{M}}\Big\} \nonumber\\
& \hspace{12ex}\times\exp\Big\{-M\big(\mathring{K}(\hat{X},V)+\mathring{L}(\hat{B},T)\big)\Big\}\cdot\prod_{j=1}^W (x_{j,1}-x_{j,2})^2(b_{j,1}+b_{j,2})^2\cdot \mathsf{A}(\hat{X}, \hat{B}, V, T).
\label{020411}
\end{align}

We will also need the following facts
\begin{eqnarray}
\prod_{j=1}^W |(x_{j,1}-x_{j,2})^2(b_{j,1}+b_{j,2})^2|=e^{O(W)},\qquad |\mathsf{A}(\hat{X},\hat{B}, V, T)|\leq  e^{O(WN^{\varepsilon_2})} \label{030902}
\end{eqnarray}
if
\begin{eqnarray*}
\mathbf{x}_{1},\mathbf{x}_2\in \widehat{\Sigma}^W,\quad  b_{j,1}=a_++o(1), \quad b_{j,2}=-a_-+o(1), \quad t_j=o(1), \quad \forall\; j=1,\ldots, N,
 \end{eqnarray*}
which always hold in these types of vicinities. The first estimate in (\ref{030902}) is trivial, and the second follows from Lemma \ref{lem.020203}.  

Now, we approximate (\ref{012640}) in the vicinities. 
For any $\vartheta\in \mathbb{L}$, we introduce the matrices
\begin{eqnarray*}
\mathcal{E}_+(\vartheta):=\bigg(\begin{array}{ccc}0 & e^{\mathbf{i}\vartheta}\\ e^{-\mathbf{i}\vartheta} & 0\end{array}\bigg),\qquad \mathcal{E}_-(\vartheta):=\bigg(\begin{array}{ccc}0 & e^{\mathbf{i}\vartheta}\\ -e^{-\mathbf{i}\vartheta} & 0\end{array}\bigg).
\end{eqnarray*}
Then, with the parameterization above, expanding $\hat{X}_j$ in (\ref{021613}) and $T_j$ in  (\ref{122708}) up to the second order, we can write
\begin{eqnarray}
&&\hat{X}_j=D_\varkappa+\frac{\mathbf{i}}{\sqrt{M}}D_\varkappa\text{diag}(\mathring{x}_{j,1},\mathring{x}_{j,2})+\frac{1}{M}R^x_j,\quad \varkappa=\pm, +,-,\nonumber\\
&&T_j=I+\frac{\mathring{t}_j}{\sqrt{M}}\mathcal{E}_+(\sigma_j)+\frac{1}{M}R^t_j.\label{011401}
\end{eqnarray}
For $\varkappa=\pm$, we also expand $V_j$ in (\ref{122708}) up to the second order, namely,
\begin{eqnarray}
V_j=I+\frac{\mathring{v}_j}{\sqrt{M}}\mathcal{E}_-(\theta_j)+\frac{1}{M}R^v_j. \label{011402}
\end{eqnarray} 
We just take (\ref{011401}) and (\ref{011402}) as the definition of $R^x_j$, $R^t_j$ and $R^v_j$. Note that $R^x_j$ is actually $\varkappa$-dependent. However, this dependence is irrelevant for our analysis thus is suppressed from the notation. It is elementary that
\begin{eqnarray}
||R^x_j||_{\max}=O(\mathring{x}_{j,1}^2+\mathring{x}_{j,2}^2),\quad ||R^t_j||_{\max}=O(\mathring{t}_j^2),\quad ||R^v_j||_{\max}=O(\mathring{v}_j^2). \label{030315}
\end{eqnarray}
Here $||\cdot||_{\max}$ represents the max-norm of a matrix.

Recall the facts (\ref{0129111}) and (\ref{020345}) 
\begin{eqnarray}
&&M\mathring{L}(\hat{B},T)=M\left(\mathring{\ell}_{++}(\mathbf{b}_1)+\mathring{\ell}_{--}(\mathbf{b}_2)\right)+M\ell_S(\hat{B},T),\nonumber\\
&&M\mathring{K}(\hat{X},V)=M\left(-\mathring{\ell}_{++}(\mathbf{x}_1)-\mathring{\ell}_{+-}(\mathbf{x}_2)\right)+M\ell_S(\hat{X},V). \label{012646}
\end{eqnarray}
In light of (\ref{020343})-(\ref{020345}), we can also represent $M\mathring{K}(\hat{X},V)$ in the following two alternative ways
\begin{align}
&M\mathring{K}(\hat{X},V)=M\left(-\mathring{\ell}_{++}(\mathbf{x}_1)-\mathring{\ell}_{++}(\mathbf{x}_2)\right)+M\ell_S(\hat{X},V)+M\big(K(D_{+},I)-K(D_{\pm},I)\big), \label{012647}\\
&M\mathring{K}(\hat{X},V)= M\left(-\mathring{\ell}_{+-}(\mathbf{x}_1)-\mathring{\ell}_{+-}(\mathbf{x}_2)\right)+M\ell_S(\hat{X},V)+M\big(K(D_{-},I)-K(D_{\pm},I)\big). \label{012648}
\end{align}
We will use three representations of $M\mathring{K}(\hat{X},V)$ in (\ref{012646}), (\ref{012647}) and (\ref{012648}) for Type I',  II and  III vicinities respectively. In addition,
we introduce the matrices
\begin{eqnarray}
\mathbb{A}_+:=(1+a_+^2)I+a_+^2S,\qquad \mathbb{A}_-:=(1+a_-^2)I+a_-^2S. \label{012510}
\end{eqnarray}
Then, we have the following lemma.
\begin{lem} \label{lem.012694}With the parametrization in (\ref{011401}), we have the following approximations.
\begin{itemize}
\item Let $\mathring{\mathbf{b}}_1, \mathring{\mathbf{b}}_2\in \mathbb{C}^{W}$ and $||\mathring{\mathbf{b}}_1||_\infty, ||\mathring{\mathbf{b}}_2||_\infty=o(\sqrt{M})$, we have
\begin{eqnarray}
\qquad\quad M\left(\mathring{\ell}_{++}(\mathbf{b}_1)+\mathring{\ell}_{--}(\mathbf{b}_2)\right)=\frac12\mathring{\mathbf{b}}_1'\mathbb{A}_+\mathring{\mathbf{b}}_1+\frac12\mathring{\mathbf{b}}_2'\mathbb{A}_-\mathring{\mathbf{b}}_2+R^b,\quad R^b=O\Big{(}\frac{\sum_{a=1,2}||\mathring{\mathbf{b}}_{a}||_3^3}{\sqrt{M}}\Big{)}.
\label{012302}
\end{eqnarray}
\item Let $\varkappa=\pm$ and $\mathring{\mathbf{x}}_1,\mathring{\mathbf{x}}_2\in \mathbb{C}^{W}$ and $||\mathring{\mathbf{x}}_1||_\infty,||\mathring{\mathbf{x}}_2||_\infty=o(\sqrt{M})$, we have
\begin{eqnarray}
\qquad\quad M\left(-\mathring{\ell}_{++}(\mathbf{x}_1)-\mathring{\ell}_{+-}(\mathbf{x}_2)\right)=\frac12\mathring{\mathbf{x}}_1'\mathbb{A}_+\mathring{\mathbf{x}}_1+\frac12\mathring{\mathbf{x}}_2'\mathbb{A}_-\mathring{\mathbf{x}}_2+R^x_\pm,\quad R^x_\pm=O\Big{(}\frac{\sum_{a=1,2}||\mathring{\mathbf{x}}_{a}||_3^3}{\sqrt{M}}\Big{)}. \label{012303}
\end{eqnarray}
\item In the Type II vicinity, we have
\begin{eqnarray}
\qquad \quad M\left(-\mathring{\ell}_{++}(\mathbf{x}_1)-\mathring{\ell}_{++}(\mathbf{x}_2)\right)=\frac12\mathring{\mathbf{x}}_1'\mathbb{A}_+\mathring{\mathbf{x}}_1+\frac12\mathring{\mathbf{x}}_2'\mathbb{A}_+\mathring{\mathbf{x}}_2+R^x_+,\quad R^x_+=O\Big{(}\frac{\Theta^{\frac32}}{\sqrt{M}}\Big{)}. 
\label{012301}
\end{eqnarray}
\item In the Type III vicinity, we have
\begin{eqnarray}
\qquad \quad M\left(-\mathring{\ell}_{+-}(\mathbf{x}_1)-\mathring{\ell}_{+-}(\mathbf{x}_2)\right)=\frac12\mathring{\mathbf{x}}_1'\mathbb{A}_-\mathring{\mathbf{x}}_1+\frac12\mathring{\mathbf{x}}_2'\mathbb{A}_-\mathring{\mathbf{x}}_2+R^x_-,\quad R^x_-=O\Big{(}\frac{\Theta^{\frac32}}{\sqrt{M}}\Big{)}. 
\label{011530}
\end{eqnarray}
\end{itemize}
Here $R^b$ $R^x_\pm$, $R^x_+$ and $R^x_-$ are remainder terms of the Taylor expansion of the function $\ell(\mathbf{a})$ defined in (\ref{020401}).
\end{lem}
\begin{rem} Here we stated (\ref{012302}) and (\ref{012303}) in the domains much larger than the Type I' vicinity for further discussion. In addition, the restriction $||\mathring{\mathbf{b}}_a||_\infty$ and $||\mathring{\mathbf{x}}_a||_\infty$ for $a=1,2$ is imposed to avoid the ambiguity of the definition of the logarithmic term in the function $\ell(\mathbf{a})$. 
\end{rem}
\begin{proof}
It follows from  the Taylor expansion of the function $\ell(\mathbf{a})$ easily.
\end{proof}

Then, according to (\ref{012646})-(\ref{012648}), what remains is to  approximate $M\ell_S(\hat{B},T)$ and $M\ell_S(\hat{X},V)$ in the vicinities. Recalling the definition in (\ref{020401}) and the parameterization in (\ref{011615}), we can rewrite
\begin{align}
M\ell_S(\hat{B},T)&=\frac12\sum_{j,k} \mathfrak{s}_{jk}|s_j\mathring{t}_ke^{\mathbf{i}\sigma_k}-s_k\mathring{t}_je^{\mathbf{i}\sigma_j}|^2\nonumber\\
&\hspace{2ex}\times \Big(a_+-a_-+\frac{a_+\mathring{b}_{j,1}-a_-\mathring{b}_{j,2}}{\sqrt{M}}\Big)\Big(a_+-a_-+\frac{a_+\mathring{b}_{k,1}-a_-\mathring{b}_{k,2}}{\sqrt{M}}\Big)\nonumber\\
&=: \frac{(a_+-a_-)^2}{2}\sum_{j,k} \mathfrak{s}_{jk}|\mathring{t}_ke^{\mathbf{i}\sigma_k}-\mathring{t}_je^{\mathbf{i}\sigma_j}|^2+R^{t,b}. \label{012306}
\end{align}
We take the above equation as the definition of $R^{t,b}$. Now, we set
\begin{eqnarray*}
\tau_{j,1}:=\mathring{t}_j\cos\sigma_j,\qquad \tau_{j,2}:=\mathring{t}_j\sin \sigma_j,\quad \forall\; j=2,\ldots, W
\end{eqnarray*}
and change the variables and  the measure as
\begin{eqnarray}
(\mathring{t}_j,\sigma_j)\to(\tau_{j,1},\tau_{j,2}),\qquad 2\mathring{t}_j{\rm d} \mathring{t}_j\frac{{\rm d} \sigma_j}{2\pi}\to\frac{1}{\pi} {\rm d} \tau_{j,1}{\rm d} \tau_{j,2}. \label{012315}
\end{eqnarray}

In the  Type I' vicinity, we can do the same thing for $M\ell_S(\hat{X},V)$, namely,
\begin{eqnarray}
M\ell_S(\hat{X},V)=:\frac{(a_+-a_-)^2}{2}\sum_{j,k} \mathfrak{s}_{jk}|\mathring{v}_ke^{\mathbf{i}\theta_k}-\mathring{v}_je^{\mathbf{i}\theta_j}|^2+R^{v,x}_\pm, \label{012307}
\end{eqnarray} 
where $R^{v,x}_\pm$ is the remainder term.
Then we set
\begin{eqnarray*}
\upsilon_{j,1}:=\mathring{v}_j\cos\theta_j,\qquad \upsilon_{j,2}:=\mathring{v}_j\sin\theta_j,\quad \forall\; j=2,\ldots, W
\end{eqnarray*}
and change the variables and measure as
\begin{eqnarray}
(\mathring{v}_j,\theta_j)\to (\upsilon_{j,1},\upsilon_{j,2}),\qquad 2\mathring{v}_j{\rm d} \mathring{v}_j\frac{{\rm d} \theta_j}{2\pi}\to\frac{1}{\pi} {\rm d} \upsilon_{j,1}{\rm d} \upsilon_{j,2}. \label{012316}
\end{eqnarray}

Now, we introduce the vectors
\begin{eqnarray*}
\boldsymbol{\tau}_a=(\tau_{2,a},\ldots, \tau_{W,a}),\qquad \boldsymbol{\upsilon}_a=(\upsilon_{2,a},\ldots, \upsilon_{W,a}),\qquad a=1,2.
\end{eqnarray*}
With this notation, we can  rewrite (\ref{012306}) and (\ref{012307}) as
\begin{eqnarray}
&&M\ell_S(\hat{B},T)=-(a_+-a_-)^2\sum_{a=1,2}\boldsymbol{\tau}'_aS^{(1)}\boldsymbol{\tau}_a+R^{t,b},\nonumber\\
 &&M\ell_S(\hat{X},V)=-(a_+-a_-)^2\sum_{a=1,2}\boldsymbol{\upsilon}'_aS^{(1)}\boldsymbol{\upsilon}_a+R^{v,x}_\pm. \label{011510}
\end{eqnarray}

According to (\ref{012315}) and (\ref{012316}), we can express (\ref{012318}) as an integral over $\mathring{\mathbf{b}}$, $\mathring{\mathbf{x}}$, $\mathring{\boldsymbol{\tau}}$ and $\mathring{\boldsymbol{\upsilon}}$-variables. However, we need to specify the domains of $\mathring{\boldsymbol{\tau}}$ and $\mathring{\boldsymbol{\upsilon}}$-variables in advance. Our aim is to restrict the integral in the domains 
\begin{eqnarray}
\boldsymbol{\tau}_a\in \mathring{\Upsilon}_S,\quad \boldsymbol{\upsilon}_a\in \mathring{\Upsilon}_S,\qquad a=1,2. \label{012320}
\end{eqnarray}
Taking $\mathring{\mathbf{t}}$ for instance, we see that
\begin{eqnarray*}
(t_j-t_k)^2\leq |t_j e^{\mathbf{i}\sigma_j}-t_k e^{\mathbf{i}\sigma_k}|^2=(\tau_{j,1}-\tau_{k,1})^2+(\tau_{j,2}-\tau_{k,2})^2,
\end{eqnarray*}
which actually implies
\begin{eqnarray}
\boldsymbol{\tau}_a\in \mathring{\Upsilon}_S\quad \text{for}\quad a=1,2\Longrightarrow \mathring{\mathbf{t}}\in \mathring{\Upsilon}_S. \label{012651}
\end{eqnarray}
However the reverse of (\ref{012651}) may not be true. That means, (\ref{012320}) is stronger than $(\mathring{\mathbf{t}},\mathring{\mathbf{v}},\boldsymbol{\sigma},\boldsymbol{\theta})\in \mathring{\Upsilon}_S\times \mathring{\Upsilon}_S\times \mathbb{L}^{W-1}\times \mathbb{L}^{W-1}$.
To show the truncation to (\ref{012320}) from $(\mathring{\mathbf{t}},\mathring{\mathbf{v}},\boldsymbol{\sigma},\boldsymbol{\theta})\in \mathring{\Upsilon}_S\times \mathring{\Upsilon}_S\times \mathbb{L}^{W-1}\times \mathbb{L}^{W-1}$ is harmless in the integral (\ref{012318}), we need 
to bound $R^{t,b}$ and $R^{v,x}_\pm$ in terms of $\boldsymbol{\tau}'_aS^{(1)}\boldsymbol{\tau}_a$ and $\boldsymbol{\upsilon}'_aS^{(1)}\boldsymbol{\upsilon}_a$, respectively. More specifically, we need the following lemma.
\begin{lem} \label{lem.012320}In the Type I' vicinity, we have
\begin{align}
&|R^{t,b}|\leq O\Big(\frac{\Theta^{\frac12}}{\sqrt{M}}\Big)\sum_{a=1,2} \big(-\boldsymbol{\tau}_a'S^{(1)}\boldsymbol{\tau}_a\big)+O\Big(\frac{\Theta^{\frac72}}{M}\Big)\sum_{a=1,2} \big(-\boldsymbol{\tau}_a'S^{(1)}\boldsymbol{\tau}_a\big)^{1/2}+O\Big(\frac{\Theta^{7}}{M^2}\Big),\nonumber\\
&|R^{v,x}_\pm|\leq O\Big(\frac{\Theta^{\frac12}}{\sqrt{M}}\Big)\sum_{a=1,2} \big(-\boldsymbol{\upsilon}_a'S^{(1)}\boldsymbol{\upsilon}_a\big)+O\Big(\frac{\Theta^{\frac72}}{M}\Big)\sum_{a=1,2}\big(-\boldsymbol{\upsilon}_a'S^{(1)}\boldsymbol{\upsilon}_a\big)^{1/2}+O\Big(\frac{\Theta^{7}}{M^2}\Big).  \label{011513}
\end{align}
\end{lem}
\begin{proof} Since the proofs of these two bounds are nearly the same, we only state the details for the first one.  By (\ref{012658}), we have
\begin{eqnarray*}
s_j=1+O\Big(\frac{\Theta^{2}}{M}\Big),\quad \forall\; j=2,\ldots, W.
\end{eqnarray*}
Then it is not difficult to see that
\begin{eqnarray}
&&|s_j\mathring{t}_k e^{\mathbf{i}\sigma_k}-s_k\mathring{t}_je^{\mathbf{i}\sigma_j}|^2=\Big{|}\mathring{t}_k e^{\mathbf{i}\sigma_k}-\mathring{t}_je^{\mathbf{i}\sigma_j}+O\Big(\frac{\Theta^{3}}{M}\Big)\Big{|}^2\nonumber\\
&&=\sum_{a=1,2}(\tau_{j,a}-\tau_{k,a})^2+O\Big(\frac{\Theta^{3}}{M}\Big)\sum_{a=1,2}|\tau_{j,a}-\tau_{k,a}|+O\Big(\frac{\Theta^{6}}{M^2}\Big). \label{012311}
\end{eqnarray}
Now, by the fact  from Definition \ref{defi.030301}
\begin{eqnarray}
||\mathring{\mathbf{b}}_a||_\infty\leq ||\mathring{\mathbf{b}}_a||_2=O(\Theta^{\frac12}) \quad \text{for}\quad  a=1,2,
\end{eqnarray} 
we have
\begin{eqnarray}
\qquad\Big{(}a_+-a_-+\frac{a_+\mathring{b}_{j,1}-a_-\mathring{b}_{j,2}}{\sqrt{M}}\Big{)}\Big{(}a_+-a_-+\frac{a_+\mathring{b}_{k,1}-a_-\mathring{b}_{k,2}}{\sqrt{M}}\Big{)}=(a_+-a_-)^2+O\Big(\frac{\Theta^{\frac12}}{\sqrt{M}}\Big). \label{012312}
\end{eqnarray}
Substituting (\ref{012311}) and (\ref{012312}) into (\ref{012306}) yields
\begin{align*}
|R^{t,b}|&\leq O\Big(\frac{\Theta^{\frac12}}{\sqrt{M}}\Big)\sum_{a=1,2}\sum_{j,k}\mathfrak{s}_{jk}(\tau_{j,a}-\tau_{k,a})^2+O\Big(\frac{\Theta^{3}}{M}\Big)\sum_{a=1,2}\sum_{j,k}\mathfrak{s}_{jk}|\tau_{j,a}-\tau_{k,a}|+O\Big(\frac{\Theta^{7}}{M^2}\Big)\nonumber\\
&\leq O\Big(\frac{\Theta^{\frac12}}{\sqrt{M}}\Big)\sum_{a=1,2}\big(-\boldsymbol{\tau}_a'S^{(1)}\boldsymbol{\tau}_a\big)+O\Big(\frac{\Theta^{\frac72}}{M}\Big)\sum_{a=1,2} \big(-\boldsymbol{\tau}_a'S^{(1)}\boldsymbol{\tau}_a\big)^{1/2}+O\Big(\frac{\Theta^{7}}{M^2}\Big).
\end{align*}
Here we used Cauchy Schwarz inequality and $\sum_{j,k}\mathfrak{s}_{jk}=O(W)=O(\Theta)$ in the second step. 
The bound for $|R^{v,x}_\pm|$ can be proved analogously. Hence, we completed the proof of Lemma \ref{lem.012320}.
\end{proof}
Roughly speaking, by (\ref{011510}) and Lemma \ref{lem.012320} we have
\begin{eqnarray*}
&&M\ell_S(\hat{B},T)=-\big(a_+-a_-+o(1)\big)^2\sum_{a=1,2}\boldsymbol{\tau}'_aS^{(1)}\boldsymbol{\tau}_a+o\Big{(}\sum_{a=1,2}\big(- \boldsymbol{\tau}_a'S^{(1)}\boldsymbol{\tau}_a\big)^{1/2}\Big{)}+o(1),\nonumber\\
&&M\ell_S(\hat{X},V)=-\big(a_+-a_-+o(1)\big)^2\sum_{a=1,2}\boldsymbol{\upsilon}'_aS^{(1)}\boldsymbol{\upsilon}_a+o\Big{(}\sum_{a=1,2}\big(- \boldsymbol{\upsilon}_a'S^{(1)}\boldsymbol{\upsilon}_a\big)^{1/2}\Big{)}+o(1).
\end{eqnarray*}
Then it is obvious that if one of $\boldsymbol{\tau}_1$, $\boldsymbol{\tau}_2$, $\boldsymbol{\upsilon}_1$ and  $\boldsymbol{\upsilon}_2$ is not in $\mathring{\Upsilon}_S$, we will get (\ref{012351}). Hence, using (\ref{030902}), we can discard the integral outside the vicinity,  analogously to the proof of Lemma \ref{lem.020211}. More specifically, in the sequel, we can and do assume
\begin{eqnarray}
\boldsymbol{\tau}_1,\boldsymbol{\tau}_2,\boldsymbol{\upsilon}_1,\boldsymbol{\upsilon}_2\in \mathring{\Upsilon}_S. \label{011512}
\end{eqnarray}
Now, plugging (\ref{011512}) into (\ref{011513}) in turn yields the bound 
\begin{eqnarray}
|R^{t,b}|=O\Big(\frac{\Theta^{\frac32}}{\sqrt{M}}\vee\frac{\Theta^{4}}{M}\Big),\quad |R^{v,x}_\pm|=O\Big(\frac{\Theta^{\frac32}}{\sqrt{M}}\vee\frac{\Theta^{4}}{M}\Big). \label{012371}
\end{eqnarray}

By the discussion above, for the Type I vicinity, we can write (\ref{012318}) as
\begin{eqnarray}
&&2^W\mathcal{I}(\Upsilon^b_+, \Upsilon^b_-, \Upsilon^x_+, \Upsilon^x_-, \Upsilon_S,\Upsilon_S)=\frac{M^{2}}{(n!)^24^{W}\pi^{4W+2}}\cdot\int_{\mathring{\Upsilon}} \prod_{j=1}^W {\rm d}  \mathring{b}_{j,1} \int_{\mathring{\Upsilon}} \prod_{j=1}^W {\rm d}  \mathring{b}_{j,2} \nonumber\\
&&  \times \int_{\mathring{\Upsilon}} \prod_{j=1}^W {\rm d}  \mathring{x}_{j,1} \int_{\mathring{\Upsilon}} \prod_{j=1}^W {\rm d}  \mathring{x}_{j,2}  \int_{\mathring{\Upsilon}_S} \prod_{j=2}^W  {\rm d}  \tau_{j,1}\int_{\mathring{\Upsilon}_S} \prod_{j=2}^W  {\rm d}  \tau_{j,2} \int_{\mathring{\Upsilon}_S} \prod_{j=2}^W  {\rm d}  \upsilon_{j,1}\int_{\mathring{\Upsilon}_S} \prod_{j=2}^W  {\rm d}  \upsilon_{j,2} \nonumber\\
&&\times  \exp\Big{\{}-\frac12\mathring{\mathbf{b}}_1'\mathbb{A}_+\mathring{\mathbf{b}}_1-\frac12\mathring{\mathbf{b}}_2'\mathbb{A}_-\mathring{\mathbf{b}}_2-R^b\Big{\}}\cdot \exp\Big{\{}-\frac12\mathring{\mathbf{x}}_1'\mathbb{A}_+\mathring{\mathbf{x}}_1-\frac12\mathring{\mathbf{x}}_2'\mathbb{A}_-\mathring{\mathbf{x}}_2-R^x_\pm\Big{\}}\nonumber\\
&&\times\exp\Big{\{}(a_+-a_-)^2\sum_{a=1,2}\boldsymbol{\tau}'_aS^{(1)}\boldsymbol{\tau}_a-R^{t,b}\Big{\}}\cdot \exp\Big{\{}(a_+-a_-)^2\sum_{a=1,2}\boldsymbol{\upsilon}'_aS^{(1)}\boldsymbol{\upsilon}_a-R^{v,x}_\pm\Big{\}}\nonumber\\
&&\times \prod_{j=1}^{W}\exp\Big{\{}\mathbf{i}\frac{\mathring{x}_{j,1}+\mathring{x}_{j,2}}{\sqrt{M}}\Big{\}} \cdot \prod_{j=1}^W (x_{j,1}-x_{j,2})^2(b_{j,1}+b_{j,2})^2\cdot \mathsf{A}(\hat{X}, \hat{B}, V, T)+O(e^{-\Theta}),
\label{012672}
\end{eqnarray}
where the error term stems from the truncation of the vicinity $(\mathring{\mathbf{t}},\mathring{\mathbf{v}}, \boldsymbol{\sigma},\boldsymbol{\theta})\in \mathring{\Upsilon}_S\times \mathring{\Upsilon}_S\times \mathbb{L}^{W-1}\times \mathbb{L}^{W-1}$ to $(\boldsymbol{\tau}_1, \boldsymbol{\tau}_2,\boldsymbol{\upsilon}_1, \boldsymbol{\upsilon}_2)\in \mathring{\Upsilon}_S\times \mathring{\Upsilon}_S\times \mathring{\Upsilon}_S\times \mathring{\Upsilon}_S$. 

Now,  for the Type II and III vicinities, the discussion on $\ell_S(\hat{B},T)$ is of course the same. For $\ell_S(\hat{X},V)$, we make the following approximation.  For the Type II vicinity,  using the notation in (\ref{011520}),we can write
\begin{align}
M\ell_S(\hat{X},V)&=\frac{-Ma_+^2}{2}\sum_{j,k} \mathfrak{s}_{jk}^v\Big{(}\frac{\mathring{x}_{j,1}}{\sqrt{M}}-\frac{\mathring{x}_{j,2}}{\sqrt{M}}+O\Big{(}\frac{\mathring{x}_{j,1}^2+\mathring{x}_{j,2}^2}{M}\Big{)}\Big{)}\Big{(}\frac{\mathring{x}_{k,1}}{\sqrt{M}}-\frac{\mathring{x}_{k,2}}{\sqrt{M}}+O\Big{(}\frac{\mathring{x}_{k,1}^2+\mathring{x}_{k,2}^2}{M}\Big{)}\Big{)}\nonumber\\
&=:-\frac{a_+^2}{2}\sum_{j,k} \mathfrak{s}_{jk}^v(\mathring{x}_{j,1}-\mathring{x}_{j,2})(\mathring{x}_{k,1}-\mathring{x}_{k,2})+R^{v,x}_{+}.\label{011531}
\end{align}
It is easy to see that 
\begin{eqnarray}
R^{v,x}_{+}=O\Big{(}\frac{||\mathring{\mathbf{x}}_1||_3^3+||\mathring{\mathbf{x}}_2||_3^3}{\sqrt{M}}\Big{)}=O\Big{(}\frac{\Theta^{\frac32}}{\sqrt{M}}\Big{)},\quad \text{for}\quad \mathring{\mathbf{x}}_1, \mathring{\mathbf{x}}_2\in \mathring{\Upsilon}.  \label{01260101}
\end{eqnarray}
Combining (\ref{020341}), (\ref{020344}),  (\ref{012301}) and (\ref{011531})  we obtain
\begin{align*}
M(K(\hat{X},V)-K(D_{+}, I))&=\frac12\mathring{\mathbf{x}}_1'\mathbb{A}_+\mathring{\mathbf{x}}_1+\frac12\mathring{\mathbf{x}}_2'\mathbb{A}_+\mathring{\mathbf{x}_2}-\frac{a_+^2}{2}\sum_{j,k} \mathfrak{s}_{jk}^v(\mathring{x}_{j,1}-\mathring{x}_{j,2})(\mathring{x}_{k,1}-\mathring{x}_{k,2})+R^x_++R^{v,x}_{+}\nonumber\\
&=\frac12\mathring{\mathbf{x}}'\mathbb{A}^v_+\mathring{\mathbf{x}}+R^x_++R^{v,x}_{+},
\end{align*}
where
\begin{eqnarray}
\mathring{\mathbf{x}}:=(\mathring{\mathbf{x}}_1',\mathring{\mathbf{x}}_2')',\qquad \mathbb{A}^v_+:=(1+a_+^2)I_{2W}+a_+^2\mathbb{S}^v,\label{012605}
\end{eqnarray} 
and recall that $\mathbb{S}^v$ is defined in (\ref{021721}). Analogously, for the Type III vicinity, we can write
\begin{eqnarray}
M(K(\hat{X},V)-K(D_{-}, I))=\frac12\mathring{\mathbf{x}}'\mathbb{A}^v_-\mathring{\mathbf{x}}+R^x_-+R^{v,x}_{-}, \label{01260909}
\end{eqnarray}
where
\begin{eqnarray}
 \mathbb{A}^v_-:=(1+a_-^2)I_{2W}+a_-^2\mathbb{S}^v,\qquad R^{v,x}_{-}=O\Big(\frac{\Theta^{\frac32}}{\sqrt{M}}\Big). \label{01260202}
\end{eqnarray}
Consequently, by (\ref{012647}) and (\ref{012648}) we can write (\ref{020411}) for $\kappa=+,-$ as
\begin{eqnarray}
&&\mathcal{I}(\Upsilon^b_+, \Upsilon^b_-, \Upsilon^x_{\varkappa}, \Upsilon^x_{\varkappa}, \Upsilon_S,\mathbb{I}^{W-1})=\exp\Big\{M\big(K(D_{\pm},I)-K(D_{\varkappa},I)\big)\Big\}\cdot\frac{(-a_\varkappa^2)^{W}}{(n!)^2}\cdot \frac{M^{W+1}}{8^W\pi^{3W+3}}\nonumber\\
&&\hspace{5ex}\times\int_{\mathbb{L}^{W-1}} \prod_{j=2}^W\frac{{\rm d} \theta_j}{2\pi} \int_{\mathbb{I}^{W-1}} \prod_{j=2}^W 2v_j {\rm d} v_j\int_{\mathring{\Upsilon}} \prod_{j=1}^W {\rm d}  \mathring{b}_{j,1} \int_{\mathring{\Upsilon}} \prod_{j=1}^W {\rm d}  \mathring{b}_{j,2} \int_{\mathring{\Upsilon}} \prod_{j=1}^W {\rm d}  \mathring{x}_{j,1} \int_{\mathring{\Upsilon}} \prod_{j=1}^W {\rm d}  \mathring{x}_{j,2} \nonumber\\
&& \hspace{5ex} \times \int_{\mathring{\Upsilon}_S} \prod_{j=2}^W  {\rm d}  \tau_{j,1}\int_{\mathring{\Upsilon}_S} \prod_{j=2}^W  {\rm d}  \tau_{j,2} \cdot   \exp\Big{\{}-\frac12\mathring{\mathbf{b}}_1'\mathbb{A}_+\mathring{\mathbf{b}}_1-\frac12\mathring{\mathbf{b}}_2'\mathbb{A}_-\mathring{\mathbf{b}}_2-R^b\Big{\}}\nonumber\\
&&\hspace{5ex}\times \exp\Big{\{}-\frac12\mathring{\mathbf{x}}'\mathbb{A}_{\varkappa}^v\mathring{\mathbf{x}}-R^x_{\varkappa}-R^{v,x}_{{\varkappa}}\Big{\}}\cdot\exp\Big{\{}(a_+-a_-)^2\sum_{a=1,2}\boldsymbol{\tau}'_aS^{(1)}\boldsymbol{\tau}_a-R^{t,b}\Big{\}}\nonumber\\
&&\hspace{5ex}\times\prod_{j=1}^{W} \exp\Big{\{}\mathbf{i}\frac{\mathring{x}_{j,1}+\mathring{x}_{j,2}}{\sqrt{M}}\Big{\}}\cdot \prod_{j=1}^W (x_{j,1}-x_{j,2})^2(b_{j,1}+b_{j,2})^2\cdot \mathsf{A}(\hat{X}, \hat{B}, V, T)+O(e^{-\Theta}). \label{030302}
\end{eqnarray}

\subsection{Steepest descent paths in the vicinities}
In order to estimate the integrals (\ref{012672}) and (\ref{030302}) properly,  we need to control various remainder terms in (\ref{012672}) and (\ref{030302}) to reduce these integrals to Gaussian ones. The final result is collected in Proposition \ref{pro.020401} at the end of this section. As a preparation, we shall further deform the contours of $\mathring{\mathbf{b}}$-variables and $\mathring{\mathbf{x}}$-variables to the steepest descent paths. We mainly provide the discussion for the  $\mathring{\mathbf{b}}$-variables, that for the $\mathring{\mathbf{x}}$-variables is analogous. 

For simplicity, in this section, we assume $0\leq E\leq \sqrt{2}-\kappa$, the case $-\sqrt{2}+\kappa\leq E\leq 0$ can be discussed similarly. We introduce the eigendecomposition of $S$ as
\begin{eqnarray*}
S=\mathsf{U}\hat{S}\mathsf{U}'.
\end{eqnarray*}
Note that $\mathsf{U}$ is an orthogonal matrix thus the entries are all real. Now, we perform the change of coordinate
\begin{eqnarray*}
\mathbf{c}_a=(c_{1,a},\ldots, c_{W,a})':=\mathsf{U}'\mathring{\mathbf{b}}_a,\quad a=1,2.
\end{eqnarray*}
Obviously, for the differentials, we have
\begin{eqnarray*}
\prod_{j=1}^W {\rm d}\mathring{b}_{j,a}=\prod_{j=1}^W {\rm d} c_{j,a},\quad a=1,2.
\end{eqnarray*}
In addition, for the domains, it is elementary to see
\begin{eqnarray}
\mathring{\mathbf{b}}_a\in \mathring{\Upsilon}\Longleftrightarrow \mathbf{c}_a\in \mathring{\Upsilon},\quad a=1,2. \label{032020}
\end{eqnarray}

Now, we introduce the notation
\begin{eqnarray*}
\gamma_j^+:=\frac{1}{\sqrt{1+a_+^2+a_+^2\lambda_j(S)}},\qquad \gamma_j^-:=\frac{1}{\sqrt{1+a_-^2+a_-^2\lambda_j(S)}},
\end{eqnarray*}
and set the diagonal matrices
\begin{eqnarray*}
\mathbb{D}_+:=\text{diag}(\gamma_1^+,\ldots, \gamma_W^+),\qquad \mathbb{D}_-:=\text{diag}(\gamma_1^-,\ldots, \gamma_W^-).
\end{eqnarray*}
By the assumption $0\leq E\leq\sqrt{2}-\kappa$ and (\ref{0129120}), it is not difficult to check 
\begin{eqnarray}
 |\gamma_j^+|\sim 1,\quad  |\gamma_j^-|\sim 1,\qquad \arg \gamma_j^+\in \big(-\frac{\pi}{8},0],\qquad  \arg\gamma_j^-\in [0,\frac{\pi}{8}\big),\quad \forall\; j=1,\ldots, W. \label{012671}
\end{eqnarray}
With the notation introduced above,  we have
\begin{eqnarray*}
\mathring{\mathbf{b}}_1'\mathbb{A}_+\mathring{\mathbf{b}}_1=\mathbf{c}_{1}'\mathbb{D}_+^{-2}\mathbf{c}_1,\quad \mathring{\mathbf{b}}_2'\mathbb{A}_-\mathring{\mathbf{b}}_2=\mathbf{c}_{2}'\mathbb{D}_-^{-2}\mathbf{c}_2.
\end{eqnarray*}

To simplify the following discussion, we enlarge the domain of the $\mathbf{c}$-variables to 
\begin{eqnarray*}
\mathbf{c}_a\in \Upsilon_\infty\equiv \Upsilon_\infty(\varepsilon):=[-\Theta^{\frac{1}{2}}, \Theta^{\frac{1}{2}}]^{W},\quad a=1,2.
\end{eqnarray*}
Obviously, $\mathring{\Upsilon}\subset \Upsilon_\infty$. It is easy to check that (\ref{012351}) also holds when $\mathbf{c}_a\in \Upsilon_\infty\setminus\mathring{\Upsilon}$ for either $a=1$ or $2$, according to (\ref{032020}), thus such a modification of the domain will only produce an error term of order $O(\exp\{-\Theta\})$ in the integral (\ref{012672}), by using (\ref{030902}).

Now we do the scaling
\begin{eqnarray*}
\mathbf{c}_1\to \mathbb{D}_+\mathbf{c}_{1},\qquad \mathbf{c}_2\to \mathbb{D}_-\mathbf{c}_{2}.
\end{eqnarray*}
Consequently, we have
\begin{eqnarray}
\mathring{\mathbf{b}}_1=\mathsf{U}\mathbb{D}_+\mathbf{c}_1,\qquad\mathring{\mathbf{b}}_2=\mathsf{U}\mathbb{D}_-\mathbf{c}_2, \label{011557}
\end{eqnarray}
thus
\begin{eqnarray}
\mathring{\mathbf{b}}_1'\mathbb{A}_+\mathring{\mathbf{b}}_1=\sum c_{j,1}^2,\quad \mathring{\mathbf{b}}_2'\mathbb{A}_-\mathring{\mathbf{b}}_2=\sum c_{j,2}^2. \label{012681}
\end{eqnarray}
Accordingly, we should adjust the change of differentials as
\begin{eqnarray*}
\prod_{j=1}^W {\rm d} \mathring{b}_{j,1}\to  \det \mathbb{D}_+\cdot \prod_{j=1}^W {\rm d} c_{j,1},\quad \prod_{j=1}^W {\rm d} \mathring{b}_{j,2}\to  \det \mathbb{D}_-\cdot \prod_{j=1}^W {\rm d} c_{j,2}.
\end{eqnarray*}
In addition, the domain of $\mathbf{c}_1$ should be changed from $\Upsilon_\infty$ to $\prod_{j=1}^W \mathbb{J}_j^+$, where
\begin{eqnarray*}
\mathbb{J}_j^+:=(\gamma_j^+)^{-1}[-\Theta^{\frac{1}{2}}, \Theta^{\frac{1}{2}}],
\end{eqnarray*}
and that of $\mathbf{c}_2$ should be changed from $\Upsilon_\infty$ to $\prod_{j=1}^W \mathbb{J}_j^-$, where
\begin{eqnarray*}
\mathbb{J}_j^-:=(\gamma_j^-)^{-1}[-\Theta^{\frac{1}{2}}, \Theta^{\frac{1}{2}}].
\end{eqnarray*}
By the fact $\det\mathbb{D}_+\mathbb{D}_-=1/\sqrt{\det\mathbb{A}_+\mathbb{A}_-}$, we can write (\ref{012672}) as
\begin{eqnarray}
&&2^W\mathcal{I}(\Upsilon^b_+, \Upsilon^b_-, \Upsilon^x_+, \Upsilon^x_-, \Upsilon_S,\Upsilon_S)=\frac{M^{2}}{(n!)^24^{W}\pi^{4W+2}}\cdot \frac{1}{\sqrt{\det \mathbb{A}_+\mathbb{A_-}}} \nonumber\\
&&  \times \int_{\prod_{j=1}^W \mathbb{J}_j^+} \prod_{j=1}^W {\rm d}  c_{j,1} \int_{\prod_{j=1}^W \mathbb{J}_j^-} \prod_{j=1}^W {\rm d}  c_{j,2}\int_{\mathring{\Upsilon}} \prod_{j=1}^W {\rm d}  \mathring{x}_{j,1} \int_{\mathring{\Upsilon}} \prod_{j=1}^W {\rm d}  \mathring{x}_{j,2}  \nonumber\\
&&\times\int_{\mathring{\Upsilon}_S} \prod_{j=2}^W  {\rm d}  \tau_{j,1}\int_{\mathring{\Upsilon}_S} \prod_{j=2}^W  {\rm d}  \tau_{j,2} \int_{\mathring{\Upsilon}_S} \prod_{j=2}^W  {\rm d}  \upsilon_{j,1}\int_{\mathring{\Upsilon}_S} \prod_{j=2}^W  {\rm d}  \upsilon_{j,2} \;  \prod_{j=1}^{W} \exp\Big\{\mathbf{i}\frac{\mathring{x}_{j,1}+\mathring{x}_{j,2}}{\sqrt{M}}\Big\} \nonumber\\
&&\times \exp\Big\{-\frac12\sum c_{j,1}^2-\frac12\sum c_{j,2}^2-R^b\Big\}\cdot \exp\Big\{-\frac12\mathring{\mathbf{x}}_1'\mathbb{A}_+\mathring{\mathbf{x}}_1-\frac12\mathring{\mathbf{x}}_2'\mathbb{A}_-\mathring{\mathbf{x}}_2-R^x_\pm\Big\}\nonumber\\
&&\times\exp\Big\{(a_+-a_-)^2\sum_{a=1,2}\boldsymbol{\tau}'_aS^{(1)}\boldsymbol{\tau}_a-R^{t,b}\Big\}\cdot \exp\Big\{(a_+-a_-)^2\sum_{a=1,2}\boldsymbol{\upsilon}'_aS^{(1)}\boldsymbol{\upsilon}_a-R^{v,x}_\pm\Big\}\nonumber\\
&&\times \prod_{j=1}^W (x_{j,1}-x_{j,2})^2(b_{j,1}+b_{j,2})^2\cdot \mathsf{A}(\hat{X}, \hat{B}, V, T)+O(e^{-\Theta}). \label{011560}
\end{eqnarray}
In (\ref{011560}), all the $\mathbf{b}$ and $\mathring{\mathbf{b}}$-variables in the integrand should be regarded as the functions of the $\mathbf{c}$-variables, see (\ref{011557}). 

Now, we consider the integrand as a function of $\mathbf{c}$-variables on the disks, namely, 
\begin{eqnarray*}
c_{j,1}\in \mathbb{O}_j^+:=\big\{z\in\mathbb{C}: |z|\leq \Theta^{\frac{1}{2}}|\gamma_j^+|^{-1}\big\},\quad c_{j,2}\in \mathbb{O}_j^-:=\big\{z\in\mathbb{C}: |z|\leq \Theta^{\frac{1}{2}}|\gamma_j^-|^{-1}\big\}.
\end{eqnarray*}
For $\mathbf{c}_1\in \prod_{j=1}^W\mathbb{O}_j^+$ and $\mathbf{c}_2\in \prod_{j=1}^W\mathbb{O}_j^-$, by (\ref{012671}) and (\ref{011557}) we have
\begin{eqnarray}
||\mathring{\mathbf{b}}_1||_\infty, ||\mathring{\mathbf{b}}_2||_\infty\leq O(\Theta). \label{011580}
\end{eqnarray}
Here we used the elementary fact $||U\mathbf{a}||_\infty\leq \sqrt{W}||\mathbf{a}||_\infty$ for any $\mathbf{a}\in \mathbb{C}^W$ and and unitary matrix $U$.
Then, we deform the contour of $c_{j,1}$ from $\mathbb{J}_j^+$ to 
\begin{eqnarray*}
(-\Sigma_{j}^+)\cup\mathbb{L}_j^+\cup\Sigma_j^+
\end{eqnarray*}
for each $j=1,\ldots, W$,
where
\begin{eqnarray*}
\mathbb{L}_j^+:=\mathbb{R}\cap \mathbb{O}_j^+,\qquad \Sigma_j^+=\left\{z\in\partial \mathbb{O}_j^+: 0\leq \arg z\leq -\arg \gamma_j^+\right\}.
\end{eqnarray*}
 It is not difficult to see that 
\begin{eqnarray*}
\Re c_{j,1}^2\geq  \Theta, \quad \text{for}\quad c_{j,1}\in(-\Sigma_j^+)\cup\Sigma_j^+,
\end{eqnarray*}
by using (\ref{012671}). Consequently, by (\ref{012681}), we have
\begin{eqnarray*}
\Big{|}\exp\Big{\{}-\frac12\mathring{\mathbf{b}}_1'\mathbb{A}_+\mathring{\mathbf{b}}_1\Big{\}}\Big{|}=\Big{|}\exp\Big{\{}-\frac12\sum_{j=1}^W c_{j,1}^2\Big{\}}\Big{|}\leq O(e^{-\Theta}).
\end{eqnarray*}
Then using (\ref{030902}),  we can get rid of the integral over $\Sigma_j^+$ and $-\Sigma_j^+$, analogously to the discussion in Section \ref{s.8}.
Similarly, we can perform the same argument for $\mathbf{c}_2$. Consequently, we can restrict the integral in (\ref{011560}) to the domain 
\begin{eqnarray*}
\mathbf{c}_1\in \prod_{j=1}^W\mathbb{L}_j^+,\qquad \mathbf{c}_2\in \prod_{j=1}^W \mathbb{L}_j^-.
\end{eqnarray*}
So we can assume that $\prod_{j=1}^W \mathbb{J}_j^+$ and $\prod_{j=1}^W \mathbb{J}_j^-$ are replaced with $\prod_{j=1}^W\mathbb{L}_j^+$ and $\prod_{j=1}^W\mathbb{L}_j^-$ respectively in (\ref{011560}).

By (\ref{012302}), (\ref{011580}) and the fact $||\mathbf{a}||_3^3\leq ||\mathbf{a}||_\infty||\mathbf{a}||_2^2$ for any vector $\mathbf{a}$, we see that
\begin{eqnarray}
|R^b|\leq C\frac{||\mathring{\mathbf{b}}_1||_\infty+||\mathring{\mathbf{b}}_2||_\infty}{\sqrt{M}}\big(||\mathring{\mathbf{b}}_1||_2^2+||\mathring{\mathbf{b}}_2||_2^2\big)\leq \frac{\Theta}{\sqrt{M}}\big(||\mathbf{c}_1||_2^2+||\mathbf{c}_2||_2^2\big) \label{011601}
\end{eqnarray}
for some positive constant $C$, where in the last step we also used the fact that $||\mathbf{b}_a||_2= O(||\mathbf{c}_a||_2)$ for $a=1,2$, which is implied by (\ref{011557}) and (\ref{012671}).
Consequently, we have
\begin{eqnarray*}
\exp\Big\{-\frac12||\mathbf{c}_1||_2^2-\frac12||\mathbf{c}_2||_2^2-R^b\Big\}=\exp\Big\{-\big(\frac12+o(1)\big)||\mathbf{c}_1||_2^2-\big(\frac12+o(1)\big)||\mathbf{c}_2||_2^2\Big\}.
\end{eqnarray*}
This allows us to take a step further to truncate $\mathbf{c}_1$ and $\mathbf{c}_2$ according to their $2$-norm, namely
\begin{eqnarray}
\mathbf{c}_1,\mathbf{c}_2\in \mathring{\Upsilon}.  \label{012688}
\end{eqnarray}
Similarly to the discussion in the proof of Lemma \ref{lem.020211}, such a truncation will only produce an error of order $\exp\{-\Theta\}$ to the integral, by using (\ref{030902}).  

Now, analogously to (\ref{011557}), we can change $\mathring{\mathbf{x}}$-variables to $\mathbf{d}$-variables, defined by
\begin{eqnarray*}
\mathbf{d}_1=(d_{1,1},\ldots, d_{W,1}):=\mathbb{D}_+^{-1}\mathsf{U}'\mathring{\mathbf{x}}_1,\qquad \mathbf{d}_2=(d_{1,2},\ldots, d_{W,2}):=\mathbb{D}_-^{-1}\mathsf{U}'\mathring{\mathbf{x}}_2.
\end{eqnarray*}
Thus accordingly, we change the differentials
\begin{eqnarray*}
\prod_{j=1}^W {\rm d} \mathring{x}_{j,1}\to  \det \mathbb{D}_+\cdot \prod_{j=1}^W {\rm d} d_{j,1},\quad \prod_{j=1}^W {\rm d} \mathring{x}_{j,2}\to  \det \mathbb{D}_-\cdot \prod_{j=1}^W {\rm d} d_{j,2}.
\end{eqnarray*}
In addition, like (\ref{012688}), we deform the domain to
\begin{eqnarray*}
\mathbf{d}_1,\mathbf{d}_2\in \mathring{\Upsilon}.
\end{eqnarray*}
Finally, from (\ref{011560}), we arrive at the representation
\begin{eqnarray}
&&\hspace{-2ex}2^W\mathcal{I}\big(\Upsilon^b_+, \Upsilon^b_-, \Upsilon^x_+, \Upsilon^x_-, \Upsilon_S,\Upsilon_S\big)\nonumber\\
&&=\frac{M^{2}}{(n!)^24^{W}\pi^{4W+2}}\cdot \frac{1}{\det \mathbb{A}_+\mathbb{A_-}}\cdot  \int_{\mathring{\Upsilon}} \prod_{j=1}^W {\rm d}  c_{j,1} \int_{\mathring{\Upsilon}} \prod_{j=1}^W {\rm d}  c_{j,2}\int_{\mathring{\Upsilon}} \prod_{j=1}^W {\rm d}  d_{j,1} \int_{\mathring{\Upsilon}} \prod_{j=1}^W {\rm d}  d_{j,2}  \nonumber\\
&&\times\int_{\mathring{\Upsilon}_S} \prod_{j=2}^W  {\rm d}  \tau_{j,1}\int_{\mathring{\Upsilon}_S} \prod_{j=2}^W  {\rm d}  \tau_{j,2} \int_{\mathring{\Upsilon}_S} \prod_{j=2}^W  {\rm d}  \upsilon_{j,1}\int_{\mathring{\Upsilon}_S} \prod_{j=2}^W  {\rm d}  \upsilon_{j,2} \; \prod_{j=1}^{W} \exp\Big\{\mathbf{i}\frac{\mathring{x}_{j,1}+\mathring{x}_{j,2}}{\sqrt{M}}\Big\} \nonumber\\
&&\times \exp\Big\{-\frac12||\mathbf{c}_1||_2^2-\frac12||\mathbf{c}_2||_2^2-R^b\Big\}\cdot \exp\Big\{-\frac12||\mathbf{d}_1||_2^2-\frac12||\mathbf{d}_2||_2^2-R^x_\pm\Big\}\nonumber\\
&&\times\exp\Big\{(a_+-a_-)^2\sum_{a=1,2}\boldsymbol{\tau}'_aS^{(1)}\boldsymbol{\tau}_a-R^{t,b}\Big\}\cdot \exp\Big\{(a_+-a_-)^2\sum_{a=1,2}\boldsymbol{\upsilon}'_aS^{(1)}\boldsymbol{\upsilon}_a-R^{v,x}_\pm\Big\}\nonumber\\
&&\times \prod_{j=1}^W (x_{j,1}-x_{j,2})^2(b_{j,1}+b_{j,2})^2\cdot \mathsf{A}(\hat{X}, \hat{B}, V, T)+O(e^{-\Theta}), \label{011603}
\end{eqnarray}
in which $\mathbf{x}$ and $\mathring{\mathbf{x}}$-variables should be regarded as functions of the $\mathbf{d}$-variables, as well, $\mathbf{b}$ and $\mathring{\mathbf{b}}$-variables should be regarded as functions of the $\mathbf{c}$-variables. 

Now, in the Type II and III vicinities, we only do the change of coordinates for the $\mathring{\mathbf{b}}$-variables, which is enough for our purpose. Consequently, we have
\begin{align}
&\mathcal{I}\big(\Upsilon^b_+, \Upsilon^b_-, \Upsilon^x_\varkappa, \Upsilon^x_\varkappa, \Upsilon_S,\mathbb{I}^{W-1}\big)=\exp\big\{M(K(D_{\pm},I)-K(D_{\varkappa},I))\big\}\cdot\frac{(-a_\varkappa^2)^{W}}{(n!)^2}\cdot \frac{M^{W+1}}{8^W\pi^{3W+3}} \nonumber\\
&\times\frac{1}{\sqrt{\det \mathbb{A}_+\mathbb{A}_-}} \cdot\int_{\mathbb{L}^{W-1}} \prod_{j=2}^W\frac{{\rm d} \theta_j}{2\pi} \int_{\mathbb{I}^{W-1}} \prod_{j=2}^W 2v_j {\rm d} v_j\int_{\mathring{\Upsilon}} \prod_{j=1}^W {\rm d}  c_{j,1} \int_{\mathring{\Upsilon}} \prod_{j=1}^W {\rm d}  c_{j,2} \nonumber\\
&\times\int_{\mathring{\Upsilon}} \prod_{j=1}^W {\rm d}  \mathring{x}_{j,1} \int_{\mathring{\Upsilon}} \prod_{j=1}^W {\rm d}  \mathring{x}_{j,2}  \int_{\mathring{\Upsilon}_S} \prod_{j=2}^W  {\rm d}  \tau_{j,1}\int_{\mathring{\Upsilon}_S} \prod_{j=2}^W  {\rm d}  \tau_{j,2}  \; \prod_{j=1}^{W} \exp\Big\{\mathbf{i}\frac{\mathring{x}_{j,1}+\mathring{x}_{j,2}}{\sqrt{M}}\Big\} \nonumber\\
&\times  \exp\Big\{-\frac12||\mathbf{d}_1||_2^2-\frac12||\mathbf{d}_2||_2^2-R^b\Big\}\cdot \exp\Big\{(a_+-a_-)^2\sum_{a=1,2}\boldsymbol{\tau}'_aS^{(1)}\boldsymbol{\tau}_a-R^{t,b}\Big\}\nonumber\\
&\times \exp\Big\{-\frac12\mathring{\mathbf{x}}'\mathbb{A}_\varkappa^v\mathring{\mathbf{x}}-R^x_\varkappa-R^{v,x}_{\varkappa}\Big\}\prod_{j=1}^W (x_{j,1}-x_{j,2})^2(b_{j,1}+b_{j,2})^2\cdot \mathsf{A}(\hat{X}, \hat{B}, V, T)+O(e^{-\Theta}). \label{011604}
\end{align}
By (\ref{012671}) and (\ref{011557}),  it is easy to see
\begin{eqnarray}
\mathbf{c}_1,\mathbf{c}_2\in \mathring{\Upsilon}\Longrightarrow \mathring{\mathbf{b}}_1, \mathring{\mathbf{b}}_2\in \mathring{\Upsilon} \label{012690}
\end{eqnarray}
Similarly, we have
\begin{eqnarray}
\mathbf{d}_1,\mathbf{d}_2\in \mathring{\Upsilon}\Longrightarrow \mathring{\mathbf{x}}_1, \mathring{\mathbf{x}}_2\in \mathring{\Upsilon}. \label{012691}
\end{eqnarray}
We keep the terminology ``Type I', II and III vicinities''  for the slightly modified domains defined in terms of $\mathbf{c}$, $\mathbf{d}$, $\boldsymbol{\tau}$ and $\boldsymbol{\upsilon}$-variables. More specifically, we redefine the vicinities as follows.
\begin{defi} We slightly modify Definition \ref{defi.030301} as follows.
\begin{itemize}
\item Type I' vicinity: $\hspace{6ex}\displaystyle \mathbf{c}_1,\mathbf{c}_2,\mathbf{d}_1,\mathbf{d}_2\in \mathring{\Upsilon}, \hspace{5ex} \boldsymbol{\tau}_1,\boldsymbol{\tau}_2,\boldsymbol{\upsilon}_1,\boldsymbol{\upsilon_2}\in \mathring{\Upsilon}_S$.\\
\item Type II vicinity: $\hspace{5ex}\displaystyle \mathbf{c}_1,\mathbf{c}_2,\mathring{\mathbf{x}}_1,\mathring{\mathbf{x}}_2\in \mathring{\Upsilon}, \hspace{5ex} \boldsymbol{\tau}_1,\boldsymbol{\tau}_2\in \mathring{\Upsilon}_S,\hspace{5ex} V_j\in \mathring{U}(2)$ for all $j=2,\ldots, W$, \\$~~~~~~~~~~~~~~~~~~~~~~~~~~~~~~~~~~~~~~~~\text{where $\mathring{\mathbf{x}}$-variables are defined in (\ref{011401}) with $\varkappa=+$}$.\\
\item Type III vicinity: $\hspace{5ex}\displaystyle \mathbf{c}_1,\mathbf{c}_2,\mathring{\mathbf{x}}_1,\mathring{\mathbf{x}}_2\in \mathring{\Upsilon}, \hspace{5ex} \boldsymbol{\tau}_1,\boldsymbol{\tau}_2\in \mathring{\Upsilon}_S,\hspace{5ex} V_j\in \mathring{U}(2)$ for all $j=2,\ldots, W$, \\$~~~~~~~~~~~~~~~~~~~~~~~~~~~~~~~~~~~~~~~~\text{where $\mathring{\mathbf{x}}$-variables are defined in (\ref{011401}) with $\varkappa=-$.}$
\end{itemize}
\end{defi}
Now, recall the remainder terms $R^b$, $R^x_{\pm}$, $R^x_+$ and $R^x_-$ in Lemma \ref{lem.012694}, $R^{t,b}$ and $R^{v,x}_\pm$ in (\ref{011510}), $R_{+}^{v,x}$ in (\ref{011531}) and $R_{-}^{v,x}$ in (\ref{01260909}). In light of (\ref{012690}) and  (\ref{012691}), the bounds on these remainder terms are the same as those obtained in Section \ref{s.9.1}. For the convenience of the reader, we collect them as the following proposition. 
\begin{pro}\label{pro.020401}Under Assumptions \ref{assu.1} and \ref{assu.4}, we have the following estimate, in the  vicinities. 
\begin{eqnarray*}
&(i):&R^{t,b}=O\Big(\frac{\Theta^{4}}{M}\Big),    \quad R^{v,x}_\pm=O\Big(\frac{\Theta^{4}}{M}\Big),\\
&(ii):& R^b=O\Big(\frac{\Theta^{2}}{\sqrt{M}}\Big), \quad  R^x=O\Big(\frac{\Theta^{2}}{\sqrt{M}}\Big),\quad R_{+}^x=O\Big(\frac{\Theta^{\frac32}}{\sqrt{M}}\Big),\quad R_{-}^x=O\Big(\frac{\Theta^{\frac32}}{\sqrt{M}}\Big),\\
&(iii):& R_{+}^{v,x}=O\Big(\frac{\Theta^{\frac32}}{\sqrt{M}}\Big),\quad R_{-}^{v,x}=O\Big(\frac{\Theta^{\frac32}}{\sqrt{M}}\Big).
\end{eqnarray*}
\end{pro}
\begin{proof} Note that, (i) can be obtained from (\ref{011513}), and (ii) follows from Lemma \ref{lem.012694}, and (iii) is implied by (\ref{01260101}) and (\ref{01260202}). Hence, we completed the proof.
\end{proof}
Analogously, in the vicinities, $||\mathring{\mathbf{b}}_1||_{2}^2$, $||\mathring{\mathbf{b}}_2||_{2}^2$, $||\mathring{\mathbf{x}}_1||_2^2$, $||\mathring{\mathbf{x}}_2||_2^2$, $||\mathring{\mathbf{t}}||_\infty$ and $||\mathring{\mathbf{v}}||_\infty$ are still bounded by $\Theta$. 
\section{Integral over the Type I vicinities} \label{s.11}
With (\ref{011603}), we estimate the integral over the Type I vicinity in this section. At first, in the Type I' vicinity, we have $||\mathring{\mathbf{x}}_a||_\infty=O(\Theta^{\frac12})$ and  $||\mathring{\mathbf{b}}_a||_\infty=O(\Theta^{\frac12})$ for $a=1,2$. Consequently, according to the parametrization in (\ref{011615}), we have
\begin{eqnarray}
x_{j,1}-x_{j,2}=a_+-a_-+O\Big{(}\frac{\Theta^{\frac12}}{\sqrt{M}}\Big{)},\qquad b_{j,1}+b_{j,2}=a_+-a_-+O\Big{(}\frac{\Theta^{\frac12}}{\sqrt{M}}\Big{)}, \label{011610}
\end{eqnarray}
which implies
\begin{eqnarray}
\prod_{j=1}^W (x_{j,1}-x_{j,2})^2(b_{j,1}+b_{j,2})^2=(a_+-a_-)^{4W}\Big{(}1+O\Big{(}\frac{\Theta^{\frac32}}{\sqrt{M}}\Big{)}\Big{)}. \label{0125321}
\end{eqnarray}
Hence, what remains is to estimate the function $\mathsf{A}(\hat{X}, \hat{B}, V, T)$. We have the following lemma.
\begin{lem} \label{lem.012401}Suppose that the assumptions in Theorem \ref{lem.012802} hold. In the Type I' vicinity, for any given positive integer $n$, there is $N_0=N_0(n)$, such that for all $N\geq N_0$ we have 
\begin{eqnarray*}
|\mathsf{A}(\hat{X}, \hat{B}, V, T)|\leq \frac{\Theta^2W^{C_0}}{M(N\eta)^{n+\ell}}\cdot |\det \mathbb{A}_+|^2\cdot \det (S^{(1)})^2. 
\end{eqnarray*}
for some positive constant $C_0$ and some integer $\ell=O(1)$, both of which are independent of $n$.
\end{lem}   
With (\ref{011603}), (\ref{0125321}) and Lemma \ref{lem.012401}, we can prove Lemma \ref{lem.010602}.
\begin{proof}[Proof of Lemma \ref{lem.010602}] Using (\ref{011603}), (\ref{0125321}), Lemma \ref{lem.012401}, Proposition \ref{pro.020401} with (\ref{021105}), the fact $\det\mathbb{A}_+=\overline{\det\mathbb{A}_-}$ and the trivial estimate
\begin{eqnarray*}
M\Theta^2W^{C_0}\frac{1}{(N\eta)^{\ell}}\leq N^{C_0}
\end{eqnarray*} 
for sufficiently large constant $C_0$, 
we have
\begin{eqnarray*}
&&2^W|\mathcal{I}(\Upsilon^b_+, \Upsilon^b_-, \Upsilon^x_+, \Upsilon^x_-, \Upsilon_S,\Upsilon_S)|\leq  \frac{N^{C_0}}{(N\eta)^{n}}\cdot \frac{1}{(2\pi^2)^{2W}}\cdot \det (S^{(1)})^2\cdot (a_+-a_-)^{4W}\nonumber\\
&&  \times\int_{\mathring{\Upsilon}} \prod_{j=1}^W {\rm d}  c_{j,1} \int_{\mathring{\Upsilon}} \prod_{j=1}^W {\rm d}  c_{j,2}\int_{\mathring{\Upsilon}} \prod_{j=1}^W {\rm d}  d_{j,1} \int_{\mathring{\Upsilon}} \prod_{j=1}^W {\rm d}  d_{j,2}\int_{\mathring{\Upsilon}_S} \prod_{j=2}^W  {\rm d}  \tau_{j,1}\int_{\mathring{\Upsilon}_S} \prod_{j=2}^W  {\rm d}  \tau_{j,2}  \nonumber\\
&&\times \int_{\mathring{\Upsilon}_S}\prod_{j=2}^W  {\rm d}  \upsilon_{j,1}\int_{\mathring{\Upsilon}_S} \prod_{j=2}^W  {\rm d}  \upsilon_{j,2} \; \exp\left\{-\frac12\big(||\mathbf{c}_1||_2^2+||\mathbf{c}_2||_2^2+||\mathbf{d}_1||_2^2+||\mathbf{d}_2||_2^2\big)\right\}\nonumber\\
&&\times\exp\left\{(a_+-a_-)^2\left(\boldsymbol{\tau}'_1S^{(1)}\boldsymbol{\tau}_1+\boldsymbol{\tau}'_1S^{(1)}\boldsymbol{\tau}_2+\boldsymbol{\upsilon}'_1S^{(1)}\boldsymbol{\upsilon}_1+\boldsymbol{\upsilon}'_1S^{(1)}\boldsymbol{\upsilon}_2\right)\right\}+O(e^{-\Theta}).
\end{eqnarray*}
Then, by elementary Gaussian integral we obtain (\ref{020436}). Hence, we completed the proof of Lemma \ref{lem.010602}.
\end{proof}

The remaining part of this section will be dedicated to the proof of Lemma \ref{lem.012401}. Recall the definitions of the functions $\mathsf{A}(\cdot)$, $\mathsf{Q}(\cdot)$, $\mathsf{P}(\cdot)$ and $\mathsf{F}(\cdot)$ in (\ref{013115}), (\ref{030311}), (\ref{013116}) and (\ref{013117}).  Using the strategy in Section \ref{s.7} again, we ignore the irrelevant factor $\mathsf{Q}(\cdot)$ at the beginning. Hence,  we bound $\mathsf{P}(\cdot)$ and $\mathsf{F}(\cdot)$ at first, and modify the bounding procedure slightly to take $\mathsf{Q}(\cdot)$ into account in the end, resulting a proof of Lemma \ref{lem.012401}.
\subsection{ $\mathsf{P}(\hat{X}, \hat{B}, V, T)$ in the Type I' vicinity} \label{s.10.1} As mentioned above, we should always regard $\mathbf{b}$ or $\mathring{\mathbf{b}}$-variables as functions of $\mathbf{c}$-variables, regard $\mathbf{x}$ or $\mathring{\mathbf{x}}$-variables as functions of $\mathbf{d}$-variables. Our aim, in this section, is to prove the following lemma.
\begin{lem} \label{lem.011602}Suppose that the assumptions in Theorem \ref{lem.012802} hold. In the Type I' vicinity, we have
\begin{eqnarray}
\mathsf{P}(\hat{X},\hat{B}, V, T)\leq \frac{W^{2+\gamma}\Theta^2}{M}|\det\mathbb{A}_+|^2\det (S^{(1)})^2. \label{031502}
\end{eqnarray}
\end{lem}
Before commencing the formal proof, we introduce more notation below.
In the sequel, we will use the notation
\begin{eqnarray}
\mathring{\kappa}_j\equiv\mathring{\kappa}_j(\hat{X},\hat{B},V,T):=|\mathring{x}_{j,1}|+|\mathring{x}_{j,2}|+|\mathring{b}_{j,1}|+|\mathring{b}_{j,2}|+|\mathring{v}_j|+|\mathring{t}_j|= O(\Theta), \label{030320}
\end{eqnarray}
where the bound holds in the Type I' vicinity, according to the facts $||\mathring{\mathbf{x}}_a||_\infty=O(\Theta^{\frac12})$, $||\mathring{\mathbf{b}}_a||_\infty=O(\Theta^{\frac12})$ for $a=1,2$, $||\mathring{\mathbf{t}}||_\infty=O(\Theta)$ and $||\mathring{\mathbf{v}}||_\infty=O(\Theta)$.

Recalling (\ref{020447}) with $\boldsymbol{\varpi}_j$ defined in (\ref{011611}) and $\hat{\boldsymbol{\varpi}}_j$ in (\ref{020201}).
Now, we write
\begin{align}
\boldsymbol{\varpi}_j
&=\exp\Big{\{} -M\log\det\big(1+M^{-1}V_j^*\hat{X}_j^{-1}V_j \Omega_jT_j^{-1}\hat{B}_j^{-1}T_j\Xi_j\big)\Big{\}}\nonumber\\
&=:\exp\Big{\{}-TrV_j^*\hat{X}_j^{-1}V_j \Omega_jT_j^{-1}\hat{B}_j^{-1}T_j\Xi_j\Big{\}}\; \exp\Big\{\sum_{\ell=2}^4\frac{(-1)^{\ell-1}}{\ell M^{\ell-1}}\Delta_{\ell,j}\Big\}, \label{011612}
\end{align}
where
\begin{eqnarray}
\Delta_{\ell,j}:= Tr \big(V_j^*\hat{X}_j^{-1}V_j \Omega_jT_j^{-1}\hat{B}_j^{-1}T_j\Xi_j\big)^\ell. \label{012101}
\end{eqnarray}
The second step of (\ref{011612}) follows from the Taylor expansion of the logarithmic function. The expansion terminates at the 4th order term since $\Delta_{\ell,j}$ is a homogeneous polynomial of $\Omega_j$ and $\Xi_j$-variables with degree $2\ell$, regarding all the complex variables as fixed parameters.
Now, we expand the first factor of (\ref{011612}) around the Type I' saddle point, namely
\begin{eqnarray}
\exp\Big{\{}-TrV_j^*\hat{X}_j^{-1}V_j \Omega_jT_j^{-1}\hat{B}_j^{-1}T_j\Xi_j\Big{\}}=:\exp\Big{\{}-TrD_{\pm}^{-1}\Omega_jD_{\pm}^{-1}\Xi_j\Big{\}}\; \exp\Big{\{}-\frac{1}{\sqrt{M}}\Delta_j\Big{\}}. \label{011613}
\end{eqnarray}
We take (\ref{011613}) as the definition of $\Delta_j$, which is of the form
\begin{eqnarray*}
\Delta_j=\sum_{\alpha,\beta=1}^4\mathring{\mathfrak{p}}_{j,\alpha,\beta} \cdot \omega_{j,\alpha}\xi_{j,\beta}
\end{eqnarray*}
for some function $\mathring{\mathfrak{p}}_{j,\alpha,\beta}$ of $\mathring{\mathbf{x}}$, $\mathring{\mathbf{b}}$, $\mathring{\mathbf{v}}$ and $\mathring{\mathbf{t}}$-variables, satisfying
\begin{eqnarray}
\mathring{\mathfrak{p}}_{j,\alpha,\beta}=O(\mathring{\kappa}_j),\qquad \forall\; \alpha,\beta=1,\ldots,4. \label{011621}
\end{eqnarray}
One can check (\ref{011621}) easily by using (\ref{011401})-(\ref{030315}). Analogously, we can also write
\begin{eqnarray}
\quad\Delta_{\ell,j}=\sum_{\substack{\alpha_1,\ldots,\alpha_\ell,\\ ~~\beta_1,\ldots, \beta_\ell=1}}^4\mathring{\mathfrak{p}}_{\ell,j,\boldsymbol{\alpha},\boldsymbol{\beta}}\prod_{i=1}^\ell \omega_{j,\alpha_i}\xi_{j,\beta_i},\qquad \boldsymbol{\alpha}:=(\alpha_1,\ldots,\alpha_\ell),\quad \boldsymbol{\beta}:=(\beta_1,\ldots, \beta_\ell), \label{012102}
\end{eqnarray}
where
\begin{eqnarray}
\qquad\mathring{\mathfrak{p}}_{\ell,j,\boldsymbol{\alpha},\boldsymbol{\beta}}=O(1),\qquad \forall\; \ell=2,\ldots, 4; \alpha_1,\ldots,\alpha_\ell, \beta_1,\ldots,\beta_\ell=1,\ldots,4. \label{012103}
\end{eqnarray}
The bound on $\mathring{\mathfrak{p}}_{\ell,j,\boldsymbol{\alpha},\boldsymbol{\beta}}$ in (\ref{012103}) follows from the fact that all the $V_j$, $\hat{X}_j^{-1}$, $T_j$, $T_j^{-1}$ and $\hat{B}_j^{-1}$-entries are bounded in the Type I' vicinity, uniformly in $j$. 
Consequently, we can write for $j\neq p,q$
\begin{eqnarray}
\exp\Big\{-\frac{1}{\sqrt{M}}\Delta_j+\sum_{\ell=2}^4\frac{(-1)^{\ell-1}}{\ell M^{\ell-1}}\Delta_{\ell,j}\Big\}=1+\sum_{\ell=1}^4 M^{-\frac{\ell}{2}}\sum_{\substack{\alpha_1,\ldots,\alpha_\ell,\\ ~~\beta_1,\ldots, \beta_\ell=1}}^4\mathring{\mathfrak{q}}_{\ell,j,\boldsymbol{\alpha},\boldsymbol{\beta}}\prod_{i=1}^\ell \omega_{j,\alpha_i}\xi_{j,\beta_i}. \label{020448}
\end{eqnarray}
In a similar manner, we can also write for $k=p,q$,
\begin{eqnarray}
\exp\Big\{-\frac{1}{\sqrt{M}}\Delta_k+\sum_{\ell=2}^4\frac{(-1)^{\ell-1}}{\ell M^{\ell-1}}\Delta_{\ell,k}\Big\}\hat{\boldsymbol{\varpi}}_k=\hat{\mathfrak{p}}_0(\cdot)\bigg(1+\sum_{\ell=1}^4 M^{-\frac{\ell}{2}}\sum_{\substack{\alpha_1,\ldots,\alpha_\ell,\\ ~~\beta_1,\ldots, \beta_\ell=1}}^4\mathring{\mathfrak{q}}_{\ell,k,\boldsymbol{\alpha},\boldsymbol{\beta}}\prod_{i=1}^\ell \omega_{k,\alpha_i}\xi_{k,\beta_i}\bigg), \label{020449}
\end{eqnarray}
where  $\hat{\mathfrak{p}}_0(\cdot)=\det \hat{X}_k/\det \hat{B}_k$, which is introduced in (\ref{020201}), and  $\mathring{\mathfrak{q}}_{\ell,j,\boldsymbol{\alpha},\boldsymbol{\beta}}$ is some function of $\hat{X}$, $\hat{B}$, $V$ and $T$-variables, satisfying the bound
\begin{eqnarray}
\mathring{\mathfrak{q}}_{\ell,j,\boldsymbol{\alpha},\boldsymbol{\beta}}=O((1+\mathring{\kappa}_j)^\ell),\quad \forall\; \ell=1,\ldots, 4,\quad j=1,\ldots,W. \label{011640}
\end{eqnarray}
Obviously, we have $\hat{\mathfrak{p}}_0(\cdot)=O(1)$ in Type I' vicinity.

Now, in order to distinguish $\ell,\boldsymbol{\alpha} $ and $\boldsymbol{\beta}$ for different $j$, we index them as $\ell_j, \boldsymbol{\alpha}_j$ and $\boldsymbol{\beta}_j$, where
\begin{eqnarray*}
\boldsymbol{\alpha}_j\equiv \boldsymbol{\alpha}_j(\ell_j):=(\alpha_{j,1},\ldots, \alpha_{j,\ell_j}),\quad \boldsymbol{\beta}_j\equiv \boldsymbol{\beta}_j(\ell_j):=(\beta_{j,1},\ldots, \beta_{j,\ell_j}).
\end{eqnarray*} 
In addition, we introduce the vector
\begin{eqnarray*}
\vec{\ell}:=(\ell_1,\ldots, \ell_W),\quad \vec{\boldsymbol{\alpha}}\equiv \vec{\boldsymbol{\alpha}}(\vec{\ell}):=(\boldsymbol{\alpha}_1,\ldots, \boldsymbol{\alpha}_W),\quad \vec{\boldsymbol{\beta}}\equiv \vec{\boldsymbol{\beta}}(\vec{\ell}):=(\boldsymbol{\beta}_1,\ldots, \boldsymbol{\beta}_W).
\end{eqnarray*}
Let $||\vec{\ell}||_1=\sum_{j=1}^W\ell_j$ be the $1$-norm of $\vec{\ell}$. Note that $\vec{\boldsymbol{\alpha}}$ and $\vec{\boldsymbol{\beta}}$ are $||\vec{\ell}||_1$-dimensional. With these notations, using (\ref{020447}), (\ref{011612}), (\ref{011613}), (\ref{020448}) and (\ref{020449}) we have the representation
\begin{align}
\mathcal{P}( \Omega, \Xi, \hat{X}, \hat{B}, V, T)&=\hat{\mathfrak{p}}_0(\hat{X}_p,\hat{B}_p)\hat{\mathfrak{p}}_0(\hat{X}_q,\hat{B}_q)\cdot\exp\Big{\{} -\sum_{j,k}\tilde{\mathfrak{s}}_{jk} Tr  \Omega_j\Xi_k-\sum_{j=1}^W TrD_{\pm}^{-1}\Omega_jD_{\pm}^{-1}\Xi_j\Big{\}}\nonumber\\
&\hspace{-10ex}\times \bigg{(}1+\sum_{\substack{\vec{\ell}\in \llbracket0,4\rrbracket^W,\\ \text{s.t.} ||\vec{\ell}||_1\geq 1}} M^{-\frac{||\vec{\ell}||_1}{2}}\sum_{\vec{\boldsymbol{\alpha}},\vec{\boldsymbol{\beta}}\in \llbracket1,4\rrbracket^{||\vec{\ell}||_1}}\prod_{j=1}^W\mathring{\mathfrak{q}}_{\ell_j, j, \boldsymbol{\alpha}_j,\boldsymbol{\beta}_j}\cdot \prod_{j=1}^W\prod_{i=1}^{\ell_j}\omega_{j,\alpha_{j,i}}\xi_{j,\beta_{j,i}} \bigg{)}, \label{011695}
\end{align}
where we made the convention 
\begin{eqnarray}
\mathring{\mathfrak{q}}_{0,j, \emptyset,\emptyset}=1,\quad \prod_{i=1}^{0}\omega_{j,\alpha_{j,i}}\xi_{j,\beta_{j,i}}=1, \quad \forall\; j=1,\ldots, W. \label{011641}
\end{eqnarray}

According to (\ref{011640}) and (\ref{011641}), we have
\begin{eqnarray}
\prod_{j=1}^W |\mathring{\mathfrak{q}}_{\ell_j, j, \boldsymbol{\alpha}_j,\boldsymbol{\beta}_j}|\leq e^{O(||\vec{\ell}||_1)} \prod_{j=1}^W (1+\mathring{\kappa}_j)^{\ell_j}. \label{011698}
\end{eqnarray}
In addition, we can decompose the sum
\begin{eqnarray}
\sum_{\substack{\vec{\ell}\in \llbracket0,4\rrbracket^W,\\ \text{s. t.} ||\vec{\ell}||_1\geq 1}}=\sum_{\mathfrak{m}=1}^{4W} \sum_{\substack{\vec{\ell}\in \llbracket0,4\rrbracket^W,\\ \text{s. t.} ||\vec{\ell}||_1=\mathfrak{m}}}. \label{011696}
\end{eqnarray}
It is easy to see 
\begin{eqnarray}
\sharp\{\vec{\ell}\in \llbracket0,4\rrbracket^W:  ||\vec{\ell}||_1=\mathfrak{m}\}\leq \binom{4W}{\mathfrak{m}}. \label{011697}
\end{eqnarray}
Moreover, it is obvious that
\begin{eqnarray}
\sharp\{\vec{\boldsymbol{\alpha}},\vec{\boldsymbol{\beta}}\in \llbracket1,4\rrbracket^{||\vec{\ell}||_1}\}=16^{||\vec{\ell}||_1}. \label{012450}
\end{eqnarray}
Therefore, it suffices to investigate the integral
\begin{eqnarray*}
\mathfrak{P}_{\vec{\ell},\vec{\boldsymbol{\alpha}},\vec{\boldsymbol{\beta} }}:=\int {\rm d} \Omega {\rm d} \Xi \exp\Big{\{} -\sum_{j,k}\tilde{\mathfrak{s}}_{jk} Tr  \Omega_j\Xi_k-\sum_{j=1}^W TrD_{\pm}^{-1}\Omega_jD_{\pm}^{-1}\Xi_j\Big{\}}\prod_{j=1}^W\prod_{i=1}^{\ell_j}\omega_{j,\alpha_{j,i}}\xi_{j,\beta_{j,i}}
\end{eqnarray*}
for each combination $(\vec{\ell},\vec{\boldsymbol{\alpha}},\vec{\boldsymbol{\beta}})$, and then sum it up for $(\vec{\ell},\vec{\boldsymbol{\alpha}},\vec{\boldsymbol{\beta}})$ to get the estimate of $\mathsf{P}(\hat{X},\hat{B},V,T)$. 
Specifically, we have the following lemma.
\begin{lem} \label{lem.011601}With the notation above, we have 
\begin{eqnarray}
\mathfrak{P}_{\vec{\ell},\vec{\boldsymbol{\alpha}},\vec{\boldsymbol{\beta} }}=0,\quad  \text{if} \quad ||\vec{\ell}||_1=0\quad  \text{or}\quad  1.\label{020471}
\end{eqnarray}
 Moreover, we have
\begin{eqnarray}
|\mathfrak{P}_{\vec{\ell},\vec{\boldsymbol{\alpha}},\vec{\boldsymbol{\beta} }}|\leq |\det\mathbb{A}_+|^2\det (S^{(1)})^2 \big(||\vec{\ell}||_1-1\big)! (2W^\gamma)^{(||\vec{\ell}||_1-1)}, \quad \text{if}\quad  ||\vec{\ell}||_1\geq 2.\label{011699}
\end{eqnarray}

\end{lem}
We postpone the proof of Lemma \ref{lem.011601} and prove Lemma \ref{lem.011602} at first.
\begin{proof}[Proof of Lemma \ref{lem.011602}] By  (\ref{013116}), (\ref{011695}) and (\ref{020471}) and the fact that $\hat{\mathfrak{p}}_0(\cdot)=O(1)$, we have
\begin{eqnarray}
|\mathsf{P}(\hat{X}, \hat{B}, V, T)|\leq C\sum_{\substack{\vec{\ell}\in \llbracket0,4\rrbracket^W,\\ \text{s. t.} ||\vec{\ell}||_1\geq 2}} M^{-\frac{||\vec{\ell}||_1}{2}}\sum_{\vec{\boldsymbol{\alpha}},\vec{\boldsymbol{\beta}}\in \llbracket1,4\rrbracket^{||\vec{\ell}||_1}} \prod_{j=1}^W |\mathring{\mathfrak{q}}_{\ell_j, j, \boldsymbol{\alpha}_j,\boldsymbol{\beta}_j}|\cdot |\mathfrak{P}_{\vec{\ell},\vec{\boldsymbol{\alpha}},\vec{\boldsymbol{\beta}}}|.\label{0116100}
\end{eqnarray}
Substituting the bounds (\ref{011698}), (\ref{012450}) and (\ref{011699}) into (\ref{0116100}) yields
\begin{eqnarray}
&&|\mathsf{P}(\hat{X}, \hat{B}, V, T)|\leq |\det\mathbb{A}_+|^2\det (S^{(1)})^2\nonumber\\
&&\hspace{10ex}\times\sum_{\substack{\vec{\ell}\in \llbracket0,4\rrbracket^W,\\ \text{s. t.} ||\vec{\ell}||_1\geq 2}} e^{O(||\vec{\ell}||_1)}\cdot M^{-\frac{||\vec{\ell}||_1}{2}}\cdot\big(||\vec{\ell}||_1-1\big)! (2W^\gamma)^{(||\vec{\ell}||_1-1)}\cdot \prod_{j=1}^W (1+\mathring{\kappa}_j)^{\ell_j}. \label{0116105}
\end{eqnarray}

Now, from (\ref{030320}) we have
\begin{eqnarray}
\prod_{j=1}^W (1+\mathring{\kappa}_j)^{\ell_j}\leq \Theta^{||\vec{\ell}||_1}, \label{0116110}
\end{eqnarray}
which can absorb the irrelevant factor $e^{O(||\vec{\ell}||_1)}$. 
Using (\ref{011696}), (\ref{011697}) and (\ref{0116110}),  we have
\begin{align}
&\hspace{-10ex}\sum_{\substack{\vec{\ell}\in \llbracket0,4\rrbracket^W,\\ \text{s.t.} ||\vec{\ell}||_1\geq 2}} e^{O(||\vec{\ell}||_1)}\cdot M^{-\frac{||\vec{\ell}||_1}{2}}\cdot\big(||\vec{\ell}||_1-1\big)! (2W^\gamma)^{(||\vec{\ell}||_1-1)}\cdot \prod_{j=1}^W (1+\mathring{\kappa}_j)^{\ell_j}\nonumber\\
&\leq \sum_{\mathfrak{m}=2}^{4W} \binom{4W}{\mathfrak{m}}\cdot  M^{-\frac{\mathfrak{m}}{2}} \cdot \Theta^{\mathfrak{m}}  \cdot\mathfrak{m}! W^{(\mathfrak{m}-1)\gamma}\nonumber\\
&\leq \sum_{\mathfrak{m}=2}^{4W} (4W)^{\mathfrak{m}} \cdot M^{-\frac{\mathfrak{m}}{2}} \cdot \Theta^{\mathfrak{m}} \cdot W^{(\mathfrak{m}-1)\gamma}=O\Big(\frac{W^{2+\gamma}\Theta^2}{M}\Big), \label{0116150}
\end{align}
where the last step follows from (\ref{021102}) and (\ref{021105}).
Now, substituting (\ref{0116150}) into (\ref{0116105}), we can complete the proof of Lemma \ref{lem.011602}. 
\end{proof}
Hence, what remains is to prove Lemma \ref{lem.011601}. We will need the following technical lemma whose proof is postponed.
\begin{lem} \label{lem.011691}For any index sets $\mathsf{I},\mathsf{J}\subset\{ 1,\ldots, W\}$ with $|\mathsf{I}|=|\mathsf{J}|=\mathfrak{m}\geq 1$, we have the following bounds for the determinants of the submatrices of $S$, $\mathbb{A}_+$ and $\mathbb{A}_-$  defined in (\ref{012510}).
\begin{itemize}
\item For $(\mathbb{A}_+)^{(\mathsf{I}|\mathsf{J})}$ and $(\mathbb{A}_-)^{(\mathsf{I}|\mathsf{J})}$, we have
\begin{eqnarray}
\frac{|\det (\mathbb{A}_+)^{(\mathsf{I}|\mathsf{J})}|}{|\det \mathbb{A}_+|}\leq 1,\quad  \frac{|\det (\mathbb{A}_-)^{(\mathsf{I}|\mathsf{J})}|}{|\det \mathbb{A}_-|}\leq 1. \label{0116300}
\end{eqnarray}
\item For $S^{(\mathsf{I}|\mathsf{J})}$, we have 
\begin{eqnarray}
\frac{|\det S^{(\mathsf{I}|\mathsf{J})}|}{|\det S^{(1)}|}\leq (\mathfrak{m}-1)!(2W^\gamma)^{(\mathfrak{m}-1)}. \label{0116301}
\end{eqnarray}
\end{itemize}
\end{lem}
\begin{proof}[Proof of Lemma \ref{lem.011601}] Recall the definition in (\ref{012110}).
Furthermore, we introduce the matrix
\begin{eqnarray}
\mathbb{H}=(a_+^{-2}\mathbb{A}_+)\oplus S\oplus S\oplus (a_-^{-2}\mathbb{A}_-). \label{020488}
\end{eqnarray}
Using the fact $a_+a_-=-1$, we can write
\begin{eqnarray*}
-\sum_{j,k}\tilde{\mathfrak{s}}_{jk} Tr  \Omega_j\Xi_k-\sum_{j=1}^W TrD_{\pm}^{-1}\Omega_jD_{\pm}^{-1}\Xi_j=-\vec{\Omega}\mathbb{H}\vec{\Xi}'.
\end{eqnarray*}
Now, by the Gaussian integral of the Grassmann variables in (\ref{0129102}), we see that
\begin{eqnarray}
|\mathfrak{P}_{\vec{\ell},\vec{\boldsymbol{\alpha}},\vec{\boldsymbol{\beta} }}|=|\det \mathbb{H}^{(\mathsf{I}|\mathsf{J})}| \label{012501}
\end{eqnarray}
for some index sets $\mathsf{I}, \mathsf{J}\subset \{1,\ldots, 4W\}$ determined by $\vec{\boldsymbol{\alpha}}$ and $\vec{\boldsymbol{\beta}}$ such that
\begin{eqnarray*}
|\mathsf{I}|=|\mathsf{J}|=||\vec{\ell}||_1.
\end{eqnarray*}
Here we  mention that (\ref{012501}) may fail when at least two components in $\boldsymbol{\alpha}_j$ coincide for some $j$. But  $\mathfrak{P}_{\vec{\ell},\vec{\boldsymbol{\alpha}},\vec{\boldsymbol{\beta} }}=0$ in this case because of $\chi^2=0$ for any Grassmann variable $\chi$.

Now, obviously, there exists index sets $\mathsf{I}_\alpha, \mathsf{J}_\alpha\subset \{1,\ldots, W\}$ for $\alpha=1,\ldots, 4$ such that
\begin{eqnarray*}
\mathbb{H}^{(\mathsf{I}|\mathsf{J})}=(a_+^{-2}\mathbb{A}_+)^{(\mathsf{I}_1|\mathsf{J}_1)}\oplus S^{(\mathsf{I}_2| \mathsf{J}_2)}\oplus S^{(\mathsf{I}_3|\mathsf{J}_3)}\oplus (a_-^{-2}\mathbb{A}_-)^{(\mathsf{I}_4|\mathsf{J}_4)},\quad \sum_\alpha|\mathsf{I}_\alpha|=\sum_\alpha|\mathsf{J}_\alpha|=||\vec{\ell}||_1.
\end{eqnarray*}
It suffices to consider the case 
\begin{eqnarray*}
|\mathsf{I}_\alpha|=|\mathsf{J}_\alpha|,\quad \forall\; \alpha=1,2,3,4.
\end{eqnarray*}
Otherwise, $\det \mathbb{H}^{(\mathsf{I}|\mathsf{J})}$ is obviously $0$, in light of the block structure of $\mathbb{H}$, see the definition (\ref{020488}). Now, note that, since $\det S=0$, we have
\begin{eqnarray*}
\det \mathbb{H}^{(\mathsf{I}|\mathsf{J})}=0,\qquad \text{if}\quad ||\vec{\ell}||_1=0, 1.
\end{eqnarray*}
In addition, if $||\vec{\ell}||_1=2$, by using (\ref{011690}) below, one has
\begin{eqnarray*}
|\det \mathbb{H}^{(\mathsf{I}|\mathsf{J})}|=\left\{
\begin{array}{ccc}
|\det\mathbb{A}_+\mathbb{A}_-|\det (S^{(1)})^2,  &\text{if} \quad  |\mathsf{I}_2|=|\mathsf{I}_3|=|\mathsf{J}_2|=|\mathsf{J}_3|=1,\\\\
0,      & \text{otherwise}.
\end{array}
\right.
\end{eqnarray*}
For more general $\vec{\ell}$, by Lemma \ref{lem.011691}, we have
\begin{eqnarray*}
|\det \mathbb{H}^{(\mathsf{I}|\mathsf{J})}|\leq |\det\mathbb{A}_+\mathbb{A}_-|\det (S^{(1)})^2\big(||\vec{\ell}||_1-1\big)! (2W^\gamma)^{(||\vec{\ell}||_1-1)}.
\end{eqnarray*}
Then, by the fact $|\det\mathbb{A}_+\mathbb{A}_-|=|\det \mathbb{A}_+|^2$, we can conclude the proof of Lemma \ref{lem.011601}.
\end{proof}
To prove Lemma \ref{lem.011691}, we will need the following lemma.
\begin{lem} \label{lem.030320}For the weighted Laplacian $S$, we have
\begin{eqnarray}
\det S^{(i|j)}=(-1)^{j-i}\det S^{(i)}, \quad \forall\; i,j=1,\ldots, W \label{011690}
\end{eqnarray}
\end{lem}
\begin{rem}A direct consequence of (\ref{011690}) is 
\begin{eqnarray}
\det S^{(1)}=\ldots=\det S^{(W)}. \label{012730}
\end{eqnarray}
\end{rem}
\begin{proof}[Proof of Lemma \ref{lem.030320}] Without loss of generality, we assume $j>i$ in the sequel.
We introduce the matrices
\begin{eqnarray*}
P_{ij}:=I_{i-1}\oplus\bigg(\begin{array}{ccc}~ &I_{j-i-1}\\
1 &~\end{array}\bigg)\oplus I_{W-j}, \qquad E_{j}:=I-2\mathbf{e}_j\mathbf{e}_j^*-\sum_{\ell\neq j} \mathbf{e}_\ell\mathbf{e}_j^*.
\end{eqnarray*}
It is not difficult to check
\begin{eqnarray}
S^{(i|j)}=S^{(i)} P_{ij} E_j. \label{012710}
\end{eqnarray}
Then, by the fact $\det P_{ij}E_j=(-1)^{j-i}$, we can get the conclusion.
\end{proof}

\begin{proof}[Proof of Lemma \ref{lem.011691}] At first, by the definition in (\ref{012510}), (\ref{0129120}) and the fact $\Re a_+^2=\Re a_-^2>0$, it is easy to see that the singular values of $\mathbb{A}_+$ and $\mathbb{A}_-$ are all larger than $1$. With the aid of the rectangular matrix $(\mathbb{A}_+)^{(\mathsf{I}|\emptyset)}$ as an intermediate matrix, we  can use Cauchy interlacing property twice to see that the $k$-th largest singular value of $(\mathbb{A}_+)^{(\mathsf{I}|\mathsf{J})}$ is always smaller than the $k$-th largest singular value of $\mathbb{A}_+$. Consequently, we have the first inequality of  (\ref{0116300}). In the same manner, we can get the second inequality of (\ref{0116300})

Now, we prove (\ref{0116301}).  At first,  we address the case that $\mathsf{I}\cap \mathsf{J}\neq \emptyset$. In light of (\ref{012730}),  without loss of generality, we assume that $1\in \mathsf{I}\cap \mathsf{J}$. Then $S^{(\mathsf{I}|\mathsf{J})}$ is a submatrix of $S^{(1)}$. Therefore, we can  find two permutation matrices $P$ and $Q$, such that
\begin{eqnarray*}
PS^{(1)}Q=\bigg(\begin{array}{ccccc}
\mathrm{A} &\mathrm{B}\\
\mathrm{C} &\mathrm{D}
\end{array}\bigg),
\end{eqnarray*}
where $\mathrm{D}=S^{(\mathsf{I}|\mathsf{J})}$. Now, by Schur complement, we know that
\begin{eqnarray*}
\frac{|\det S^{(\mathsf{I}|\mathsf{J})}|}{|\det S^{(1)}|}=|\det(\mathrm{A}-\mathrm{B}\mathrm{D}^{-1}\mathrm{C})^{-1}|.
\end{eqnarray*}
Moreover, $(\mathrm{A}-\mathrm{B}\mathrm{D}^{-1}\mathrm{C})^{-1}$ is the $(|\mathsf{I}|-1)$ by $(|\mathsf{I}|-1)$ upper-left corner of 
\begin{eqnarray*}
(PS^{(1)}Q)^{-1}=Q^{-1}(S^{(1)})^{-1}P^{-1}.
\end{eqnarray*}
That means $\det S^{(\mathsf{I}|\mathsf{J})}/\det S^{(1)}$ is the determinant of a sub matrix of $(S^{(1)})^{-1}$ (with dimension $|\mathsf{I}|-1$), up to a sign. Then, by Assumption \ref{assu.1} (iii), we can easily get
\begin{eqnarray*}
|\det S^{(\mathsf{I}|\mathsf{J})}|/|\det S^{(1)}|\leq (|\mathsf{I}|-1)! W^{(|\mathsf{I}|-1)\gamma}.
\end{eqnarray*}

Now, for the case $\mathsf{I}\cap \mathsf{J}=\emptyset$, we can fix one $i\in \mathsf{I}$ and $j\in \mathsf{J}$. Due to (\ref{011690}), it suffices to consider 
\begin{eqnarray}
\frac{\det S^{(\mathsf{I}|\mathsf{J})}}{\det S^{(i|j)}}. \label{010311}
\end{eqnarray}
By similar discussion, one can see that (\ref{010311}) is the determinant of a sub matrix of  $(S^{(i|j)})^{-1}$ with dimension $|\mathsf{I}|-1$. Hence, it suffices to investigate the bound of the entries of $(S^{(i|j)})^{-1}$. From (\ref{012710}) we have
\begin{eqnarray}
(S^{(i|j)})^{-1}=E_j^{-1}P_{ij}^{-1} (S^{(i)})^{-1}. \label{012530}
\end{eqnarray} 
Observe that
\begin{eqnarray*}
P_{ij}^{-1}=I_{i-1}\oplus\bigg(\begin{array}{ccc}~ &1\\
I_{j-i-1} &~\end{array}\bigg)\oplus I_{W-j},\qquad E_j^{-1}=E_j.
\end{eqnarray*}
Then, it is elementary to see that the entries of $(S^{(i|j)})^{-1}$ are bounded by $2W^\gamma$, in light of (\ref{012530}) and Assumption \ref{assu.1} (iii). Consequently, we have
\begin{eqnarray*}
\frac{|\det S^{(\mathsf{I}|\mathsf{J})}|}{|\det S^{(i|j)}|}\leq (|\mathsf{I}|-1)! (2W^\gamma)^{|\mathsf{I}|-1},
\end{eqnarray*}
which implies (\ref{0116301}). Hence, we completed the proof of Lemma \ref{lem.011691}.
\end{proof}
\subsection{ $\mathsf{F}(\hat{X},\hat{B}, V, T)$ in the Type I' vicinity} \label{s.10.2}
Neglecting the $X^{[1]}$, $\mathbf{y}^{[1]}$ and $\mathbf{w}^{[1]}$-variables in $\mathsf{Q}(\cdot)$ at first, we investigate the integral $\mathsf{F}(\hat{X},\hat{B}, V, T)$ in the Type I' vicinity in this section.  We have the following lemma.
\begin{lem} \label{lem.031410}Suppose that the assumptions in Theorem \ref{lem.012802} hold. In the Type I' vicinity, we have
\begin{eqnarray}
\mathsf{F}(\hat{X},\hat{B}, V, T)= O\Big{(}\frac{1}{(N\eta)^{n+2}}\Big{)}. \label{031503}
\end{eqnarray}
\end{lem}

Recalling the functions $\mathbb{G}(\hat{B},T)$ and $\mathbb{F}(\hat{X},V)$ defined in (\ref{120801}) and (\ref{120802}), we further introduce
\begin{eqnarray}
\mathring{\mathbb{G}}(\hat{B},T)=\exp\big\{(a_+-a_-)N\eta\big\} \mathbb{G}(\hat{B},T),\quad  \mathring{\mathbb{F}}(\hat{X},V)=\exp\big\{-(a_+-a_-)N\eta\big\}\mathbb{F}(\hat{X},V). \label{0121255}
\end{eqnarray}
Then, we have the decomposition
\begin{eqnarray}
\mathsf{F}(\hat{X},\hat{B}, V,T)=\mathring{\mathbb{G}}(\hat{B},T) \mathring{\mathbb{F}}(\hat{X},V).  \label{012140}
\end{eqnarray} 
Hence, we can estimate $\mathring{\mathbb{F}}(\hat{X},V)$ and $\mathring{\mathbb{G}}(\hat{B},T)$ separately in the sequel.
\subsubsection{Estimate of $\mathring{\mathbb{F}}(\hat{X},V)$}
We have the following lemma.
\begin{lem} \label{lem.032010}Suppose that the assumptions in Theorem \ref{lem.012802} hold. In the Type I' vicinity, we have
\begin{eqnarray}
\mathring{\mathbb{F}}(\hat{X},V)=O\Big{(}\frac{1}{N\eta}\Big{)}. \label{011805}
\end{eqnarray}
\end{lem}
\begin{proof}
Using (\ref{011401}) and (\ref{011402}), we can write
\begin{eqnarray}
X_j=P_1^*V_j^*\hat{X}_j V_j P_1=P_1^*D_{\pm} P_1+O\Big(\frac{\Theta}{\sqrt{M}}\Big), \label{011905}
\end{eqnarray}
where the remainder term represents a $2\times 2$ matrix whose max-norm is bounded by $\Theta/\sqrt{M}$. Using (\ref{011905}) and recalling $N=MW$ yields
\begin{eqnarray}
\exp\Big{\{} M\eta\sum_{j=1}^W Tr X_jJ\Big{\}}=\exp\Big{\{} N\eta Tr P_1^*D_{\pm}  P_1J\Big{\}}\Big{(}1+O\Big{(}\frac{\Theta N\eta}{\sqrt{M}}\Big{)}\Big{)}. \label{0303905}
\end{eqnarray}
Substituting (\ref{0303905}) into (\ref{012150}) and (\ref{120802}), we can write
\begin{align*}
\mathbb{F}(\hat{X},V)
&=\int {\rm d}  \mu(P_1) {\rm d} X^{[1]} \; \exp\Big{\{} N\eta Tr P_1^*D_{\pm}  P_1J\Big{\}}\nonumber\\
&\hspace{2ex}\times\prod_{k=p,q}\frac{1}{\det^2(X_k^{[1]})} \cdot \exp\Big{\{} \mathbf{i} Tr X_k^{[1]}JZ-\sum_j\tilde{\mathfrak{s}}_{jk} TrX_jX_k^{[1]}J\Big{\}}\nonumber\\
&\hspace{2ex}\times\prod_{k,\ell=p,q}\exp\Big{\{}\frac{\tilde{\mathfrak{s}}_{k\ell}}{2M} TrX_k^{[1]}JX_\ell^{[1]}J\Big{\}}\cdot\Big{(}1+O\Big{(}\frac{\Theta N\eta}{\sqrt{M}}\Big{)}\Big{)}.
\end{align*}

Recalling the parametrization of $P_1$ in (\ref{122708}), we have
\begin{eqnarray*}
Tr P_1^*D_{\pm}  P_1J=(1-2v^2)(a_+-a_-).
\end{eqnarray*}
Consequently, we have
\begin{align*}
\mathring{\mathbb{F}}(\hat{X},V)
&=\int  {\rm d} X^{[1]}\int v{\rm d} v\int \frac{{\rm d} \theta}{\pi} \; \exp\Big{\{} -2(a_+-a_-)N\eta v^2 \Big{\}}\prod_{k=p,q}\frac{1}{\det^2(X_k^{[1]})}\nonumber\\
&\hspace{2ex}\times  \prod_{k=p,q}\exp\Big{\{} \mathbf{i} Tr X_k^{[1]}JZ-\sum_j\tilde{\mathfrak{s}}_{jk} TrX_jX_k^{[1]}J\Big{\}} (1+o(1)).
\end{align*}
Obviously, by the fact that $X^{[1]}$-variables are all bounded and $|\det X_k^{[1]}|=1$ for $k=p,q$, it is easy to see that
\begin{eqnarray*}
|\mathring{\mathbb{F}}(\hat{X},V)|\leq C\int_0^{1} v {\rm d} v\;  \exp\Big{\{} -2(a_+-a_-)N\eta v^2 \Big{\}}=O\Big{(}\frac{1}{N\eta}\Big{)}.
\end{eqnarray*}
Therefore, we completed the proof.
\end{proof}

\subsubsection{Estimate of $\mathring{\mathbb{G}}(\hat{B},T)$} Recall the definition of $\mathring{\mathbb{G}}(\hat{B},T)$ from (\ref{0121255}), (\ref{120801}) and (\ref{0125130}). In this section, we will prove the following lemma.
\begin{lem}\label{lem.01202011} Suppose that the assumptions in Theorem \ref{lem.012802} hold. In the Type I' vicinity, we have
\begin{eqnarray}
\mathring{\mathbb{G}}(\hat{B},T)=O\Big{(}\frac{1}{(N\eta)^{n+1}}\Big{)}. \label{020501}
\end{eqnarray}
\end{lem}
Note that $y_p^{[1]}$, $y_q^{[1]}$ and $t$ in the parametrization of $Q_1$ (see (\ref{122708}) )
are not bounded, we shall truncate them with some appropriate bounds at first, whereby we can neglect some irrelevant terms in the integrand, in order to simplify the integral. More specifically, we will do the truncations 
\begin{eqnarray}
t\leq \frac{1}{(N\eta)^{1/4}} \label{121901}
\end{eqnarray} 
and
\begin{eqnarray}
y_p^{[1]}, y_q^{[1]}\leq (N\eta)^{\frac{1}{8}}. \label{011920}
\end{eqnarray}
Accordingly, we set 
\begin{align}
\widehat{\mathbb{G}}(\hat{B},T)&:=e^{(a_+-a_-)N\eta}\int_{\mathbb{L}}\frac{{\rm d} \sigma}{2\pi}\int_{\mathbb{I}^2}{v}_p^{[1]}{v}_q^{[1]}{\rm d} {v}_p^{[1]}{\rm d} {v}_q^{[1]} \int_0^{(N\eta)^{\frac18}} {\rm d}  y_p^{[1]} \int_0^{(N\eta)^{\frac18}} {\rm d}  y_q^{[1]} \nonumber\\
&\hspace{2ex}\times  \int_{0}^{(N\eta)^{-\frac14}} 2t{\rm d} t \int_{\mathbb{L}^2}{\rm d} \sigma_p^{[1]}{\rm d} \sigma_q^{[1]} \; g(Q_1,T,\hat{B}, \mathbf{y}^{[1]}, \mathbf{w}^{[1]}), \label{0120100}
\end{align}
where we have used the parameterization of $\mathbf{w}^{[1]}$ in (\ref{0204100}). 
We will prove the following lemma.
\begin{lem} \label{lem.011901}Suppose that the assumptions in Theorem \ref{lem.012802} hold. In the Type I' vicinity, we have \begin{eqnarray*}
\mathring{\mathbb{G}}(\hat{B},T)=\widehat{\mathbb{G}}(\hat{B},T)+O(e^{-N^{\varepsilon}})
\end{eqnarray*}
for some positive constant $\varepsilon$.
\end{lem}
\begin{proof}
At first, by (\ref{011901})-(\ref{011903}), we have for any $j$,
\begin{align}
\Re Tr B_jY_k^{[1]}J \geq \frac{y_k^{[1]}}{(s+t)^2}\cdot \frac{\text{min}\{\Re b_{j,1}, \Re b_{j,2}\}}{(s_j+t_j)^2}\geq c\frac{y_k^{[1]}}{1+2t^2},\quad k=p,q,  \label{122007}
\end{align}
for some positive constant $c$, where the last step follows from the facts that $\Re b_{j,1}, \Re b_{j,2}=\Re a_++o(1)$ and  $t_j=o(1)$ in the Type I' vicinity.
In addition, it is not difficult to get 
\begin{eqnarray*}
Tr B_j J=\Big{(}a_+-a_-+O\Big{(}\frac{\Theta}{\sqrt{M}}\Big{)}\Big{)}(1+2t^2),\quad \forall\; j=1,\ldots, W,
\end{eqnarray*} 
which implies that
\begin{eqnarray}
\qquad M\eta\sum_{j=1}^WTr B_j J=(a_+-a_-)N\eta+2\Big{(}a_+-a_-+O\Big{(}\frac{\Theta}{\sqrt{M}}\Big{)}\Big{)}N\eta t^2+O\Big{(}\frac{\Theta N\eta}{\sqrt{M}}\Big{)}.  \label{011910}
\end{eqnarray}

Note that the second and third factors in the definition of $g(\cdot)$ in (\ref{0125130}) can be bounded by $1$,  according to (\ref{0204105}).  Then, as a consequence of (\ref{122007}) and (\ref{011910}), we have
\begin{align}
e^{(a_+-a_-)N\eta}|
g(\cdot)| \leq C(y_p^{[1]}y_q^{[1]})^{n+3}\exp\{-c'N\eta t^2
\}\exp\Big{\{}-c\frac{y_p^{[1]}+y_q^{[1]}}{1+2t^2}\Big{\}},\label{011912}
\end{align}
for some positive constants $C$, $c$ and $c'$. By integrating $y_p^{[1]}$ and $y_q^{[1]}$ out at first, we can easily see that the truncation (\ref{121901}) only produces an error of order $O(\exp\{-N^{\varepsilon}\})$ to the integral $\mathring{\mathbb{G}}(\hat{B},T)$, for some positive constant $\varepsilon=\varepsilon(\varepsilon_2)$ by the assumption $\eta\geq N^{-1+\varepsilon_2}$ in (\ref{021101}). Then one can substitute the bound (\ref{121901}) to the last factor of the r.h.s. of (\ref{011912}), thus
\begin{eqnarray*}
\exp\Big{\{}-c\frac{y_p^{[1]}+y_q^{[1]}}{1+2t^2}\Big{\}}\leq \exp\left\{-\frac{c}{2}(y_p^{[1]}+y_q^{[1]})\right\}.
\end{eqnarray*}
We can also do the truncation (\ref{011920})
in the integral $\mathring{\mathbb{G}}(\hat{B},T)$, up to an error of order $O(\exp\{-N^{\varepsilon}\})$, for some positive constant $\varepsilon$.
Therefore, we completed the proof of Lemma \ref{lem.011901}.
\end{proof}
With the aid of Lemma \ref{lem.011901}, it suffices to work on $\widehat{\mathbb{G}}(\hat{B},T)$ in the sequel. We have the following lemma.
\begin{lem} \label{lem.0120110}We have
\begin{eqnarray*}
\widehat{\mathbb{G}}(\hat{B},T)=O\Big{(}\frac{1}{(N\eta)^{n+1}}\Big{)}.
\end{eqnarray*}
\end{lem}
\begin{proof}[Proof of Lemma \ref{lem.0120110}] 
Recall the parameterization of $\mathbf{w}^{[1]}_k$ in (\ref{0204100}) again. To simplify the notation, we set
\begin{eqnarray*}
\mathfrak{w}_k^{[1]}={u}_k^{[1]}{v}_k^{[1]},\quad k=p,q.
\end{eqnarray*}
Similarly to (\ref{011905}), using $t=o(1)$ from (\ref{121901}), we have the expansion
\begin{eqnarray*}
B_j=Q_1^{-1}T_j^{-1}\hat{B}_jT_jQ_1=Q_1^{-1}D_{\pm} Q_1+O\Big{(}\frac{\Theta}{\sqrt{M}}\Big{)}.
\end{eqnarray*}
Consequently, we have
\begin{eqnarray}
-M\eta \sum_{j=1}^WTrB_j J=-N\eta(a_+-a_-)(1+2t^2)+O\Big{(}\frac{\Theta N\eta}{\sqrt{M}}\Big{)}. \label{011930}
\end{eqnarray}
In addition, for $k=p,q$, using the fact $\sum_{j}\tilde{\mathfrak{s}}_{jk}=1$, we have
\begin{align}
\sum_j\tilde{\mathfrak{s}}_{jk} Tr B_jY_k^{[1]}J &= Tr Q_1^{-1}D_{\pm} Q_1 Y_k^{[1]}J+\frac{\Theta}{\sqrt{M}}TrR_k Y_k^{[1]}\nonumber\\
&= y_k^{[1]}\Big{(}(a_+-a_-)t^2+\big(a_+({u}_k^{[1]})^2-a_-({v}_k^{[1]})^2\big)\Big{)}\nonumber\\
&\hspace{2ex}+y_k^{[1]}\Big{(}(a_+-a_-)\big(e^{-\mathbf{i}(\sigma_k^{[1]}+\sigma)}+e^{\mathbf{i}(\sigma_k^{[1]}+\sigma)}\big)\mathfrak{w}_k^{[1]} s t\Big{)}+\frac{\Theta}{\sqrt{M}}TrR_k Y_k^{[1]}, 
\label{0125110}
\end{align}
where $R_k$ is a $2\times 2$ matrix independent of $Y_k^{[1]}$, satisfying
\begin{eqnarray}
||R_k||_{\max}=O(1). \label{0125131}
\end{eqnarray}
Observe that the  term in (\ref{011930}) is obviously independent of $\mathbf{w}^{[1]}$-variables. In addition, for $k=p$ or $q$, we have
\begin{eqnarray}
\mathbf{i}TrY_k^{[1]}JZ=\big(-\eta +\mathbf{i}E(1-2({v}_k^{[1]})^2)\big)y_k^{[1]}, \label{0125101}
\end{eqnarray}
and for $k,\ell=p$ or $q$, we have
\begin{eqnarray}
TrY_k^{[1]}J Y_{\ell}^{[1]} J=y_{k}^{[1]}y_{\ell}^{[1]}\Big{(}(\mathfrak{w}_{k}^{[1]}\mathfrak{w}_{\ell}^{[1]})^2+\mathfrak{w}_{k}^{[1]}\mathfrak{w}_{\ell}^{[1]}\left(e^{\mathbf{i}(\sigma_k^{[1]}-\sigma_{\ell}^{[1]})}-e^{\mathbf{i}(\sigma_{\ell}^{[1]}-\sigma_k^{[1]})}\right)\Big{)}. \label{0125102}
\end{eqnarray}
 Moreover,  we have
\begin{eqnarray}
 \Big{(}\left(\mathbf{w}^{[1]}_q(\mathbf{w}^{[1]}_q)^*\right)_{12}\left(\mathbf{w}^{[1]}_p(\mathbf{w}^{[1]}_p)^*\right)_{21}\Big{)}^n=\left(\mathfrak{w}_p^{[1]}\mathfrak{w}_q^{[1]}\right)^ne^{\mathbf{i}n(\sigma_p^{[1]}-\sigma_q^{[1]})}. \label{0125103}
\end{eqnarray}

Substituting (\ref{011930}), (\ref{0125110}) and (\ref{0125101})-(\ref{0125103}) to the definition of $g(\cdot)$ in (\ref{0125130}) and  reordering the factors properly, we can write the integrand in (\ref{0120100}) as
\begin{eqnarray}
&&\exp\{(a_+-a_-)N\eta\} g(\cdot)\nonumber\\\nonumber\\
&&=\exp\{{\mathbf{i}n(\sigma_p^{[1]}-\sigma_q^{[1]})}\}
\cdot\exp\Big{\{}-(a_+-a_-)st\sum_{k=p,q}y_k^{[1]}\mathfrak{w}_k^{[1]}\left(e^{-\mathbf{i}(\sigma_k^{[1]}+\sigma)}+e^{\mathbf{i}(\sigma_k^{[1]}+\sigma)}\right)\Big{\}}\nonumber\\
&&\times\exp\Big{\{}-\frac{\Theta}{\sqrt{M}}\sum_{k=p,q} Tr R_k Y_k^{[1]}\Big{\}}\cdot\exp\Big{\{}-\frac{1}{M} \tilde{\mathfrak{s}}_{pq}y_p^{[1]}y_q^{[1]}\mathfrak{w}_p^{[1]}\mathfrak{w}_q^{[1]}\left(e^{\mathbf{i}(\sigma_p^{[1]}-\sigma_q^{[1]})}-e^{\mathbf{i}(\sigma_q^{[1]}-\sigma_p^{[1]})}\right)\Big{\}}\nonumber\\
&&\times\prod_{k=p,q}(y_k^{[1]})^{n+3}(\mathfrak{w}_k^{[1]})^n\cdot\prod_{k,\ell=p,q}\exp\Big{\{}-\frac{1}{2M} \tilde{\mathfrak{s}}_{k\ell}y_k^{[1]}y_\ell^{[1]}\left(\mathfrak{w}_k^{[1]}\mathfrak{w}_\ell^{[1]}\right)^2\Big{\}}\cdot\exp\big\{-2N\eta(a_+-a_-)t^2\big\} \nonumber\\
&&\times\prod_{k=p,q}\exp\Big{\{} -y_k^{[1]}\Big{(}\left(a_+({u}_k^{[1]})^2-a_-({v}_k^{[1]})^2\right)+(a_+-a_-)t^2+\eta-\mathbf{i}E\left(1-2({v}_k^{[1]})^2\right)\Big{)}\Big{\}}\nonumber\\
&&\times\Big{(}1+O\Big{(}\frac{\Theta N\eta}{\sqrt{M}}\Big{)}\Big{)},\nonumber\\
\label{011950}
\end{eqnarray}
where the last factor is independent of the $\mathbf{w}^{[1]}$-variables. Here, we put the factors containing $\sigma_p^{[1]}$ and $\sigma_q^{[1]}$ together, namely, the first two lines on the r.h.s. of (\ref{011950}).

 For further discussion, we write for $k=p,q$
\begin{eqnarray}
Tr R_k Y_k^{[1]}=y_k^{[1]}\big(\mathfrak{r}_{k}^+ e^{\mathbf{i}\sigma_k}+\mathfrak{r}_k^-e^{-\mathbf{i}\sigma_k}+\mathfrak{r}_k\big), \label{0125140}
\end{eqnarray}
where $\mathfrak{r}_{k}^+ $, $\mathfrak{r}_k^-$ and $\mathfrak{r}_k$ are all polynomials of ${u}_k^{[1]}$ and ${v}_k^{[1]}$, with bounded degree and bounded coefficients, in light of (\ref{0125131}), the definition of $Y_k^{[1]}$ in (\ref{0204110}) and the parametrization in (\ref{0204100}). 

Now, we start to estimate the integral (\ref{0120100}) by using (\ref{011950}). We deal with the integral over $\sigma_p^{[1]}$ and $\sigma_q^{[1]}$ at first. These variables are collected in the integral of the form
\begin{align*}
\mathcal{I}_\sigma(\ell_1,\ell_2)&:=\int_{\mathbb{L}^2} {\rm d} \sigma_p^{[1]}{\rm d} \sigma_q^{[1]} \exp\big\{{\mathbf{i}(n+\ell_1)\sigma_p^{[1]}\big\}\exp\big\{-\mathbf{i}(n+\ell_2)\sigma_q^{[1]}}\big\}
\exp\Big{\{}-\frac{\Theta}{\sqrt{M}}\sum_{k=p,q} Tr R_k Y_k^{[1]}\Big{\}}\nonumber\\
&\hspace{5ex}\times \exp\Big{\{}-(a_+-a_-)st\sum_{k=p,q}y_k^{[1]}\mathfrak{w}_k^{[1]}\big(e^{-\mathbf{i}(\sigma_k^{[1]}+\sigma)}+e^{\mathbf{i}(\sigma_k^{[1]}+\sigma)}\big)\Big{\}}\nonumber\\
&\hspace{5ex}\times \exp\Big{\{}-\frac{1}{M} \tilde{\mathfrak{s}}_{pq}y_p^{[1]}y_q^{[1]}\mathfrak{w}_p^{[1]}\mathfrak{w}_q^{[1]}\big(e^{\mathbf{i}(\sigma_p^{[1]}-\sigma_q^{[1]})}-e^{\mathbf{i}(\sigma_q^{[1]}-\sigma_p^{[1]})}\big)\Big{\}}
\end{align*}
with integers $\ell_1$ and $\ell_2$ independent of $n$. Note that according to (\ref{011950}), it suffices to consider $\mathcal{I}_\sigma(0,0)$  for the proof of (\ref{020501}). We study $\mathcal{I}_\sigma(\ell_1,\ell_2)$ for general  $\ell_1$ and $\ell_2$ here, which will be used later.  

Now, we set
\begin{eqnarray}
&& c_{p,q}:=\tilde{\mathfrak{s}}_{pq}y_p^{[1]}y_q^{[1]}\mathfrak{w}_p^{[1]}\mathfrak{w}_q^{[1]},\nonumber\\
&&c_{k,1}:=-(a_+-a_-)st y_k^{[1]}\mathfrak{w}_k^{[1]}e^{-\mathbf{i}\sigma}-\frac{\Theta}{\sqrt{M}} y_k^{[1]}\mathfrak{r}_k^-,\quad k=p,q,\nonumber\\
&&c_{k,2}:=-(a_+-a_-)st y_k^{[1]}\mathfrak{w}_k^{[1]}e^{\mathbf{i}\sigma}-\frac{\Theta}{\sqrt{M}} y_k^{[1]}\mathfrak{r}_k^+,\quad k=p,q.\label{012090}
\end{eqnarray}
In addition, we introduce 
\begin{eqnarray}
d_{p,q}:=y_p^{[1]}y_q^{[1]},\quad d_{k}:=\Big{(}t +\frac{\Theta}{\sqrt{M}}\Big{)}y_k^{[1]},\quad k=p,q. \label{0125160}
\end{eqnarray}
Obviously, when (\ref{121901}) is satisfied, we have
\begin{eqnarray}
c_{p,q}=O(d_{p,q}),\quad c_{k,1}=O(d_{k}),\quad c_{k,2}=O(d_{k}),\quad k=p,q. \label{0120200}
\end{eqnarray}
With the aid of the notation defined in (\ref{0125140}) and (\ref{012090}),  we can write
\begin{align}
\mathcal{I}_\sigma(\ell_1,\ell_2)&=\exp\Big{\{}-\frac{\Theta}{\sqrt{M}}(y_p^{[1]}\mathfrak{r}_p+y_q^{[1]}\mathfrak{r}_q)\Big{\}}\int_{\mathbb{L}^2} {\rm d} \sigma_p{\rm d} \sigma_q \; \exp\left\{\mathbf{i}(n+\ell_1)\sigma_p^{[1]}\right\}      \exp\left\{-\mathbf{i}(n+\ell_2)\sigma_q^{[1]}\right\}
 \nonumber\\
&\times \prod_{k=p,q}\exp\left\{c_{k,1} e^{-\mathbf{i}\sigma_k^{[1]}}+c_{k,2} e^{\mathbf{i}\sigma_k^{[1]}}\right\}
\exp\Big\{-\frac{c_{p,q}}{M}e^{\mathbf{i}(\sigma_p^{[1]}-\sigma_q^{[1]})}+\frac{c_{p,q}}{M}e^{\mathbf{i}(\sigma_q^{[1]}-\sigma_p^{[1]})}\Big\} .\label{020503}
\end{align}
We have the following lemma.
\begin{lem}\label{lem.0120111}Under the truncation (\ref{121901}) and (\ref{011920}), we have
\begin{eqnarray*}
|\mathcal{I}_\sigma(\ell_1,\ell_2)|\leq C\Big{(}\Big{(}\frac{d_{p,q}}{M}\Big{)}^{n+\ell_3}+d_{p}^{2(n+\ell_3)}+d_{q}^{2(n+\ell_3)}\Big{)},\qquad \ell_3:=\frac{\ell_1+\ell_2}{2}
\end{eqnarray*}
for some positive constant $C$.
\end{lem}
\begin{proof}
At first, by Taylor expansion, we have
\begin{align}
&\exp\left\{\mathbf{i}(n+\ell_1)\sigma_p^{[1]}\right\}\exp\left\{-\mathbf{i}(n+\ell_2)\sigma_q^{[1]}\right\}\exp\left\{-\frac{c_{p,q}}{M}e^{\mathbf{i}(\sigma_p^{[1]}-\sigma_q^{[1]})}+\frac{c_{p,q}}{M}e^{\mathbf{i}(\sigma_q^{[1]}-\sigma_p^{[1]})}\right\}\nonumber\\
&=\sum_{n_1,n_2=0}^\infty \frac{(-1)^{n_1}}{(n_1)!(n_2)!} \Big{(}\frac{c_{p,q}}{M}\Big{)}^{n_1+n_2} \exp\left\{\mathbf{i}(n+\ell_1+n_1-n_2)\sigma_p^{[1]}\right\}\exp\left\{-\mathbf{i}(n+\ell_2+n_1-n_2)\sigma_q^{[1]}\right\}.  \label{032022}
\end{align}
Now, for any $m_1,m_2\in\mathbb{Z}$, we denote
\begin{align}
&\tilde{\mathcal{I}}_\sigma(m_1,m_2):=\int_{\mathbb{L}^2} {\rm d} \sigma_p^{[1]}{\rm d} \sigma_q^{[1]}\; \exp\{\mathbf{i}m_1\sigma_p^{[1]}\}\exp\{-\mathbf{i}m_2\sigma_q^{[1]}\}\prod_{k=p,q}\exp\left\{c_{k,1} e^{-\mathbf{i}\sigma_k^{[1]}}+c_{k,2} e^{\mathbf{i}\sigma_k^{[1]}}\right\}\nonumber\\
&\;\;\;=4\pi^2\sum_{n_3=0}^{\infty}\mathbf{1}(n_3+m_1\geq 0)\frac{(c_{p,1})^{n_3+m_1}(c_{p,2})^{n_3}}{n_3!(n_3+m_1)!} \sum_{n_4=0}^{\infty}\mathbf{1}(n_4+m_2\geq 0)\frac{(c_{q,1})^{n_4}(c_{q,2})^{n_4+m_2}}{n_4!(n_4+m_2)!}.
\label{020502}
\end{align}
Setting 
\begin{eqnarray}
m_1:=n+\ell_1+n_1-n_2,\qquad m_2:=n+\ell_2+n_1-n_2, \label{020507}
\end{eqnarray} 
and using (\ref{032022}), we can rewrite (\ref{020503}) as 
\begin{eqnarray}
\qquad \mathcal{I}_\sigma(\ell_1,\ell_2)=\exp\Big{\{}-\frac{\Theta}{\sqrt{M}}\big(y_p^{[1]}\mathfrak{r}_p+y_q^{[1]}\mathfrak{r}_q\big)\Big{\}}\sum_{n_1,n_2=0}^\infty \frac{(-1)^{n_1}}{(n_1)!(n_2)!} \Big{(}\frac{c_{p,q}}{M}\Big{)}^{n_1+n_2}\tilde{\mathcal{I}}_\sigma(m_1,m_2). \label{020505}
\end{eqnarray}

For simplicity, we employ the notation
\begin{eqnarray}
m_3:=m_3(\ell_1,n_1,n_2, n_3)=m_1+n_3, \quad m_4:=m_4(\ell_2, n_1,n_2, n_4)=m_2+n_4. \label{012093}
\end{eqnarray}
Consequently, by (\ref{020505}) and (\ref{020502}) we obtain
\begin{align}
|\mathcal{I}_\sigma(\ell_1,\ell_2)|&\leq 4\pi^2\Big{|}\exp\big\{-\frac{\Theta}{\sqrt{M}}(y_p^{[1]}\mathfrak{r}_p+y_q^{[1]}\mathfrak{r}_q)\big\}\Big{|}\sum_{n_1,n_2=0}^\infty \frac{1}{(n_1)!(n_2)!} \Big{|}\frac{c_{p,q}}{M}\Big{|}^{n_1+n_2}\nonumber\\
&\times \sum_{n_3=0}^{\infty}\mathbf{1}(m_3\geq 0)\frac{|c_{p,1}|^{m_3}|c_{p,2}|^{n_3}}{n_3!m_3!} \cdot\sum_{n_4=0}^{\infty}\mathbf{1}(m_4\geq 0)\frac{|c_{q,1}|^{n_4}|c_{q,2}|^{m_4}}{n_4!m_4!}\nonumber\\
&\leq C\max_{n_1,n_2,n_3,n_4,m_3,m_4\in \mathbb{N}} \Big{|}\Big{(}\frac{c_{p,q}}{M}\Big{)}^{n_1+n_2}(c_{p,1})^{m_3}(c_{p,2})^{n_3} (c_{q,1})^{n_4}(c_{q,2})^{m_4}\Big{|}\nonumber\\
&\leq C \max_{n_1,n_2,n_3,n_4,m_3,m_4\in\mathbb{N}}\Big{|}\Big{(}\frac{c_{p,q}}{M}\Big{)}^{n_1+n_2}(c_{p,1})^{m_3}(c_{q,2})^{m_4}\Big{|}\label{011961}
\end{align}
for some positive constant $C$,
where in the last step we used the fact 
\begin{eqnarray}
|c_{k,1}|<1,\quad |c_{k,2}|< 1, \quad k=p,q\label{012097}
\end{eqnarray}
which can be seen directly from the definition in (\ref{0120200}),  the truncations (\ref{121901}) and (\ref{011920}) and the assumption $\eta\leq M^{-1}N^{\varepsilon_2}$. Analogously, we also have
\begin{eqnarray}
\Big{|}\frac{c_{p,q}}{M}\Big{|}<1. \label{030401}
\end{eqnarray}
According to the definitions (\ref{020507}) and (\ref{012093}), we have
\begin{eqnarray*}
2(n_1+n_2)+m_3+m_4\geq 2n+\ell_1+\ell_2.
\end{eqnarray*}
Hence, by using (\ref{012097}) and (\ref{030401}), we have the trivial bound
\begin{eqnarray*}
\max_{n_1,n_2,n_3,n_4,m_3,m_4\geq 0}\Big{|}\Big{(}\sqrt{\frac{|c_{p,q}|}{M}}\Big{)}^{2(n_1+n_2)}(c_{p,1})^{m_3}(c_{q,2})^{m_4}\Big{|}\leq 
\Big{(}\frac{|c_{p,q}|}{M}\Big{)}^{n+\ell_3}+|c_{p,1}|^{2(n+\ell_3)}+|c_{q,2}|^{2(n+\ell_3)}.
\end{eqnarray*}
Therefore, we completed the proof by using (\ref{0120200}).
\end{proof}
Now, we return to the proof of Lemma \ref{lem.0120110}.
Using (\ref{011950}) and Lemma \ref{lem.0120111} with $\ell_1=\ell_2=0$ to (\ref{0120100}), and integrating the bounded variables ${v}_p^{[1]}$,  ${v}_q^{[1]}$ and $\sigma$ out, we can get
\begin{align*}
|\widehat{\mathbb{G}}(\hat{B},T)| &\leq C\int_0^{(N\eta)^{\frac18}} {\rm d}  y_p^{[1]} \int_0^{(N\eta)^{\frac18}} {\rm d}  y_q^{[1]} \int_{0}^{(N\eta)^{-\frac14}} 2t{\rm d} t\cdot (d_{p,q})^{n+3} \nonumber\\
&\times \Big{(}\Big{(}\frac{d_{p,q}}{M}\Big{)}^{n}+d_{p}^{2n}+d_{q}^{2n}\Big{)} \exp\big\{-2N\eta(a_+-a_-)t^2\big\}\nonumber\\
&\times  \exp\Big{\{} -\frac{\sqrt{4-E^2}}{2}\sum_{k=p,q} y_k^{[1]}\Big{\}}(1+o(1))
\end{align*}
where the last two factors come from the facts
\begin{align}
&\Big{|}\exp\Big{\{} -\sum_{k=p,q} y_k^{[1]}\Big{(}a_+({u}_k^{[1]})^2-a_-({v}_k^{[1]})^2\Big{)}\Big{\}}\Big{|}= \exp\Big{\{} -\frac{\sqrt{4-E^2}}{2}\sum_{k=p,q} y_k^{[1]}\Big{\}}, \label{030402}\\
&\Big{|}\exp\Big{\{}-\sum_{k=p,q}y_k^{[1]}\Big{(}(a_+-a_-)t^2+\eta-\mathbf{i}E(1-2({v}_k^{[1]})^2)\Big{)}\Big{\}}\; \Big{(}1+O\Big(\frac{\Theta N\eta}{\sqrt{M}}\Big)\Big{)}\Big{|}=1+o(1).\nonumber
\end{align}
In (\ref{030402}) we used the fact $({u}_k^{[1]})^2+({v}_k^{[1]})^2=1$.
Now, we integrate $y_p^{[1]}$ and $y_q^{[1]}$ out. Consequently, by the definition in (\ref{0125160}), we have
\begin{eqnarray*}
|\widehat{\mathbb{G}}(\hat{B},T)| \leq C \int_{0}^{(N\eta)^{-\frac14}} 2t{\rm d} t \; \Big{(}\frac{1}{M^n}+\Big{(}\frac{\Theta}{\sqrt{M}}\Big{)}^{2n}+t^{2n}\Big{)} \exp\big\{-2N\eta(a_+-a_-)t^2\big\}=O\Big{(}\frac{1}{(N\eta)^{n+1}}\Big{)},
\end{eqnarray*}
where in the last step we have used the assumption $\eta\leq M^{-1}N^{\varepsilon_2}$ in (\ref{021101}), Assumption \ref{assu.4}, the definition of $\Theta$ in (\ref{021102}) and the fact $N=MW$. Hence, we completed the proof of Lemma \ref{lem.0120110}.
\end{proof}
Finally, we can prove Lemma \ref{lem.01202011}, and further prove Lemma \ref{lem.031410}.
\begin{proof}[Proof of Lemma \ref{lem.01202011} ] This is a direct consequence of Lemmas \ref{lem.011901} and \ref{lem.0120110}.
\end{proof}
\begin{proof}[Proof of Lemma \ref{lem.031410}] This is a direct consequence of (\ref{012140}), Lemma \ref{lem.032010} and Lemma \ref{lem.01202011}.
\end{proof}
\subsection{Summing up:  Proof of Lemma \ref{lem.012401}}
In this section, we slightly modify the discussions in Sections \ref{s.10.1} and \ref{s.10.2} to prove Lemma \ref{lem.012401}. The combination of Lemmas \ref{lem.011602} and \ref{lem.031410} would directly imply Lemma  \ref{lem.012401} if the $\mathsf{Q}(\cdot)$ factor were not present in the definition of $\mathsf{A}(\cdot)$. Now we should take $\mathsf{Q}(\cdot)$ into account. This argument is similar to the corresponding discussion in Section \ref{s.7.4}.
\begin{proof}[Proof of Lemma \ref{lem.012401}]  At first, we observe that $\kappa_1$, $\kappa_2$ and $\kappa_3$ in (\ref{020101}) are obviously independent of $n$. Then, by the fact $\kappa_1=W^{O(1)}$, it suffices to consider one monomial of the form 
\begin{eqnarray*}
\mathfrak{p}_1\Big{(}t,s, (y^{[1]}_p)^{-1},(y^{[1]}_q)^{-1}\Big{)}\mathfrak{p}_2\Big{(}\Big{\{}e^{\mathbf{i}\sigma_k^{[1]}}, e^{-\mathbf{i}\sigma_k^{[1]}}\Big{\}}_{k=p,q}\Big{)}\mathfrak{q}\Big{(}\Big{\{}\frac{\omega_{i,a}\xi_{j,b}}{M}\Big{\}}_{\substack{i,j=1,\ldots,W\\a,b=1,\ldots,4}}\Big{)},
\end{eqnarray*}
where the degrees of $\mathfrak{p}_1(\cdot)$, $\mathfrak{p}_2(\cdot)$ and $\mathfrak{q}(\cdot)$ are all $O(1)$, and independent of $n$, in light of the fact $\kappa_3=O(1)$ in (\ref{020101}). Especially, the order of $(y^{[1]}_p)^{-1}$and $(y^{[1]}_q)^{-1}$ are not larger than $2$, which can be easily seen from the definition of $\mathcal{Q}(\cdot)$ in (\ref{012130}). 

Now, we reuse the notation $\mathsf{P}_{\mathfrak{q}}(\hat{X}, \hat{B}, V, T)$ and $\mathsf{F}_{\mathfrak{p}}(\hat{X}, \hat{B}, V, T)$ in (\ref{020511}), by redefining them as
\begin{eqnarray*}
&&\mathsf{P}_{\mathfrak{q}}(\hat{X}, \hat{B}, V, T):=\int {\rm d} \Omega {\rm d}\Xi \; \mathcal{P}(\cdot)\cdot \mathfrak{q}\Big{(}\Big{\{}\frac{\omega_{i,a}\xi_{j,b}}{M}\Big{\}}_{\substack{i,j=1,\ldots,W\\a,b=1,\ldots,4}}\Big{)},\nonumber\\
&&\mathsf{F}_{\mathfrak{p}}(\hat{X}, \hat{B}, V, T):=\int {\rm d} X^{[1]} {\rm d}\mathbf{y}^{[1]} {\rm d}\mathbf{w}^{[1]} {\rm d}\mu(P_1) {\rm d}\nu(Q_1)\; \mathcal{F}(\cdot)\nonumber\\&&\hspace{5ex}\times \mathfrak{p}_1\Big{(}t,s, (y^{[1]}_p)^{-1},(y^{[1]}_q)^{-1}\Big{)}\mathfrak{p}_2\Big{(}\Big{\{}e^{\mathbf{i}\sigma_k^{[1]}}, e^{-\mathbf{i}\sigma_k^{[1]}}\Big{\}}_{k=p,q}\Big{)}.
\end{eqnarray*}
It is easy to check $\mathcal{P}(\cdot)\mathfrak{q}(\cdot)$ also has an expansion of the form in (\ref{011695}). Hence, $\mathsf{P}_{\mathfrak{q}}(\cdot)$ can be bounded in the same way as $\mathsf{P}(\cdot)$, thus we have
\begin{eqnarray*}
|\mathsf{P}_{\mathfrak{q}}(\hat{X}, \hat{B}, V, T)|\leq \frac{W^{2+\gamma}\Theta^{2}}{M}|\det\mathbb{A}_+|^2\det (S^{(1)})^2.
\end{eqnarray*} 
For $\mathsf{F}_{\mathfrak{p}}(\cdot)$, the main modification is to use Lemma \ref{lem.0120111} with general $\ell_1$ and $\ell_2$ independent of $n$, owing to the function $\mathfrak{p}_2(\cdot)$. In addition, by the truncations in (\ref{121901}) and (\ref{011920}), we can bound $\mathfrak{p}_1(\cdot)$ by some constant $C$. Hence, it suffices to replace $n$ by $n+\ell_3$ in the proof of Lemma \ref{lem.0120110}. Finally, we can get
\begin{eqnarray*}
\mathsf{F}_{\mathfrak{p}}(\hat{X}, \hat{B}, V, T)=O\left(\frac{1}{(N\eta)^{n+\ell_3}}\right),
\end{eqnarray*}
with some finite integer $\ell_3$ independent of $n$. Consequently, we completed the proof of Lemma \ref{lem.012401}.
\end{proof}
\section{Integral over the Type II and III vicinities} \label{s.10} In this section, we prove Lemma \ref{lem.122801}.  We only present the discussion for $\mathcal{I}(\Upsilon^b_+, \Upsilon^b_-, \Upsilon^x_+, \Upsilon^x_+, \Upsilon_S,\mathbb{I}^{W-1})$, i.e. integral over the Type II vicinity . The discussion on  $\mathcal{I}(\Upsilon^b_+, \Upsilon^b_-, \Upsilon^x_-, \Upsilon^x_-, \Upsilon_S,\mathbb{I}^{W-1})$ is analogous. We start from  (\ref{011604}). In this section, we will regard $V$-variables as fixed parameters, and consider the integrand as a function of all the other variables.

 Similarly, we shall provide an estimate for the integrand. At first, under the parameterization (\ref{011615}) with $\varkappa=+$, we see that
\begin{eqnarray*}
x_{j,1}-x_{j,2}=\frac{\mathbf{i}a_+}{\sqrt{M}}(\mathring{x}_{j,1}-\mathring{x}_{j,2})+O\Big{(}\frac{\Theta}{M}\Big{)},\qquad b_{j,1}+b_{j,2}=a_+-a_-+O\Big{(}\frac{\Theta^{\frac12}}{\sqrt{M}}\Big{)}.
\end{eqnarray*}
Consequently, we have
\begin{align}
\prod_{j=1}^W(x_{j,1}-x_{j,2})^2(b_{j,1}+b_{j,2})^2&=\frac{(-a_+^2)^W}{M^W}(a_+-a_-)^{2W}\Big(1+O\Big{(}\frac{\Theta^{\frac32}}{\sqrt{M}}\Big{)}\Big)\nonumber\\
&\times\prod_{j=1}^W\Big{(}\mathring{x}_{j,1}-\mathring{x}_{j,2}+O\Big{(}\frac{\Theta}{\sqrt{M}}\Big{)}\Big{)}^2. \label{012167}
\end{align}

Now, analogously to the case of Type I vicinity, what remains is to estimate $\mathsf{A}(\hat{X}, \hat{B}, V, T)$.  Our aim is to prove the following lemma.
\begin{lem} \label{lem.012601} Suppose that the assumptions in Theorem \ref{lem.012802} hold. In the Type II vicinity, we have
\begin{eqnarray} 
|\mathsf{A}(\hat{X}, \hat{B}, V, T)|\leq e^{-cN\eta} |\det\mathbb{A}_+|^2\det (S^{(1)})^2 \label{012610}
\end{eqnarray}
for some positive constant $c$.
\end{lem}
With the aid of (\ref{012167}) and Lemma \ref{lem.012601}, we can prove Lemma \ref{lem.122801}.
\begin{proof}[Proof of Lemma \ref{lem.122801}] Recall (\ref{011604}). At first, by the definition of $\mathbb{A}_+^v$ in (\ref{012605}), (\ref{0129997}) and the fact $\Re a_+^2>0$, we can see that
\begin{eqnarray}
\Re(\mathring{\mathbf{x}}'\mathbb{A}_+^v\mathring{\mathbf{x}})\geq ||\mathring{\mathbf{x}}||_2^2 \label{012611}
\end{eqnarray}
for all $\{V_j\}_{j=2}^W\in (\mathring{U}(2))^{W-1}$. Substituting (\ref{011537}), (\ref{012167}), (\ref{012610}), (\ref{012611}) and the estimates in Proposition \ref{pro.020401} into (\ref{011604}) yields
\begin{eqnarray*}
&&| \mathcal{I}(\Upsilon^b_+, \Upsilon^b_-, \Upsilon^x_+,\Upsilon^x_+, \Upsilon_S,\mathbb{I}^{W-1})|\nonumber\\
 &&\leq e^{-cN\eta}\cdot\frac{(a_+-a_-)^{2W}}{8^W\pi^{3W-1}}\cdot |\det S^{(1)}|^2\cdot |\det \mathbb{A}_+|\cdot \int_{\mathbb{L}^{W-1}} \prod_{j=2}^W\frac{{\rm d} \theta_j}{2\pi}\int_{\mathbb{I}^{W-1}} \prod_{j=2}^W 2v_j {\rm d}  v_j\nonumber\\
 &&\times\int_{\mathbb{R}^{W-1}} \prod_{j=2}^W {\rm d} \tau_{j,1}\int_{\mathbb{R}^{W-1}} \prod_{j=2}^W {\rm d} \tau_{j,2}  \int_{\mathbb{R}^W} \prod_{j=1}^W {\rm d}  c_{j,1} \int_{\mathbb{R}^W} \prod_{j=1}^W {\rm d}  c_{j,2} \int_{\mathbb{R}^W} \prod_{j=1}^W {\rm d}  \mathring{x}_{j,1} \int_{\mathbb{R}^W} \prod_{j=1}^W {\rm d}  \mathring{x}_{j,2}  \nonumber\\
&& \times \exp\{(a_+-a_-)^2\boldsymbol{\tau}_1'S^{(1)}\boldsymbol{\tau}_1\}\exp\{(a_+-a_-)^2\boldsymbol{\tau}_2'S^{(1)}\boldsymbol{\tau}_2\}\exp\{-\frac12||\mathring{\mathbf{c}}_{1}||_2^2-\frac12||\mathring{\mathbf{c}}_{2}||_2^2\}\nonumber\\
&&\times \exp\{-\frac12||\mathring{\mathbf{x}}_{1}||_2^2-\frac12||\mathring{\mathbf{x}}_{2}||_2^2\}\prod_{j=1}^W\Big(\mathring{x}_{j,1}-\mathring{x}_{j,2}+O\Big(\frac{\Theta}{\sqrt{M}}\Big)\Big)^2,
\end{eqnarray*}
where we  absorbed several factors by $\exp\{-cN\eta\}$.  We also enlarged the domains to the full ones. Then, using the trivial facts
\begin{eqnarray*}
\int_{\mathbb{L}^{W-1}} \prod_{j=2}^W\frac{{\rm d} \theta_j}{2\pi}\int_{\mathbb{I}^{W-1}} \prod_{j=2}^W 2v_j {\rm d}  v_j=1
\end{eqnarray*}
and performing the Gaussian integral for the remaining variables, we can get
\begin{eqnarray}
| \mathcal{I}(\Upsilon^b_+, \Upsilon^b_-, \Upsilon^x_+,\Upsilon^x_+, \Upsilon_S,\mathbb{I}^{W-1})|\leq C|\det S^{(1)}|\cdot |\det \mathbb{A}_+|\cdot\Big{(}1+O\Big{(}\frac{\Theta}{\sqrt{M}}\Big{)}\Big{)}^{W}. \label{012703}
\end{eqnarray}

Observe that
\begin{eqnarray}
|\det \mathbb{A}_+|\leq |1+a_+^2|^W\leq 2^W. \label{012701}
\end{eqnarray}
Moreover, by Assumption \ref{assu.1} (ii), we see that $|\mathfrak{s}_{ii}|\leq (1-c_0)/2$ for some small positive constant $c_0$. Consequently, since $S^{(1)}$ is negative definite, we have 
\begin{eqnarray}
|\det S^{(1)}|\leq \prod_{i\neq 1} |\mathfrak{s}_{ii}| \leq \Big{(}\frac{1-c_0}{2}\Big{)}^W \label{012702}
\end{eqnarray}
by Hadamard's inequality.
Substituting (\ref{012701}) and (\ref{012702}) into (\ref{012703}) yields
\begin{eqnarray}
| \mathcal{I}(\Upsilon^b_+, \Upsilon^b_-, \Upsilon^x_+,\Upsilon^x_+, \Upsilon_S,\mathbb{I}^{W-1})|=O(e^{-c W}) \label{031570}
\end{eqnarray}
for some positive constant $\delta$. Hence, we proved the first part of Lemma \ref{lem.122801}. The second part can be proved analogously. 
\end{proof}

In the sequel, we prove Lemma \ref{lem.012601}. We also ignore the factor $\mathsf{Q}(\cdot)$ from the discussion at first.
\subsection{ $\mathsf{P}(\hat{X}, \hat{B}, V, T)$ in the Type II vicinity} \label{s.11.1} Our aim, is to prove the following lemma.
\begin{lem} \label{lem.01210001}Suppose that the assumptions in Theorem \ref{lem.012802} hold. In the Type II vicinity, we have
\begin{eqnarray}
\mathsf{P}(\hat{X},\hat{B}, V, T)\leq \frac{W^{2+\gamma}\Theta^2}{M}|\det\mathbb{A}_+|^2\det (S^{(1)})^2. \label{012161}
\end{eqnarray}
\end{lem} 
\begin{proof} We will follow the strategy in Section \ref{s.10.1}.  We regard all $V$-variables as fixed parameters. Now, we define the function
\begin{eqnarray*}
\mathring{\iota}\equiv\mathring{\iota}_j(\hat{X},\hat{B},T):=|\mathring{x}_{j,1}|+|\mathring{x}_{j,2}|+|\mathring{b}_{j,1}|+|\mathring{b}_{j,2}|+|\mathring{t}_j|.
\end{eqnarray*}
Then, we recall the representation (\ref{011612}) and the definition of $\Delta_{\ell,j}$ in (\ref{012101}). We still adopt the representation (\ref{012102}). It is easy to see that in the Type II vicinity, we also have the bound (\ref{012103}) for $\mathring{\mathfrak{p}}_{\ell,j,\boldsymbol{\alpha},\boldsymbol{\beta}}$. The main difference is the first factor of the r.h.s. of (\ref{011612}). We expand it around the saddle point as
\begin{eqnarray*}
\exp\Big{\{}-TrV_j^*\hat{X}_j^{-1}V_j \Omega_jT_j^{-1}\hat{B}_j^{-1}T_j\Xi_j\Big{\}}=:\exp\Big{\{}-TrD_{+}^{-1}\Omega_jD_{\pm}^{-1}\Xi_j\Big{\}}\; \exp\Big{\{}-\frac{1}{\sqrt{M}}\widehat{\Delta}_j\Big{\}}.\label{01161300000}
\end{eqnarray*}
We take the formula above as the definition of $\widehat{\Delta}_j$, which is of the form
\begin{eqnarray*}
\widehat{\Delta}_j=\sum_{\alpha,\beta=1}^4 \hat{p}_{j,\alpha,\beta} \cdot \omega_{j,\alpha}\xi_{j,\beta},
\end{eqnarray*}
where $\hat{p}_{j,\alpha,\beta}$ is a function of $\hat{X}$, $\hat{B}$, $V$ and $T$-variables, satisfying
\begin{eqnarray*}
\hat{p}_{j,\alpha,\beta}=O(\mathring{\iota}).
\end{eqnarray*}

Let
\begin{eqnarray*}
\widehat{\mathbb{H}}=(a_+^{-2}\mathbb{A}_+)\oplus S\oplus (a_+^{-2}\mathbb{A}_+)\oplus S.
\end{eqnarray*}
Recalling the notation in (\ref{012110}), we can write
\begin{eqnarray*}
-\sum_{j,k}\tilde{\mathfrak{s}}_{jk} Tr  \Omega_j\Xi_k-\sum_{j=1}^W TrD_{+}^{-1}\Omega_jD_{\pm}^{-1}\Xi_j=-\vec{\Omega}\widehat{\mathbb{H}}\vec{\Xi}'.
\end{eqnarray*}
Now, via replacing $\Delta_{1,j}$ by $\widehat{\Delta}_{1,j}$, $\mathring{\kappa}_j$ by $\mathring{\iota}_j$, $\mathbb{H}$ by $\widehat{\mathbb{H}}$ in the proof of Lemma \ref{lem.011602}, we can perform the proof of Lemma \ref{lem.01210001} in the same way. We leave the details to the reader.
\end{proof}

\subsection{ $\mathsf{F}(\hat{X}, \hat{B}, V, T)$ in the Type II vicinity} \label{s.11.2} In this section, we will prove the following lemma.
\begin{lem} \label{lem.031440}Suppose that the assumptions in Theorem \ref{lem.012802} hold. In the Type II vicinity, we have
\begin{eqnarray}
\mathsf{F}(\hat{X}, \hat{B}, V, T)=O\Big{(}\frac{\exp\{-(a_+-a_-)N\eta\}}{(N\eta)^{n+1}}\Big{)}.\label{012160}
\end{eqnarray}
\end{lem}
\begin{proof}
Recall  the decomposition (\ref{012140}). Note that Lemma \ref{lem.01202011} is still applicable. Hence, it suffices to estimate $\mathring{\mathbb{F}}(\hat{X},V)$.
Now, note that in the Type II vicinity, it s obvious to see that
\begin{eqnarray*}
Tr X_jJ=O\Big{(}\frac{|\mathring{x}_{j,1}|+|\mathring{x}_{j,2}|}{\sqrt{M}}\Big{)},\quad \forall\; j=1,\ldots, W,
\end{eqnarray*}
which implies that
\begin{eqnarray*}
\sum_{j=1}^W Tr X_jJ=O\Big{(}\frac{||\mathring{\mathbf{x}}_{1}||_1+||\mathring{\mathbf{x}}_{2}||_1}{\sqrt{M}}\Big{)}=O\Big{(}\frac{\Theta}{\sqrt{M}}\Big{)}.
\end{eqnarray*}
Consequently, we have
\begin{eqnarray*}
\exp\Big{\{} M\eta\sum_{j=1}^W Tr X_jJ\Big{\}}=\exp\{O(\Theta\sqrt{M}\eta)\}=1+o(1)
\end{eqnarray*}
by our assumption on $\eta$. From (\ref{012150}) we can also see that all the other factors of $f(P_1, V, \hat{X}, X^{[1]})$ are $O(1)$. Hence, by the definition (\ref{0121255}), we have
\begin{eqnarray*}
\mathring{\mathbb{F}}(\hat{X},V)=O(\exp\{-(a_+-a_-)N\eta\}),
\end{eqnarray*}
which together with Lemma \ref{lem.01202011} yields the conclusion.
\end{proof}
\subsection{Summing up:  Proof of Lemma \ref{lem.012601} } Analogously, we shall slightly modify the proofs of Lemma \ref{lem.01210001} and Lemma \ref{lem.031440}, in order to take $\mathsf{Q}(\cdot)$ into account. The proof can then be performed in the same manner as Lemma \ref{lem.012401}. We omit the details here.

\section{Proof of Theorem \ref{lem.012802}} \label{s.12} 
 The conclusion for Case 1 is a direct consequence of the discussions in Sections \ref{s.4.4}--\ref{s.10}. The proofs of Case 2 and Case 3 can be performed analogously, with slight modifications, which will be stated below. \\\\
$\bullet$ (Case 1)\\\\
In this case, by using Lemmas \ref{lem.010601}, \ref{lem.121601}, \ref{lem.010602} and \ref{lem.122801}, we can get (\ref{090102}) immediately.\\\\
$\bullet$ (Case 2) \\\\
In this case, we shall slightly modify the discussions in Sections \ref{s.4.4}-\ref{s.10} for Case 1, according to the decomposition of supermatrices in (\ref{020904}). Now, at first, in (\ref{020907}) and (\ref{020908}),  for  $A=\breve{\mathcal{S}}$, $\breve{{X}}$, $\breve{{Y}}$, $\breve{{\Omega}}$ or $\breve{{\Xi}}$,  we replace $A_p^{\langle1\rangle}$ and $A_q^{\langle1\rangle}$ by $A_p^{\langle1,2\rangle}$ and $A_q$ respectively, and replace $A_q^{[1]}$ by $A_p^{[2]}$. In addition, in the last three lines of (\ref{020908}), we shall also replace $\tilde{s}_{jq}$ by $\tilde{s}_{jp}$, and replace $\tilde{s}_{pq}$ and $\tilde{s}_{qp}$ by $\tilde{s}_{pp}$, and in the first line, we replace $\bar{\phi}_{1,q,1}\phi_{1,p,1}\bar{\phi}_{2,p,1}\phi_{2,q,1}$ by $\bar{\phi}_{1,p,2}\phi_{1,p,1}\bar{\phi}_{2,p,1}\phi_{2,p,2}$. Then, in (\ref{0204110}) and (\ref{0204100}), for 
$A=X$, $Y$, $\Omega$, $\Xi$, $\boldsymbol{\omega}$, $\boldsymbol{\xi}$, $\mathbf{w}$, $y$, $\tilde{u}$, $\tilde{v}$ or $\sigma$, we replace $A_q^{[1]}$ by $A_p^{[2]}$. With these modifications, it is easy to check the proof in Sections \ref{s.4.4}--\ref{s.10} applies to Case 2 as well.  The main point is we can still gain the factor $1/(N\eta)^{n+1}$ from  integral of $g(\cdot)$ defined in (\ref{0125130}) (with $y_q^{[1]}$ and $\mathbf{w}_q^{[1]}$ replaced by $y_p^{[2]}$ and $\mathbf{w}_p^{[2]}$). Heuristically,  we can go back to (\ref{020920}), and replace $\sigma_q^{[1]}$ by $\sigma_p^{[2]}$ therein. It is then quite clear the same estimate holds. Consequently, Lemmas \ref{lem.010601}, \ref{lem.121601}, \ref{lem.010602} and \ref{lem.122801} still hold under the replacement of the variables described above. Hence, (\ref{090102}) holds in Case 2.\\\\
$\bullet$ (Case 3) \\\\
Analogously, in this case, we can also mimic the discussions for Case 1 with slight modifications. We also start from (\ref{020907}) and (\ref{020908}). For $A=\breve{\mathcal{S}}$, $\breve{{X}}$, $\breve{{Y}}$, $\breve{{\Omega}}$, $\breve{{\Xi}}$, $\boldsymbol{\omega}$ and $\boldsymbol{\xi}$,  we replace $A_q^{\langle1\rangle}$ by $A_q$, and replace $A_q^{[1]}$ by $0$.  In addition, in the first line of (\ref{020908}),  we replace $\bar{\phi}_{1,q,1}\phi_{1,p,1}\bar{\phi}_{2,p,1}\phi_{2,q,1}$ by $\bar{\phi}_{1,p,1}\phi_{1,p,1}\bar{\phi}_{2,p,1}\phi_{2,p,1}$. Consequently, after using superbosonization formula, we will get the factor $(y_p^{[1]}|(\mathbf{w}_p^{[1]}(\mathbf{w}_p^{[1]})^*)_{12}|)^{2n}$ instead of $(y_p^{[1]}y_q^{[1]}(\mathbf{w}^{[1]}_q(\mathbf{w}^{[1]}_q)^*)_{12}(\mathbf{w}^{[1]}_p(\mathbf{w}^{[1]}_p)^*)_{21})^n$ in (\ref{090230}). Then, for the superdeterminant terms
\begin{eqnarray*}
\prod_{k=p,q}\frac{\det(X_k-\Omega_k(Y_k)^{-1}\Xi_k)}{\det Y_k},\quad  \prod_{k=p,q}\frac{y_k^{[1]}\Big(y_k^{[1]}-\boldsymbol{\xi}_k^{[1]}(X_k^{[1]})^{-1}\boldsymbol{\omega}_k^{[1]}\Big)^{2}}{\det^2 (X_k^{[1]})}.
\end{eqnarray*}
we shall only keep the factors with $k=p$ and delete those with $k=q$.
Moreover, we shall also replace $A_q^{[1]}$ by $0$ for $A=X$, $Y$, $\Omega$, $\Xi$, $\boldsymbol{\omega}$, $\boldsymbol{\xi}$, $\mathbf{w}$, $y$, $\tilde{u}$, $\tilde{v}$ or $\sigma$ in (\ref{090230}). In addition,  ${\rm d A^{[1]}}$ shall be redefined as the differential of $A_p^{[1]}$-variables only, for $A=X$, $\mathbf{y}$, $\mathbf{w}$, $\boldsymbol{w}$ and $\boldsymbol{\xi}$. One can check step by step that such a modification does not require any essential change of our discussions for Case 1. Especially, note that our modification has nothing to do with the saddle point analysis on the Gaussian measure $\exp\{-M(K(\hat{X},V)+L(\hat{B},T))\}$. Moreover, the term $\mathcal{P}(\cdot)$ in (\ref{012870}) can be redefined  by deleting the factor with $k=q$ in the last term therein. Such a modification does not change our analysis of $\mathcal{P}(\cdot)$. In addition, the irrelevant  term  $\mathcal{Q}(\cdot)$ can also be defined accordingly. Specifically, we shall delete the factor with $k=q$ in the last term of (\ref{122806}) and replace $A_q^{[1]}$ by $0$ for $A=\Omega$, $\Xi$, $\boldsymbol{\omega}$, $\boldsymbol{\xi}$, $\mathbf{w}$, $y$. It is routine to check that Lemma \ref{lem.0202081} still holds under such a modification. Analogously, we can redefine the functions $\mathcal{F}(\cdot)$, $f(\cdot)$ and $g(\cdot)$ in (\ref{012130})-(\ref{0125130}). Now, the main difference between Case 3 and Case 1 or 2 is that the factor $(y_p^{[1]}|(\mathbf{w}_p^{[1]}(\mathbf{w}_p^{[1]})^*)_{12}|)^{2n}$ does not produce oscillation in the integral of $g(\cdot)$ any more. Heuristically, the counterpart of (\ref{020920}) in Case 3 reads 
\begin{align*}
&e^{(a_+-a_-)N\eta}\int d\mathbf{y}^{[1]} {\rm d}\mathbf{w}^{[1]} d\nu(Q_1)\cdot g(\hat{B}, T, Q_1, \mathbf{y}^{[1]},\mathbf{w}^{[1]})\nonumber\\
&\sim \int_0^\infty 2t d t\int_{\mathbb{L}} d\sigma_p^{[1]}  \cdot e^{-cN\eta t^2+c_1e^{-\mathbf{i}\sigma_p^{[1]}}t}\sim \frac{1}{N\eta}.
\end{align*}
Hence, (\ref{090102}) holds for Case 3.\\\\
Therefore, we completed the proof of Theorem \ref{lem.012802}.

\section{Further comments}\label{s.13}
In this section, we make some comments on possible further improvements on our results. \\\\
$\bullet$ (Comment on how to remove the prefactor  $N^{C_0}$ in (\ref{090102}))\\\\
As mentioned in Remark \ref{rem.031501}, we have used $N^{C_0}$ to replace $M\Theta^2W^{C_0}/(N\eta)^{\ell}$. However, the latter is also artificial. It can be improved to some $n$-dependent constant $C_n$ via a more delicate analysis on $\mathsf{A}(\cdot)$, i.e. the integral of $\mathcal{P}(\cdot)\mathcal{Q}(\cdot)\mathcal{F}(\cdot)$. Such an improvement stems from the cancellation in the Gaussian integral. At first, a finer analysis will show that the factor $\mathcal{Q}(\cdot)$ can really be ignored, in the sense that it does not play any role in the estimate of the order of $\mathbb{E}|G_{ij}(z)|^{2n}$. Hence, for simplicity, we just focus on the product $\mathsf{P}(\cdot)\mathsf{F}(\cdot)$ instead of $\mathsf{A}(\cdot)$. Then, we go back to Lemma \ref{lem.011602} and Lemma \ref{lem.031410}. Recall the decomposition (\ref{012140}). A more careful analysis on $\mathsf{F}(\cdot)$ leads us to the following expansion, up to the subleading order terms of the factors  $\mathring{\mathbb{G}}(\cdot)$ and $\mathring{\mathbb{F}}(\cdot)$,
\begin{align}
\mathsf{F}(\cdot)=\mathring{\mathbb{G}}(\cdot)\mathring{\mathbb{F}}(\cdot)\sim &\frac{1}{(N\eta)^{n+2}}\Big(1+\frac{M\eta}{\sqrt{M}}\sum_{j=1}^W \mathsf{l}_j(\mathring{x}_{j,1},\mathring{x}_{j,2},\mathring{v}_j)+\cdots\Big)\nonumber\\
&\times \Big(1+\frac{M\eta}{\sqrt{M}}\sum_{j=1}^W \mathsf{l}'_j(\mathring{b}_{j,1},\mathring{b}_{j,2}, \mathring{t}_j)+\cdots\Big), \label{031550}
\end{align}
where $\mathsf{l}_j(\cdot)$'s and $\mathsf{l}'_j(\cdot)$'s are some linear combinations of the arguments. Analogously, we shall write down the leading order term of $\mathsf{P}(\cdot)$ in terms of $\mathring{\mathbf{x}}$, $\mathring{\mathbf{b}}$, $\mathring{\mathbf{t}}$ and $\mathring{\mathbf{v}}$ explicitly, instead of bounding it crudely by using  (\ref{0116110}). Then it can be seen that the leading order term of  $\mathsf{P}(\cdot)$ is a linear combination of $\mathring{x}_{j,1}\mathring{x}_{k,2}$, $\mathring{b}_{j,1}\mathring{b}_{k,2}$, $\mathring{x}_{j,\alpha}\mathring{b}_{k,\beta}$, $\upsilon_{j,\alpha}\tau_{k,\beta}$ for $j,k=1,\ldots, W$ and $\alpha,\beta=1,2$, in which all the coefficients are of order $1/M$. Observe that the Gaussian integral in (\ref{011603}) will kill the linear terms. Consequently, in the expansion (\ref{031550}), the first term that survives after the Gaussian integral is actually
\begin{eqnarray}
\frac{1}{(N\eta)^{n+2}}\cdot \frac{M\eta}{\sqrt{M}}\sum_{j=1}^W \mathsf{l}_j(\mathring{x}_{j,1},\mathring{x}_{j,2},\mathring{v}_j)\cdot \frac{M\eta}{\sqrt{M}}\sum_{j=1}^W \mathsf{l}'_j(\mathring{b}_{j,1},\mathring{b}_{j,2}, \mathring{t}_j). \label{031566}
\end{eqnarray}
Replacing $\mathsf{A}(\cdot)$ by the product of the leading order term of $\mathsf{P}(\cdot)$ and  (\ref{031566})  in the integral (\ref{011603}) and taking the Gaussian integral over $\mathbf{c}$, $\mathbf{d}$, $\boldsymbol{\tau}$ and $\boldsymbol{\upsilon}$-variables yield the true order $1/(N\eta)^n$, without  additional $N$-dependent prefactors.\\\\
$\bullet$ (Comment on the restriction  $|E|\leq \sqrt{2}-\kappa$)\\\\
This restriction is used in several places, we mention the two most important ones. 

The most critical issue is the term $\ell_S(\hat{B},T)$ defined in (\ref{020401}). A direct consequence of $|E|\leq \sqrt{2}-\kappa$ is that $\Re(b_{j,1}+b_{j,2})(b_{k,1}+b_{k,2})\geq 0$ for $\mathbf{b}_1\in \Gamma^W$ and $\mathbf{b}_2\in\bar{\Gamma}^W$, thus $\Re\ell_S(\hat{B},T)\geq 0$. However, once $|E|>\sqrt{2}$, $\Re\ell_S(\hat{B},T)$ can be negative for $\mathbf{b}_1\in \Gamma^W$ and $\mathbf{b}_2\in\bar{\Gamma}^W$. Consequently, the measure $\exp\{-M\ell_S(\hat{B},T)\}$ is not well defined, considering the domain of the $\mathbf{t}$-variables is not compact. Actually, such a problem is unavoidable if we independently deform the contours of $b_{j,1}$'s (resp. $b_{j,2}$'s) from $\mathbb{R}_+$  to any contour passing through the saddle point $a_+$  (resp. $a_-$), starting from the $0$, since one can always choose $b_{j,1}$ and $b_{k,1}$ to be $a_+$ while $b_{j,2}$ and $b_{k,2}$ to be very close to $0$, such that $\Re(b_{j,1}+b_{j,2})(b_{k,1}+b_{k,2})< 0$ in case $|E|>\sqrt{2}$. A possible way to solve this problem is to change the variables $(b_{j,1}, b_{j,2})$ to $(\mathsf{r}_{j,1},\mathsf{r}_{j,2})$,  defined by $\mathsf{r}_{j,1}=(b_{j,1}+b_{j,2})/2$ and $\mathsf{r}_{j,2}=(b_{j,1}-b_{j,2})/2$, and perform the saddle point analysis with respect to the latter. It is not difficult to calculate that the saddle point of $L(\hat{B},T)$ restricted on these $\mathsf{r}$-variables is $\mathsf{r}_{j,1}=\sqrt{4-E^2}$ and $\mathsf{r}_{j,2}=\mathbf{i}E$. To guarantee the positivity of $\ell_S(\hat{B},T)$, it suffices to choose a contour for $\mathsf{r}_{j,1}$ passing through $\sqrt{4-E^2}$, staying in the sector $\tilde{\mathbb{K}}:=\{\omega\in\mathbb{C}:|\arg\omega|<\pi/4\}$. However,  the explicit way to choose the contours for $\mathsf{r}_{j,1}$ and $\mathsf{r}_{j,2}$ and the subsequent analysis in terms of $\mathsf{r}$-variables instead of $\mathbf{b}$-variables would be more involved and we leave it to future work. 

Another point where $|E|<\sqrt{2}-\kappa$ is used is the proof of Lemma \ref{lem.122801} in Section \ref{s.10}. In (\ref{012611}), we used $\Re \mathbb{A}_+^v\geq I$, which is a consequence of $\Re a_+^2=(2-E^2)/2\geq 0$ and (\ref{0129997}), see the definition of $\mathbb{A}_+^v$ in (\ref{012605}). If $|E|>\sqrt{2}$, the bound in (\ref{012611}) should be weakened to $\Re(\mathring{\mathbf{x}}'\mathbb{A}_+^v\mathring{\mathbf{x}})\geq c(E)||\mathring{\mathbf{x}}||_2^2$ for some $E$-dependent constant $c(E)<1$, which is uniform on $V_j\in \mathring{U}(2)$ for all $j=2,\ldots, W$. This modification, however, is not critical. An easy way to remedy this situation is to impose an additional condition $\det S^{(1)}\leq (c(E)-c)^W$ for arbitrarily small constant $c$, to guarantee the estimate (\ref{031570}).

\end{document}